\documentclass[11pt,leqno]{amsart}
\usepackage[a4paper, margin=1.5in]{geometry}

\usepackage{glaudo}

\title{Solvable Descent and the Grunwald Problem for Solvable Groups}

\author{Julian L.~Demeio}
\address{Julian L.~Demeio \\ 
Institut f\"ur Algebra, Zahlentheorie und Diskrete Mathematik \\
Leibniz Universit\"at Hannover \\
Welfengarten 1  \\
Hannover \\
30167 \\
Germany
}
\email{demeiojulian@yahoo.it}

\begin{document}
	
	\maketitle
	
	\begin{abstract}
		We prove a suitable fibration theorem over quasi-trivial tori that, through an approach developed by Harpaz and Wittenberg, implies so-called {\em solvable descent}. In particular, this gives a positive answer to the Grunwald problem for solvable groups up to the necessary Brauer--Manin obstruction, providing a generalizion of Shafarevich's positive answer to the Inverse Galois Problem for solvable groups. This also provides an alternative proof of Shafarevich's result that avoids his ``shrinking procedure''.
        
        For the fibration theorem, we first adapt the starting ideas of Shafarevich for the creation of local lifts. To deal then with the Brauer--Manin obstruction (i.e.\ the relevant local-to-global obstruction), we compute its ``triple variation'' on grids of fibers. The resulting expression is a linear combination of Red\'ei symbols on the base. Customizing these and employing a combinatorial principle first noted by Alexander Smith in the context of Class and Selmer Groups, one infers the vanishing of the obstruction in at least one fiber.
        
	\end{abstract}

    \setcounter{tocdepth}{1}
    \tableofcontents

    \vspace{-10mm}
	
	\section{Introduction}\label{Sec1}
	We prove a fibration theorem (Theorem \ref{Thm: fibration}) that combined with earlier work of Harpaz and Wittenberg \cite{HW18} proves:
	
	\begin{theorem}\label{Thm: homspaces}
		Let $X$ be a homogeneous space under a linear algebraic semi-simple simply connected group over a number field $K$ with finite solvable stabilizers. Then $X(K)$ is dense in $X(K_{\Omega})^{\Br X}$.
	\end{theorem}

    This answers positively the Grunwald Problem for Solvable Groups up to the necessary Brauer--Manin obstruction, see Corollary \ref{Cor: GrunwaldBM} below. Moreover, by a result of Lucchini-Arteche \cite[Corollary 6.3]{LA19} on the unramified Brauer group of homogeneous spaces with finite stabilizers, Theorem \ref{Thm: homspaces} also implies a positive answer to the (tame) Grunwald problem for finite solvable group schemes: 
	\begin{corollary}\label{Thm: Grunwald}
		Let $G$ be a solvable finite group, and $K$ be a number field. Let $S$ be a finite set of places of $K$ not containing any finite place dividing the order of $G$, and let, for each $v \in S$, $L_v/K_v$ be a Galois extension whose Galois group embeds in $G$. Then there exists a Galois extension $L/K$ with group $G$ such that for every $v \in S$ the extension of $K_v$ obtained by completing $L$ at a place dividing $v$ is isomorphic to $L_v/K_v$.
	\end{corollary}
	
	\medskip
	
	{\bf The Grunwald Problem.} Its history starts with the celebrated Grunwald--Wang theorem, i.e.\  Corollary \ref{Thm: Grunwald} when $G$ is abelian, $K=\Q$ and $S$ contains only the $2$-adic place. The important contribution of Wang \cite{Wang} was to notice that the theorem does not hold with an empty $S$, by proving that no extension of $\Q$ with group $G=\Z/8\Z$ is inert over $2$. Inspired by their theorem, the Grunwald Problem asks the following question for a finite group $G$, and a number field $K$ with local completions $K_v, \ v \in M_K$:
	\begin{center}
		(GP) Does the restriction map $\Hom(\Gamma_K,G)/{\text{conj}} \to \prod_v\Hom(\Gamma_v,G)/{\text{conj}}$ have dense image?
	\end{center}
	Here ${\text{conj}}$ stands for conjugation by elements of $G$, $\Gamma_K \coloneqq \Gal(\oK/K)$ and $\Gamma_v \coloneqq \Gal(\oK_v/K_v)$.
	
	An interesting case where a positive answer to (GP) is known is when there is a positive answer to {\em Noether's problem }for $G$. I.e.\ when, for one (equivalently all) embedding(s) $G \hookrightarrow GL_n(K)$, the quotient $\A^n/G$ is stably rational. This is known to be the case for various $G$, notoriously including $G=S_n, n \in \N$. See e.g.\ \cite{ParkCity} for more on this topic. However, for general $G$, Noether's problem has a negative answer \cite[p.223]{BGbook}, and thus a different approach to (GP) is necessary. 
	
	When $K=\Q$ and $G$ is abelian, the counterexample of Wang shows that the general answer to (GP) is negative, but the Grunwald--Wang theorem indicates that the obstruction to a positive answer should be supported at finitely many places, i.e.\ that the closure of $\Hom(\Gamma_K,G)/{\text{conj}}$ in $\prod_v\Hom(\Gamma_v,G)/{\text{conj}}$ should be a product 
	\[
	(\star) \quad \mathcal A \times \prod_{v \notin S}\Hom(\Gamma_v,G)/{\text{conj}},
	\]	
    where $\mathcal A$ is a non-empty subset of $\prod_{v \in S}\Hom(\Gamma_v,G)/{\text{conj}}$ for some finite set of places $S$ of $K$. The set $\mathcal A$ was later conjecturally linked to the Brauer--Manin obstruction to weak approximation on the space $X=\SL_{n,K}/G$, for an(y) embedding $G \hookrightarrow SL_n(K)$. See Proposition 2.4 of \cite{DLAN17} and the preceding discussion.
	A celebrated conjecture of Colliot-Thélène and Sansuc would then imply a positive answer to the following refined version of the Grunwald Problem:
	\begin{center}
		(GP$_{\text{BM}}$) Is the image of $\Hom(\Gamma_K,G)/{\text{conj}}$ dense in $\left( \prod_v\Hom(\Gamma_v,G)/{\text{conj}}\right)^{BM}$?
	\end{center}
	
	We refer the reader to \cite[Section 3]{PalS} for a precise definition of the ``Brauer--Manin set'' $\left( \prod_v\Hom(\Gamma_v,G)/{\text{conj}}\right)^{BM}$ in the more general context of embedding problems, and we only point out here that it is a set of the form $(\star)$.
	
	Some groups for which a positive answer to (GP) or (GP$_{\text{BM}}$) is known are:
	\begin{itemize}
		\item[(GP$_{\text{BM}}$)] abelian groups $G$ by Wang \cite{Wang};
		\item[(GP$_{\text{BM}}$)] iterated semi-direct products $G=A_1 \rtimes\left(A_2 \rtimes \cdots \rtimes A_r\right)$ of abelian groups by Harari \cite{Harari2007};
		\item[(GP)] solvable groups $G$ with order prime to the number of roots of unity in $K$ by Neukirch \cite{Neukirch};
		\item[(GP$_{\text{BM}}$)] supersolvable groups $G$ by Harpaz and Wittenberg \cite{HW18}.
	\end{itemize}
	
	\medskip
	
	Theorem \ref{Thm: homspaces} implies (GP$_{\text{BM}}$) for solvable groups, and the analogous statement for solvable group schemes (see e.g.\ \cite[Section 9]{PalS} for the implication):
	\begin{corollary}\label{Cor: GrunwaldBM}
		Let $G$ be a finite solvable group scheme over a number field $K$. Then the restriction $H^1(\Gamma_K,G) \to \left(\prod_v H^1(\Gamma_v,G)\right)^{BM}$ has dense image. 
	\end{corollary}
	
	We refer again to \cite[Section 3]{PalS} for a precise definition of the ``Brauer--Manin set'' $\left(\prod_v H^1(\Gamma_v,G)\right)^{BM}$. When $G$ is a constant group scheme, then we have a natural identification $H^1(\Gamma_K,G)=\Hom(\Gamma_K,G)/{\text{conj}}$ (and analog local identifications). Hence Corollary \ref{Cor: GrunwaldBM} gives, as a special case, a positive answer to (GP$_{\text{BM}}$) for solvable groups $G$.
	
	\vskip1mm
	
	Using the descent theory developed by Harpaz and Wittenberg (see \cite{HW18} and Section \ref{Sec3}), Theorem \ref{Thm: homspaces} follows from the following fibration theorem, to whose proof the bulk of this paper is devoted:
	
	\begin{theorem}[Main Theorem]\label{Thm: fibration}
        Let $V \to Q$ be a smooth morphism with rationally connected fibers over a quasi-trivial torus $Q$ defined over a number field $K$. Assume there exists a $K$-torus $T$ endowed with a group-epimorphism $T \to Q$, and an action on $V$ such that the projection $V \to Q$ is $T$-equivariant. Then:
		\[
		V(K_{\Omega})_{\bullet}^{\Br_{ur} V} = \overline{\bigcup_{q \in Q(K)} V_q(K_{\Omega})_{\bullet}^{\Br_{ur} V_q}}.
		\]
	\end{theorem}
	
	A special case of Theorem \ref{Thm: fibration}, when a certain additional splitting assumption is satisfied (see \cite[Théorème 4.2]{HW18}) was already proven by Harpaz and Wittenberg \cite{HW18}. Harpaz and Wittenberg also proved Theorem \ref{Thm: fibration} conditionally under their homonymous conjecture \cite[Conjecture 9.1]{HW14}.

    \vskip1mm

    Theorem \ref{Thm: fibration} easily reduces to the case where $Q$ is simple as a quasi-trivial torus, and in this case we actually prove a slightly stronger variant of the main theorem, Theorem \ref{Thm: fibrationnew}, with a higher degree of archimedean approximation.

    \medskip

	{\bf The fibration method.} We follow the general pattern of the fibration method, working on a relative compactification of the fibration. Generally speaking, given a smooth proper morphism $f:X\to Y$, the ``fibration method'' is a technique to prove that:
	\begin{equation}\label{Eq:fibration}
	X(K_{\Omega})^{\Br_{ur} X} = \overline{\bigcup_{y \in Y(K)} X_y(\A_K)^{\Br X_y}}, \tag{$\star$}
	\end{equation}
    although variants where the set of archimedean points is replaced by the set of their connected components are also common.

	\medskip

	Identity \eqref{Eq:fibration} is equivalent to saying that for any (arbitrarily large) finite set of places $S$ of $K$, and collection of local points $(x_v)_{v \in S} \in X(K_S)^{\Br_{ur} X}$, there exists a $y \in Y(K)$ and a Brauer--Manin unobstructed adelic point $(x'_v)_{v \in M_K}$ of the fiber $X_y$, approximating $x_v$ arbitrarily well for $v \in S$. One may think of this problem as being the combination of the following two parts:
	\begin{enumerate}
		\item ({\em local solubility of the fibers}) finding a $y \in Y(K)$ that approximates $f(x_v), v \in S$ for which $X_y(\A_K) \neq \emptyset$;
		\item ({\em avoiding the Brauer--Manin obstruction}) ensuring that there is a \underline{Brauer--} \underline{Manin unobstructed} adelic point $(x_v(y))_{v \in M_K} \in X_y(\A_K)^{\Br X_y}$, with $y$ as above, that approximates $x_v$ for $v \in S$.
	\end{enumerate}
	
	\medskip

	{\bf Outline of the paper and strategy of the proof of Theorem \ref{Thm: fibration}.}

    \vskip1mm
    
	Section \ref{Sec2} is reserved to settling notation and preliminaries. 

    \vskip1mm
    
    In Section \ref{Sec3}, we reduce the general case of Theorem \ref{Thm: fibration} to the case where $Q=R_{E/K}\G_m$ with $E$ a field and, in this case, we actually formulate a slightly stronger result, Theorem \ref{Thm: fibrationnew}, with a higher degree of archimedean approximation. This stronger result is what we prove in the rest of the paper. 

    \vskip1mm
    
    In Section \ref{SSec3.1}, we also show how solvable descent follows from Theorem \ref{Thm: fibration} via the work of Harpaz and Wittenberg \cite{HW24}. We also use here a base-change result of theirs (see also Theorem \ref{Prop:HWbasechange}) to reduce Theorem \ref{Thm: fibrationnew} to the case where $E/K$ is Galois. Sections \ref{Sec4}-\ref{Sec12} are dedicated to the proof of Theorem \ref{Thm: fibrationnew} in the Galois case.
    
	\vskip1mm
	
	In Section \ref{Sec4} we introduce two kinds of arithmetic symbols that play central roles in the paper: half-spin symbols and (generalized) Redéi symbols. Both arise in some fashion from the Brauer--Manin pairing on the fibers of $f$, as we later show. 
	
	\vskip1mm

	In Section \ref{Sec5} we introduce, in a general setting, the {\em horizontal Brauer group} 
	\[
	\Br_{\text{hor}}(X/Y):=H^2(Y, \tau_{\geq 1}Rf_*\G_m)
	\]
    of a smooth proper morphism $f:X\to Y$ with geometrically integral fibers.
	The main motivation behind this definition is that if $Y$ is a $K$-algebraic variety and $H^1(X_{\bar y},\mu_{\infty})=H^2(X_{\bar y},\cO)=0$ for geometric points $\bar y$ of $Y$, then $\Br_{\text{hor}}(X/Y) \cong \Br X_y/\im \Br K$ for $y$ in a Hilbertian subset $H \subset Y(K)$. I.e., $\Br_{\text{hor}}(X/Y)$ parametrizes the Brauer groups of the fibers modulo constants in a Hilbertian set of parameters. The horizontal Brauer group sits in an exact sequence
	\[
	\Br Y \to[f^*] \Br X \to \Br_{\text{hor}}(X/Y) \to[\partial] H^3(Y,\G_m) \to[f^*] H^3(X,\G_m).
	\]
	The morphism $\Br X \to \Br_{\text{hor}}(X/Y)$ factors through $\Br X_{\eta}$, where $\eta$ denotes here the generic point of $Y$, as $\Br X \subset \Br X_{\eta} \to[r] \Br_{\text{hor}}(X/Y)$. Finally, we also show
    in this section that if $Y$ is a form of a torus, the following is a complex:
    \begin{equation}\label{COmplex}
        \Br X_{\eta} \to[r] \Br_{\text{hor}}(X/Y) \to[\bar \partial] H^3(Y\otimes_K\oK,\G_m)^{\Gamma_K}.
    \end{equation}
    
	In Section \ref{Sec6} we provide a general criterion for finding representatives for the quotient $\Br X_{\eta}/f^*\Br \eta$ in $\Br X_{\eta}$ with prescribed residues. We then apply this criterion to the fibration $f:X \to Q$ of Theorem \ref{Thm: fibrationnew} (this is the relative compactification of the fibration $V\to Q$ of Theorem \ref{Thm: fibration}), in the special case that $Q=R_{E/K}\G_m$ with $E$ a field and $E/K$ Galois.
    We then compute, via functoriality of residues, the local Brauer pairing associated to the representatives just obtained, which reveals (the local components of) half-spin symbols arising, hence their treatment at the beginning of the paper.
	
	\vskip1mm
	
	In Section \ref{Sec7} we use the rational connectedness of the fibers of $f$, combined with the Theorem of Graber--Harris--Starr \cite{GHS}, to produce a finite étale multisection of the shape:
    \[
    s:Q^{L,m,(1)} \to X,\quad Q^{L,m,(1)}:=\Spec L[y^{\pm1},(y^{\sigma})^{\pm\frac1m}]_{\sigma \in \Gamma \s \{\id\}},
    \]
    where $\Gamma =\Gal(E/K)$, and $\{y^{\sigma}\}_{\sigma \in \Gamma}$ is an equivariant basis of characters of $Q$, and $L$ is a finite (field) extension of $E$. See the main text for details.

    \vskip1mm
	
	In Section \ref{Sec8} we reduce the general case of Theorem \ref{Thm: fibrationnew} to the case where some additional assumptions are satisfied: $E/K$ is Galois, the complex \eqref{COmplex} is exact, and the multisection produced in Section \ref{Sec7} satisfies, in addition, that $\mu_m \subset E^*$ and $s^*B=0$ for a finite subset $B\subset (\Br X)_+$ generating the quotient $\Br X_{\eta}/f^*\Br \eta$. The reductions are obtained via the same base--change method of Harpaz and Wittenberg used in Section \ref{SSec3.1}. 

	\vskip1mm
	
	In Section \ref{Sec9}, we follow the ideas of the paper \cite{Shafarevich} of Shafarevich (see also \cite[Theorem IX.9.3.2]{GermanBook} for an English version) to prove an arithmetic approximation result, which, when applied to our setting, produces parameters $q \in Q(K)$ whose images in $Q(K_v)$ at large bad $v$ lift to $K_v$-points of the multisection, and whose associated half-spin symbols vanish. We also produce such parameters $q$ in the shape of grids
    \[
    \cM=q_0 \cdot \{q_{1,1},\ldots,q_{1,M}\}\cdots \{q_{k,1},\ldots,q_{k,M}\} \subset E^* =Q(K),
    \]
    with customazible ``Redéi variations''. Our interest in these Redéi variations is that they appear in the expression for the triple variation of the Brauer--Manin pairing on the fibers associated to $\Br_{\text{hor}}(X/Q)$, as shown later in Section \ref{Sec11}.

	\vskip1mm
	
	In Section \ref{Sec10} we use the approximation result of Section \ref{Sec9} and the multisection $s$ to produce adelic points on the fibers above suitable grids $\cM \subset Q(K)$.
	
	\vskip1mm
	
	In Section \ref{Sec11} we compute the triple variation on suitable grids $\cM$ of the Brauer--Manin pairing associated to an element $b \in \Br_{\text{hor}}(X/Q)$, for the specific choice of adelic points on the fibers made in Section \ref{Sec10}. The resulting expression is a linear combination of the Redéi variations of the grid in the base, with linear weights depending on the class $\bar \partial(b)\in H^3(Q \otimes_K \oK,\G_m)^{\Gamma_K}$. The computation is done using \v{C}ech cochains.
	
	\vskip1mm
	
	In Section \ref{Sec12} we conclude the proof of Theorem \ref{Thm: fibrationnew} under the additional assumptions to which we reduced in Section \ref{Sec8}. Using the results of Section \ref{Sec9} we produce grids above which one may construct adelic points on the respective fibers via the construction in Section \ref{Sec10}, and such that these adelic points are orthogonal to $\im \Br X_{\eta}$.
	The key idea to avoid then the Brauer--Manin obstruction associated to $\Br_{\text{hor}}(X/Q)/\im \Br X_{\eta}$ is to employ the following baby-case of Smith's combinatorics \cite[Section 7]{Smith}:
	
	\begin{lemma}[Smith]\label{Lem:Smith0}
		Let $A$ be a finite abelian group, and $k \in \N$. For every $M \in \N$ larger than some constant $M_0(A)$, there exists a function $g:\{1,\ldots,M-1\}^k \to A$ such that all functions $f:\{1,\ldots,M\}^k\to A$ with $\Delta^{(k)}(f)=g$ are surjective.
	\end{lemma}
	
	Here the {\em $k$-th variation} operator $\Delta^{(k)}$ denotes the composition $\partial_1\circ \cdots \circ \partial_k$, where $(\partial_if)(a_1,\ldots,a_i,\ldots,a_k)\coloneqq f(a_1,\ldots,a_i+1,\ldots,a_k)-f(a_1,\ldots,a_i,\ldots,a_k)$. 

    \vskip1mm

	We apply Smith's lemma following this simple idea: given an ``unkonwn'' function $f: \{1,\ldots, M\}^k \to A$ for which we can customize the $k$-th variation $\Delta^{(k)}$ to be any function we want (for instance, the function $g$), then $f$ can be forced to attain the value $0$ (as it is surjective, it attains in particular the value $0$). In our case, we take for the ``unknown'' function the Brauer--Manin obstruction associated to $\Br_{\text{hor}}(X/Q)/\im \Br X_{\eta}$ on the fibers $X_q$, for $q$ in a suitable grid $\cM \cong \{1,\ldots,M\}^3$ in $Q(K)$. Since, by the results of Section \ref{Sec11}, the triple variation of the Brauer--Manin obstruction is expressible as a combination of Redéi-variations of the grid $\cM$, and these Redéi-variations are customizable by the result of Section \ref{Sec9}, this concludes the proof.

	\vskip1mm
	
	The paper contains three appendices. Appendix \ref{AppA} is dedicated to proving a combination of Chebotarev and Hecke density Theorems, which plays an important role in giving enough archimedean approximation in Section \ref{Sec9} to ensure that the whole argument carries through. In Appendix \ref{AppB} we give our definition of relative étale cohomology (a definition already appears in Friedlander's book \cite{Friedlander}, but this involves the theory of simplicial schemes, that we preferred to avoid). This is used in our computation of Section \ref{Sec11} to express the Brauer--Manin pairing on the fibers in terms of \v{C}ech cochains. In Appendix \ref{AppC} we give a slight generalization of a Theorem of Nakaoka proving, under some technical hypothesis, that the Hochschild--Serre spectral sequence of wreath products splits at the second page. This is used in some technical passages in Section \ref{Sec11}.
	
	\vskip1mm
	
	Finally, it feels important to emphasize that Smith's combinatorics is vastly more general than Lemma \ref{Lem:Smith0}. As a small example, in the same context as the lemma one may infer that $f$ attains all values in the group $A$ with multiplicities that tend to be uniformly equidistributed as $M \to \infty$. See \cite[Section 7]{Smith}, or Section 3 of the work \cite{CKML} on Class Groups by Chen, Koymans, Milovic and Pagano.

    \vskip1mm

    Part of the results of this paper had been previously obtained independently by Harpaz and Wittenberg in an unpublished note, as they told me in private communication. Namely, they had the results of Section \ref{Sec9} and applied them to create local lifts for fibration quasi-trivial tori, i.e.\ the content of Section \ref{Sec10}.

    {
    \acknowledgement{The author is thankful to Jean--Louis Colliot-Th\'el\`ene for his very active interest in the paper, and several useful comments. He is also grateful to Carlo Pagano for his detailed exposition of his work with Koymans for a study group in Basel in the year 2022-2023, as the last part of this paper was inspired by the ``variational'' ideas appearing in this other branch of Number Theory. 
    Finally, he thanks Jessica Alessandrì, Elyes Boughattas, Martin Bright, Adam Morgan, Yonathan Harpaz, Peter Koymans, Jesse Pajwani, Happy Uppal, and Olivier Wittenberg for helpful and interesting private conversations. 
    
    This project was finalised while the author was a guest at the Lodha Mathematical Sciences Institute (LMSI) in Mumbai, to which he is grateful for their hospitality and for offering wonderful working conditions.}
    }

    \bigskip

    \paragraph{Concept map.} In the diagram below, an arrow from $X$ to $Y$ indicates that $X$ is used (either in the current mathematical literature or in this paper) to prove $Y$. Each arrow has up to two labels: the one on the right or bottom of the arrow indicates the ingredients used for the implication, if any (again, either in the current mathematical literature or in this paper); while the one on the left of or above the arrow indicates the section of this paper where the details of the implication are presented, when necessary.
    
    
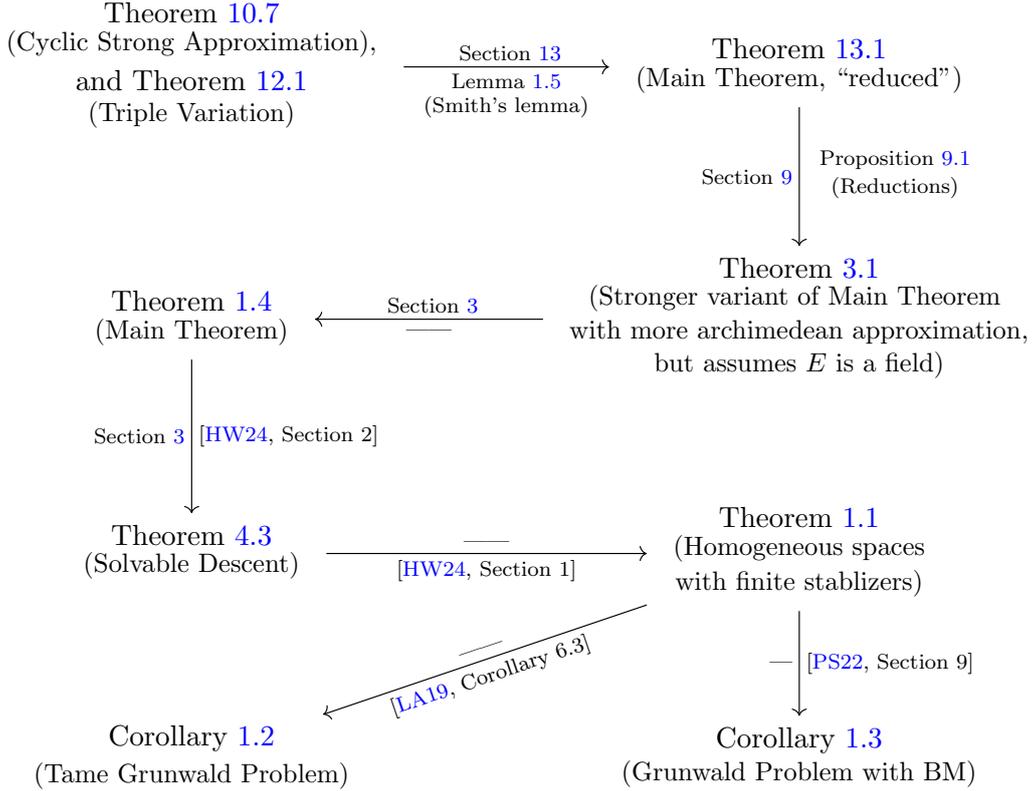
\[
    \hspace{-4mm}\begin{tikzcd}
	\begin{array}{c} {\text{\shortstack{Theorem \ref{MainApprThm} \\ {\small (Cyclic Strong Approximation),} \\[2pt] and  Theorem \ref{Thm:3variation} \\[1pt] {\small (Triple Variation)}}}} \end{array} 
    &&
    \begin{array}{c} {\text{\shortstack{Theorem \ref{Thm: fibrationreduced} \\ {\small (Main Theorem, ``reduced'')}}}} \end{array} 
    \\ \\
	\begin{array}{c} {\text{\shortstack{Theorem \ref{Thm: fibration} \\ {\small (Main Theorem)}}}} \end{array}
    &&
    \begin{array}{c} {\text{\shortstack{Theorem \ref{Thm: fibrationnew}\\  {\small (Stronger variant of Main Theorem } \\ {\small with more archimedean approximation,} \\ {\small but assumes $E$ is a field)}}}} \end{array} 
    \\ \\
	\begin{array}{c} {\text{\shortstack{Theorem \ref{Thm:FAB} \\ {\small (Solvable Descent)}}}} \end{array}
    &&
    \begin{array}{c} {\text{\shortstack{ Theorem \ref{Thm: homspaces} \\ {\small (Homogeneous spaces} \\ {\small with finite stablizers)}}}}\end{array}
    \\ \\
    \begin{array}{c} {\text{\shortstack{Corollary \ref{Thm: Grunwald} \\[1pt] {\small (Tame Grunwald Problem)}}}} \end{array} 
    &&
    \begin{array}{c} {\text{\shortstack{Corollary \ref{Cor: GrunwaldBM} \\ {\small (Grunwald Problem with BM)}}}} \end{array} 
	\arrow["{\begin{array}{c}{{\text{\shortstack{Lemma \ref{Lem:Smith0} \\ (Smith's lemma)}}}}\end{array}}"',"\text{ Section \ref{Sec12}}", from=1-1, to=1-3]
	\arrow["\begin{array}{c} \text{\shortstack{Proposition \ref{PropReductions} \\ (Reductions)}} \end{array}","{\text{ Section \ref{Sec8}}}"', from=1-3, to=3-3]
	\arrow["{{\text{\cite[Section 2]{HW24}}}}","{\text{ Section \ref{Sec3}}}"', from=3-1, to=5-1]
	\arrow["{\text{ Section \ref{Sec3}}}"',"{\text{------}}",from=3-3, to=3-1]
	\arrow["{\text{\cite[Section 1]{HW24}}}"',"{\text{------}}", from=5-1, to=5-3]
    \arrow["{\text{\cite[Corollary 6.3]{LA19}}}"',"{\text{------}}", sloped, from=5-3, to=7-1]
    \arrow["{\text{\cite[Section 9]{PalS}}}","{\text{---}}"', from=5-3, to=7-3]
    \end{tikzcd}\]
        
    \bigskip

	\section{Notation and terminology}\label{Sec2}
    
    \subsection{Arithmetic}
	
	\paragraph{Number fields.} For a number field $K$, we let $\cO_K$ be its ring of integers, $M_K$ be its set of places, and $M_K^{\fin}$ (resp.\ $M_K^{\infty}$) be its set of finite (resp.\ infinite) places. We let $\A_K$ and $\I_K$ the topological groups of adéles $\A_K\coloneqq \prod'_vK_v$ and idéles $\I_K \coloneqq  \prod'_vK_v^*$, where the products are restricted with respect to $\cO_v \subset K_v$ and $\cO_v^* \subset K_v^*$, respectively. We let $K_{\Omega}:=\prod_vK_v$, $K_{\text{fin}}:=\prod_{v \in M_K^{\text{fin}}}K_v$, and $K_{\infty}:=\prod_{v \in M_K^{\infty}} K_v = K \otimes \R$. We let $K_{\text{dir}}:=(K_{\infty}\s \{0\})/\R_{>0}$, where $\R_{>0}$ acts by multiplication via the embedding $\R \subset K \otimes \R=K_{\infty}$.
	
	For $v \in M_K^{\fin}$, we let $\mathcal{P}_{v}$ be the associated prime ideal in $\mathcal{O}_K$. For a prime ideal $\mathcal{P}$  (resp.\ a finite place $v \in M_K^{\fin}$) of $K$, we let $\F_{\mathcal{P}}$ (resp.\ $\F_v$) be the residue field $\mathcal{O}_K/\mathcal{P}$ (resp.\ $\mathcal{O}_K/\mathcal{P}_v$).
	
	For a finite extension $E/K$, we let $\Delta_{E/K}$ be its relative discriminant, meant as an ideal in $\cO_K$.     
	For a finite Galois extension $E/K$, and a prime $\mathcal{P}$ of $E$ not dividing $\Delta_{E/K}$, we denote by $\Frob_{\mathcal{P}} \in \Gal(E/K)$ its Frobenius.
		
	\paragraph{Hilbert symbols.} Let $n$ be a natural number, and $K$ be a number field containing $\mu_n$. For $v \in M_K$, we denote by:
	\[
	(-,-)_{n,v}: K_v^{*}/(K_v^{*})^n \times K_v^{*}/(K_v^{*})^n \to \mu_n
	\]
	the local Hilbert symbol of order $n$. Recall that $(a,b)_{n,v}$ is defined as the image of $(a,b)$ under the cup product:
	\[
	K_v^{*}/(K_v^{*})^n \times K_v^{*}/(K_v^{*})^n \to[\cup] H^2(K_v, \mu_n) \otimes_{\Z} \mu_n \to[\inv_v] \mu_n.
	\]

    \paragraph{Power residues.}
	
	For a prime $\mathcal{P}\subset \cO_K$ and $\alpha \in (\cO_K/\cP)^*$, we let
	\[
	\leg{\alpha}{\mathcal{P}}_n \coloneqq \left( \alpha^{\frac{N\mathcal{P}-1}{n}} \bmod \mathcal{P} \right) \in \mu_n
	\]
	be the $n$-th power residue symbol. For $a \in K^*$ such that $v_{\cP}(a)=0$, we let $\leg{a}{\mathcal{P}}_n:=\leg{a\bmod \cP}{\mathcal{P}}_n$. Recall that $(\pi,a)_{n,\mathcal{P}}=\leg{a}{\mathcal{P}}$ for any uniformizer $\pi$ of $\mathcal{P}$. The $n$-th power residue symbol extends multiplicatively: for an ideal $\mathcal I = \prod_i \cP_i^{e_i}$ of $\cO_K$, we define 
	$$
	\leg{\alpha}{\mathcal{I}}_n \coloneqq \prod_i \leg{\alpha}{\cP_i}^{e_i}.
	$$
	
	We define $\addleg{\alpha}{\cP}_n $ as the image of $\leg{\alpha}{\cP}_n$ under the natural isomorphism $\mu_n \to[\sim] \Z/n\Z(1)$ (and analogously for multiplicative extensions, $S$-variants, ...).\footnote{We shall use this for the sole purpose of having additive instead of multiplicative notation when more convenient.}
	
	\paragraph{Class Field Theory.}
	For a number field $K$, let $C_K:=\I_K/K^*$, where $K^*$ is embedded in $\I_K$ diagonally (recall that this embedding is closed and discrete).
	
	We follow \cite[Section 15.4]{HarariBook}.  
	Let $\m=(I_{\m},S_{\m})$ be a cycle of $K$: i.e.\ $I_{\m} \subset O_K$ is an ideal and $S_\m$ is a subset of the real places of $K$. The {\em ray class group} $\Cl_{\m}O_K$ is (with Chevalley's adelic notation, see e.g.\ \cite[Sec.\ 15.4]{HarariBook}) the quotient:
	\[
	K^{*}\backslash \I_K/ U_{\m},
	\]
	where $U_{\m}:=\prod_{v \in M_K} U_v$, $U_v \coloneqq \cO_v^{*}$ for $v \nmid I_{\m}, v \notin S_\m$, $U_v  \coloneqq  \{a \in \cO_v^{*} : a \equiv 1 \bmod m_v^k\}$ for $v \mid I_{\m}$,  where $m_v^k$ is the maximum power of $m_v$ dividing $I_{\m}$, and $U_v \coloneqq \{a \in K_v^{*} : a >0\}$ for $v \in S_{\m}$. By Class Field Theory, there is a unique abelian extension $H_{\m}/K$ with Galois group $\Cl_{\m}O_K$ 
	such that, for any prime $\cP$ coprime with $I_{\m}$, $\cP$ is unramified in $H_{\m}$ and the Frobenius $\Frob(\cP) \in \Gal(H_{\m}/K) = \Cl_{\m}O_K$ is the equivalence class of the element $\left(1, \ldots, 1, \pi_{\mathcal{P}}, 1, \ldots\right) \in \I_K$, where $\pi_{\mathcal{P}}$ is a uniformizer of the maximal ideal of $\mathcal{O}_{\cP}$.
	
	For a place $v$, an element $x \in K_v^{*}$, and a cycle $\m$, we use the notation 
	\vskip-4mm
	\[
	x \equiv 1 \bmod \m
	\]
	\vskip-3mm
	\noindent to mean that $x \in U_v$.

    \subsection{Geometry and arithmetic geometry}

	\paragraph{Brauer--Manin pairing.}
	For a scheme $X$, we denote by $\Br X \coloneqq H^2(X,\G_m)$ its (cohomological) Brauer group. Recall that, if $X$ is a variety defined over a number field $K$, the {\em Brauer--Manin pairing}
	\[
	(-,-)_{BM}:X(\A_K) \times \Br X \to \qz
	\]
	is defined as $((P_v)_{v \in M_K}, B)_{BM} \coloneqq \sum_v \inv_v (B(P_v))$, where $\inv_v : H^2(K_v, \G_m) \to \qz$ denotes the usual invariant map \cite[Thm 8.9]{HarariBook}. Whenever $B \in  \im \Br K$ or $(P_v) \in X(K)$ (diagonally embedded in $X(\A_K)$), $((P_v),B)_{BM}=0$ by the Albert-Brauer-Hasse-Noether theorem   \cite[Sec. 5]{Skorobogatov}. It follows that $X(K)\subseteq X(\A_K)^{\Br X}$, where, for $\mathcal{B} \subseteq \Br X$:
	\[
	X(\A_K)^{\mathcal{B} } := \{(P_v) \in X(\A_K) \mid ((P_v),B)_{BM}=0 \text{ for all } B \in \mathcal{B} \}.
	\]
	
	When $X$ is a smooth $K$-variety, we denote by $\Br_{ur}X$ its {\em unramified Brauer group}. This is the subgroups of elements of $\Br X$ that are unramified at all places of the function field $K(X)$. (See \cite[Section 6.2]{BGbook} for a definition.) Equivalently, $\Br_{ur}X$ is the Brauer group of a(ny) smooth compactification of $X$. The Brauer--Manin pairing on a(ny) smooth compactification of $X$ induces a {\em continuous} unramified Brauer--Manin pairing:
	\[
	X(K_{\Omega}) \times \Br_{ur} X \to \qz,
	\]
	defined via $((P_v)_{v \in M_K}, B)_{BM} \coloneqq \sum_v \inv_v (B(P_v))$.  As above, $X(K)$ is a subset of $X(K_{\Omega})^{\Br_{ur} X}$, where, for $\mathcal{B} \subseteq \Br_{ur} X$:
	\[
	X(K_{\Omega})^{\mathcal{B} } := \{(P_v) \in X(K_{\Omega}) \mid ((P_v,B))_{BM}=0 \text{ for all } B \in \mathcal{B} \}.
	\]
	
	\paragraph{Quasi-trivial tori.} A {\em quasi-trivial torus} over a field $k$ of characteristic $0$ is a $k$-torus $Q$ that possesses a Galois-equivariant basis $\{y_1,\ldots,y_d\}$ of characters. Such a basis induces an identification:
	\[
	Q = \Spec \left(\ok [y_1^{\pm 1 },\ldots,y_d^{\pm  1}]\right)^{\Gamma_k}.
	\]
	To the $\Gamma_k$-set $\{y_1,\ldots,y_d\}$ corresponds a unique étale algebra $E/k$ of degree $d$. The torus $Q$ is then isomorphic to the Weil-restriction $R_{E/k}\G_m$, and we use the notation $R_{E/k} \G_m$ throughout the paper to denote a quasi-trivial torus whose corresponding étale algebra is $E/k$.

    \paragraph{Quasi-trivial tori over number fields} When $Q$ is as in the paragraph above, and $k=K$ is a number field, the Weil-restriction $R_{\cO_E/\cO_K}\G_m$, where $\cO_E$ denotes the ring of integers in the étale algebra $E$, provides a canonical arithmetic model for $Q$ over $\cO_K$. We denote this model by $\cQ$.
	
	When $E$ is a field, we also let $Q(K_{\infty})_{\text{dir}}$ be the quotient of $Q(K_{\infty})=E_{\infty}^*=\prod_{w \in M_E^{\infty}}E_w^*$ by the image of $\R_{>0}$ under the diagonal embedding $\R_{>0} \hookrightarrow \prod_{w \in M_E^{\infty}}E_w^*$. For a finite set of places $S$ of $K$ containing the archimedean ones, we let $Q(K_S)_{\text{dir}}=\prod_{v \in S \cap M_K^{\text{fin}}} Q(K_v) \times Q(K_{\infty})_{\text{dir}}$.

	\paragraph{Galois quasi-trivial tori.} A quasi-trivial torus $Q$ is {\em Galois quasi-trivial} if $E/k$ is a Galois field extension.  Equivalently, if $\Gamma_k$ acts transitively on $\{y_1,\ldots,y_d\}$ with normal stabilizers. We call $\Gamma=\Gal(E/k)$ the {\em Galois group} of the torus $Q$. In the Galois case, we frequently denote the Galois-equivariant basis of characters by $\{y^{\sigma}\}_{\sigma\in \Gamma}$ and the corresponding torus as $Q = \Spec \left(\ok [(y^{\sigma})^{\pm  1}]_{\sigma \in \Gamma}\right)^{\Gamma_k}.$
	
	\paragraph{Hilbert sets and Hilbertian sets.} Following \cite[Sec.\ 9.5]{LangDio} and \cite[Section 1]{HW14}, given a variety $X$ defined over a field $k$, we say that a subset $H \subset X$ is {\em Hilbert} if there exists a dense open $U \subset X$, and finitely many finite étale irreducible covers $\rho_i:V_i \to U$ such that $H$ is the set of points $x$ of $X$ where all the fibers $\rho_i^{-1}(x)$ are connected. We stress that here $H$ is not a subset of $X(k)$, as it contains, for instance, the generic point of $X$. We say that a subset $H \subset X(k)$ is {\em Hilbertian} if it is the intersection of a Hilbert set with $X(k)$.
	
	\paragraph{Picard functor.} 
	For a base scheme $S$, and an $S$-scheme $f:X \to S$, we indicate by $\Pic_{X/S}$ the {\em relative Picard sheaf} on $S_{\acute{E}t}$, i.e.\  the sheafification of the presheaf $T \mapsto \Pic(X \times_S T)/f_T^{*}\Pic T$ \cite{kleiman}. This sheaf is equal to the derived pushforward $R^1f^{\text{big}}_*\G_m$, where $f^{\text{big}}_*:X_{\Et} \to S_{\Et}$ denotes the pushforward on the big étale site.
	
	Recall that, when $X \to S$ is flat projective with geometrically integral fibers and $S$ is locally noetherian, this sheaf is represented by an $S$-scheme locally of finite type \cite[p. 232-06, Theorem 3.1]{FGA}, which we also denote, with a slight abuse of notation, by $\Pic_{X/S}$.

    \subsection{Derived categories and cohomologies}
	
	\paragraph{Homological algebra.} We follow \cite[Ch.\ III]{Verdier}. We denote by $\Ch(\cA)$ (resp.\ $K(\cA),$ $\cD(\cA),$ $\cD^+(\cA)$) the category of chain complexes in  an abelian category $\cA$ (resp.\ its homotopy category, its derived category, and its bounded below subcategory). For a double complex $(C^{\bullet,\bullet},d_1,d_2)$, we denote by $\Tot^\bullet =\Tot^\bullet (C^{\bullet,\bullet})$ its totalizing complex defined by $\Tot^{n}= \bigoplus_{p+q =n} C^{p,q}, d=\sum_{p+q=n} d_1+ (-1)^p d_2$.

    \vskip1mm
	
	For $A \in \cA$ and $n\in \Z$, we denote by $A[n] \in \Ch(\cA)$ the complex that is $0$ in position different from $n$ and $A$ in the $n$-th position.

    \vskip1mm
	
	For $C^\bullet \in \Ch(\cA)$ and $n\in \Z$, we let $\tau_{\geq n} C^{\bullet}$ be the corresponding {\em truncation} of $C^\bullet$, defined by
	$$
	\left(\tau_{\geq n} C\right)^i := \begin{cases}0 & \text { if } i<n \\ \Coker(C^{n-1} \to C^{n}) & \text { if } \mathrm{i}=n \\ C^i & \text { if } i>n .\end{cases}
	$$
	
	We let also $\tau_{<n} C^\bullet :=C^\bullet  /\left(\tau_{\geq n} C^\bullet \right)$, and $\tau_{[n,m]} C^\bullet  := \tau_{\leq m}\tau_{\geq n} C^\bullet $ for $n \leq m$.
	
	\paragraph{Sites and cohomology.} For a scheme $S$, we denote by $S_{\et}$ its small étale site. The notation $H^\bullet (S,F)$ will always indicate {\em étale cohomology} of the étale sheaf $F/S$.
	
	\paragraph{Coverings and \v{C}ech cochains.} A {\em covering} $\cU$ of a scheme $S$ will always mean an étale covering, i.e.\ a collection $\{\phi_i:U_i \to S\}_{i\in I}$ where each $U_i \to S$ is étale and $\bigcup_i \phi_i(U_i)=S$. A {\em refinement} $\psi:\cV \to \cU$ of coverings $\cV=\{V_j \to S\}_{j \in J},\,\cU=\{\phi_i:U_i \to S\}_{i\in I}$ is the collection of a map $\tau:J \to I$, and $S$-morphisms $\psi_j:V_j \to U_{\tau(j)}$ for all $j$. The {\em category of coverings of $S$} is the category whose objects are coverings and whose morphisms are refinements.
	
	\vskip1mm
	
	The {\em \v{C}ech comple}x $\check{C}^{\bullet}(\cU/S,F)$ of an étale sheaf $F/S$ with respect to a covering $\cU=\{U_i\to S\}_{i \in I}$ is defined by
	\[
	\check{C}^{p}(\cU/S,F) := \prod_{i_0,\ldots,i_p} F(U_{i_0\cdots i_p}), \ \ d(s)_{i_0,\ldots,i_{p+1}} := \sum_{j=0}^{p+1} s_{i_0, \ldots, \widehat{i_j}, \ldots,i_{p+1}}|_{U_{i_0,\ldots,i_{p+1}}},
	\]
	where $U_{i_0,\ldots,i_p}:=U_{i_0} \times_S \cdots \times_S U_{i_p}$.

    \vskip1mm
	
	Given two refinements $\phi,\phi':\cU_1 \rightrightarrows \cU_2$, their associated refinement maps $\phi^*,(\phi')^*:C^{\bullet}(\cU_2/S,P) \to C^{\bullet}(\cU_1/S,P)$ are homotopic. A (natural) homotopy is \cite[p.96]{LECcompleto}:
	\[
	\phi^*-(\phi')^*=\d K_{\phi_1,\phi'} + K_{\phi,\phi'} \d, \quad K_{\phi,\phi'}:C^{\bullet}(\cU_2/S,P) \to C^{\bullet-1}(\cU_1/S,P),
	\]
	\begin{equation}\label{TheHomotopy}
	(K_{\phi,\phi'} s)_{i_0,\ldots,i_{p-1}}=\sum_{0 \leq r \leq p-1}\, (-1)^r s_{\tau(i_0),\ldots,\tau(i_r),\tau'(i_r),\ldots,\tau'(i_{p-1})}|_{U_{i_0,\ldots,i_{p-1}}},
	\end{equation}
	where the restriction denotes pullback along $(\phi_{i_0},\ldots,\phi_{i_r},\phi'_{i_r},\ldots,\phi_{i_{p-1}})$. When $\phi'=\phi$, $K_{\phi,\phi}$ is $2$-homotopically trivially. I.e.\ one has
	\begin{equation}\label{Equation}
	\d K^{(2)}_{\phi} + K^{(2)}_{\phi} \d = K_{\phi,\phi},
	\end{equation}
	for a $2$-homotopy $K^{(2)}_{\phi}:C^{\bullet}(\cU_2/S,P) \to C^{\bullet-2}(\cU_1/S,P)$. Explicitly, one may take
	\[
	(K^{(2)}_{\phi}s)_{i_0,\ldots,i_{p-2}}=\sum_{0 \leq r \leq p-2} \, (-1)^r s_{\tau(i_0),\ldots,\tau(i_r),\tau(i_r),\tau(i_r),\ldots,\tau(i_{p-1})}|_{U_{i_0,\ldots,i_{p-1}}},
	\]
	where the restriction denotes pullback along $(\phi_{i_0},\ldots,\phi_{i_r},\phi_{i_r},\phi_{i_r},\ldots,\phi_{i_{p-2}})$. 
	(We leave the straightforward computation verifying \eqref{Equation} to the interested reader).
	
	\paragraph{Galois coverings.} When $\cU$ consists of a single Galois finite étale cover $S' \to S$, there is a natural identification \cite[Example III.2.6]{LECcompleto}:
	\begin{equation}\label{EqIdentificationCechGroup}
	\check{C}^{\bullet}(S'/S,F) = \check{C}^{\bullet}(G,F(S')) , \, G := \Gal(S'/S).
	\end{equation}
	We stress that this is really an identification of chain complexes and not just of the associated cohomology groups!

    \subsection{Topology}

    \paragraph{Closure.} For a subset $A$ of a topological space, we denote by $\overline{A}$ its closure in said space. 

	\paragraph{Direction.} The {\em direction} of a vector $v \neq 0$ in a real vector space $V$ is its image in $(V\s \{0\})/\R_{>0},$ where $\R_{>0}$ acts by scalar multiplication. 

    \paragraph{Factorization into relative connected components.}
    For a proper continuous map $f:X \to Y$ between locally compact Hausdorff topological spaces, we let $X \to \pi_{0,rel}(X/Y)$ $ \to Y$ be its factorization obtained by collapsing the connected components of the fibers, i.e.\ such that $X \to \pi_{0,rel}(X/Y)$ has connected fibers and $\pi_{0,rel}(X/Y) \to Y$ has totally disconnected fibers. (This is also known as the {\em monotone-light} factorization. See e.g.\ \cite{Bauer}  for a proof of its existence, or \cite[Theorem 1]{Walker} for a statement in English.)

    \paragraph{Covering spaces.} A continuous map $f:X \to Y$ between topological spaces is a {\em covering space} if, locally on $Y$, it decomposes as a direct product $Y \times A$ with $A$ discrete.

    \section{Main Theorem, stronger variant}\label{Sec3}

    We actually prove the following stronger version of Theorem \ref{Thm: fibration} in the case where $E$ is a field. The only improvement is that the theorem below gives a higher degree of archimedean approximation.

    \begin{theorem}\label{Thm: fibrationnew}
        Let $E/K$ be a finite extension of number fields, and $Q=R_{E/K}\G_m$ the corresponding quasi-trivial torus. Let $f:X \to Q$ be a smooth proper morphism with rationally connected fibers. Assume there exists a torus $T$ endowed with a group-epimorphism $T \to Q$, and an action on $X$ such that $f$ is $T$-equivariant. Then:
		\[
		X(K_{\Omega})_{dir}^{\Br_{ur} X} = \overline{\bigcup_{q \in Q(K)} X_q(\A_K)_{\bullet}^{\Br X_q}}.
		\]
    \end{theorem}

    Here $X(K_{\infty})_{dir}:={\pi_{0,rel}}(X(K_{\infty})/Q(K_{\infty}))/\R_{>0}$, where the implied $\R_{>0}$-action on  ${\pi_{0,rel}}(X(K_{\infty})/Q(K_{\infty}))$ is the unique one lifting the diagonal $\R_{>0}$-action on $ \prod_{w \mid \infty}E_w^*=Q(K_{\infty})$. This lift does exist and it is unique because $X(K_{\infty}) \to Q(K_{\infty})$ is a topological fiber bundle with compact smooth fibers by Ehresmann's theorem \cite[Theorem 9.3]{VoisinI} (which applies because $f$ is smooth and proper with a smooth base) and thus ${\pi_{0,rel}}(X(K_{\infty})/Q(K_{\infty})) \to Q(K_{\infty})$ is a finite covering space.
	
    \begin{proposition}Theorem \ref{Thm: fibrationnew} implies Theorem \ref{Thm: fibration}. More specifically:
    \begin{enumerate}[label=(\arabic*)]
        \item If Theorem \ref{Thm: fibrationnew} holds, then Theorem \ref{Thm: fibration} holds when $E$ is a field.
        \item If Theorem \ref{Thm: fibration} holds when $E$ is a field, then  Theorem \ref{Thm: fibration} holds always.
    \end{enumerate}
    \end{proposition}
    
    \begin{proof}[Proof of (1).]
        By Hironaka's strong resolution of singularities \cite{strong_resolution}, we may find a relative compactification $f:X \to Q$ where the $T$-action extends to $X$ and $f$ is $T$-equivariant.
        Take now  $((x_v)_{v \in M_K^{\text{fin}}},x_{\infty}) \in (V(K_{\text{fin}}) \times V(K_{\infty}))^{\Br_{ur} V}=V(K_{\Omega})^{\Br_{ur} V}$. Theorem \ref{Thm: fibrationnew} gives 
        \[
        X(K_{\Omega})_{dir}^{\Br_{ur}(X)} = \overline{\bigcup_{q \in Q(K)}X_q(\A_K)_{\bullet}^{\Br X_q}},
        \]
        so there exists $q \in Q(K)$ and $((x_v(q))_{v \in M_K^{\text{fin}}},x_{\infty}(q)) \in X_q(\A_K)^{\Br X_q}$ whose image in $X(K_{\Omega})_{dir}$ approximates that of $((x_v)_{v \in M_K^{\text{fin}}},x_{\infty})$. For non-archimedean $v$, let $x_v(q)'$ be a point arbitrarily close to $x_v(q)$, but lying in $V_q$. We claim that there exists a point $x_{\infty}(q)' \in V_q(K_{\Omega})$ lying in the same connected component of $V(K_{\Omega})$ as $x_{\infty}$, and lying on the same connected component of $X_q(K_{\Omega})$ as $x_{\infty}(q)$. The claim concludes the proof, as then the point $((x_v(q)')_v,x_{\infty}(q)') \in V_q(K_{\text{fin}}) \times V_q(K_{\infty})= V_q(K_{\Omega}) \subset X_q(\A_K)$ is orthogonal to $\Br X_q=\Br_{ur}V_q$ by continuity of the Brauer--Manin pairing, and thus lies in $V_q(K_{\Omega})^{\Br_{ur}V_q}$ and provides the sought approximation.

        Let us now prove the claim. For an element $x \in X(K_{\infty})$, we denote by $[x]$ its image in $\pi_{0,rel}(X(K_{\infty})/Q(K_{\infty}))$.
        The image of $x_{\infty}(q)$ in $X(K_{\infty})_{dir}$ is close to that of $x_{\infty}$. I.e.\ there exists $t_q \in \R_{>0}$ such that $t_q^{-1} \cdot [x_{\infty}(q)] $ is arbitrarily close to  $[x_{\infty}]$. In particular, projecting to $Q(K_{\infty})$, we have that $q \cdot t_q^{-1}$ is arbitrarily close to $f(x_{\infty})$.
        Thus, by the inverse function theorem, there exists $\tilde x_{\infty} \in X_{t_q^{-1}\cdot q}(K_{\infty})$ arbitrarily close to $x_{\infty}$. 

        Let now $\alpha:\R_{>0} \to T(K_{\infty})$ be a lift (as a homomorphism) of the diagonal map $\iota_{diag}: \R_{>0} \to Q(K_{\infty})$ along the epimorphism $\pi: T(K_{\infty}) \to Q(K_{\infty})$. Such a lift is provided, for instance, by the function $\alpha(t):=\exp_{Lie} (\log(t) \cdot w)$, where $w \in ((\d\pi)_e)^{-1}(v)$ denotes an inverse image along the (surjective) differential  $(\d\pi)_e:T_e(T(K_{\infty})) \to T_e(Q(K_{\infty}))$ of the unique tangent vector $v \in T_e(Q(K_{\infty}))$ at the origin $e$ of $Q(K_{\infty})$ such that $\iota_{diag}(t)=\exp_{Lie}(\log(t) \cdot v)$. Here `` $\exp_{Lie}$'' denotes the exponential map of a Lie group.

        Define $x_{\infty}(q)':=\alpha(t_q) \cdot \tilde x_{\infty} \in X(K_{\infty})$. By construction, this point projects to $q \in Q(K_{\infty})$, and thus lies in $X_q(K_{\infty})$. Moreover, since $x_{\infty}\in V(K_{\infty})$ and $V$ is preserved by the $T$-action, we also have $x_{\infty}(q)' \in V_q(K_{\infty})$, and, since $\alpha(\R_{>0})$ acts trivially on the set of connected components of $V(K_{\infty})$, it lies in the same connected component of $V(K_{\infty})$ as $x_{\infty}$. Finally, note that the action $(t,x)\mapsto \alpha(t) \cdot x$ of $\R_{>0}$ on $X(K_{\infty})$ preserves the connected components of the fibers and thus descends to an action on $\pi_{0,rel}(X(K_{\infty})/Q(K_{\infty}))$. This action lifts the $\R_{>0}$-action on $Q(K_{\infty})$ by construction, and so is the $\R_{>0}$-action with which we had already endowed $\pi_{0,rel}(X(K_{\infty})/Q(K_{\infty}))$.
        
        Hence $t_q^{-1}\cdot [x_{\infty}(q)']= [\tilde x_{\infty}]$, and this is close to $[\tilde x_{\infty}]$ and thus also to $t_q^{-1}\cdot [x_{\infty}(q)]$. So, $t_q^{-1}\cdot [x_{\infty}(q)']$ is close to $t_q^{-1}\cdot [x_{\infty}(q)]$. However, both lie in the same (finite, and thus discrete) fiber of ${\pi_{0,rel}}(X(K_{\infty})/Q(K_{\infty})) \to Q(K_{\infty})$ above $t_q^{-1} \cdot q$, and so in fact $t_q^{-1}\cdot [x_{\infty}(q)']=t_q^{-1}\cdot [x_{\infty}(q)]$ and $[x_{\infty}(q)']=[x_{\infty}(q)]$, i.e.\ $x_{\infty}(q)'$ and $x_{\infty}(q)$ lie on the same connected component of $X_q(K_{\Omega})$, proving the claim as wished.
    \end{proof}

    \begin{proof}[Proof of (2).]
        We proceed by induction on $[E:K]$. Assume $E$ is not a field, write $E=E_1\oplus E_2$ with $[E_1:K], [E_2:K]<[E:K]$, and let $Q=Q_1 \times_K Q_2$ be the corresponding decomposition in quasi-trivial factors. Let $\pi: T \to Q$ be the epimorphism given by our setting, and  $T_i:=\pi^{-1}(Q_i),\, i=1,2$ (these are subtori of $T$). Consider the composition:
		\[
		  V \to Q= Q_1 \times_K Q_2 \to[pr_1] Q_1.
		\]
		We have:
		\begin{itemize}
			\item For each $q_1 \in Q_1(K)$, the fibration $V_{q_1}:=V \times_Q (\{q_1\} \times_K Q_2) \to \{q_1\} \times_K Q_2$ is $T_2$-equivariant and so satisfies \eqref{Eq:fibration2} by the induction hypothesis;
			\item The fibration $V \to Q_1$ is $T_1$-equivariant and so again satisfies \eqref{Eq:fibration2} by the induction hypothesis.
		\end{itemize}
		Combining the above we deduce \eqref{Eq:fibration2} for $V \to Q$, as wished.
    \end{proof}
    
	\section{Harpaz--Wittenberg descent theory}\label{SSec3.1}
	
	In this section we show how to deduce {\em solvable descent} (Theorem \ref{Thm:FAB} below) for rationally connected varieties from the fibration Theorem \ref{Thm: fibration}. We follow \cite{HW24}, where Harpaz and Wittenberg already (implicitly) proved this implication and used it in their proof of {\em supersolvable descent} (see \cite[Theorem 1.4]{HW24}).
	
	\begin{definition}
		Let $X$ be a variety over a field $k$. A {\em finite descent type} is a finite integral cover $Y \to X_{\overline k}$ such that $Y\to X$ is Galois (i.e.\ such that the field extension $\overline k (Y) /k(X)$ is Galois).
	\end{definition}
	
	\begin{definition}
		Let $\bar Y \to X_{\overline k}$ be a finite descent type. A {\em torsor of type }$\bar Y$ is a cover $Y \to X$ such that $Y_{\overline k} \cong \bar Y$ as $X_{\overline k}$-schemes.
	\end{definition}
	
	\begin{theorem}[Solvable Descent]\label{Thm:FAB}
		Let $X$ be a smooth geometrically integral variety over a number field $K$. Let $\bar Y$ be a rationally connected solvable descent type over $X$. Then
		\[
		X\left(K_{\Omega}\right)^{\operatorname{Br}_{\mathrm{ur}}(X)}=\overline{\bigcup_{f: Y \rightarrow X} f\left(Y\left(K_{\Omega}\right)^{\operatorname{Br}_{\mathrm{ur}}(Y)}\right)},
		\]
		where $f:Y \to X$ varies among all torsors of type $\bar Y$.
	\end{theorem}
	
	\begin{proof}[Proof assuming Theorem \ref{Thm: fibration}.]
		As mentioned at the beginning of the section, this is implicitly proven in \cite{HW24}. For completeness, we fill the details of this argument.
		
		It suffices to prove this when $\bar G=\operatorname{Aut}(\bar Y/ X_{\oK})$ is abelian, and then proceed by induction on the derived series of $G$ as in \cite[Section 3C]{HW24}, to which we refer the reader for further details.
		
		Since $\bar G$ is abelian, the external action of $\Gal(X_{\oK}/X)=\Gamma_K$ on $\bar G$ induces a structure of $\Gamma_K$-module on $\bar G$, and a corresponding finite abelian group scheme $G/K$.
		
		We may assume without loss of generality that $X\left(K_{\Omega}\right)^{\operatorname{Br}_{\mathrm{ur}}(X)} \neq \emptyset$. Then by \cite[Lemma 3.7]{HW24}, the descent type $\bar Y \to X$ descends to a $G$-torsor $Y \to X$. (This actually only requires $X(K_{\Omega})^{\Bet(X)} \neq \emptyset$, see {\em loc.cit.}).
		
		Let now $S$ be a finite set of places of $K$ containing the archimedean ones, and let $(x_v)_{v \in S}$ be in $X(K_S)$. We wish to prove that there exists a $G$-twist $Y'$ of $Y$ and an adelic point $(y_v) \in Y'(K_{\Omega})^{\Br_{\mathrm{ur}}Y'}$ whose projection to $X$ approximates $x_v$ arbitrarily well for $v \in S$.
		
		By \cite[Lemma 3.8]{HW24}, the group $G$ fits in a short exact sequence $1 \to G \to T \to Q\to 1$, where $T$ and $Q$ are tori and $Q$ is quasi-trivial. Let $Z$ be the contracted product $Z=Y \times^G T$. The second projection $Z \to T/G=Q$ is a fibration whose fibers are $G$-twists of $Y$. The first projection $Z \to Y/G=X$ is a $T$-torsor. 
		
		Following the arguments of \cite[Section 3B]{HW24}, there exists a cocycle $\sigma \in H^1(K,G)$ such that a $K_{\Omega}$-point of $X$ arbitrarily near to $(x_v)_{v \in S}$ lifts to the $T$-torsor $Z_{\sigma} \coloneqq Y_{\sigma} \times^G T$. Thus, after replacing $Y$ with $Y_{\sigma}$ (and $Z$ with $Z_{\sigma}$), we may assume that the point $(x_v)_{v \in S}$ lifts to some $(z_v)_{v \in S} \in Z(K_v)$ and that $Z(K_{\Omega}) \neq \emptyset$. 
		
		Theorem \ref{Thm: fibration} gives:
		\begin{equation}\label{Eq:Fibration}
		Z\left(K_{\Omega}\right)_{\bullet}^{\mathrm{Br}_{\mathrm{ur}}(Z)}=\overline{\bigcup_{q \in Q(K)} Z_q\left(K_{\Omega}\right)_{\bullet}^{\mathrm{Br}_{\mathrm{ur}}\left(Z_q\right)}}.
		\end{equation}
		In fact, one may then find a $q \in Q(K)$ and a point $(z'_v)_{v \in M_K} \in Z_q(\A_K)^{\Br_{\mathrm{ur}}Z_q}$ that approximates $z_v$ for non-archimedean $v \in S$ and such that $z'_v$ and $z_v$ lie on the same connected component of $Z(K_v)$ for archimedean $v$. Note that the fibers of $Z \to Q$ are $G$-twists of $Y$, and let $Y' \coloneqq Z_q$. The projection of $z'_v$ to $X$ approximates $x_v$ arbitrarily well for non-archimedean $v \in S$ and lies on the same connected component of $x_v$ in $X(K_v)$ for archimedean $v$. For all archimedean $v$, we move $z'_v$ on its connected component in $Z_q(K_v)$ so that its projection to $X$ is equal to $x_v$. Since moving within the archimedean connected components does not interfere with the Brauer--Manin obstruction, after this modification we still have $(z'_v)_{v \in M_K} \in Z_q(K_{\Omega})^{\Br_{\mathrm{ur}}Z_q}=Y'(K_{\Omega})^{\Br_{\mathrm{ur}}Y'}$, and this point provides the sought lift.
	\end{proof} 
	
	We also borrow the following useful tool from Harpaz and Wittenberg's work. It appears implicitly in the proof of Théorème 4.2 of \cite{HW18}, but we restate it here for completeness, and add a variant with $X(K_{\Omega})_{dir}$ in place of $X(K_{\Omega})$:
	\begin{theorem}[Harpaz--Wittenberg base change]\label{Prop:HWbasechange}
		Let $f:X \to Q$ be a smooth proper morphism over a quasi-trivial torus $R_{E/K}\G_{m,E}$, and let $E \subset E'$ be an inclusion of finite étale $K$-algebras. Assume that, for all $c \in E^*$, letting $\phi_c:R_{E'/K}\G_{m,E'} \to R_{E/K}\G_{m,E}$ be the morphism $c \cdot N_{E'/E}$ (norm followed by translation by $c$), and defining $X'$ through the following cartesian diagram:
		\[
		\begin{tikzcd}
		X' \arrow[d,"\psi_c"'] \arrow[r, "f'"] & Q' \arrow[d, "\phi_c"] \\
		X \arrow[r, "f"]               & Q   ,                  
		\end{tikzcd}
		\]
		the identity
		\[
		X'(K_{\Omega})^{\Br_{ur}X'}=\overline{\bigcup_{q' \in Q'(K)} X'_{q'}(\A_K)^{\Br X'_{q'}}}
		\]
		holds. Then
		\[
		X(K_{\Omega})^{\Br_{ur}X}=\overline{\bigcup_{q \in Q(K)} X_q(\A_K)^{\Br X_q}}.
		\]
        The same statement holds after replacing $X(K_{\Omega}),X'(K_{\Omega}),X_q(\A_K),$ and $X'_{q'}(\A_K)$ by $X(K_{\Omega})_{dir},X'(K_{\Omega})_{dir},X_q(\A_K)_{\bullet},$ and $X'_{q'}(\A_K)_{\bullet}$, respectively. 
	\end{theorem}
    
	\begin{proof}
		For the first part, see the proof of \cite[Théorème 4.2]{HW18}. We include the argument for completeness.

        It suffices to prove the inclusion ``$\subset$'' as the other one is immediate. By \cite[Théorème 2.1]{HW18}, we have
        \begin{equation}\label{Lift}
            X(K_{\Omega})^{\Br_{ur}X} \subset \overline{\bigcup_{c \in E^*} \psi_c\left(X'_c(K_{\Omega})^{\Br_{ur}X'_c}\right)},
        \end{equation}
        where we are temporarily emphasizing the dependence of $X'$ on $c$ via a subscript. Thus, our assumption gives:
        \begin{align*}
            X(K_{\Omega})^{\Br_{ur}X} \subset & \overline{\bigcup_{c \in E^*} \psi_c(X'_c(K_{\Omega}))^{\Br_{ur}X'_c}} = \overline{\bigcup_{c \in E^*} \psi_c\left(\overline{\bigcup_{q' \in Q'(K)} (X'_{c})_{q'}(\A_K)^{\Br (X'_{c})_{q'}}}\right)} \\ 
            \subset  & \overline{\bigcup_{c \in E^*} \psi_c\left({\bigcup_{q' \in Q'(K)} (X'_{c})_{q'}(\A_K)^{\Br (X'_{c})_{q'}}}\right)} \\
            = & \overline{\bigcup_{c \in E^*} \bigcup_{q' \in Q'(K)} \psi_c\left({(X'_{c})_{q'}(\A_K)^{\Br (X'_{c})_{q'}}}\right)} \subset \overline{\bigcup_{q \in Q(K)}\left({X_q(\A_K)^{\Br X_q}}\right)},
        \end{align*}
        which proves the first part.

        \vskip1mm

        For the second part, note that the map $\phi_c:Q'(K_{\infty}) \to Q(K_{\infty})$ satisfies $\phi_c(t \cdot q') = t^{[E':E]}\cdot \phi_c(q')$ for $t \in \R_{>0}$ and $q' \in Q'(K_{\infty})$. This equivariance induces quotient maps $\phi_c:Q'(K_{\infty})_{\text{dir}} \to Q(K_{\infty})_{\text{dir}}$ and $\psi_c:X'(K_{\infty})_{\text{dir}} \to X(K_{\infty})_{\text{dir}}$ of $\phi_c:Q'(K_{\infty}) \to Q(K_{\infty})$ and  $\psi_c:X'(K_{\infty}) \to X(K_{\infty})$, respectively. Accordingly, \eqref{Lift} induces an inclusion
        \[
        X(K_{\Omega})_{\text{dir}}^{\Br_{ur}X} \subset \overline{\bigcup_{c \in E^*} \psi_c\left(X'_c(K_{\Omega})_{\text{dir}}^{\Br_{ur}X'_c}\right)},
        \]
        and the whole sequence of inclusions above still holds after the replacements in the statement.
	\end{proof}
    
    We shall use Theorem \ref{Prop:HWbasechange} repeatedly to make reductions in proving Theorem \ref{Thm: fibrationnew} (see Proposition \ref{PropReductions}). Applying it with $E'$ equal to the Galois closure of $E$ we immediately get the first:
    \begin{corollary}
        Assume Theorem \ref{Thm: fibrationnew} holds when:
        \begin{enumerate}[label=$(R_{\arabic*})$]
            \item The extension $E/K$ is Galois.
        \end{enumerate}
        Then Theorem \ref{Thm: fibrationnew} holds.
    \end{corollary}
	
	\section{Arithmetic symbols}\label{Sec4}
	
	This section is dedicated to arithmetic symbols. Lemma \ref{Lem:reciprocity} below, which is likely known to experts, tells that certain spin symbols \cite{SPIN} of primes are trivial up to local terms at small places. We name the lemma ``spin reciprocity''. Together with the classical reciprocity for power-symbols \cite[Theorem VI.8.3]{NeukirchBook}, this plays an important role in the proof of Theorem \ref{ThmShafarevich}. In Subsection \ref{Ssec4.1} we introduce what we call {\em half-spin symbols}, while in Subsection \ref{Ssec4.2} we lay out our conventions on {\em generalized Redéi symbols}, mostly following \cite{KimMori}.

    \vskip1mm
	
	Fix in this section a number field $E$ and a natural number $n$ such that $\mu_n \subset E^{*}$.
	
	\begin{lemma}[Spin reciprocity]\label{Lem:reciprocity}
		Let $\tau:E \to E$ be an involution such that that $\zeta_n^{\tau}=\zeta_n^{-1}$. Let $S$ be a finite set of places of $E$ containing those dividing $\Delta_{E/E^{\tau}},n$ or $\infty$, and such that $\tau(S)=S$. For a prime element $p$ of $\cO_{E,S}$ such that $p \cdot \cO_{E,S} \neq p^{\tau}\cdot \cO_{E,S}$, we have 
		\begin{equation}\label{Eq: Legendre rec}
		\leg{p}{p^{\tau}\cdot \cO_{E,S}}_{n}=\prod_{v \in S} (p,p^{\tau}-p)_{w,n}.
		\end{equation}
	\end{lemma}
	
	\begin{proof}
		We claim that for all $x \in \cO_{E,S}$ and $v\notin S$ such that $v(x)=0$, we have $\prod_{w \in \{v,v^{\tau}\}}(x,x^{\tau}-x)_{w,n}=1$.
		
		The claim is clear if $v(x^{\tau}-x)=0$. Otherwise, if  $e\coloneqq v(x^{\tau}-x)>0$ and $v^{\tau} \neq v$, we have $\leg{x}{v}_n\leg{x}{v^{\tau}}_n=\leg{x}{v}_n\leg{x^{\tau}}{v}_n^{\tau}=\leg{x}{v}_n\leg{x^{\tau}}{v}_n^{-1}=\leg{x}{v}_n\leg{x}{v}_n^{-1}=1$, where in the penultimate identity we use that $x \equiv x^{\tau} \bmod v$. Hence $(x,x^{\tau}-x)_{v,n}(x,x^{\tau}-x)_{v^{\tau},n}=\leg{x}{v}_n^e\leg{x}{v^{\tau}}_n^e=1$. Finally, assume $v(x^{\tau}-x)>0$ and $v^{\tau} = v$. Then $v$ is inert above the place $w$ of $E^{\tau}$ under it, and the extension $\F_w \subset \F_v$ has degree $2$. Thus $x \in \F_v^*$ is an $(Nw+1)$-power. However, the hypothesis on $\zeta_n$ implies that $n$ divides $Nw+1$: indeed, $\tau$ reduces  $\bmod\,v$ to the Frobenius automorphism of the quadratic extension $\F_v/\F_w$, so $\zeta_n^{\tau} \equiv \zeta_n^{Nw} \bmod v$ and hence $\zeta_n^{\tau} = \zeta_n^{Nw}$, and using $\zeta_n^{\tau}=\zeta_n^{-1}$ we get $\zeta_n^{Nw+1}= 1$. Thus $x \in \F_v^*$ is also an $n$-th power, hence $\leg{x}{v}_n=1$, and $(x,x^{\tau}-x)_{v,n}=\leg{x}{v}^{v(x^{\tau}-x)}_n=1$. 
		
		Now, using the claim and Hilbert's product formula we infer that $1=\linebreak \prod_{v \in S \cup \{v_p,v_p^{\tau}\}} (p,p^{\tau}-p)_{v,n}$, where $v_p$ denotes the place of $E$ corresponding to $p$. This concludes the proof after observing that $(p,p^{\tau}-p)_{v_p^{\tau},n}=1$ while $(p,p^{\tau}-p)_{v_p,n}=\leg{p^{\tau}}{p\cdot \cO_{E,S}}=\leg{p}{p^{\tau}\cdot \cO_{E,S}}^{\tau}=\leg{p}{p^{\tau}\cdot \cO_{E,S}}^{-1}$, where the latter holds because of the identity $\zeta_n^{\tau}=\zeta_n^{-1}$. 
	\end{proof}
	
	\subsection{Half-spin symbols}\label{Ssec4.1}
	
	In this subsection, let $\tau :E \to E$ be an involution such that $\zeta_n^{\tau}=\zeta_n$, and $S$ a finite set of places of $E$ such that $\tau(S)=S$ and that contains the ones dividing $n$, $\infty$ or $\Delta_{E/E^{\tau}}$.
	
	\begin{definition}\label{Def:SR}
		For $x \in \cO_{E,S}$ coprime with its conjugate $x^{\tau}$, the {\em half-spin symbol} is the power-residue symbol
		\[
		\operatorname{HS}(x) \coloneqq \leg{x}{\mathcal I_x}_n,
		\]
		where $\mathcal I_x$ denotes the unique prime ideal of $\cO_{E^{\tau},S}$ such that $\cI_x \cdot \cO_{E,S}=(x-x^{\tau})\cdot \cO_{E,S}$.
	\end{definition}
	
	Such an ideal $\cI_x$ always exists as $(x-x^{\tau})\cdot \cO_{E,S}$ is a $\tau$-invariant ideal and $S$ contains all ramified primes for $E/E^{\tau}$. To make sense of the symbol $\leg{x}{\mathcal I_x}_{n}$ note that, even though $x$ does not necessarily belong to $\cO_{E^{\tau},S}$, its reduction modulo $\mathcal I_x\cdot \cO_{E,S}$ lies in $\cO_{E^{\tau},S}/\mathcal I_x \subset \cO_{E,S}/\mathcal I_x\cdot \cO_{E,S}$.
		
	\vskip2mm
	
	The definition of $\operatorname{HS}(x)$ depends of course on $E,\tau,S$ and $n$. We shall often emphasize the dependence on this data (or a subset of this data) by including them as a subscript, e.g.\ as in $\operatorname{HS}_{E,\tau,S,n}(x), \operatorname{HS}_{E,\tau,S}(x)$, etc...
	
	\vskip1mm
	
	The first point of the following proposition shows that the square of the half-spin symbols recovers, up to local factors at $S$, the spin symbol $\operatorname{Spin}_S(x)=\leg{x^{\tau}}{x\cdot \cO_{E,S}}_n$, hence the name.
	
	\begin{proposition}\label{PropSelfResidues}            
		Let $x \in \mathcal O_{E,S}$ be coprime with $x^{\tau}$. The following hold:
		\begin{enumerate}[label={\rm (\arabic*)}]
			\item \makebox[\linewidth][c]{
				$\begin{aligned}[t]
				\operatorname{HS}_S(x)^2=\leg{x^{\tau}}{x\cdot \cO_{E,S}}_n\cdot \prod_{v \in S}(x^{\tau}-x,x)_{n,v}.
				\end{aligned}$}
			\item If $y$ is an element of $\mathcal O_{E,S}$ that is coprime in $\cO_{E,S}$ with both $x$ and $y^{\tau}$, we have
            \begin{multline*}
                \operatorname{HS}_S(xy)\operatorname{HS}_S(x)^{-1}\operatorname{HS}_S(y)^{-1}\leg{y^{\tau}}{x\cdot \cO_{E,S}}_n \\ =\prod_{v\in S}(N(x)/y_2,z_2/x_2)_{n,v}\cdot (-x_2,z_2)_{n,v},
            \end{multline*}
            where $x_2 \coloneqq \sqrt D \cdot (x-x^{\tau}), y_2 \coloneqq \sqrt D \cdot (y-y^{\tau}), z_2 \coloneqq \sqrt D \cdot (xy-(xy)^{\tau})$ for any $D \in E^{\tau}$ such that $E=E^{\tau}(\sqrt D)$.
			\item If $y$ is an element of $\mathcal O_{E,S}$ that is coprime in $\cO_{E,S}$ with both $x$ and $y^{\tau}$, is sufficiently close to $1$ for all non-archimedean $v\in S$ and its Minkowski direction is sufficently close to that of $1$, then
			\[
			\operatorname{HS}_S(xy)\operatorname{HS}_S(x)^{-1}\operatorname{HS}_S(y)^{-1}\leg{y^{\tau}}{x\cdot \cO_{E,S}}_n=1.
			\]
		\end{enumerate}  
	\end{proposition}
	
	The ``sufficiently close'' in point (3) depends on $x$.
		
	\subsubsection{Proof of Proposition \ref{PropSelfResidues}}\label{SSec:Proof}
	\begin{proof}[Proof of (1)] 
		Power-reciprocity states $\prod_{w \in S}(b,a)_n=\leg{a}{b\cdot \cO_{E,S}}_n\cdot \leg{b}{a\cdot \cO_{E,S}}_n^{-1}$ whenever $a,b\in \cO_{E,S}$ are coprime. Applying it to $x,x^{\tau}-x$ and using $\leg{x^{\tau}}{x \cdot \cO_{E,S}}_n=\leg{x^{\tau}-x}{x \cdot \cO_{E,S}}_n$, we get
		\begin{align*}
		\leg{x}{(x^{\tau}-x)\cdot \cO_{E,S}}_n\cdot \leg{x^{\tau}}{x \cdot \cO_{E,S}}_n^{-1}= \prod_{w \in S}(x^{\tau}-x,x)_{n,w}.
		\end{align*}
		Since $\leg{x}{(x^{\tau}-x)\cdot \cO_{E,S}}_n=\leg{x}{\cI \cdot \cO_{E,S}}_n=\leg{x}{\cI}_n^2=\operatorname{HS}(x)^2$, this concludes the proof. \end{proof}
	
	To prove (2), we use the following.
	
	\begin{lemma}\label{LemTripleProductOfLegendre}
		Let $K$ be a number field, $n$ be a natural number such that $\mu_n \subset K$, and $S$ be a finite set of places of $K$ containing those dividing $n$ and the archimedean ones. Let $a,b,c \in \cO_{K,S}$ be three non-zero elements such that, letting  $\mathfrak d \coloneqq (a,b,c)\cdot \cO_{K,S}$ be their gcd ideal, the ideals $\mathfrak a\coloneqq (a \cdot \cO_{K,S})\cdot \mathfrak d^{-1}, \mathfrak b\coloneqq (b \cdot \cO_{K,S})\cdot \mathfrak d^{-1}, \mathfrak c\coloneqq (c \cdot \cO_{K,S})\cdot \mathfrak d^{-1}$ of $\cO_{K,S}$ are pairwise coprime. Then 
		\[
		\leg{a/b}{\mathfrak c}_n\leg{b/c}{\mathfrak a}_n\leg{c/a}{\mathfrak b}_n\leg{-1}{\mathfrak d}_n=\prod_{v \in S} (b,a)_{n,v}(c,b)_{n,v}(a,c)_{n,v}.
		\]
	\end{lemma}
	\begin{proof}
		By the Hilbert product formula, the right hand side is equal to \linebreak $\prod_{v \in M_K \s S} (a,b)_{n,v}(b,c)_{n,v}(c,a)_{n,v}$. Fix now a place $v \notin S$, choose a uniformizer $\pi_v$ of $K_v$, and write $a=\bar a_v \cdot \pi_v^{\alpha_v}, b=\bar b_v \cdot \pi_v^{\beta_v}, c=\bar c_v \cdot \pi_v^{\gamma_v}$, where $\bar a_v, \bar b_v, \bar c_v$ are $v$-adic units. Write $\alpha'_v,\beta'_v,\gamma'_v,\delta_v$ for the $v$-adic valuations of $\mathfrak a,\mathfrak b,\mathfrak c,\mathfrak d$, respectively. By the hypothesis of coprimality of $\mathfrak a, \mathfrak b$ and $\mathfrak c$, only one among $\alpha'_v,\beta'_v$ and $\gamma'_v$ can be positive. Note that $\alpha_v-\alpha'_v=\beta_v-\beta'_v=\gamma_v-\gamma'_v=\delta_v$. We get:
		\begin{align}
		\notag (a,b)_{n,v}(b,c)_{n,v}(c,a)_{n,v} &= \leg{\bar a_v /\bar b_v}{\pi_v}^{\gamma_v}\leg{\bar b_v /\bar c_v}{\pi_v}^{\alpha_v}\leg{\bar c_v /\bar a_v}{\pi_v}^{\beta_v} \leg{-1}{\pi_v}^{\alpha_v\beta_v+ \beta_v \gamma_v+\gamma_v\alpha_v}\\
		\notag &= \leg{\bar a_v /\bar b_v}{\pi_v}^{\gamma'_v+\delta_v}\leg{\bar b_v /\bar c_v}{\pi_v}^{\alpha'_v+\delta_v}\leg{\bar c_v /\bar a_v}{\pi_v}^{\beta'_v+\delta_v} \leg{-1}{\pi_v}^{3 \delta_v^2}\\
		\label{Eqqela} &= \leg{\bar a_v /\bar b_v}{\pi_v}^{\gamma'_v}\leg{\bar b_v /\bar c_v}{\pi_v}^{\alpha'_v}\leg{\bar c_v /\bar a_v}{\pi_v}^{\beta'_v} \leg{-1}{\pi_v}^{\delta_v}
		\end{align}
		If $\gamma'_v > 0$, then $\alpha'_v=\beta'_v=0$. Thus $\bar a_v /\bar b_v \equiv a/b \bmod \pi_v$, and $\leg{\bar a_v /\bar b_v}{\pi_v}^{\gamma'_v}$ is equal to $\leg{a/b}{\pi_v}^{\gamma'_v}$, i.e.\ the local $v$-adic contribution to the Legendre symbol $\leg{a/b}{\mathfrak c}_n$. Analogously, when $\beta'_v>0$ or $\alpha'_v>0$, the factors $\leg{\bar b_v /\bar c_v}{\pi_v}^{\alpha'_v}$ and $ \leg{\bar c_v /\bar a_v}{\pi_v}^{\beta'_v}$ are equal to the local $v$-adic contributions of the symbols $\leg{b/c}{\mathfrak a}_n$ and $\leg{c/a}{\mathfrak b}_n$. Lastly, $\leg{-1}{\pi_v}^{\delta_v}$ is the local $v$-adic contribution to $\leg{-1}{\mathfrak d}_n$. Thus the product of \eqref{Eqqela} over all $v \notin S$ is $\leg{a/b}{\mathfrak c}_n\leg{b/c}{\mathfrak a}_n\leg{c/a}{\mathfrak b}_n\leg{-1}{\mathfrak d}_n$, proving the statement.
	\end{proof} 
	
	\begin{proof}[Proof of (2)]
		Let $K=E^{\tau}$. Since the extension $E=K(\sqrt{D})$ is unramified over $S$, we have $D\cdot \cO_{K,S}=\mathfrak d_s^2$ for some fractional ideal $\mathfrak d_s$. After possibly replacing $D$ with $D \cdot u^2$ for an appropriate $u \in K^*$, we assume without loss of generality that $D$ is coprime with $N(x)$ and $N(y)$ in $\cO_{K,S}$. 
		
		\vskip1mm
		
		The proof is based on applying Lemma \ref{LemTripleProductOfLegendre} to $x_2,y_2,z_2\in \cO_{K,S}$, but we first need to verify the coprimality hypothesis.
		For an element $a \in \cO_{E,S}$, write $\cI_a \subset \cO_{K,S}$ for the ideal such that $\cI_a \cdot \cO_{E,S}=(a-a^{\tau})\cdot \cO_{E,S}$. We have $x_2 \cdot \cO_{K,S}= \mathfrak d_s \cdot \cI_x, \ y_2 \cdot \cO_{K,S}= \mathfrak d_s \cdot \cI_y, \ z_2 \cdot \cO_{K,S}= \mathfrak d_s \cdot \cI_z$. 
		From the identity
		\begin{equation}\label{Ide}
		z-z^{\tau}=(xy)-(xy)^{\tau}=y(x-x^{\tau})+x^{\tau}(y-y^{\tau})
		\end{equation}
		and the coprimality of the pairs $x,x^{\tau}$ and $y,y^{\tau}$ in $\cO_{E,S}$, it follows that, if $\cJ \subset \cO_{K,S}$ divides two among the triple $x-x^{\tau},y-y^{\tau},z-z^{\tau}$, then it must divide the third as well. Hence $(\cI_x,\cI_y,\cI_z)=(\cI_x,\cI_y)=(\cI_y,\cI_z)=(\cI_z,\cI_x)$. In other words, writing $\cI_x=\cI'_x\cdot \mathfrak d,\cI_y=\cI'_y\cdot \mathfrak d,\cI_z=\cI'_z\cdot \mathfrak d$ with $\mathfrak d= (\cI_x,\cI_y,\cI_z) \subset \cO_{K,S}$, the three ideals $\cI'_x,\cI'_y,\cI'_z$ are pairwise coprime. Writing $x_2 \cdot \cO_{K,S}=\cI'_x \cdot \mathfrak d \cdot \mathfrak d_s, y_2 \cdot \cO_{K,S}=\cI'_y \cdot \mathfrak d \cdot \mathfrak d_s, z_2 \cdot \cO_{K,S}=\cI'_z \cdot \mathfrak d \cdot \mathfrak d_s$, we see that we have now verified the coprimality hypothesis. 
		
		\vskip1mm
		
		We may thus apply Lemma \ref{LemTripleProductOfLegendre} and get that
		\begin{equation}\label{TripleProd}
		\leg{z_2y_2^{-1}}{\cI'_x}_n\leg{x_2z_2^{-1}}{\cI'_y}_n\leg{y_2x_2^{-1}}{\cI'_z}_n\leg{-1}{\mathfrak d \mathfrak d_s}_n
		\end{equation}
		is equal to the product of local symbols $\prod_{v \in S}(y_2,x_2)_{n,v} \cdot (z_2,y_2)_{n,v} \cdot (x_2,z_2)_{n,v}$.
		
		\vskip1mm 
		
		We have $z_2y_2^{-1}=y\frac{x_2}{y_2}+x^{\tau}$. Since the ideal $\cI_x'$ and $\cI_y'$ are coprime in $\cO_{K,S}$, the fractional ideal $(x_2/y_2)\cdot \cO_{K,S}=\cI'_x/\cI'_y$ is reduced to minimal terms, and we then have that $z_2y_2^{-1} \equiv x^{\tau} \equiv x \bmod \cI'_x$. Analogously, using $z_2x_2^{-1}=y+x^{\tau}\frac{y_2}{x_2}$ and $y_2x_2^{-1}=(x^{\tau})^{-1}\frac{z_2}{x_2}-y(x^{\tau})^{-1}$, one finds that $z_2x_2^{-1} \equiv y \bmod \cI'_y$ and $y_2x_2^{-1} \equiv -y\cdot (x^{\tau})^{-1} \equiv -z\cdot N(x)^{-1}\bmod \cI'_z$ (for the latter we also notice that $x^{\tau}$ is coprime with $y(x-x^{\tau})$ by hypothesis, and hence with $z-z^{\tau}$ by \eqref{Ide}, and a fortiori with $\cI'_z$), respectively. Thus the product \eqref{TripleProd} is equal to
		\[
		\leg{x}{\cI'_x}_n\leg{y}{\cI'_y}_n^{-1}\leg{-z\cdot N(x)^{-1}}{\cI'_z}_n\leg{-1}{\mathfrak d \mathfrak d_s}_n.
		\]
		Multiplying by the trivial term $\leg{x}{\mathfrak d}_n\leg{y}{\mathfrak d}^{-1}_n\leg{z\cdot N(x)^{-1}}{\mathfrak d}_n$ (note that $z=xy$ and $N(x)=xx^{\tau} \equiv x^2 \bmod\,\mathfrak d$), we get that this is equal to the product
		\begin{align*}
		&\leg{x}{\cI_x}_n\leg{y}{\cI_y}_n^{-1}\leg{-z\cdot N(x)^{-1}}{\cI_z}_n\leg{-1}{\mathfrak d_s}_n=\\ 
		&\leg{x}{\cI_x}_n\leg{y}{\cI_y}_n^{-1}\leg{z}{\cI_z}_n\leg{- N(x)}{\cI_z}_n^{-1}\leg{-1}{\mathfrak d_s}_n=\\ 
		&\operatorname{HS}(x)\operatorname{HS}(y)^{-1}\operatorname{HS}(z)\leg{- N(x)}{\cI_z}_n^{-1}\leg{-1}{\mathfrak d_s}_n.
		\end{align*}
		Since $x \equiv x^{\tau} \bmod \cI_x$, we have $\operatorname{HS}(x)^2=\leg{x^2}{\cI_x}=\leg{N(x)}{\cI_x}$, and so the product above is equal to
		\[
		\operatorname{HS}(x)^{-1}\operatorname{HS}(y)^{-1}\operatorname{HS}(z)\leg{N(x)}{\cI_x}\leg{- N(x)}{\cI_z}_n^{-1}\leg{-1}{\mathfrak d_s}_n.
		\]
		Further multiplying by the trivial term $\leg{N(x)}{\mathfrak d_s}_n\leg{N(x)}{\mathfrak d_s}_n^{-1}$, this becomes equal to
		\[
		\operatorname{HS}(x)^{-1}\operatorname{HS}(y)^{-1}\operatorname{HS}(z)\leg{-N(x)}{x_2 \cdot \cO_{K,S}}_n\leg{-N(x)}{z_2 \cdot \cO_{K,S}}_n^{-1}.
		\]
		Albert-Brauer-Hasse-Noether's theorem gives that $\leg{-N(x)}{z_2 \cdot \cO_{K,S}}_n=\leg{z_2}{N(x) \cdot \cO_{K,S}}_n \cdot \prod_{v \in S}(-N(x),z_2)_{n,v}$ and $\leg{N(x)}{x_2 \cdot \cO_{K,S}}_n=\leg{x_2}{N(x) \cdot \cO_{K,S}}_n \cdot \prod_{v \in S}(N(x),x_2)_{n,v}$. Summarizing the above we get
		\begin{align}
		\label{EqProductRS} &\operatorname{HS}(x)^{-1}\operatorname{HS}(y)^{-1}\operatorname{HS}(z)\leg{z_2/x_2}{N(x) \cdot \cO_{K,S}}_n\\
		\notag =&\prod_{v\in S}(-1,z_2)_{n,v}(N(x),z_2/x_2)_{n,v} \cdot (y_2,x_2)_{n,v} \cdot (z_2,y_2)_{n,v} \cdot (x_2,z_2)_{n,v}\\
		\notag =&\prod_{v\in S}(N(x)/y_2,z_2/x_2)_{n,v}\cdot (-x_2,z_2)_{n,v}.
		\end{align}
		This concludes the proof once we show that $\leg{z_2/x_2}{N(x) \cdot \cO_{K,S}}_{n}=\leg{y}{x^{\tau}\cdot \cO_{E,S}}_n$, which we do in the next paragraph.
		
		Let $x\cdot \cO_{E,S}=\prod_i\cP_i^{e_i}$ be the factorization in prime ideals. We have $N(x) \cdot \cO_{K,S}=\prod_i N(\cP_i)^{e_i}$, and each ideal $N(\cP_i)$ is prime as otherwise $\cP_i$ would be inert over $K$, and thus divide $x^{\tau}$ as well, contradicting the assumption that $x$ and $x^{\tau}$ are coprime. Since $\frac{z_2}{x_2}=y+x^{\tau}\frac{y_2}{x_2}$ and thus $\frac{z_2}{x_2}=y^{\tau}+x\frac{y_2}{x_2}$, we have
		\[
		\leg{z_2/x_2}{N(x) \cdot \cO_{K,S}}_n=\prod_i \leg{z_2/x_2}{N(\cP_i)}_n^{e_i}=\prod_i \leg{y^{\tau}}{\cP_i}_n^{e_i}=\leg{y}{x^{\tau}\cdot \cO_{E,S}}_n. \qedhere
		\]            
	\end{proof} 
	
	\begin{proof}[Proof of (3)]
		For $v \in S$, we have $\lim\limits_{y \to 1}\frac{z_2}{x_2}=\lim\limits_{y \to 1}\frac{y \cdot x_2+x^{\tau}\sqrt D \cdot (y-y^{\tau})}{x_2}=1$, so the non-archimedean contribution on the right hand side of (2) tends to $1$. The archimedean contribution also tends to $1$ by the same argument, since the (archimedean) Hilbert symbols in question are invariant under the operation of multiplying $y$ by a positive real factor.
	\end{proof}

	\subsection{Generalized Redéi symbols}\label{Ssec4.2}
	
	Rédei symbols are triple symbols $[a,b,c] \in \pm1$, defined for $a,b,c$ in $\Q^*/\Q^{*2}$ satisfying the coprimality condition \linebreak $\gcd(\Delta(a),\Delta(b),\Delta(c))=1$ (where $\Delta(x)$ denotes the discriminant of the quadratic field associated to $x$) and the interlinking condition $(a,b)_p=(b,c)_p=(a,c)_p$ for all primes $p \leq \infty$ \cite{Stevenhagen}. Several generalizations of these symbols to aribitrary number fields and higher orders $n$ have been proposed, see for instance Section 2 of Koymans and Smith's paper \cite{KS}. We adopt here one such generalization, presented systematically by Kim and Morishita in \cite{KimMori}. In {\em loc.cit.}\ such generalization is presented under the name of ``triple symbol'', but we adopt the term ``generalized Redéi symbol'' in this paper. We also prove in this section a well-posedness property (Lemma \ref{Lemma3point7}) not included in \cite{KimMori}.
	
	\vskip1mm
	
	\begin{definition*}
		For a local field $k$ of characteristic $0$ and a $\Gamma_k$-module $M$, we say that an inhomogeneous cochain $\alpha \in C^p(\Gamma_k,M), p \geq 0$ is {\em unramified} if it lies in the image of the natural map 
		$C^p(\Gal(k^{ur}/k),M^{\Gal(\bar k/k^{ur})}) \to C^p(\Gamma_k,M)$ (with the convention that $k^{ur}=k$ if $k$ is real). 
	\end{definition*}
	
	\begin{definition*}
		Given a number field $K$ with algebraic closure $\oK$, embeddings $\oK \subset \overline{K_v}, v \in M_K$ in the algebraic closures of its completions, and a $\Gamma_K$-module $M$, we say that an inhomogeneous cochain $\alpha\in C^p(\Gamma_K,M), p \geq 0$ is {\em completely unramified} if, for every place $v$, the restriction $\alpha|_{\Gamma_{K_v}} \in C^p(\Gamma_{K_v},M)$ is unramified. 
	\end{definition*}
	
	\vskip1mm
	
	To each element $x \in E^*$, we may associate a character $\chi_x:\Gamma_E \to (\Z/n\Z)(1)$ via the Kummer isomorphism $E^*/E^{*n} \cong H^1(\Gamma_E,(\Z/n\Z)(1))$. We denote by $S_x \subset M_E$ the set of places the character $\chi_x$ ramifies. Choose, for each place $v$ of $E$, an embedding $\bar E \subset \overline {E_v}$. 
	
	\vskip1mm
	
	For three elements $a,b,c \in E^*/E^{*n}$, consider the following condition.
	\begin{center}
		($C$) the characters $\chi_a,\chi_b,\chi_c$ are tame (i.e.\ for all places $v$, the wild inertia group at $v$ is contained in the kernels of all three characters), the sets $S_a, S_b,S_c$ are disjoint, and for every place $v$, the Hilbert symbols $(a,b)_v,(b,c)_v,(a,c)_v$ are trivial. 
	\end{center}
	Condition ($C$) guarantees that the cup-product $\chi_a \cup \chi_b \cup \chi_c \in Z^3(\Gamma_E, (\Z/n\Z)(3))$ is completely unramified \cite[Proposition 2]{KimMori}. For $a,b,c \in E^*$ satisying ($C$), we define (as in \cite[Section 1]{KimMori}) the generalized {\em Redéi symbol} as
	\begin{equation}\label{DefRedei}
	[a,b,c]_n\coloneqq \sum_{v \in M_E} \inv_v ([\xi|_{\Gamma_{E_v}}-\eta_v]) \in (\Z/n\Z)(2),
	\end{equation}
	where $\xi \in C^2(\Gamma_E,(\Z/n\Z)(3))$ is a primitive of the cocycle $\chi_a \cup \chi_b \cup \chi_c$ (in the sense that $\d \xi=\chi_a \cup \chi_b \cup \chi_c$), for every place $v$, $\eta_v \in C^2(\Gal(E_v^{\text{ur}}/E_v),(\Z/n\Z)(3))$ is a primitive of the restriction $\xi|_{\Gamma_{E_v}} \in C^2(\Gamma_{E_v},(\Z/n\Z)(3))$, and $\inv_v$ denotes (the Tate twist of) the local invariant map $H^2(\Gamma_{E_v},(\Z/n\Z)(1)) \to \Z/n\Z$.
	
	\begin{remark}
		In \cite{KimMori}, the symbol is defined with values in $\mu_n$ rather than in $(\Z/n\Z)(2)$. Although we are working under the assumption that $\mu_n \subset E$, and so the two groups are (non-canonically) isomorphic, this choice is not the natural one, as the coefficients of the cup product $\chi_a \cup \chi_b \cup \chi_c$ lie naturally in $(\Z/n\Z)(3)$. The invariant map takes one twist away, leaving $(\Z/n\Z)(2)$ as the natural value group for the generalized Redéi symbol. This will be important in later sections, as for instance our natural choice gives $[\sigma(a),\sigma(b),\sigma(c)]_n=\sigma([a,b,c]_n)$ for automorphisms $\sigma$ of $E$.
	\end{remark}
	
	The well-posedeness of  \eqref{DefRedei} comes from the following lemma, see \cite[Section 2]{KimMori}. We denote by $\Gamma_S\coloneqq \Gal(E_S/E)$ the Galois group of the maximal extension $E_S/E$ unramified outside $S$. 
	
	\begin{lemma}\label{Lem:KMWellPosed}
		For every triple $a,b,c \in E^*/E^{*n}$ satisfying ($C$), there exist primitives $\xi$ and $\eta_v$ (for all $v$) as above, and the expression \eqref{DefRedei} is independent of their choice. 
		
		Moreover, for any set of places $S$ containing $S_a \cup S_b \cup S_c$ and all the places dividing $n$ and $\infty$, the primitive $\xi$ may be chosen to lie in $C^2(\Gamma_S, (\Z/n\Z)(3))$.
	\end{lemma}
	
	
	The following lemma does not appear in \cite{KimMori}.
	
	\begin{lemma}\label{Lemma3point7}
		The symbol $[a,b,c]_n$ is independent of the choice of embeddings $\bar E \subset \overline{E_v}$.
	\end{lemma}
	
	Before proving the lemma, we need to recall how conjugation acts on group cochains. To an element $g \in G$, one associates a {\em conjugation map} $c_g: C^{\bullet}(G,M) \to C^{\bullet}(G,M)$ of complexes, defined by $(c_g(\alpha))(\sigma_1,\ldots,\sigma_p)\coloneqq g \cdot \alpha(g^{-1}\sigma_1g, \ldots, g^{-1}\sigma_pg)$. This map is homotopically equivalent to the identity. More precisely, we have $c_g-\id =\d K_g + K_g \d$ for the chain homotopy
	\begin{gather}\label{Homotopy}
	\notag K_g:C^{\bullet}(G,M) \to C^{\bullet-1}(G,M), \\
	(K_g(\alpha))(\sigma_1,\ldots,\sigma_{p-1})\coloneqq g \cdot\alpha(g^{-1},\sigma_1,\ldots,\sigma_{p-1})\, +\\
	\notag\sum_{i=1}^{p-1}(-1)^i g \cdot \alpha\left(g^{-1} \sigma_1 g, \ldots, g^{-1} \sigma_i g, g^{-1}\sigma_{i},\sigma_{i+1}, \ldots, \sigma_{p-1}\right)\, +\\
    \notag (-1)^pg \cdot \alpha\left(g^{-1} \sigma_1 g, \ldots, g^{-1} \sigma_{p-1} g\right).
	\end{gather}
	
	\begin{proof}[Proof of Lemma \ref{Lemma3point7}]
		Suppose given two different families of embeddings $\iota_v, \iota'_v:\bar E \hookrightarrow \bar E_v, v \in M_K$. Since the images $\iota_v(\bar E)$ and $\iota'_v(\bar E)$ coincide (both images are just the algebraic closure of $E$ inside $\bar E_v$), there exists, for all $v$, a $g_v \in \Gamma_E$ such that $\iota'_v=\iota_v\circ g_v$. The two decomposition representations $\rho_v,\rho'_v:\Gamma_{E_v} \hookrightarrow \Gamma_E$ induced by the two embeddings are then related by $\rho'_v=\text{conj}_{g_v} \circ \rho_v$, where $\text{conj}_{g_v}: \gamma \mapsto g_v \gamma g_v^{-1}$ denotes conjugation by $g_v$ in $\Gamma_E$. 
		
		For the sake of the proof, we denote, for $\xi \in C^n(\Gamma_E,(\Z/n\Z)(3)), n \geq 0$, by $\xi|_{\Gamma_v}$ and $\xi|'_{\Gamma_v}$ its restrictions to $\Gamma_v$ via $\rho_v$ and $\rho'_v$, respectively. The relation $\rho'_v=\text{conj}_{g_v} \circ \rho_v$ gives that $\xi|'_{\Gamma_v}=c_{g_v}(\xi)|_{\Gamma_v}$. We denote by $[a,b,c]_n$ and $[a,b,c]'_n$ the two Redéi symbols given, respectively, by the first and second family of embeddings.  Using the homotopy $K_{g_v}$, we get 
		\[
		c_{g_v}(\xi)=\xi+\d K_{g_v}(\xi)+ K_{g_v}(\d \xi)=\xi+ \d K_{g_v}(\xi)+ K_{g_v}(\chi_a \cup \chi_b \cup \chi_c).
		\]
		for each $v$. Restricting to $\Gamma_v$ and subtracting $\eta_v+ K_{g_v}(\chi_a \cup \chi_b \cup \chi_c)|_{\Gamma_v}$ from both sides, we get
		\begin{equation}\label{LocCose}
		\left(\xi|'_{\Gamma_v}-\eta_v- K_g(\chi_a \cup \chi_b \cup \chi_c)|_{\Gamma_v}\right)\equiv \left(\xi|_{\Gamma_v}-\eta_v \right), 
		\end{equation}
		modulo coborders. We claim that $K_g(\chi_a \cup \chi_b \cup \chi_c)|_{\Gamma_v}$ is unramified for all $v$. The statement then follows from the claim by summing local invariants of both sides of \eqref{LocCose} over all $v$.
		
		To prove the claim, we apply \eqref{Homotopy} to $\alpha=\chi_a \cup \chi_b \cup \chi_c$. Using $(\chi_a \cup \chi_b \cup \chi_c)(\sigma_1,\sigma_2,\sigma_3)=\chi_a(\sigma_1)\chi_b(\sigma_2)\chi_c(\sigma_3)$ (definition of cup-product) and the formula $\chi_x(\sigma\sigma')=\chi_x(\sigma)+\chi_x(\sigma')$, we get:
		\begin{multline}\label{HomotopyValue}
		K_g(\chi_a \cup \chi_b \cup \chi_c)(\sigma_1,\sigma_2)=\chi_a(g^{-1})\chi_b(\sigma_1)\chi_c(\sigma_2)-\chi_a(\sigma_1)\chi_b(g^{-1}\sigma_1)\chi_c(\sigma_2)+\\\chi_a(\sigma_1)\chi_b(\sigma_2)\chi_c(g^{-1}\sigma_2)-\chi_a(\sigma_1)\chi_b(\sigma_2)\chi_c(\sigma_2).
		\end{multline}
		
		If all three cocycles $\chi_a, \chi_b,\chi_c$ are unramified at $v$, then \eqref{HomotopyValue} is also clearly unramified. On the other hand, suppose one of them ramifies at $v$, say $\chi_a$. Then $\chi_b$ and $\chi_c$ are unramified at $v$. Moreover, as argued in the proof of \cite[Proposition 2]{KimMori}, condition ($C$) gives that the restrictions of the cup-products $\chi_a \cup \chi_b$ and $\chi_a \cup \chi_c$ to $Z^2(\Gamma_{E_v},(\Z/n\Z)(1))$ are trivial as cochains. In other words, we have $\chi_a(\sigma_1)\chi_b(\sigma_2)=0$ and $\chi_a(\sigma_1)\chi_c(\sigma_2)=0$ for all $\sigma_1,\sigma_2 \in \Gamma_{E_v}$, and \eqref{HomotopyValue} simplifies to 
		\[
		\left(K_g(\chi_a \cup \chi_b \cup \chi_c)|_{\Gamma_{E_v}}\right)(\sigma_1,\sigma_2)=\chi_a(g^{-1})\chi_b(\sigma_1)\chi_c(\sigma_2),
		\]
		which is unramified, proving the claim. 
	\end{proof}
	
	\vskip2mm
	
	We now discuss about the relation of generalized Redéi symbols with Heisenberg extensions. The Heisenberg group of matrices is:
	\[
	\operatorname{Heiss}_n:=\left\{ \begin{pmatrix}
	1 & x & z \\ 0 & 1 & y \\0 & 0 & 1
	\end{pmatrix}:  x,y \in (\Z/n\Z)(1), \ z \in (\Z/n\Z)(2)\right\}.
	\]
	This sits in a short exact sequence:
	\begin{equation}\label{EqHeiss}
	1 \to (\Z/n\Z)(2) \to \operatorname{Heiss}_n \to (\Z/n\Z)(1)^2 \to 1,
	\end{equation}
	where the first embedding is the inclusion in the $z$-coordinate, and the surjection is the projection on the $x$- and $y$-coordinates.
	
	\vskip1mm
	
	Given two $(\Z/n\Z)$-linearly independent classes $[a],[b] \in E^*/E^{*n}$, the extension \linebreak $E(\sqrt[n]{a},\sqrt[n]{b})/E$ is Galois with group $(\Z/n\Z)(1)^2$ by Kummer theory. We have:
	\begin{lemma}\label{Lem3.8}
		Given two $(\Z/n\Z)$-linearly independent classes $[a],[b] \in E^*/E^{*n}$, there exists a $\operatorname{Heiss}_n$-extension $F/E$ with fixed field $F^{(\Z/n\Z)(2)}=E(\sqrt[n]{a},\sqrt[n]{b})$ if and only if the Hilbert symbol $(a,b)_v$ is trivial for all $v$. Moreover, if such an extension $F$ exists, and $S$ is a set of places of $E$ containing $S_a \cup S_b \cup S_n \cup S_{\infty}$ such that $\Cl \cO_{E,S}=0$, then there exists one that is unramified outside of $S$.
	\end{lemma}
	
	We denote an extension $F$ as in the lemma by $E(\sqrt[n]{a},\sqrt[n]{b})_2$.
	
	\begin{proof}
		Sequence \eqref{EqHeiss} gives an exact sequence of sets
		\begin{equation}\label{SES}
		H^1(\Gamma_E,\operatorname{Heiss}_n) \to H^1(\Gamma_E,(\Z/n\Z)(1)^2) \to[\delta] H^2(\Gamma_E,(\Z/n\Z)(2)).
		\end{equation}
		For $\phi \coloneqq (\chi_a,\chi_b):\Gamma_E \to (\Z/n\Z)(1)^2$, one has $\delta(\phi)=\chi_a \cup \chi_b$, see \cite[Lemma 2]{KimMori}.
		
		By Galois theory, the existence of $F$ is equivalent to the existence of a surjection $\Gamma_E \to \operatorname{Heiss}_n$ lifting the surjection $\Gamma_E\to \Gal(E(\sqrt[n]{a},\sqrt[n]{b})/E) = (\Z/n\Z)(1)^2$. The latter surjection is just $(\chi_a,\chi_b)$. By \eqref{SES}, the existence of such a lift is then equivalent to $\chi_a \cup \chi_b$ being trivial in $H^2(\Gamma_E,(\Z/n\Z)(2))$. Indentifying the latter with $(\Br E)_n(1)$, the cup-product $\chi_a \cup \chi_b$ corresponds to the global Hilbert symbol $(a,b)$, and by the Albert--Brauer--Hasse--Noether theorem, this is trivial if and only if all the local Hilbert symbols $(a,b)_v$ are trivial.
		
		For the last statement, write $F=E(\sqrt[n]{a},\sqrt[n]{b})(\sqrt[n]{c})$ with $c \in E(\sqrt[n]{a},\sqrt[n]{b})^*$. Since $F$ is Galois over $E(\sqrt[n]{a},\sqrt[n]{b})$, we have $F=E(\sqrt[n]{a},\sqrt[n]{b})(\sqrt[n]{c'})$ for every $E(\sqrt[n]{a},\sqrt[n]{b})/E$-conjugate $c'$ of $c$. Thus $c'/c$ is an $n$-th power in $E(\sqrt[n]{a},\sqrt[n]{b})$ for any such $c'$. It follows that, for all places $v$ of $E$ not in $S$, and all places $w \mid v$ of $E(\sqrt[n]{a},\sqrt[n]{b})$, the residue class of $w(c) \bmod n$ is independent of $w$. We denote by $v(c) \in \Z$ a representative of this class. Since $\Cl \cO_{E,S}=0$, there is a $d \in E^*$ such that $d \cdot \cO_{E,S}= \prod_{\cP \notin S} \cP^{-v(c)}$ (note that this is a finite product). Now $F'=E(\sqrt[n]{a},\sqrt[n]{b})(\sqrt[n]{cd})$ is unramified over $E(\sqrt[n]{a},\sqrt[n]{b})$ (and thus over $E$) outside $S$. The extension $F'/E$ is again Galois with group $\operatorname{Heiss}_n$ and provides the desired lift.
	\end{proof}
	
	For the remainder of this section, fix a finite set of places $S_0$ containing $\{v:v|n\cdot \infty\}$ such that $\Cl \cO_{E,S_0}=0$.
	
	\vskip1mm
	
	For $a,b,c \in E^*$ we introduce a stronger variant of ($C$):
	\begin{center}
		($C'$)$_{S_0}$ The elements $a,b,c$ are $v$-adic $n$-th powers for all $v \in S_0$; the valuation $v(x)$ is coprime with $n$ for every $v \in S_x, x=a,b,c$; and ($C$) holds.
	\end{center}
	
	\begin{proposition}
		Let $a,b,c \in E^*$ be algebraic numbers satisfying ($C'$)$_{S_0}$. Then, for every finite place $v\in S_c$, we have that $\Frob_v(E(\sqrt[n]{a},\sqrt[n]{b})_2/E)$ lies in the center $(\Z/n\Z)(2)$ of $\operatorname{Heiss}_n$, and
		\[
		[a,b,c]_n=\sum_{v\in S_c}v(c)\cdot \Frob_v(E(\sqrt[n]{a},\sqrt[n]{b})_2/E) \in (\Z/n\Z)(2).
		\]
	\end{proposition}
	\begin{proof}
		We denote by $pr_x,pr_y:\operatorname{Heiss}_n \to (\Z/n\Z)(1), \ pr_z:\operatorname{Heiss}_n \to (\Z/n\Z)(2)$ the projections to the $x$, $y$ and $z$ coordinates. The multiplication rule for matrices gives that $pr_z(\sigma_1\sigma_2)-pr_z(\sigma_1)-pr_z(\sigma_2)=pr_x(\sigma_1)pr_y(\sigma_2)$ for all $\sigma_1,\sigma_2 \in \operatorname{Heiss}_n$. In other words,
		\[
		\d pr_z = pr_x \cup pr_y \quad \text{in }C^2(\operatorname{Heiss}_n,(\Z/n\Z)(2)).
		\]
		Pulling back this last identity along $\rho:\Gamma_E \to \Gal(E(\sqrt[n]{a},\sqrt[n]{b})_2/E)=\operatorname{Heiss}_n$, we get $\d \phi = \chi_a \cup \chi_b$ with $\phi=pr_z \circ \rho$. Thus $\d(\phi\cup \chi_c)=\chi_a\cup\chi_b\cup\chi_c$, and
		\[
		[a,b,c]_n= \sum_{v} \inv_v((\phi\cup \chi_c)|_{\Gamma_v}-\eta_v),
		\]
		with $\eta_v$ unramified such that $\d \eta_v=(\chi_a\cup\chi_b\cup\chi_c)|_{\Gamma_v}$. For $v \in S_0$, ($C'$)$_{S_0}$ implies that $c$ is a $v$-adic $n$-th power, so $\chi_c|_{\Gamma_v}$ is trivial, and thus so is $(\phi\cup \chi_c)|_{\Gamma_v}$, meaning that $v$-adic contributions to the Redéi symbol are trivial. For $v \in S_a$, since $v(c) \neq 0$, we have $(a,c)_v=\leg{c}{v}_n^{v(a)}$. So the triviality of this symbol, combined with $(v(a),n)=1$ (again from ($C'$)$_{S_0}$), gives that $c$ is a $v$-adic $n$-th power, and same for $v \in S_b$, and so these $v$ also contribute trivially to the Redéi symbol. For $v \notin S_0 \cup S_a \cup S_b \cup S_c$, the cocycle $\phi\cup \chi_c$ is unramified, and the contribution is also trivial.
		
		The only places remaining are the $v \in S_c$. Arguing as above, both $a$ and $b$ are $v$-adic $n$-th powers, so $v$ splits in $E(\sqrt[n]{a},\sqrt[n]{b})$ and $\Frob_v:=\Frob_v(E(\sqrt[n]{a},\sqrt[n]{b})_2/E)$ lies in $(\Z/n\Z)(2)$. Moreover, $\d\left( \phi|_{\Gamma_v}\right)=\chi_a|_{\Gamma_v} \cup \chi_b|_{\Gamma_v}$ is trivial, meaning $\phi|_{\Gamma_v}$ is already a cocycle, and thus so is $(\phi\cup \chi_c)|_{\Gamma_v}$. Then the local contribution at $v$ to the Redéi symbol is $\inv_v(\phi|_{\Gamma_v} \cup \chi_c|_{\Gamma_v})$. Since the cocycle $\phi|_{\Gamma_v}$ is unramified, this local invariant is precisely $v(c)\cdot \phi(\Frob_v)=v(c)\cdot\Frob_v(E(\sqrt[n]{a},\sqrt[n]{b})_2/E)$.
	\end{proof}
	
	\section{Horizontal Brauer group}\label{Sec5}
	
    In this section, we use the language of derived categories to define the {\em horizontal Brauer group} of a smooth proper morphism $f:X \to Y$ with geometrically integral fibers. The main interest lies in the case where $f$ has rationally connected fibers and $Y$ is a smooth variety over a number field $K$, where such group parametrizes the quotients $\Br X_y/f^{*}\Br y$ for $y$ in a Hilbertian subset $H \subset Y(K)$. See Theorem \ref{Surj}. We prove this theorem in Subsection \ref{SSec5.1}. Finally, Subsection \ref{SSec5.2} is devoted to a certain complex that arises naturally when the basis $Y$ is geometrically a torus.

    \vskip1mm
    
    We also refer the interested reader to Section 5 of the paper \cite{BBL} by Bright, Browning and Loughran for a similar approach to generate the Brauer groups of the fibers of rationally connected fibration, via the sheaf $R^2f_*\mu_{\infty}$. Although the approach in {\em loc.cit.} is simpler to ours (it does not require derived categories), it does not parametrize the Brauer groups of the fibers modulo constants exactly, but only generates them as the image of $H^0(Y,R^2f_*\mu_{\infty})$.
	
	\begin{definition}\label{Def:HBG}
		Let $f:X \to Y$ be a smooth proper morphism of schemes with geometrically integral fibers. The {\em horizontal Brauer group of $f$} is
		\[
		\Br_{\text{hor}}(X/Y) \coloneqq H^2(Y, \tau_{\geq 1}R^\bullet f_*\G_m).
		\]
	\end{definition}
	
	To every cartesian square
	\[
	\begin{tikzcd}
	X' \arrow[d, "f'"] \arrow[r, "g'"] & X \arrow[d, "f"] \\
	Y' \arrow[r, "g"]                  & Y  ,             
	\end{tikzcd}
	\]
	where $f$ (and thus also $f'$) is smooth proper with geometrically integral fibers, we associate a natural pullback map:
	\begin{equation}\label{Pullback}
	g^{*}:\Br_{\text{hor}}(X/Y) \to \Br_{\text{hor}}(X'/Y'),
	\end{equation}
	constructed as follows. Consider the composition of morphisms
	\begin{align*}
	(g^{*}\circ \tau_{\geq 1} \circ R^\bullet f_*)\G_m  &=(\tau_{\geq 1}\circ g^{*}\circ R^\bullet f_*)\G_m \\
	&\to (\tau_{\geq 1}\circ R^\bullet f'_*\circ (g')^{*})\G_m = (\tau_{\geq 1}\circ R^\bullet f'_*)((g')^{*}\G_m) \\
	&\to (\tau_{\geq 1}R^\bullet f'_*)\G_m
	\end{align*}
	in the derived category $\cD^+(\Sh_{Y'})$, where the two maps are give by the natural transformation $g^{*}\circ R^\bullet f_* \to R^\bullet f'_*\circ (g')^{*}$ and the natural map $(g')^{*}\G_m \to \G_m$, respectively. (We warn the reader that the natural map $(g')^{*}\G_m \to \G_m$ may not be an isomorphism, as we are working in the small étale site, whereas it would be an isomorphism in the big étale site.)
	This composition induces maps on cohomology
	\[
	H^n(Y',g^{*}\tau_{\geq 1}R^\bullet f_*\G_m) \to H^n(Y',\tau_{\geq 1}R^\bullet f'_*\G_m),
	\]
	for all $n \geq 0$. Composing these with the natural pullback maps
	\[
	g^{*}:H^n(Y,\tau_{\geq 1}R^\bullet f_*\G_m) \to H^n(Y',g^{*}\tau_{\geq 1}R^\bullet f_*\G_m),
	\]
	we get natural maps
	\[
	g^{*}:H^n(Y,\tau_{\geq 1}R^\bullet f_*\G_m) \to H^n(Y',\tau_{\geq 1}R^\bullet f'_*\G_m),\ \ n \geq 0,
	\]
	which we still denote by $g^{*}$ with a slight abuse of notation. For $n=2$, this gives \eqref{Pullback}. When $Y'$ is a point $y$ of $Y$, the pullback \eqref{Pullback} defines a {\em specialization map}
	\[
	\Br_{\text{hor}}(X/Y) \to \Br_{\text{hor}}(X_y/y).
	\]
	
	\begin{theorem}\label{Surj}
		Let $f:X \to Y$ be a smooth proper morphism of schemes with geometrically integral fibers. The following hold true.
		\begin{enumerate}[label={\rm (\roman*)}]
			\item There is a natural exact sequence:
			\begin{equation}\label{Eq:LES}
			\Br Y \to[f^{*}] \Br X \to \Br_{\text{hor}}(X/Y) \to[\partial] H^3(Y,\G_m) \to[f^{*}] H^3(X,\G_m)
			\end{equation}
		\end{enumerate}
		Assume, moreover, that $Y$ is regular, integral, of residual characteristic $0$, and that $R^1f_*\mu_{\infty}=R^2f_*\cO_X=0$. Then:
		\begin{enumerate}[label={\rm (\roman*)}, resume]
			\item the specialization map
			\[
			\Br_{\text{hor}} (X/Y) \to \Br_{\text{hor}}(X_y/y)
			\]
			is an isomorphism for $y$ in a Hilbert subset $H \subset Y$,
			\item $\Br_{\text{hor}}(X/Y)$ is finite,
		\end{enumerate}
	\end{theorem}
	
	Our proof of Theorem \ref{Surj} follows similar lines as that of Proposition 4.1 of \cite{HW14}, which covers points (ii)-(iv) in the case where $H^3(Y,\G_m)=0$.
	
	\vskip1mm
	
	The generic point $\eta$ of $Y$ lies in all Hilbert subsets of $Y$. Hence point (ii) gives, under its hypotheses, an identification
	\[
	\Br_{\text{hor}}(X/Y)=\Br_{\text{hor}}(X_{\eta}/\eta),
	\]
	which in turn identifies the morphism $\Br X_{\eta} \to \Br_{\text{hor}}(X_{\eta}/\eta)$ appearing in \eqref{Eq:LES} for $X_{\eta} \to \eta$ with a morphism
	\begin{equation}\label{EqMorf}
	r:\Br X_{\eta} \to \Br_{\text{hor}}(X/Y).
	\end{equation}
	
	When $Y$ is a smooth variety over a number field $K$, the vanishing $H^3(K,\G_m)=0$ \cite[VII.11.4, p. 199]{CasselsFrohlich} induces identifications $\Br_{\text{hor}}(X_y/y) = \Br X_y/\Br y$ for all $y \in Y(K)$. Thus the specialization $\Br_{\text{hor}}(X/Y) \to \Br_{\text{hor}}(X_y/y)$ induces a Brauer--Manin pairing on the fibers:
	\[
	(-,-)_{BM}:\Br_{\text{hor}}(X/Y) \times \bigcup_{y \in Y(K)} X_y(\A_K) \to \qz.
	\] 
	
	\subsection{Proof of Theorem \ref{Surj}}\label{SSec5.1}
	
	\begin{proof}[Proof of (i)]
		We have $f_*\G_m=\G_m$ (this holds in general for any flat proper morphism with geometrically integral fibers). Hence, applying the functor $\tau_{\geq 1}$ to $R^\bullet f_*\G_m$, we get a natural exact triangle:
		\begin{equation}\label{Triangle}
		\G_m[0] \to R^\bullet f_*\G_m \to \tau_{\geq 1}R^\bullet f_*\G_m \to[+1],
		\end{equation}
		in $\cD^+(\Sh_Y)$. Taking hypercohomology and using the identity $H^n(Y,R^\bullet f_*\G_m)=H^n(X,\G_m)$ arising from 
		Grothendieck's theorem on the composition of derived functors \cite[Tag 015L]{stacks-project} (which applies as $f_*:\Sh(X) \to \Sh(Y)$ sends injectives to injectives), this triangle induces a long exact sequence:
		\begin{equation*}
		\cdots \to  H^{n}(Y,\G_m) \to[f^{*}] H^{n}(X,\G_m) \to H^n(Y, \tau_{\geq 1}R^\bullet f_*\G_m) \to[\partial] H^{n+1}(Y,\G_m) \to \cdots .
		\end{equation*}
		of which \eqref{Eq:LES} is a subsequence.
	\end{proof}
	
	From here on, we assume that $Y$ is regular and integral and that $R^1f_*\mu_{\infty}=R^2f_*\cO_X=0$. Applying $R^\bullet f_*$ to the exact sequence
	\[
	1 \to \mu_{\infty} \to \G_m \to \G_m \otimes \Q \to 1,
	\]
	of étale sheaves on $X$, we obtain a natural ``$\infty$-Kummer'' exact triangle:
	\begin{equation}\label{Eq:cone}
	R^\bullet f_*\mu_{\infty} \to R^\bullet f_*\G_m \to R^\bullet f_*(\G_m\otimes \Q) \to[+1].
	\end{equation}
	We have
	\begin{equation}\label{Eq:Vanishing}
	R^2f_*(\G_m \otimes \Q)=0.
	\end{equation}
	In fact, $R^2f_*\G_m$ is the sheafification in the small étale site $Y_{\et}$ of the presheaf $Y' \mapsto \Br (X \times_Y Y')$. Since $f$ is smooth and $Y$ is regular, the scheme $X \times_Y Y'$ is regular for every $Y' \in Y_{\et}$, and hence its Brauer group is torsion \cite[Lemma 3.5.3]{BGbook}. It follows that $R^2f_*\G_m$ is also torsion, and the identity $R^2f_*(\G_m \otimes \Q)=(R^2f_*\G_m) \otimes \Q$ gives \eqref{Eq:Vanishing}.
	
	By the just-proven vanishing $R^2f_*(\G_m \otimes \Q)=0$, and the assumed vanishing $R^1f_*\mu_{\infty}=0$, we may truncate \eqref{Eq:cone} to obtain a natural exact triangle
	\begin{equation}\label{Lemma}
	(R^1f_*\G_m)[2] \otimes \Q \to R^2f_*\mu_{\infty}[2] \to \tau_{[1,2]}R^\bullet f_*\G_m \to[+1].
	\end{equation}
	
	\vskip1mm
	
	The following lemma is an adaptation of \cite[Proposition 2.9]{Zhu} to our slightly different context. 
	
	\begin{lemma}\label{Lem9}
		The sheaf $R^1f_*\G_m$ is represented by a torsion-free finitely generated isotrivial $Y$-scheme. Its stalk at a geometric point $\bar y$ of $Y$ is equal to $\Pic X_{\bar y}=\NS(X_{\bar y})$. 
	\end{lemma}
	\begin{proof}
		The smooth base-change theorem \cite[Theorem VI.4.1]{LECcompleto} and Grauert's theorem \cite[pp. 288-291]{Hartshorne} show that the two vanishing hypotheses are equivalent to the vanishing of $H^1(X_{\bar y},\mu_{\infty})$ and $H^2(X_{\bar y},\cO_{X_{\bar y}})$, respectively, for all geometric points $\bar y$ of $Y$. The vanishing of the former implies that $X_{\bar y}$ has no non-trivial connected abelian finite étale covers. Hence the following hold: $\Pic X_{\bar y} \cong \NS X_{\bar y}$,  $\NS X_{\bar y}$ is torsion-free, and $H^1(X_{\bar y},\cO_{X_{\bar y}})=0$ for all $\bar y$ \cite[Corollary 5.1.3]{BGbook}. Again by Grauert's theorem, the last of these three conditions is equivalent to the vanishing of $R^1f_*\cO_X$.
		
		The rest of the argument follows \cite{Zhu}, but we include it for completeness. We claim that the relative Picard scheme $\Pic_{X/Y}$ is a disjoint union of finite étale covers of $Y$. In fact, the smoothness of $f$ guarantees that each closed subscheme of $\operatorname{Pic}_{X / Y}$ which is of finite type is proper over $Y$ \cite[p. 232, Theorem 3]{FGA}. The vanishings $R^1f_*\cO_X=R^2f_*\cO_X=0$ (and the fact that both commute with base-change) imply that $\mathrm{Pic}_{X / Y}$ is formally étale over $C$. Together, these facts imply that each connected component of $\mathrm{Pic}_{X / Y}$ is finite étale over $C$, proving the claim.
		
		Denote by $\bar \eta$ the generic geometric point $\Spec \overline{\kappa(\eta)}$ of $Y$. Choose a basis $v_1,\ldots,v_r$ of constant sections for the group-scheme $\Pic_{X_{\bar \eta}/\bar \eta}$. The section $v_1$ dominates a connected component $Y_1$ of $\Pic_{X/Y}$, thus the base-change $\Pic_{X/Y} \times_Y Y_1$ is a $Y_1$-group scheme equipped with a canonical section. Taking further finite étale base-changes, we find an étale cover $Y' \to Y$ such that the base-change $\Pic_{X/Y} \times_Y Y'$ has $r$ canonical sections, which restrict to  $v_1,\ldots,v_r$ on the generic geometric fiber. Combined with the previous paragraph, this guarantees that $\Pic_{X/Y} \times_Y Y'$ is a (torsion-free, finitely generated) constant group scheme. Thus $\Pic_{X/Y}$ is isotrivial by definition.
		
		Finally, the geometric stalk $(\Pic_{X/Y})_{\bar y}$ of $\Pic_{X/Y}$ is equal to $\Pic (X_{\cO_{\bar y}^{\text{sh}}})=\Pic_{X/Y}(\cO_{\bar y}^{\text{sh}})$, where $\cO_{\bar y}^{\text{sh}}$ denotes the strict Henselianization of $Y$ at $\bar y$ \cite[p.38]{LECcompleto}. Since $\Pic_{X/Y}$ is étale over $Y$, we have  $\Pic_{X/Y}(\cO_{\bar y}^{\text{sh}})=\Pic_{X/Y}(\bar y)$ by the properties of strictly Henselian rings. Finally, we have $\Pic_{X/Y}(\bar y)=\Pic X_{\bar y}$ by representability.
	\end{proof}
	
	\begin{lemma}\label{Cor1}
		There is a natural exact sequence 
		$$
		H^0(Y,\Pic_{X/Y}) \otimes \Q \to H^0(Y,R^2f_*\mu_{\infty}) \to \Br_{\text{hor}}(X/Y) \to 0.
		$$
	\end{lemma}
	\begin{proof}
		Choose a geometric point $\bar y$ of $Y$. By Lemma \ref{Lem9}, there is a finite étale morphism $p:Y' \to Y$ such that the restriction $\Pic_{X/Y}|_{Y'}$ is a constant sheaf with coefficient group $\Z^r$. We may assume that $Y'$ is integral and Galois over $Y$, with group $G$. Since $Y'$ is normal, $H^1(Y',\Z)=0$, and the Hochschild-Serre spectral sequence gives an isomorphism $H^1(Y,\Pic_{X/Y}) \cong H^1(G,\Pic_{X/Y}(Y'))$, which is finite. Thus $H^1(Y,\Pic_{X/Y})\otimes \Q=0$. Noting that $H^2(Y,\tau_{\geq 1}M)=H^2(Y,\tau_{[1,2]}M)$ for any complex of sheaves $M$, the sequence of the statement is contained in the cohomology sequence of \eqref{Lemma}.
	\end{proof}
	
	\begin{lemma}\label{Lem10}
		There exists a finite étale connected cover $\pi:Y' \to Y$ such that, for any morphism $p:Z \to Y$ with $Z$ regular and integral for which the fiber product $Z \times_Y Y'$ is also integral, the pullback
		\[
		p^{*}:\Br_{\text{hor}}(X/Y) \to \Br_{\text{hor}}(X_Z/Z),
		\]
		where $X_Z \coloneqq X \times_Y Z$, is an isomorphism.
	\end{lemma}
	
	We observe that the integrality hypothesis is automatically satisfied if the fibers of $p$ are geometrically connected. In fact, in this case, the fibers of the base change $p':Z':=Z \times_Y Y' \to Y'$ are also geometrically connected. So, since $Y'$ is connected, $Z'$ is too. Since $Z'$ is also regular, as it is a finite étale cover of $Z$, it is thus also integral.
	
	\begin{proof}
		Let $\bar y$ be a geometric point of $Y$, of which we make a more specific choice further down. By Lemma \ref{Lem9}, we may identify the sheaf $R^1f_*\G_m$ with the $\pi_1(Y,\bar y)$-module $\NS(X_{\bar y})$. Moreover, by smooth-proper base change \cite[Corollary VI.4.2]{LECcompleto}, we may identify the sheaf $R^2f_*\mu_{\infty}$, which is a direct limit of locally constant sheaves, with the $\pi_1(Y,\bar y)$-module $H^2(X_{\overline{y}},\mu_{\infty})$, see \cite[p.\ 156]{LECcompleto}. 
		
		Under these identifications, Lemma \ref{Cor1} gives a canonical exact sequence:
		\begin{equation}\label{Eq:ident}
		(\NS X_{\overline{y}} \otimes \Q)^{\pi_1(Y,\overline{y})}\to H^2(X_{\overline{y}},\mu_{\infty})^{\pi_1(Y,\overline{y})}\to \Br_{\text{hor}}(X/Y) \to 0.
		\end{equation}
		Let now $\bar z$ be the generic geometric point of $Z$, and let $\bar y\coloneqq p(\bar z)$. Comparing \eqref{Eq:ident} for $Y$ and $Z$, and in virtue of the isomorphism $X_{\bar y} \cong (X_Z)_{\bar z}$, we get a commutative diagram
		\[
		\begin{tikzcd}
		{(\NS X_{\overline{y}} \otimes \Q)^{\pi_1(Y,\overline{y})}} \arrow[r] \arrow[r] \arrow[d] & {H^2(X_{\overline{y}},\mu_{\infty})^{\pi_1(Y,\overline{y})}} \arrow[d] \arrow[r] & \Br_{\text{hor}}(X/Y) \arrow[d, "p^{*}"] \arrow[r] & 0 \\
		{(\NS X_{\overline{y}} \otimes \Q)^{\pi_1(Z,\bar z)}} \arrow[r]                 & {H^2(X_{\overline{y}},\mu_{\infty})^{\pi_1(Z,\bar z)}} \arrow[r]       & \Br_{\text{hor}}(X_Z/Z) \arrow[r]              & 0
		\end{tikzcd}
		\]
		The action of $\pi_1(Y,\overline{y})$ on $H^2(X_{\overline{y}}, \mu_{\infty})$ factors through a finite quotient by \cite[Lemma 4.2]{HW14} (in summary because $H^2(X_{\overline{y}}, \mu_{\infty})$ is a finite extension of the finite group $\Br X_{\overline{y}}$ by $\NS(X_{\overline{y}}) \otimes \qz$, on which $\pi_1(Y,\overline{y})$ acts through a finite quotient as $\NS(X_{\overline{y}})$ is finitely generated and the action is continuous, see {\em loc.cit.} for more details), which corresponds to a finite irreducible Galois étale cover $\pi:Y' \to Y$, say with group $G$. The assumption that $Y' \times_Y Z$ is integral implies that the composition  $\pi_1(Z,\overline{y}'') \to \pi_1(Y,\overline{y}) \to G$ is surjective. We conclude with a diagram chase.
	\end{proof}
	
	We go back to the proof of points (ii) and (iii) of Theorem \ref{Surj}.
	
	\begin{proof}[Proof of (ii)]
		This follows from Lemma \ref{Lem10}, taking $Z$ to be a point. 
	\end{proof}
	
	\begin{proof}[Proof of (iii)]
		The well-known argument used to prove that the Brauer groups of rationally connected varieties are finite modulo constants (see e.g.\ \cite[Sec.\ 5.5]{BGbook}) works with minimal modifications.
		Namely, point (ii) provides an identification $\Br_{\text{hor}}(X/Y)=\Br_{\text{hor}}(X_{\eta}/\eta)$. The hypercohomology spectral sequence 
		$$
		H^i(\kappa(\eta),H^j(\tau_{\geq 1} R^\bullet f_*\G_m)) \Rightarrow H^{i+j}(\kappa(\eta),\tau_{\geq 1} R^\bullet f_*\G_m)
		$$
		yields an exact sequence
		\begin{equation}\label{SES0}
		0 \to H^1(\kappa(\eta), \Pic X_{\overline{\eta}}) \to \Br_{\text{hor}}(X_{\eta}/\eta) \to \Br X_{\overline{\eta}}.
		\end{equation}
		As $X_{\eta}$ is rationally connected, the group $\Pic X_{\overline{\eta}}$ is finitely generated and torsion-free, making the cohomology group $H^1(\kappa(\eta), \Pic X_{\overline{\eta}})$ finite \cite[Theorem 5.5.1]{BGbook}. Again by the rational connectedness of $X_{\eta}$, the group $\Br X_{\overline{\eta}}$ is also finite \cite[Theorem 5.5.2]{BGbook}.
	\end{proof}
	
	\subsection{The $r-\bar \partial$ complex}\label{SSec5.2}
	In this subsection, let $f:X \to Y$ be a smooth proper morphism with geometrically integral fibers such that $R^1f_*\mu_{\infty}=R^2f_*\cO_X=0$. Assume that $Y$ is a smooth integral variety over a field $k$ of characteristic $0$, and let $\eta$ be its generic point. Recall that we have maps
	\[
	r:\Br X_{\eta} \to \Br_{\text{hor}}(X/Y), \ \ \partial: \Br_{\text{hor}}(X/Y) \to H^3(Y,\G_m).
	\]
    We also let $\bar \partial$ be the composition
    \[
    \bar \partial: \Br_{\text{hor}}(X/Y) \to H^3(Y,\G_m)\to H^3(Y \otimes_k \ok,\G_m)^{\Gamma_K}.
    \]

	\begin{lemma}\label{Prop612}
		If the restriction
		\begin{equation}\label{EqRestr}
		H^3(Y \otimes_k\ok,\G_m) \to H^3(\eta \otimes_k \ok,\G_m)
		\end{equation}
		is injective, then the two morphisms
        \begin{equation}\label{erdelta}
            \Br X_{\eta} \to[r] \Br_{\text{hor}}(X/Y) \to[\bar \partial] H^3(Y\otimes_k \ok,\G_m)^{\Gamma_K}
        \end{equation}
        form a complex. The restriction \eqref{EqRestr} is injective if $Y\otimes_k \ok$ is a $\ok$-torus.
	\end{lemma}
	\begin{proof}
		By the naturality of $\partial$, the image $\partial(r(\Br X_{\eta}))$ lies in the kernel of \linebreak $H^3(Y, \G_m) \to H^3(\eta,\G_m)$. Any element of this kernel must map to zero also in $H^3(\eta \otimes_k \ok,\G_m)$. So, if $H^3(Y \otimes_k \ok,\G_m) \to H^3(\eta \otimes_k \ok,\G_m)$ is injective, it maps to zero in $H^3(Y\otimes_k \ok,\G_m)$ as well.
		
		For the second part, let $d$ be a natural number, and $t_1,\ldots,t_d$ be algebraically independent variables. For a field $F$, we denote by $F((t))$ its field of Laurent series.
		Gille and Pianzola \cite[Proposition 3.1(2)]{GP} prove that the composition
		\[
		H^n(\ok[t_1^{\pm1},\ldots, t_d^{\pm1}], \G_m) \to H^n(\ok (t_1,\ldots,t_d), \G_m) \to H^n(\ok((t_1))\cdots ((t_d)), \G_m),
		\]
		is an isomorphism for all $n \geq 2$. In particular, the first map is injective. Setting $n=3$ and choosing a $\ok$-isomorphism $Y \otimes_k \ok \cong \Spec \ok[t_1^{\pm1}, \ldots,t_d^{\pm1}]$ concludes the proof.
	\end{proof}
	
	\begin{remark}\label{Rmk:variousparts}
        In Section \ref{Sec8}, we shall reduce the general case of Theorem \ref{Thm: fibrationnew} to the case where \eqref{erdelta} is exact. When guaranteeing Brauer--Manin orthogonality of suitable adelic points on the fibers of $f$, we shall first obtain Brauer--Manin orthogonality to $\Br X_{\eta}$ by combining Shafarevich's ideas with the study of half-spin symbols (this is done in Section \ref{Sec9} and \ref{Sec10}, see in particular Proposition \ref{Proprop} for an explicit expression of the Brauer--Manin pairing associated to $\Br X_{\eta}$ for a certain suitable choice of adelic points on the fibers), and afterwards we shall use the triple variation method mentioned in the introduction to deal with the cokernel of $r:\Br X_{\eta} \to \Br_{\text{hor}}(X/Y)$ (see Section \ref{Sec11}, and the closing argument in Section \ref{Sec12}).
	\end{remark}
	
	\section{Representatives for $\Br X_{\eta}/f^*\Br \eta$}\label{Sec6}

	For a smooth proper morphism $f:X \to Y$ with geometrically integral fibers and satisfying $R^1f_*\mu_{\infty}=R^2f_*\cO_X=0$, we investigate the following two questions:
	\begin{itemize}
		\item Given an element of $\Br X_{\eta}/f^*\Br \eta$, can a representative $b \in \Br f^{-1}(U)$ be found for some {\em explicit} dense open $U \subset Y$? 
		\item If yes, can any control be attained on the residues of such a $b$? 
	\end{itemize}
	
	The main result of this section, Theorem \ref{Teo1}, provides partial positive answers. Theorem \ref{Teo1} is formulated in Subsection \ref{SSec6.0}, proven in Subsection \ref{SSec6.2}, and applied to the setting of fibrations over Galois quasi-trivial tori in Subsection \ref{SSec6.3}. In Subsection \ref{SSec6.5}, we compute the local Brauer pairing associated to the representative just obtained. The resulting expression reveals ``local half-spin symbols'' arising, which was our motivation behind their treatment in Section \ref{Sec4}.
	\vskip3mm
	
	\subsection{Collections of augmented divisors and main theorem}\label{SSec6.0}
	
	\begin{definition}\label{DefResidueAlgebras}
		For a regular integral Noetherian scheme $Y$, and pairwise distinct integral divisors $Z_1,\ldots,Z_r \subset Y$, a {\em collection of residue algebras} $\mathfrak A$ on the divisors $Z_i$'s is a collection of finite étale algebras $\mathfrak A \coloneqq \{A_1/\kappa(Z_1), \ldots, A_r/\kappa(Z_r)\}$, where each $A_i$ is a finite étale algebra over the function field $\kappa(Z_i)$ of $Z_i$.
	\end{definition}
	
	Collections of residue algebras behave contravariantly with respect to dominant morphisms, in the following sense. For a dominant morphism $f:Y' \to Y$ of regular integral Noetherian schemes, and a collection $\mathfrak A=\{A_i/\kappa(Z_i)\}_{1 \leq i \leq r}$ on integral divisors $Z_1,\ldots,Z_r$ of $Y$ such that the image $f(Y')$ is not contained in $\cup_i Z_i$, we define the {\em pullback} of $\mathfrak A$ under $f$ as
	\[
	f^{-1}(\mathfrak A)\coloneqq \{A'_{i,j}/\kappa(Z'_{i,j})\}_{1 \leq i \leq r, 1 \leq j \leq r_i},
	\]
    where, for each $i$, $\cup_{j=1}^{r_i} Z'_{i,j}$ is the decomposition into irreducible components of the reduced (possibly empty) divisor $f^{-1}(Z_i)_{\text{red}}$, and $A'_{i,j} \coloneqq  \kappa(Z'_{i,j}) \otimes_{\kappa(Z_i)} A_i$.

    \vskip1mm
	
	(Occasionally, we shall denote the pullback $f^{-1}(\mathfrak A)$ by $\mathfrak A \times_Y Y'$.)
	
	\medskip
	
	In this subsection, given $Y$, $Z_1,\ldots,Z_r$ and $\mathfrak A$ as in Definition \ref{DefResidueAlgebras}, we denote, for each $i$, by $\iota_i:\Spec \kappa(Z_i) \hookrightarrow Y$ the embedding of the generic point of $Z_i$ in $Y$, and by $\phi_i$ the composition $\Spec A_i\to \Spec \kappa(Z_i) \to[\iota_i] Y$. We also denote by $U$ the complement $Y \s \cup_iZ_i$, and by $j:U \to Y$ the open embedding. 
	
	\vskip1mm
	
	When $Y$ is a curve, the first point in the following definition appears already in Section 2 of \cite{HW14}.
	
	\begin{notation*}
		For a collection of residue algebras $\mathfrak A=\{A_1/\kappa(Z_1), \ldots, A_r/\kappa(Z_r)\}$ on divisors $Z_1,\ldots,Z_r$ of a regular integral Noetherian scheme $Y$, we define:
		\begin{itemize}
			\item the group $(\Br Y)_+(\mathfrak A)$ as the subgroup of $\Br U$ of those elements whose residue at each $Z_i$ belongs to the kernel of the pullback $H^1(\kappa(Z_i),\qz) \to  H^1(A_i,\qz)$;
			\item the complex $\G_{m+}(\mathfrak A) \in \cD^+(\Sh_Y)$ of {\em $\mathfrak A$-augmented invertible functions} as the complex
			\begin{equation}\label{Eq11}
			j_*\G_m \to[u] \oplus_i \phi_{i,*}{\Z},
			\end{equation}
			supported in degrees $0$ and $1$, where $u$ is defined as the composition of the natural morphism $j_*\G_m \to \oplus_i \iota_{i,*}\Z$ with the natural embeddings $\iota_{i,*} {\Z} \hookrightarrow \phi_{i,*}{\Z}$.
		\end{itemize}
	\end{notation*}
	
	\vskip1mm
	
	We denote by $v:\G_m \to \G_{m+}$ the natural morphism in $\cD^+(\Sh_Y)$ induced by the commutative diagram
	\begin{equation}\label{v}
	\begin{tikzcd}
	\G_m \arrow[d] \arrow[r] & 0 \arrow[d]             \\
	j_*\G_m \arrow[r, "u"]   & {\oplus_i \phi_{i,*}\Z}.
	\end{tikzcd}
	\end{equation}

    For the following theorem, recall that we have morphisms
    \[
    r:\Br X_{\eta} \to \Br_{\text{hor}}(X/Y),\, \partial:\Br_{\text{hor}}(X/Y) \to H^3(Y,\G_m),
    \]
    for smooth proper morphisms $f:X \to Y$ such that $R^1f_*\mu_{\infty}=R^2f_*\cO_X=0$.
	
	\begin{theorem}\label{Teo1}
		Let $Y$ be a smooth geometrically integral variety over a field $k$ of characteristic $0$, and $f:X \to Y$ be a smooth proper morphism satisfying $R^1f_*\mu_{\infty}=R^2f_*\cO_X=0$. Let $Z_1,\ldots,Z_r$ be smooth integral divisors on $Y$ such that $\cup_iZ_i$ is smooth, and $\mathfrak A=\{A_i/\kappa(Z_i)\}_{i \in I}$ be a collection of residue algebras. Let $\beta \in \Br X_{\eta}/f^*\Br \eta$ be such that $\partial (r(\beta)) \in H^3(Y,\G_m)$ lies in the kernel of
		\[
		H^3(Y,\G_m) \to[v] H^3(Y,\G_{m+}).
		\]
		Then there exists $b \in (\Br X)_+(f^{-1}(\mathfrak A))$ representing $\beta$.
	\end{theorem}

	\subsection{Proof of Theorem \ref{Teo1}}\label{SSec6.2}
	We start with two general lemmas, the first of which is certainly (at least implicitly) well-known but the author could not find it explicitly stated in this generality in the literature.
	
	\begin{lemma}\label{LemResidues}
		Let $X$ be a regular Noetherian scheme of residual characteristic $0$, $Z$ be an integral smooth divisor of $X$, $\iota:\Spec \kappa(Z) \hookrightarrow X$ be the embedding of its generic point in $X$, and $U$ be a non-empty open subscheme of $X \s Z$. We denote by $j:U \hookrightarrow X$ the open embedding. The diagram
		\[
		\begin{tikzcd}
		{H^2(X,j_*\G_m)} \arrow[d, "j^{*}"] \arrow[r] & {H^2(X,\iota_*\Z)} \arrow[d, "\iota^{*}"] \\
		{H^2(U,\G_m)} \arrow[r, "\partial_Z"]       & {H^2(\kappa(Z),\Z)}                    
		\end{tikzcd}
		\]
		where the upper horizontal map is induced by the natural morphism $j_*\G_m \to \iota_*\Z$, and the bottom horizontal map is the Witt residue (after identifying $H^2(\kappa(Z),\Z)$ with $H^1(\kappa(Z),\qz)$), is commutative.
	\end{lemma}
	\begin{proof}
		
		Let $\zeta$ be the generic point of $Z$. To prove the statement we may pullback the pair $(X,Z)$ to the pair $(\Spec \cO_{X,\zeta}, \Spec \kappa(\zeta))$. In this manner we reduce to the case where $X$ is the spectrum of a DVR, $U$ is its generic point, and $Z$ is its special point. In this special case, a proof may be found within the proof of Theorem 3.6.1 of \cite{BGbook}, see in particular the argument in p.86 of {\em ibid.}.
	\end{proof}
	
	\begin{lemma}\label{LemGysinVanishing}
		Let $Z$ be a divisor of a smooth variety $Y$ defined over a field of characteristic $0$, $U \coloneqq Y \s Z$ be the complement of $Z$, and $j:U \hookrightarrow Y$ be the open embedding. If $Z$ is smooth, then $R^nj_*\G_m=0$ for all $n \geq 1$.
	\end{lemma}
	\begin{proof}		
		The étale sheaf $R^nj_*\G_m$ is the sheafification in $Y_{\et}$ of the presheaf $Y' \mapsto H^n(U \times_YY',\G_m)$ \cite[III.1.13]{LECcompleto}. For $n=1$, this presheaf is $0$ by Hilbert Theorem 90. For $n \geq 2$, this is a torsion presheaf by \cite[Lemma 3.5.3]{BGbook}. So, for all $n \geq 2$, $(R^nj_*\G_m) \otimes \Q=0$ and then $R^nj_*\G_m \cong R^nj_*\mu_{\infty}$ by the ``$\infty$-Kummer sequence'' (as used in Subsection \ref{SSec5.1}). 
		Finally, since $Z$ is smooth, we have $R^nj_*\mu_{m}=0$ for all $m\geq 1,\,n \geq 2$ by Gysin purity \cite[VI.5.1]{LECcompleto}, and so $R^nj_*\mu_{\infty}=0$ by taking the limit.
	\end{proof}
	
	We also need the following.
	
	\begin{proposition}\label{Prop61}
		For a smooth variety $Y$ over a field of characteristic $0$, and a collection of residue algebras $\mathfrak A=\{A_i/\kappa(Z_i)\}_{i=1,\ldots,r}$ on smooth integral divisors $Z_1,\ldots,Z_r \subset Y$, there is a natural isomorphism
		\[
		H^2(Y,\G_{m+}) \cong (\Br Y)_+.
		\]
		Moreover, if each $A_i$ is integral, there are natural isomorphisms
		\[
		H^0(Y,\G_{m+}) \cong k[Y]^{*}, \quad H^1(Y,\G_{m+}) \cong \Pic Y.
		\]
	\end{proposition}
	\begin{proof}
		Taking hypercohomology of the exact triangle
		\begin{equation}
		\G_{m+} \to j_*\G_m \to \oplus_i \phi_{i,*}\Z \to[+1]
		\end{equation}
		induces a long exact sequence
		\begin{align*}
		\notag 0 \to &H^0(Y,\G_{m+}) \to H^0(Y,j_*\G_m) \to \oplus_i H^0(Y,\phi_{i,*}\Z) \to  \\
		\label{Eq7}&H^1(Y,\G_{m+}) \to H^1(Y,j_*\G_m) \to \oplus_i H^1(Y,\phi_{i,*}\Z) \to \\
		\notag &H^2(Y,\G_{m+}) \to H^2(Y,j_*\G_m) \to \oplus_i H^2(Y,\phi_{i,*}\Z) \to \ldots\  .
		\end{align*}
		
		For each $i$, we have $H^1(Y,\phi_{i,*}\Z)=0$ since the Leray spectral sequence of $\phi_i$ gives an embedding $H^1(Y,\phi_{i,*}\Z) \hookrightarrow H^1(A_i,\Z)$, and the latter vanishes as $A_i$ is a normal algebra. Moreover, by Lemma \ref{LemGysinVanishing}, the Leray spectral sequence induces natural isomorphisms $H^n(Y,j_*\G_m) \cong H^n(U,\G_m)$ for all $n \geq 0$.
		
		Thus the long exact sequence above splits into two exact sequences:
		\begin{equation}\label{Eq9}
		0 \to H^0(Y,\G_{m+}) \to H^0(U,\G_m) \to \oplus_i H^0(Y,\phi_{i,*}\Z) \to H^1(Y,\G_{m+}) \to H^1(U,\G_m) \to 0,
		\end{equation}
		and
		\begin{equation}\label{Eq8}
		0 \to H^2(Y,\G_{m+}) \to H^2(U,\G_m) \to \oplus_i H^2(Y,\phi_{i,*}\Z) \to \ldots\  .
		\end{equation}
		Consider now, for each $i$, the commutative diagram:
		\[
		\begin{tikzcd}
		{H^2(Y,j_*\G_m)} \arrow[d, "j^{*}", "\cong"'] \arrow[r] & {H^2(Y,\iota_{i,*}\Z)} \arrow[d, "\iota_i^{*}", hook] \arrow[r] & {H^2(Y,\phi_{i,*}\Z)} \arrow[d, hook] \\
		{H^2(U,\G_m)} \arrow[r, "\partial_{Z_i}"]   & {H^2(\kappa(Z_i),\Z)} \arrow[r]                               & {H^2(A_i,\Z)}     ,                   
		\end{tikzcd}
		\]            
		where the third vertical map is the edge map of the Leray spectral sequences of $\phi_i$,
		the first upper horizontal map is induced by the natural morphism of sheaves $j_*\G_m \to \iota_{i,*}\Z$, the map $\partial_{Z_i}$ is the Witt residue map (after identifying $H^2(\kappa(Z_i),\Z)$ with $H^1(\kappa(Z_i),\qz)$), and the second bottom horizontal map is pullback along $\Spec A_i \to \Spec \kappa(Z_i)$. The commutativity of the first square follows from Lemma \ref{LemResidues}. Since $\iota_i$ is an embedding, the pullback $\iota_i^{*}:H^2(Y,\iota_{i,*}\Z) \to H^2(\kappa(Z_i),\Z)$ coincides with the edge map of the Leray spectral spectral sequence of $\iota_i$. Hence the commutativity of the  second square follows by the naturality of the Leray spectral sequence. The injectivity of the second and third vertical maps comes from Leray's spectral sequence and  the vanishings $R^1\phi_{i,*}\Z=R^1\iota_{i,*}\Z=0$ (one has in general $R^1f_*\Z=0$ for all morphisms of schemes $f:S \to S'$ with normal $S$). 
		
		The composition $H^2(U,\G_m) \to[\cong] H^2(Y,j_*\G_m) \to H^2(Y,\iota_{i,*}\Z) \to H^2(Y,\phi_{i,*}\Z)$ in the diagram is the morphism $H^2(U,\G_m) \to H^2(Y,\phi_{i,*}\Z)$ appearing in \eqref{Eq8}. Hence, sequence \eqref{Eq8} identifies $H^2(Y,\G_{m+})$ with the kernel of the composition
		\[
		H^2(U,\G_m) \to[(\partial_{Z_i})] \oplus_i H^2(\kappa(Z_i),\Z) \to \oplus_i H^2(A_i,\Z),
		\]
		proving that $H^2(Y,\G_{m+})$ is naturally isomorphic to $(\Br Y)_+$.
		
		Confronting \eqref{Eq9} with the analog exact sequence in the degenerate case where $A_i=\kappa(Z_i)$ for all $i$, we obtain a commutative diagram:
		\begin{equation}\label{Eq12}
		\begin{tikzcd}[column sep=tiny]
		0 \arrow[r] & {H^0(Y,\G_{m})} \arrow[r] \arrow[d, "v"] & {H^0(U,\G_m)} \arrow[r] \arrow[d, "id"] & {\oplus_i H^0(Y,\iota_{i,*}\Z)} \arrow[r] \arrow[d] & {H^1(Y,\G_{m})} \arrow[r] \arrow[d, "v"] & {H^1(U,\G_m)} \arrow[r] \arrow[d, "id"] & 0 \\
		0 \arrow[r] & {H^0(Y,\G_{m+})} \arrow[r]               & {H^0(U,\G_m)} \arrow[r]                 & {\oplus_i H^0(Y,\phi_{i,*}\Z)} \arrow[r]            & {H^1(Y,\G_{m+})} \arrow[r]               & {H^1(U,\G_m)} \arrow[r]                 & 0
		\end{tikzcd}
		\end{equation}
		where $v$ denotes the map of complexes \eqref{v}. We have $H^0(Y,\iota_{i,*}\Z) \cong H^0(\kappa(Z_i),\Z)\cong \Z$. Since $A_i$ is integral, we also have $H^0(Y,\phi_{i,*}\Z) \cong H^0(A_i,\Z) \cong \Z$. The natural map $H^0(Y,\iota_{i,*}\Z) \to H^0(Y,\phi_{i,*}\Z)$ then gets identified with the identity $id:\Z \to \Z$. 
		The five lemma applied to the diagram \eqref{Eq12} gives isomorphisms $H^i(Y,\G_{m}) \cong H^i(Y,\G_{m+})$ for $i=0,1$, concluding the proof of the first two isomorphisms.
	\end{proof}
	
	\begin{proof}[Proof of Theorem \ref{Teo1}]         
		The pullback $\mathfrak A \times_YX$ is $\{A'_i/\kappa(Z'_i)\}$, where, for each $i$, $Z'_i$ denotes the pullback $f^{-1}(Z_i)$, and $A'_i$ denotes the finite étale $\kappa(Z'_i)$-algebra $\kappa(Z'_i) \otimes_{\kappa(Z_i)}A_i$. For each $i$, we denote by $\phi'_i$ the morphism $\Spec A_i \to X$. We let $U' \coloneqq f^{-1}(U)$, and $j':U'\hookrightarrow X$ be the open embedding.
		
		To keep the notation compact, we denote by $[A^\bullet  \to B^\bullet ]$ the mapping cone of a morphism $A^\bullet  \to B^\bullet $ of chain complexes in some abelian category $\cC$.
		Consider the morphism of exact triangles
		\begin{equation*}\label{DT2}
		\hspace{-15mm}\begin{tikzcd}[column sep=tiny]
		\G_m \arrow[d, "v"] \arrow[r]                   & Rf_*\G_m \arrow[d, "v"] \arrow[r]                               & \tau_{\geq 1}Rf_*\G_m \arrow[d, "v"] \arrow[r, "+1"]                                            & {} \\
		{[j_*\G_m \to \oplus_i \phi_{i,*}\Z]} \arrow[r] & { [Rf_*(j'_*\G_m) \to \oplus_i Rf_*(\phi'_{i,*}\Z)] } \arrow[r] & { [\tau_{\geq 1}Rf_*(j'_*\G_m) \to \oplus_i \tau_{\geq 1}Rf_*(\phi'_{i,*}\Z)] } \arrow[r, "+1"] & {},
		\end{tikzcd}
		\end{equation*}
		where the vertical maps are induced by the natural inclusions $\G_m \hookrightarrow j_*\G_m$ and $\G_m \hookrightarrow j'_*\G_m$, and the exactness of the last row follows from the identifications $j_*\G_m=f_*(j'_*\G_m),$ $\phi_{i,*}\Z=f_*\phi'_{i,*}\Z$.
		Taking hypercohomology on $Y$, we get a commutative diagram with long exact rows, of which a fragment reads:
		\begin{equation}\label{Diagra1}
		\begin{tikzcd}
		\ldots \arrow[r] & \Br X \arrow[d, "v"] \arrow[r] & \Br_{\text{hor}}(X/Y) \arrow[r, "\partial"] \arrow[d, "v"] & {H^3(Y,\G_m)} \arrow[d, "v"] \arrow[r] & \ldots    \\
		\ldots \arrow[r] & (\Br X)_+ \arrow[r]            & \Br_{\text{hor}}(X/Y)_+ \arrow[r, "\partial"]              & {H^3(Y,\G_{m+})} \arrow[r]             & {\ldots,}
		\end{tikzcd}
		\end{equation}
		where $\Br_{\text{hor}}(X/Y)_+ \coloneqq H^2(Y,[\tau_{\geq 1}Rf_*(j'_*\G_m) \to \oplus_i \tau_{\geq 1}Rf_*(\phi_{i,*}\Z)])$, and we are using the identification
		\begin{align*}
		H^2(Y,{[Rf_*(j'_*\G_m) \to \oplus_i Rf_*(\phi_{i,*}\Z)] }) &\cong H^2(Y,Rf_*({[j'_*\G_m \to \oplus_i \phi_{i,*}\Z] }))\\
		&\cong H^2(X,{[j'_*\G_m \to \oplus_i \phi_{i,*}\Z] }) 
		\end{align*}    
		given by the natural isomorphism $H^2(Y,Rf_*-) \cong H^2(X,-)$ provided by Grothendieck's theorem on the composition of derived functors and the identification \linebreak $H^2(X,{[j'_*\G_m \to \oplus_i \phi_{i,*}\Z] })\cong (\Br X)_+$ given by Proposition \ref{Prop61}. 
		Consider now the element $r(\beta) \in \operatorname{Br}_{\text {hor }}(X / Y) \subset \operatorname{Br}_{\text {hor }}(X / Y)$. We have $\partial(v(r(\beta)))=v(\partial(r(\beta)))$ $=0$ by hypothesis, so $v(r(\beta))$ comes from an element $b \in (\Br X)_+$. 
		
		The proof is concluded once we show that
		\begin{center}
			$(\star)$ the image of $b$ under the map $r:\Br X_{\eta} \to \Br_{\text{hor}}(X/Y)$ is $r(\beta)$. 
		\end{center}
		By functoriality of the objects involved, to prove $(\star)$ we may restrict $Y$ to any non-empty Zariski-open $V \subset Y$ (replacing $Y$ with $V$ and $X$ with $f^{-1}(V)= V \times_YX$). Choosing $V \subset Y$ to be any non-empty Zariski-open that is contained in $U$, we reduce to the case where $\mathfrak A$ is empty, in which case the two rows in diagram \eqref{Diagra1} coincide and $v$ is the identity, and the claim becomes tautologically true.
	\end{proof}
	
	\subsection{Application to fibrations over Galois quasi-trivial tori}\label{SSec6.3}
	
	In this subsection, $k$ denotes a field of characteristic $0$ whose every finite extension contains only finitely many roots of unity (e.g.\ a number field).
	
	\begin{notation*}
		For the Galois group $\Gamma$ of a finite Galois extension $E/k$, we let:
        \begin{align*}
        &\Gamma^* := \Gamma \s \{\id\}, \\
        &\Gamma[2]^*:= \{\tau \in \Gamma^*: \tau^2=\id\},\\
		&\Gamma[2]^{**}:= \begin{cases}
		\Gamma[2]^*& \text{ if }8\nmid n,\\
		\{\tau \in \Gamma[2]^*: \chi(\tau) \equiv \pm 1 \bmod\,8\}& \text{ if }8\mid n,
		\end{cases}
        \end{align*}
		where $n=n(E)$ denotes the largest natural number such that $\mu_n \subset E$, and $\chi:\Gamma \to (\Z/n\Z)^*$ denotes the cyclotomic character. 
	\end{notation*}
	
	\begin{notation*}
		For a Galois quasi-trivial $k$-torus $Q=R_{E/k}\G_m=\Spec (E[(y^{\sigma})^{\pm1}]_{\sigma \in \Gamma})^{\Gamma}$, and $\tau \in \Gamma[2]^*$, we let:
        \begin{itemize}
			\item $\bar Z_{\tau}(Q)$ be the divisor of $Q \otimes_K \oK$ defined by the equation $y=y^{\tau}$;
			\item $Z_{\tau}(Q) \subset Q$ be the unique integral divisor such that $Z_{\tau}(Q) \otimes_k \ok$ is the Galois-orbit of $\bar Z_{\tau}(Q)$.
		\end{itemize}
	\end{notation*}
	
	We shall often write $\bar Z_{\tau},Z_{\tau}$ instead of $\bar Z_{\tau}(Q),Z_{\tau}(Q)$. The field of definition of $\bar Z_{\tau}(Q)$ is $E^{\tau}$, and thus the function field of $Z_{\tau}$ coincides with that of its geometrically integral component $\{y=y^{\tau}\}$ viewed as a divisor over $Q \otimes_K E^{\tau}$. The latter function field is the field of fractions of $E^{\tau}[y^{\sigma}]_{\sigma \in \Gamma}/(y-y^{\tau})$. We shall use the induced identification $K(Z_{\tau})=\Frac(E^{\tau}[y^{\sigma}]_{\sigma \in \Gamma}/(y-y^{\tau}))$ below, especially to view $y$ as an element of $K(Z_{\tau})$.
	
	\vskip1mm
		
	For $\tau \in \Gamma[2]^*$, let $n_{\tau}:=\max\{r\mid n : \mu_r \subset E^{\tau}\}$. Let also $Q^\sharp:= Q \s \bigcup_{\tau \in \Gamma[2]^{**}}Z_{\tau}$.
	
	\begin{theorem}\label{Teo3}
		Let $f:X \to Q$ be a smooth proper morphism with geometrically integral fibers over a Galois quasi-trivial torus $Q=R_{E/k}\G_m$ and such that $R^1f_*\mu_{\infty}=R^2f_*\cO_X=0$. Every class in $\Br X_{\eta}/f^*\Br \eta$ has a representative $b \in \Br f^{-1}(Q^\sharp) \subset \Br X_{\eta}$ such that, for all $\tau \in \Gamma[2]^{**}$:
		\begin{itemize}
			\item[$(\star)$] the residue $\partial_{Z'_{\tau}}(b) \in H^1(k(Z'_{\tau}),\qz)$ at $Z'_{\tau} \coloneqq f^{-1}(Z_{\sigma})$ is split by the cyclic extension $k(Z'_{\tau})(y^{\frac 1 {n_{\tau}}})/k(Z'_{\tau})$.
		\end{itemize}
	\end{theorem}
	
	\vskip1mm
	
	We shall prove Theorem \ref{Teo3} by first applying Theorem \ref{Teo1} to find a representative $b$ whose residue at each $Z'_{\tau}$ is split by $\kappa(Z'_{\tau})(y^{\frac 1 {n}})$. An elementary Galois-theoretic argument, presented at the end of this subsection, then shows that the residues must in fact already be split by $\kappa(Z'_{\tau})(y^{\frac 1 {n_{\tau}}})$.
	
	\subsubsection{Proof of Theorem \ref{Teo3}}	
	Consider the collection of residual algebras
	$
	\mathfrak A \coloneqq \{\kappa(Z_{\tau})(y^{\frac 1 {n}})/\kappa(Z_{\tau})\}_{\tau \in \Gamma[2]^{**}}
	$
	on the divisors $Z_{\tau}, \tau \in \Gamma[2]^{**}$ of $Q$.
	Let $\bar {\mathfrak A} = {\mathfrak A} \otimes_k \ok$ be the base-change of ${\mathfrak A}$ to $\bar Q:=Q \otimes_k \ok$, $(\Br \bar Q)_+ \coloneqq (\Br \bar Q)_+(\bar {\mathfrak A})$, and 
	\begin{equation}\label{BRXPLUS}
		(\Br X)_+:=\Br_+(\mathfrak A \times_QX).
	\end{equation}
	
	\vskip2mm
	
	The following will be our key instrument in verifying the hypotheses of Theorem \ref{Teo1}.
	
	\begin{proposition}\label{Lem5}
		The restriction $H^1(K,(\Br \bar Q)_+)\to H^1(E,(\Br \bar Q)_+)$ is injective.
	\end{proposition}
	
	We prove the proposition by providing an explicit basis of $(\Br \bar Q)_+[n]$ (see Lemma \ref{Lem612}).
	
	\vskip1mm
	
	For each unordered pair $\{\sigma_1,\sigma_2\} \in \binom{\Gamma}{2}$ we choose an ordering $(\sigma_1,\sigma_2)$ of the pair, and denote by $R \subset \Gamma^2$ the set made of the resulting ordered pairs. Let $\binom{\Gamma}{2}_{**}\coloneqq \{\{\sigma_1,\sigma_2\}\in \binom{\Gamma}{2}:\sigma_1\sigma_2^{-1} \in \Gamma[2]^{**}\}, (\Gamma^2)_{**} \coloneqq \{(\sigma_1,\sigma_2)\in \Gamma^2:\sigma_1\sigma_2^{-1} \in \Gamma[2]^{**}\}$, and $R_{**} \coloneqq R \cap (\Gamma^2)_{**}$. 
	
	\begin{lemma}\label{Lem612}
		The group $(\Br \overline Q)_+[n]$ is a free $(\Z/n\Z)$-module. The following sets form bases of its Tate twist $(\Br \overline Q)_+[n](1)$:
		\begin{itemize}
			\item the set $\{(y^{\sigma_2},y^{\sigma_1}-y^{\sigma_2})_n\}_{(\sigma_1, \sigma_2) \in R_{**}}\cup \{(y^{\sigma_2},y^{\sigma_1})_n\}_{(\sigma_1, \sigma_2) \in R}$;                \item the set $\{(y^{\sigma_2},y^{\sigma_1}-y^{\sigma_2})_n\}_{(\sigma_1, \sigma_2) \in (\Gamma^2)_{**}}\cup \{(y^{\sigma_2},y^{\sigma_1})_n\}_{(\sigma_1, \sigma_2) \in R\s R_{**}}$.
		\end{itemize}
		Moreover, we have $(\Br \overline Q)_+^{\Gamma_E}=(\Br \overline Q)_+[n]$.
	\end{lemma}
	\begin{proof}
		We prove that the first set is a basis. Using the identity
		\[
		(y^{\sigma_2},y^{\sigma_1}-y^{\sigma_2})_n-(y^{\sigma_1},y^{\sigma_2}-y^{\sigma_1})_n=(y^{\sigma_2},y^{\sigma_1})_n,
		\]
		which holds for all $(\sigma_1,\sigma_2) \in \Gamma^2, \sigma_1 \neq \sigma_2$ and follows from the relations $(a,-1)_n=0$, $(a,1-a)_n=0$ and $(a,a)_n=0$ (the first and last hold because $\sqrt{-1} \in \ok(Q)$), it is then clear that the second is as well.
		
		For $(\sigma_1,\sigma_2) \in (\Gamma^2)_{**}$, let $b_{\sigma_1,\sigma_2}\coloneqq (y^{\sigma_2},y^{\sigma_1}-y^{\sigma_2})_n$.
		We claim that 
		\begin{equation}\label{eqref}
		(\Br \bar Q)_+(1)= (\Br \bar Q)(1) \oplus \langle b_{\sigma_1,\sigma_2} \rangle_{(\sigma_1,\sigma_2) \in R_{**}},
		\end{equation}
		where in the second summand the elements are $(\Z/n\Z)$-linearly independent.
		
		Let us prove the claim. The pullback $\mathfrak A  \otimes_k \ok $ of the collection $\mathfrak A$ to $\bar Q$ is the collection $\{k(\bar Z_{\tau}^{\sigma})((y^{\sigma})^{\frac 1n})/k(\bar Z_{\tau}^{\sigma})\}_{\tau \in \Gamma[2]^{**},\sigma \in \langle \tau \rangle \s \Gamma}$. Thus $(\Br \bar Q)_+$ is the set of elements $b \in \Br \bar Q^{\sharp}$ whose residue at each $\bar Z_{\tau}^{\sigma}$ is split by $k(\bar Z_{\tau}^{\sigma})((y^{\sigma})^{\frac 1n})$, or, equivalently, whose residue at each $\bar Z_{\tau}^{\sigma}$ is a multiple of the Kummer character $\widehat{y^{\sigma}}:\Gamma_{k(\bar Z_{\tau}^{\sigma})} \to \mu_n \cong \frac1n \Z/\Z$ of order $n$.
		
		For each $(\sigma_1,\sigma_2) \in R_{**}$, the element $b_{\sigma_1,\sigma_2}\coloneqq (y^{\sigma_2},y^{\sigma_2}-y^{\sigma_1})_n$ is unramified over $\bar Q \s \bar Z_{\sigma_1\sigma_2^{-1}}^{\sigma_2}$ and its residue at $\bar Z_{\sigma_1\sigma_2^{-1}}^{\sigma_2}$ is the Kummer character $\widehat{y^{\sigma_2}}$. Combined with the discussion in the previous paragraph, this implies that every $b \in (\Br \bar Q)_+$ may be written uniquely in the form
		\[
		b=b_0+\sum_{(\sigma_1,\sigma_2) \in R_{**}}\alpha_{\sigma_1,\sigma_2} b_{\sigma_1,\sigma_2},
		\]
		with $b_0 \in (\Br \bar Q)$ and $\alpha_{\sigma_1,\sigma_2} \in (\Z/n\Z)(-1)$ (by taking $\alpha_{\sigma_1,\sigma_2}$ to be the unique class such that $\partial_{Z_{\sigma_1,\sigma_2}}(b)=\alpha_{\sigma_1,\sigma_2}\widehat{y^{\sigma_2}}$). This proves the claim.
		
		The group $(\Br \bar Q)[n]$ is the free $(\Z/n\Z)$-module with basis formed by the elements $(y^{\sigma_2},y^{\sigma_1})_n$ as $(\sigma_1,\sigma_2)$ varies in $R$, see e.g.\  \cite[Proposition 9.1.2]{BGbook}. Now \eqref{eqref} gives the basis (i) for $(\Br \bar Q)_+$.
		
		As for the last statement, taking $\Gamma_E$-invariants in \eqref{eqref} and noting that each $b_{\sigma_1,\sigma_2}$ is invariant under the action of $\Gamma_E$ and has order dividing $n$, it suffices to prove that $(\Br \bar Q)^{\Gamma_E}=(\Br \bar Q)[n]$. Since the torus $Q_E$ is split, the $\Gamma_E$-module $\Br \bar Q$ is isomorphic to $(\qz(-1))^{d(d-1)/2}, d = [E:k]$ (see again \cite[Proposition 9.1.2]{BGbook}), and so $(\Br \bar Q)^{\Gamma_E}=(\Br \bar Q)[n]$ as $n$ was defined as the maximum order of a root of unity that is contained in $E$.
	\end{proof}
	
	We also need the following simple fact.
	
	\begin{lemma}\label{Lem7}
		Let $n$ be a natural number and $a\in (\Z/n\Z)^{*}$ be such that $a^2 = 1$. Let $C_2$ be the group with $2$ elements, and $M$ be the $C_2$-module $\Z/n\Z$ where the generator of $C_2$ acts via multiplication by $a$. Then
		\[
		H^1(C_2,M)=\begin{cases}
		0 & \text{if }n\text{ is odd};\\
		\Z/2\Z & \text{if }v_2(n)\leq 2;\\
		\Z/2\Z & \text{if }v_2(n)\geq 3\text{ and }a \equiv \pm 1 \bmod 8;\\
		0 & \text{if }v_2(n)\geq 3\text{ and }a \equiv \pm 3 \bmod 8.
		\end{cases}
		\]
	\end{lemma}
	\begin{proof}
		We may assume that $n=p^m$ is a prime-power. The case $a=1$ is immediate, while $a \neq 1$ is a special case of \cite[Lemma 9.1.4]{GermanBook}.
	\end{proof}
	
	\begin{proof}[Proof of Proposition \ref{Lem5}]
		We prove that $H^1(\Gamma,(\Br \bar Q)_+^{\Gamma_E})=0$. By the inflation-restriction sequence, this is equivalent to the statement. Let $M\coloneqq (\Br \bar Q)_+^{\Gamma_E}$. For $(\sigma_1,\sigma_2) \in (\Gamma^2)_{**}$, let $B_{\sigma_1,\sigma_2} \coloneqq \langle\, (y^{\sigma_2},y^{\sigma_1}-y^{\sigma_2})_n \,\rangle < M$. For $\{\sigma_1,\sigma_2\} \in \binom{\Gamma}{2}$, let $C_{\{\sigma_1,\sigma_2\}} \coloneqq \langle\, (y^{\sigma_1},y^{\sigma_2})_n \,\rangle < M$. 
		By Lemma \ref{Lem612}, we have
		\[
		M=\bigoplus_{(\sigma_1,\sigma_2) \in (\Gamma^2)_{**}} B_{\sigma_1,\sigma_2} \oplus \bigoplus_{\{\sigma_1,\sigma_2\} \in \binom{\Gamma}{2} \s  \binom{\Gamma}{2}_{**}} C_{\{\sigma_1,\sigma_2\}}.
		\]
		The left action of $\Gamma$ on $M$ permutes the summands. More precisely, we have $\gamma(B_{\sigma_1,\sigma_2})=B_{\gamma\sigma_1,\gamma\sigma_2}$ and $\gamma(C_{\{\sigma_1,\sigma_2\}})=C_{\{\gamma\sigma_1,\gamma\sigma_2\}}$. We use the following fact \cite[Corollary III.5.4]{Brown}.
		\begin{center}
			\underline{Fact} Let $G$ be a profinite group, and $M$ a discrete $G$-module whose underlying abelian group is of the form $M=\bigoplus_{i\in I}M_i$. Assume that the $G$-action permutes the factors according to some action of $G$ on $I$. Let for each $i$,  $H_i < G$ be the isotropy subgroup of $M_i$. Let $J$ be a set of representatives for the $G$-action on $I$. Then $M \cong \oplus_{i \in J} \Ind^G_{H_i}M_i$.
		\end{center}
		The natural left $\Gamma$-action on $(\Gamma^2)_{**}$ is free. The left $\Gamma$-action on $\binom{\Gamma}{2} \s  \binom{\Gamma}{2}_{**}$ is free on the subset $\binom{\Gamma}{2} \s  \binom{\Gamma}{2}_{*} \subset \binom{\Gamma}{2} \s  \binom{\Gamma}{2}_{**}$, while the stabilizer of an element $\{\sigma_1,\sigma_2\} \in \binom{\Gamma}{2}_*  \s  \binom{\Gamma}{2}_{**}$ is the subgroup $\{1,\gamma\} < \Gamma,\ \gamma = \sigma_1\sigma_2^{-1}$. So
		\[
		M \cong \Ind^{\Gamma}_eM_0 \oplus \bigoplus_{\gamma \in \Gamma[2]^* \s \Gamma[2]^{**}} \Ind^{\Gamma}_{\{1,\gamma\}}C_{e,\gamma},
		\]
		for some abelian group $M_0$, and where $C_{e,\gamma}$ is the $\{1,\gamma\}$-module $\langle\,  y \cup y^{\gamma}\, \rangle = \linebreak (\Z/n\Z)(-1) \cdot (y \cup y^{\gamma})$, on which $\gamma$ acts via multiplication by $-\chi(\gamma)$, where $\chi$ denotes the cyclotomic character. We have $\gamma \notin \Gamma[2]^{**}$ and so either $n$ is odd or $8 \mid n,\,\chi(\gamma)\equiv \pm 3 \bmod 8$ by definition. Hence $H^1(\{1,\gamma\},C_{e,\gamma})=0$ by Lemma \ref{Lem7} and $H^1(\Gamma,M)=0$ by Shapiro's lemma.
	\end{proof}
	
	\begin{proposition}\label{Prop65}
		We have that 
		\begin{equation*}
		\Ker (H^1(K,\Br \bar Q) \to H^1(K,(\Br \bar Q)_+))=\Ker (H^1(K,\Br \bar Q) \to H^1(K,\Br \bar \eta)).
		\end{equation*}
	\end{proposition}
	\begin{proof}
		The inclusions $\Br \bar Q \subset (\Br \bar Q)_+ \subset \Br \bar \eta$ induce a commutative diagram:
		\begin{equation}\label{Diagra}
		\begin{tikzcd}
		{H^1(K,\Br \bar Q)} \arrow[r] \arrow[d, "res"] & {H^1(K,(\Br \bar Q)_+)} \arrow[r] \arrow[d, "res", hook] & {H^1(K,\Br \bar \eta)} \arrow[d, "res"] \\
		{H^1(E,\Br \bar Q)} \arrow[r]                  & {H^1(E,(\Br \bar Q)_+)} \arrow[r]                        & {H^1(E,\Br \bar \eta)}     .            
		\end{tikzcd}
		\end{equation}
		The second restriction is injective by Proposition \ref{Lem5}.
		The bottom horizontal composition is also injective. In fact, for any split torus $T=\Spec k[t_1^{\pm1},\ldots,t_r^{\pm1}]$ over a field $k$ of characteristic $0$ with generic point $\eta$, the natural maps $H^n(T \otimes_k \ok, \G_m) \to H^n(\eta \otimes_k \ok, \G_m),\, n\geq 2$ are split injections of $\Gamma_k$-modules by Gille and Pianzola's \cite[Proposition 3.1(2)]{GP} (the retraction is provided by the composition that we already used in the proof of Lemma \ref{Prop612}). Since $Q_E$ is a split $E$-torus, we get that $\Br \bar Q \subset \Br \bar \eta$ is a split injection of $\Gamma_E$-modules by taking $n=2$, and the injectivity follows.
		Now the lemma follows by a diagram chase.
	\end{proof}
	
	We are now ready to start the proof of Theorem \ref{Teo3}.
	
	\begin{proof}[Proof of Theorem \ref{Teo3}, first part.]
		Let $b \in \Br X_{\eta}$. In this first part, we prove that there exists $b_0 \in \Br \eta$ such that $b-f^{*}b_0 \in (\Br X)_+$.
		
		Let $Q^o\subset Q$ denote the complement $Q \s \operatorname{Sing}(Z_{\Gamma[2]^{**}})$ of the singular locus $\operatorname{Sing}(Z_{\Gamma[2]^{**}})$ of $Z_{\Gamma[2]^{**}}:=\bigcup_{\tau \in \Gamma[2]^{**}}Z_{\tau}$.
		We shall apply Theorem \ref{Teo1} to $f|_{f^{-1}(Q^o)}:f^{-1}(Q^o) \to Q^o$.
		
		Let $E^{p,q}_r, \leftidx^+E^{p,q}_r$ and $\leftidx^{\eta}E^{p,q}_r$ denote the $r$-th pages of the three Hochschild-Serre spectral sequences:
		\begin{align*}
		H^p(K,H^q(Q^o \otimes_K \oK,\G_m))     &\Rightarrow H^{p+q}(Q^o,\G_m), \\ 
		H^p(K,H^q(Q^o \otimes_K \oK,\G_{m+}))  &\Rightarrow H^{p+q}(Q^o,\G_{m+}), \\
		H^p(K,H^q(\eta \otimes_K \oK,\G_m))    &\Rightarrow H^{p+q}(\eta,\G_m).
		\end{align*}
		For each of these, we denote by $F_{HS}^pH^{p+q}$ the associated abutment filtration on $H^{p+q}$. Proposition \ref{Prop61} provides natural isomorphisms $\leftidx^+E^{p,q}_2 \cong E^{p,q}_2$ for all $q \leq 2$, induced by the morphism $v:\G_m \to \G_{m+}$. We have $\Pic (Q \otimes_K\oK)=0$ and $H^3(K,\oK [Q]^{*})=0$ (resp.\ $\Pic \bar \eta=0$ and $H^3(K,\oK(\eta))=0$), where the second follows from the vanishings $H^3(K,\oK^*)=0$ and $H^3(K,\widehat Q)\cong H^3(K,\Ind^E_K\Z)=H^3(E,\Z)=0$ which hold because $K$ and $E$ are number fields. These give $E^{3,0}_2= E^{2,1}_2=0$ (resp.\ $\leftidx^{\eta}E^{3,0}_2= \leftidx^{\eta}E^{2,1}_2=0$), and thus the same holds for $\leftidx^+E$. 
		
		Comparing the abutments in position $(p,q)=(1,2)$ of the three sequences, we then get commutative diagrams
		\begin{equation}\label{Diagrammonen}
		\hspace{-10mm}\begin{tikzcd}[column sep=small]
		{F^2_{HS}H^3(Q^o,\G_m)} \arrow[d, "v"] \arrow[r, hook] & {H^1(K,\Br (Q^o \otimes_K\oK))} \arrow[d] \\
		{F^2_{HS}H^3(Q^o,\G_{m+})} \arrow[r, hook]             & {H^1(K,(\Br (Q^o \otimes_K\oK))_+)}      
		\end{tikzcd}
		\begin{tikzcd}[column sep=small]
		{F^2_{HS}H^3(Q^o,\G_m)} \arrow[d] \arrow[r, hook] & {H^1(K,\Br (Q^o \otimes_K \oK))} \arrow[d] \\
		{F^2_{HS}H^3(\eta,\G_m)} \arrow[r, hook]          & {H^1(K,\Br \bar \eta)}                    
		\end{tikzcd}
		\end{equation}
		where the horizontal maps are injective, and the vertical maps in the second diagram are induced by the pullback along $\eta \to Q$.
		
		Consider now the maps
		\[
		r:\Br X_{\eta} \to \Br_{\text{hor}}(f^{-1}(Q^o)/Q^o), \ \ \partial: \Br_{\text{hor}}(f^{-1}(Q^o)/Q^o) \to H^3(Q^o,\G_m)
		\]
		associated with the morphism $f^{-1}(Q^o) \to Q^o$. The element $\partial (r( b)) \in H^3(Q^o,\G_m)$ lies in $F^2_{HS}H^3(Q^o,\G_m)$ by Lemma \ref{Prop612}. Its image in ${F^2_{HS}H^3(\eta,\G_m)} \subset H^3(\eta,\G_m)$ vanishes because $b \in \Br X_{\eta}$. 
		
		Since the complement of $Q^o$ in $Q$ has codimension at most $2$, purity for the Brauer group gives us identities  $\Br (Q^o \otimes_K \oK)=\Br (Q \otimes_k \oK), \Br (Q^o \otimes_K \oK)_+=\Br (Q \otimes_k \oK)_+$ and $(\Br f^{-1}(Q^o))_+=(\Br X)_+$. In view of these identities, Proposition \ref{Prop65} tells us that
		\begin{align*}
		&\Ker (H^1(K,\Br (Q^o \otimes_K \oK)) \to H^1(K,(\Br (Q^o \otimes_K \oK))_+)) \\
		=   &\Ker (H^1(K,\Br (Q^o \otimes_K \oK)) \to H^1(K,\Br \bar \eta)).
		\end{align*}
		A diagram chase on the two diagrams \eqref{Diagrammonen} now shows that $\partial (r( b))$ maps trivially to ${F^2_{HS}H^3(Q^o,\G_{m+})} \subset H^3(Q^o,\G_{m+})$. Theorem \ref{Teo1} then tells us that there exists $b_0 \in \Br \eta$ such that $b-f^{*}b_0 \in (\Br f^{-1}(Q^o))_+=(\Br X)_+$.
	\end{proof}
	
	To conclude the proof, we need the following elementary lemma:
	\begin{lemma}\label{Lemmm}
		Let $F$ be a field of characteristic $0$, and $F':=F(\sqrt[n]{u})$, where $u \in F^{*}$ and $n \in \N$ are such that $T^n-u$ is irreducible over the field $F(\mu_{n})$. Then the maximal abelian subextension of $F'/F$ is $F(\sqrt[m]{u})$, where $m$ is the maximal divisor of $n$ with $\mu_m \subset F^{*}$.
	\end{lemma}
	\begin{proof}
		This is an exercise in Kummer theory. 
	\end{proof}
	
	\begin{proof}[Proof of Theorem \ref{Teo3}, second and final part]
		Applying Theorem \ref{Teo1} we find $b \in (\Br X)_+$ such that $r(b)= \beta$. Recalling the definition of $(\Br X)_+$, we have that $b \in \Br f^{-1}(Q^{\sharp})$ and the residue $r_{\tau} \in H^1(\Gamma_{k(Z'_{\tau})},\qz)$ of $b$ at $Z'_{\tau}$ is split by the extension $k(Z'_{\tau})(y^{\frac1n})/k(Z'_{\tau})$ for each $\tau \in \Gamma[2]^{**}$. Since the splitting field of $r_{\tau}$ is (cyclic and hence) abelian, Lemma \ref{Lemmm} then implies that $r_{\tau}$ is split by $k(Z'_{\tau})(y^{\frac1{m_{\tau}}})$, with $m_{\tau}=\max\{d : d \mid n \text{ and }\mu_d \subset k(Z'_{\tau})\}$. The algebraic closure of $k$ in $k(Z'_{\tau})$ is $E^{\tau}$, and so $m_{\tau}=n_{\tau} = \max\{d : d \mid n \text{ and }\mu_d \subset E^{\tau}\}$, concluding the proof.
	\end{proof}
	
	\subsection{Relation with half-spin symbols}\label{SSec6.5}
	
	In this subsection, let $K$ be a number field, $Q:=R_{E/K}\G_m=\Spec (E[(y^{\sigma})^{\pm1}]_{\sigma \in \Gamma})^{\Gamma}$ be a Galois quasi-trivial $K$-torus, $n$ be the largest natural number such that $\mu_n\subset E$, and $S$ be a finite set of places of $K$ that contains the places dividing $n, \Delta_{E/K}$ or $\infty$.
	We let $\cQ:=\Spec (\cO_{E,S}[(y^{\sigma})^{\pm1}]_{\sigma \in \Gamma})^{\Gamma}$, and let $\cZ_{\tau}$ be the Zariski-closure in $\cQ$ of  $Z_{\tau}$.
	
	\vskip1mm
	
	The following lemma relates the Brauer pairing with classes whose residues are the one that appeared earlier in this section with the (local components of) half-spin symbols.
	
	\begin{lemma}\label{PropRes}
		Let $f:\cX \to \cQ$ be a flat morphism with $\cX$ excellent, regular, integral and Noetherian, and $b \in \Br f^{-1}(\cQ^{\sharp})$ be an element whose residue on each irreducible component $\cW_{\tau}$ of $\cZ'_{\tau}:=f^{-1}(\cZ_{\tau})$ is $a_{\tau}\cdot \hat y \in H^1(k(\cW_{\tau}),\qz)$, for a constant $a_{\tau} \in (\Z/n_{\tau}\Z)(-1)$ that only depends on $\tau$ (and not on the chosen component $\cW_{\tau}$). Then for every $v \notin S$ and $x \in \cX(\cO_v)$ we have
		\[
		\inv_v(b(x))= \sum_{\tau \in \Gamma[2]^{**}}a_{\tau}\sum_{\substack{w \mid v \\ w \in M_{E^{\tau}}}}\addleg{q}{\cI_w}_{n_{\tau}},
		\]
		where $q=f(x)$, and $\cI_w\subseteq \cO_{E^{\tau},w}$ denotes the power of the uniformizer with exponent $\tilde w(q-q^{\tau})$, for a place $\tilde w|w$ of $E$.
	\end{lemma}
	
	(As is tacitly assumed in the statement, the ideal $\cI_w$ is independent of the choice of $\tilde w$, since $\tilde w^{\tau}(q-q^{\tau})=\tilde w(q^{\tau}-q)=\tilde w(q-q^{\tau})$.)
	
	\begin{proof}
		The morphism $x:\Spec \cO_v \to \cX$ factors through the spectrum of the local ring $\cO_{X,\bar x}$ at the special point $\bar x$ of $x$.
		In particular, there is no loss of generality in substituting $\cX$ with $\Spec (\cO_{X,\bar x})$, and we assume that $\cX$ is local affine with special point $\bar x$.
		
		For each $\tau\in \Gamma[2]^{**}$, let $\cQ_{\tau}:=\cQ \otimes \cO_{E^{\tau},S}=\Spec (\cO_{E^{\tau},S}[(y^{\tau})^{\pm1}]_{\tau \in \Gamma})^{\tau}$. In $\cQ_{\tau}$, let $\tilde \cZ_{id,\tau}$ be the divisor defined by the $\tau$-equivariant equation $y=y^{\tau}$. Note that $\tilde \cZ_{id,\tau}\otimes_{\cO_{E^{\tau},S}}\cO_{E,S}=\cZ_{id,\tau}$. The restriction $y|_{\cZ_{id,\tau}}$ is invariant under $\tau$, and so defines by étale-descent a function on $\tilde \cZ_{id,\tau}$, which is invertible because $y$ is.
		
		\vskip1mm
		
		The point $q\in \cQ(\cO_v)$ defines by pullback a point $q_{\tau}:=q\otimes_{\cO_{K,S}}\cO_{E^{\tau},S}\in \cQ_{\tau}(\cO_v \otimes_{\cO_{K,S}}\cO_{E^{\tau},S})$. Let $\bar q_{\tau}\subset \cQ_{\tau}$ be its subscheme of closed points, and $\cO_{\cQ_{\tau},\bar q_{\tau}}:=\bigcap_{\xi \in \bar q_{\tau}}\cO_{\cQ_{\tau},\xi}$ be the associated semi-local stalk.
		Consider now a function $\tilde y_{\tau} \in \cO_{\cQ_{\tau},\bar q_{\tau}}^{\times}$ whose restriction to $\tilde\cZ_{id,\tau} \cap \Spec \cO_{\cQ_{\tau},\bar q_{\tau}}$ is equal to the restriction of $y$ (in the sense of the previous paragraph), and a uniformizer $f_{\tau} \in \cO_{\cQ_{\tau},\bar q_{\tau}}$ for the divisor $\tilde\cZ_{id,\tau} \cap \Spec \cO_{\cQ_{\tau},\bar q_{\tau}}$.
		
		\vskip1mm
		
		In a Zariski-neighbourhood of $\bar q_{\tau}$, the (twisted) Brauer class
		\[
		(f_{\tau},\tilde y_{\tau})_{n_{\tau}} \in \Br (k(\cQ_{\tau}))(1)
		\]
		can ramify only on $\tilde\cZ_{id,\tau}$ (and only if $\bar q \in \tilde\cZ_{id,\tau}$, otherwise it is unramified), with residue the Kummer character $\hat y \in H^1(k(\tilde\cZ_{id,\tau}),\qz)(1)$ of order $n_{\tau}$.
		
		Since the norm of $\tilde\cZ_{id,\tau}$ to $\cQ$ is $\cZ_{\tau}$ and the two divisors have the same function field, the functoriality of residues under corestriction \cite[Prop.~1.4.7]{BGbook} gives that the residue of 
		\[
		\cores_{\cO_{E^{\tau},S}/\cO_{K,S}}(f_{\tau},\tilde y_{\tau})_{n_{\tau}} \in \Br (k(\cQ))(1)
		\]
		at $\cZ_{\tau}$ is also $\hat y \in H^1(k(\tilde\cZ_{id,\tau}),\qz)(1)$ under the identification $k(\cZ_{\tau})=k(\tilde\cZ_{id,\tau})$, and it has no other residues in a neighbourhood of $\bar q$.
		
		So by functoriality of residues under pullback \cite[Theorem 3.7.5]{BGbook}, the residues of $f^*b_0$ with
		\[
		b_0:=\sum_{\tau \in \Gamma[2]^{**}}a_{\tau}\cores_{\cO_{E^{\tau},S}/\cO_{K,S}}(f_{\tau},\tilde y_{\tau})_{w,n_{\tau}} \in \Br k(\cQ)
		\]
		are the pullback of the residues of $b_0$. These pullbacks are none other (by the very choice of $b_0$) than the residues of $b$. Thus $b-f^*b_0$ lies in $\Br \cX$ by purity for the Brauer group. In particular:
		\[
		\inv_v(b(x))=\inv_v(b_0(q))=\sum_{\tau \in \Gamma[2]^{**}}a_{\tau}\sum_{\substack{w \mid v \\ w \in M_{E^{\tau}}}}[f_{\tau}(q_{\tau}),\tilde y_{\tau}(q_{\tau})]_{w,n_{\tau}},
		\]
		where we are using the square brackets to denote the additive analog of the Hilbert symbol (with values in $(\Z/n_{\tau}\Z)(1)$ instead of $\mu_{n_{\tau}}$). Here both $f_{\tau}(q_{\tau})$ and $\tilde y_{\tau}(q_{\tau})$ take their values in $(K_v \otimes_K E^{\tau})^*=\prod_{w\mid v, w \in M_E^{\tau}}(E^{\tau})_w^*$, and the symbol $[-,-]_{w,n}$ is meant as the taking the (additive) Hilbert symbol after projecting both components to $(E^{\tau})_w^*$.

		Since, for each $\tau \in \Gamma[2]^{**}$, $\tilde y_{\tau}(q_{\tau})$ is integral by construction, in the sense that it lies in $(\cO_v \otimes_{\cO_{K,S}}\cO_{E^{\tau},S})=\prod_{w\mid v, w \in M_E^{\tau}}\cO_{E^{\tau},w}^*$ and $w(f_{\tau}(q_{\tau}))$ is the local intersection number $(\cZ_{id,\tau},q_{\tau})_{w}=\tilde w(q-q^{\tau})$ (for $\tilde w$ any place extending $w$ to $E$, which is necessarily unramified), we have 
		\[
		(f_{\tau}(q_{\tau}),\tilde y_{\tau}(q_{\tau}))_{w,n_{\tau}}=\leg{\tilde y_{\tau}(q_{\tau})}{w}_{n_{\tau}}^{\tilde w(q-q^{\tau})}=\leg{q}{w}_{n_{\tau}}^{\tilde w(q-q^{\tau})},
		\]
		giving the sought formula.
	\end{proof}
	
	\section{A multi-section}\label{Sec7}
	
	The purpose of this section is to construct a multi-section of the fibration appearing in Theorem \ref{Thm: fibrationnew}. See Theorem \ref{Lem:ExistenceMultisection} and its Corollary \ref{Cor:ExistenceMultisection}.
	
	\vskip1mm
	
	In this section, let $k$ be a field of characteristic $0$, and $\ok$ be an algebraic closure. 
	
	\vskip1mm
	
	A {\em pointing} of a finite étale algebra $E/k$ is a $k$-homomorphism $\iota: E \to \ok$. Fix for the rest of this section a finite étale algebra $E/k$ and a pointing $\iota$.
	
	Let $Q\coloneqq \Spec \left(\ok[y_{\sigma}^{\pm 1}]_{\sigma \in \Hom_k(E,\ok)}\right)^{\Gamma_k}$ be the quasi-trivial torus associated to $E$.    
	\paragraph{Covers of $Q$.} Let $\cS := \Hom_k(E,\ok)$, and:
	\begin{align*}
	R_{\infty}&\coloneqq   \ok[y_{\sigma}^{\pm \frac 1\infty}]_{\sigma \in \cS}.\\
	\intertext{For $1 \leq n \leq \infty$ and $k \subset L \subset \ok$, we let:}
	R_{L,m}&\coloneqq \left(  \ok[y_{\sigma}^{\pm \frac 1n}]_{\sigma \in \cS}\right)^{\Gamma_L} \subset R_{\infty},\\
	\intertext{and, if $\Gamma_L$ acts trivially on $\cS$, we let:}
	\ \ R_{L,m,(1)}&\coloneqq \left(  \ok[y_{\iota}^{\pm 1},y_{\sigma}^{\pm \frac 1n}]_{\sigma \in \cS\s \{\iota\}}\right)^{\Gamma_L}= L[y_{\iota}^{\pm 1},y_{\sigma}^{\pm \frac 1n}]_{\sigma \in \cS\s \{\iota\}}.
	\end{align*}
	The algebras $R_{L,m}$ (resp.\ the algebras $R_{L,m,(1)}$) form a filtered system when ordered by inclusion. This order is equivalent (in both cases) to the one induced by the partial order defined by $(L,m) \preccurlyeq
	(L',m')$ if and only if $m|m'$ and $ L \subset L'$. We let
	\[
	Q^{\infty} \coloneqq \Spec R_{\infty}, \ \ Q^{L,m} \coloneqq \Spec R_{L,m}, \ \ Q^{L,m,(1)}\coloneqq \Spec R_{L,m,(1)}
	\]
	be the corresponding profinite étale covers of $Q$.
	
	\vskip1mm
	
	\noindent We say that a pair $(L,m)$ is {\em finite} if $m < \infty$ and $L/k$ is a finite extension. Note that 
	$$\lim\limits_{\substack{\to \\ (L,m) \text{ finite}}} R_{L,m} = R_{\infty,\ok}, \ \lim\limits_{\substack{\to \\ (L,m) \text{ finite}}} R_{L,m,(1)} = R_{\infty,\ok,(1)},$$ 
	\begin{equation}\label{EqLimit}
	\lim\limits_{\substack{\leftarrow \\ (L,m) \text{ finite}}} Q^{L,m} = Q^{\infty,\ok}, \ \lim\limits_{\substack{\leftarrow \\ (L,m) \text{ finite}}} Q^{L,m,(1)} = Q^{\infty,\ok,(1)}.
	\end{equation}
	We abbreviate $R_{\infty,\ok}, Q^{\infty,\ok}, \ldots$ to $R_{\infty}, Q^{\infty}, \ldots$.
	
	Let $\bar \Delta_{\sigma,\tau} \subset Q \otimes_k \ok, \sigma,\tau \in \cS, \sigma \neq \tau$ be the divisor defined by the equation $y_{\sigma}=y_{\tau}$. Let $\bar \Delta =\cup_{\sigma,\tau : \sigma \neq \tau} \bar \Delta_{\sigma,\tau}$. This union is Galois-equivariant, let $\Delta \subset Q$ be such that $\Delta \otimes_k \ok =\bar \Delta.$ 
	
	\begin{theorem}\label{Lem:ExistenceMultisection}
		Let $T \to Q$ be an epimorphism of $k$-tori, $X$ be a $T$-variety, and $f:X \to Q$ be a $T$-equivariant morphism. Assume that $f$ is proper with rationally connected fibers. Then there exists a multi-section $s_{\infty}$ as in the diagram:
		\[
		\begin{tikzcd}
		& {Q^{\infty,(1)}} \arrow[d] \arrow[ld, "s_{\infty}"'] \\
		X \arrow[r, "f"] & Q   .                                          
		\end{tikzcd}
		\]
		Moreover, for any such $s_{\infty}$, and any $b \in \Br f^{-1}(Q\s \Delta)$ such that the residue of $b \otimes \ok$ at the divisor $f^{-1}(\bar \Delta_{\sigma,\tau})$ is split by a power of $y_{\sigma}$ (equiv.\ $y_{\tau}$) for all $\sigma \neq \tau$, we have $s_{\infty}^{*}b=0$.            
	\end{theorem}
	
	We use the following in the proof.
	
	\begin{lemma}
		Let $X$ be a rationally connected variety defined over an algebraically closed field $k$ of characteristic $0$, and $C$ be a finite cyclic group acting on $X$. Then $X(k)^C \neq \emptyset$.
	\end{lemma}
	\begin{proof}
		We embed $C$ in $\G_m$. Let $Z$ be the contracted product $Z \times^C \G_m$. The ``second factor projection'' $\pi:Z \to \G_m/C \cong \G_m$ has a section $\sigma$ by the Graber–Harris–Starr Theorem \cite[Theorem 1.1]{GHS}. Consider the $C$-equivariant cartesian diagram:
		\[
		\begin{tikzcd}
		\G_m \arrow[d, "C"] & X \times_k \G_m \arrow[d, "C"] \arrow[l, "pr_2"] \\
		\G_m/C              & Z \arrow[l, "\pi"']                             
		\end{tikzcd}
		\]
		The section $\sigma$ of the bottom row induces a $C$-equivariant section $\sigma'$ of the top row. Projecting this section to $X$, we get a  $C$-equivariant morphism $\G_m \to X$. By the valuative criterion of properness, this morphism extends to a morphism $\P^1 \to X$. Since $0$ and $\infty$ are fixed by $C$, their images in $C$ are too.
	\end{proof}
	
	\begin{proof}[Proof of Theorem \ref{Lem:ExistenceMultisection}]
		We may assume without loss of generality that $k$ is algebraically closed. Then $E=k^d$ for a natural number $d$, and $Q=\Spec k[y_1^{\pm1},\ldots,y_d^{\pm1}]$, $Q^{\infty}=\Spec k[y_1^{\pm\frac1\omega},\ldots,y_d^{\pm\frac1\omega}], Q^{\infty,1}=\Spec k[y_1^{\pm1},y_2^{\pm\frac1\omega},\ldots,y_d^{\frac1\omega}]$. (Here $\cS$ degenerates to just the various projections $k^d \to k$, which we enumerate $1,\ldots,d$. Up to permutation of the coordinates, we may always assume that $\iota$ is the first projection).
		
		We prove the first part. The epimorphism of tori $T \to Q$ admits a multi-section $Q^{\infty} \to T$ in the category of $k$-group schemes. The $T$-action then induces a $Q^{\infty}$-action on $X$ with respect to which the fibration $f:X \to Q$ is $Q^{\infty}$-equivariant. This induces an action of the procyclic group $I \coloneqq \Ker(Q^{\infty} \to Q^{\infty,(1)}) \cong \widehat{\Z}$ on the fibers of $f$, which are rationally connected. Since $X$ is of finite type over $k$, this action must factor through a finite quotient of $I$, and thus there is at least one fixed point $x \in X(k)$ by Lemma \ref{Lem:ExistenceMultisection}. In particular $Q^{\infty}\cdot x=Q^{\infty,(1)} \cdot x$, and this orbit provides the sought multi-section $s_{\infty}:Q^{\infty,(1)} \to X$. 
		
		We now prove the second statement. Let $Q_{\Delta}^{\infty,(1)} \coloneqq Q^{\infty,(1)} \times_Q (Q \s \Delta)$, and let $n$ be the order of $b$. Let us first prove that $s_{\infty}^{*}b$ is contained in the subgroup $\Br Q^{\infty,(1)}$ of $\Br Q_{\Delta}^{\infty,(1)}$. Let $\partial_{ij}(b) \in H^1(\kappa(D_{ij}),\qz)$ be the residue of $b$ at $f^{-1}(\bar \Delta_{ij})$, and write $\partial_{ij}(b)=c_{ij}\cdot \widehat{(y_j)}_{n}$ for all ordered pairs $i <j $ with $c_{ij} $ lying in $ (\Z/n\Z)(-1)$. The algebra
		\[
		b' \coloneqq b - \sum_{i < j}c_{ij}(y_j,y_i-y_j)_{n} \quad \in \Br (X \s f^{-1}(\Delta))
		\]
		has trivial residues on all divisors of $X$. By purity for the Brauer group,  it
		belongs to $\Br X \subset \Br (X \s f^{-1}(\Delta))$. Since the functions $y_j$ with $j \geq 2$ are infinitely divisible on $Q_{\Delta}^{\infty,(1)}$, and $j$ is greater or equal than $2$ for all pairs of indices $(i,j)$ appearing in the sum above, the cyclic algebras $(y_j,y_i-y_j)_{n_{ij}}$ appearing in the sum pull back to the trivial algebra on $Q_{\Delta}^{\infty,(1)}$. Thus $s_{\infty}^{*}b=s_{\infty}^{*}b'$, which lies in $\Br Q^{\infty,(1)}$.    
		Finally, note that the Brauer group of $Q^{\infty,(1)}$, i.e.\ the Brauer group of the algebra $k[y_1^{\pm 1},y_2^{\pm \frac 1 \omega},\ldots, y_d^{\pm \frac 1 \omega}]$, is trivial (as follows e.g.\ by taking the appropriate limit in \cite[Proposition 9.1.2]{BGbook}). This proves that $s_{\infty}^{*}b=0$ and concludes the proof.
	\end{proof}
	
	\begin{corollary}\label{Cor:ExistenceMultisection}
		Let $f$ be as in Theorem \ref{Thm: fibrationnew} with $Q$ Galois, and $B \subset (\Br X)_+$ be a finite subgroup. Then there exists a finite pair $(L,m)$ and a multi-section $s$ as in the diagram
		\[
		\begin{tikzcd}
		& {Q^{L,m,(1)}} \arrow[d] \arrow[ld, "s"'] \\
		X \arrow[r, "f"] & Q                                          
		\end{tikzcd}
		\]
		such that $s^{*}B=0$.  
	\end{corollary}
	\begin{proof}
		The elements of $(\Br X)_+$ lie in $\Br f^{-1}(Q \s \Delta)$ and satisfy the assumption of Theorem \ref{Lem:ExistenceMultisection}. Thus the corollary follows from Theorem \ref{Lem:ExistenceMultisection} by a limit argument, using that $\operatorname{Mor}(Q^{\infty,(1)},X)= \lim\limits_{\substack{\to \\ (L,m) \text{ finite}}} \operatorname{Mor}(Q^{L,m,(1)},X)$ (this identity follows from \eqref{EqLimit}), where $\operatorname{Mor}(S,S')$ denotes the set of morphisms from a scheme $S$ to a scheme $S'$.
	\end{proof}
	
	\section{Reductions}\label{Sec8}

    The following proposition is the main result of this section. Its proof is an application of ``Harpaz--Wittenberg base-change'' (Theorem \ref{Prop:HWbasechange}).

    \begin{proposition}\label{PropReductions}
        Assume Theorem \ref{Thm: fibrationnew} holds when, letting $n:=\max\{r: \mu_r \subset E\}$, the morphism $f:X \to Q$ satisfies the following additional properties:
		\begin{enumerate}[label=$(R_{\arabic*})$]
			\item the extension $E/K$ is Galois;
			\item the horizontal Brauer group $\Br_{\text{hor}}(X/Q)$ has exponent dividing $n$;
			\item the complex $\Br X_{\eta} \to[r] \Br_{\text{hor}}(X/Q) \to[\bar \partial] H^3(\bar Q,\G_m)^{\Gamma_K}$ is exact;
			\item the morphism $f$ admits a multi-section $s:Q^{\oK,n,(1)} \to X$ for which there exists a subgroup $B \subset (\Br  X)_+$ surjecting onto $\Br X_{\eta}/f^*\Br \eta$ with $s^*B=0$.
		\end{enumerate}
        Then Theorem \ref{Thm: fibrationnew} holds.
    \end{proposition}

    \subsection{Reduction from $(R_1)$-$(R_3)$}

    We prove here:

    \begin{proposition}\label{PrimaProp}
        Assume Theorem \ref{Thm: fibrationnew} holds for morphisms $f$ that satisfy, in addition, properties $(R_1)-(R_3)$. Then Theorem \ref{Thm: fibrationnew} holds.
    \end{proposition}

    Before coming to the proof, we introduce, for a general smooth proper morphism $f:X \to Q$ over a quasi-trivial $K$-torus $Q$, the following:
    \begin{align*}
        &\Br_{\text{hor}}(X/Q)':=\Ker (\bar \partial: \Br_{\text{hor}}(X/Q) \to H^3(\bar Q,\G_m)^{\Gamma_K} ),\\
        &H^3(Q,\G_m)_{alg} := \Ker(H^3(Q,\G_m) \to H^3(\bar Q,\G_m)^{\Gamma_K}),\\
        &\partial_{alg}:=\iota \circ \left(\partial|_{\Br_{\text{hor}}(X/Q)'}\right):\Br_{\text{hor}}(X/Q)' \to[\partial] H^3(Q,\G_m)_{alg} \xhookrightarrow{\iota} H^1(K,\Br \bar Q),
    \end{align*}
    where $\iota: H^3(Q,\G_m)_{alg} \hookrightarrow H^1(K,\Br \bar Q)$ denotes the inclusion induced by the Hochschild--Serre spectral sequence $H^i(K,H^j(\bar Q,\G_m)) \Rightarrow H^{i+j}(Q,\G_m)$ in view of the vanishings $\Pic \bar Q=H^3(K,\oK[Q]^*)=0$ (see the proof of Theorem \ref{Teo3}).
    
    \begin{lemma}\label{Firstseq}
        The sequence
        \[
        \Br X_{\eta} \to[r] \Br_{\text{hor}}(X/Q)' \to[\phi \circ \partial_{alg}] H^1(K,\Br \bar \eta),
        \]
        where $\phi$ denotes the natural map $H^1(K,\Br \bar Q) \to H^1(K,\Br \bar \eta)$, is exact.
    \end{lemma}
    \begin{proof}
        The exact sequence $\Br X \to[r] \Br_{\text{hor}}(X/Q) \to[\partial] H^3(Q, \G_m)$ (see \eqref{EqLES}) induces an exact sequence
        \begin{equation}\label{Aseq}
           \Br X \to[r] \Br_{\text{hor}}(X/Q)' \to[\partial_{alg}] H^1(K,\Br \bar Q), 
        \end{equation}
        since $\Ker (\partial|_{\Br_{\text{hor}}(X/Q)'})=\Ker (\partial_{alg})$. 

        \vskip1mm
        
        We may give analogous definitions for $X_{\eta}/\eta$ as the ones we gave right before the lemma for $X/Q$. More precisely, let
        \begin{align*}
            &\Br_{\text{hor}}(X_{\eta}/\eta)':=\Ker (\bar \partial: \Br_{\text{hor}}(X_{\eta}/\eta) \to H^3(\bar \eta,\G_m)^{\Gamma_K} ),\\
            &H^3(\eta,\G_m)_{alg} := \Ker(H^3(\eta,\G_m) \to H^3(\bar \eta,\G_m)^{\Gamma_K}),\\
            & \partial_{alg}:= \partial \circ i:\Br_{\text{hor}}(X_{\eta}/\eta)' \to[\partial] H^3(\eta,\G_m)_{alg} \xhookrightarrow{i} H^1(K,\Br \bar \eta),
        \end{align*}
        where $i: H^3(\eta,\G_m)_{alg} \hookrightarrow H^1(K,\Br \bar \eta)$ denotes the inclusion induced by the Hochschild--Serre spectral sequence, using the vanishings $\Pic \bar \eta=H^3(K,\oK(\eta)^*)=0$. With these definitions, we have an analogous sequence to \eqref{Aseq} for $X_{\eta}/\eta$. Confronting this new sequence with \eqref{Aseq}, we get a natural commutative diagram with exact rows
        \[
        \begin{tikzcd}
        \Br X \arrow[d] \arrow[r]              & \Br_{\text{hor}}(X/Q)' \arrow[d, hook] \arrow[r, "\partial_{alg}"] & {H^1(K,\Br \bar Q)} \arrow[d, "\phi"] \\
        \Br X_{\eta} \arrow[r] \arrow[ru, "r"] & \Br_{\text{hor}}(X_{\eta}/\eta)' \arrow[r, "\partial_{alg}"]       & {H^1(K,\Br \bar \eta)} ,              
        \end{tikzcd}
        \]
        where we added the diagonal map $r$. The second vertical map is injective as it is a restriction of the natural isomorphism $\Br_{\text{hor}}(X/Q) \to[\cong] \Br_{\text{hor}}(X_{\eta}/\eta)$. A diagram chase now concludes the proof.
    \end{proof}

    \begin{lemma}\label{Lemma1}
        Assume $Q$ is Galois. Then $\partial_{alg}(\Br_{\text{hor}}(X/Q)')$ lies in the kernel of the restriction map $H^1(K, \Br \bar Q) \to H^1(E, \Br \bar Q)$ if and only if the complex $\Br X_{\eta} \to[r] \Br_{\text{hor}}(X/Q) \to[\bar \partial] H^3(\bar Q,\G_m)^{\Gamma_K}$ is exact.
    \end{lemma}
    \begin{proof}
        A diagram chase on  \eqref{Diagra} reveals that 
        \[
        \Ker(H^1(K,\Br \bar Q) \to[\phi] H^1(K,\Br \bar \eta))=\Ker(H^1(K, \Br \bar Q) \to[res] H^1(E, \Br \bar Q)).
        \]
        Thus $\partial_{alg}(\Br_{\text{hor}}(X/Q)')$ lies in the kernel of $H^1(K, \Br \bar Q) \to[res] H^1(E, \Br \bar Q)$ if and only if $\phi \circ \partial_{alg}=0$, which, by Lemma \ref{Firstseq}, happens if and only if $r(\Br X_{\eta})=\Br_{\text{hor}}(X/Q)'$.
    \end{proof}

    For the next lemma, consider a field extension $E'/E$ with $E'/K$ Galois. For every $c \in E^*$, we have a corresponding cartesian diagram arising from Harpaz--Wittenberg base-change:
    \begin{equation}\label{HWbasechange}
	\begin{tikzcd}
	X' \arrow[d, "\psi_c"] \arrow[r, "f'"] & Q' \arrow[d, "\phi_c"] \\
	X \arrow[r, "f"]               & Q   .                  
	\end{tikzcd}
    \end{equation}
    
    \begin{lemma}\label{Lem1etrequarti}
        For each $c \in E^*$, the pullback $\phi_c^*:\Br_{\text{hor}}(X/Q)' \to \Br_{\text{hor}}(X'/Q')'$ is an isomorphism.
    \end{lemma}
    \begin{proof}
		Consider the commutative diagram
		\begin{equation}\label{Diagrammon2}
		\begin{tikzcd}
		0 \arrow[r] & \Br_{\text{hor}}(X/Q)' \arrow[r] \arrow[d, "\phi_c^{*}"] & \Br_{\text{hor}}(X/Q) \arrow[r] \arrow[d, "\phi_c^{*}", "\cong"'] & {H^3(Q\otimes_k\ok,\G_m)} \arrow[d, "(\phi_c\otimes_k \ok)^{*}", hook] \\
		0 \arrow[r] & \Br_{\text{hor}}(X'/Q')' \arrow[r]                     & \Br_{\text{hor}}(X'/Q') \arrow[r]                     & {H^3(Q'\otimes_k\ok,\G_m)}                                          
		\end{tikzcd}
		\end{equation}
		The isomorphism in the middle follows from Lemma \ref{Lem10} (see also the sentence after the lemma), since the fibers of $\phi_c$ are geometrically integral, as they are forms of norm $1$ tori. Moreover, the pullback
		\begin{equation}\label{EqInjective}
		(\phi_c\otimes_k \ok)^{*}:H^3(Q \otimes_k \ok,\G_m) \to H^3(Q' \otimes_k \ok, \G_m)
		\end{equation}
		is injective, because $\phi_c \otimes_k \ok:Q' \otimes_k \ok \to Q \otimes_k \ok, x \mapsto c\cdot N_{E'\otimes_k \ok/E\otimes_k \ok}(x)$ is, after composing with the translation on $Q\otimes_k \ok$ by $c^{-1}$, a group-epimorphism of $\ok$-tori with connected kernel (the kernel is a norm $1$ torus), and all group-epimorphisms of $\ok$-tori with connected kernel have a section\footnote{By the duality between $\ok$-tori and free finitely generated $\Z$-modules, this is the dual of the fact that a saturated injection of free finitely generated $\Z$-modules admits a retraction.}. Such section provides a retraction of \eqref{EqInjective}.
		A diagram chase on \eqref{Diagrammon2} now proves the lemma.
    \end{proof}

    \begin{proof}[Proof of Proposition \ref{PrimaProp}]
        This follows by applying Harpaz--Wittenberg base-change (Theorem \ref{Prop:HWbasechange}) with any finite extension of fields $E'/E$ with $E'$ Galois over $K$, containing the roots of unity of order the exponent of $\Br_{\text{hor}}(X/Q)$, and splitting $\partial_{alg}(\Br_{\text{hor}}(X/Q)') \subset H^1(K,\Br \bar Q)$ (in the sense that $\partial_{alg}(\Br_{\text{hor}}(X/Q)')$ maps to zero under the restriction $H^1(K,\Br \bar Q) \to H^1(E',\Br \bar Q)$). Clearly $(R_1)$ and $(R_2)$ hold for $X' \to Q'$. 

        \vskip1mm

        Property $(R_3)$ holds as well for $X' \to Q'$ as Lemma \ref{Lem1etrequarti} and the naturality of $\partial_{alg}$ imply that $\partial_{alg}(\Br_{\text{hor}}(X'/Q')')$ is the image under $\phi_c^*:H^1(K,\Br \bar Q) \to H^1(K,\Br \bar Q')$ of $\partial_{alg}(\Br_{\text{hor}}(X/Q))$. This image restricts to zero in $H^1(E',\Br \bar Q')$ by construction, and thus Lemma \ref{Lemma1} gives $(R_3)$.
    \end{proof}

    \subsection{Reduction to $(R_1)-(R_3)$}

    We prove here:

    \begin{proposition}\label{SecondaProp}
        Assume Theorem \ref{Thm: fibrationnew} holds under $(R_1)-(R_4)$. Then Theorem \ref{Thm: fibrationnew} holds under $(R_1)-(R_3)$.
    \end{proposition}
    
    For the next lemma, consider a field extension $E'/E$ with $E'/K$ Galois. As before we have, for each $c \in E^*$, a corresponding Harpaz--Wittenberg base-change diagram \eqref{HWbasechange}. We let $\Gamma':=\Gal(E'/K)$, and
    \begin{align*}
        &Q'=:\Spec (\oK[(z^{\sigma'})^{\pm1}]_{\sigma' \in \Gamma'})^{\Gamma_K},\\
        &(Q')^{\text{aff}}:=\Spec (\oK[(z^{\sigma'})]_{\sigma' \in \Gamma'})^{\Gamma_K} \cong \A^{d'}_K,\, d':=[E':K],\\
        &D := \left\{\prod_{\sigma' \in \Gamma'} z^{\sigma'}=0\right\} \subset (Q')^{\text{aff}}.
    \end{align*}
    Note that $D$ is an integral divisor, but not necessarily geometrically integral. The function field of $D$ coincides with the function field $E'(z^{\sigma'})_{\sigma' \in (\Gamma')^*}$ of the divisor $\{z=0\}$ of $(Q')^{\text{aff}}\otimes_KE'$. We shall use the induced identification $K(D)=E'(z^{\sigma'})_{\sigma' \in (\Gamma')^*}$ below.
    
    \begin{lemma}\label{ResLemma}
        For each $\tau \in \Gamma[2]^{**},\,a \in (\Z/n_{\tau}\Z)(-1)$, there exists a Brauer element $b_{\tau,a} \in (\Br \eta')$ with
        \begin{align*}
            \partial_{\phi_c^{-1}(Z_{\tau})}(b_{\tau,a}) &=a \cdot  \hat{ y }\in H^1(k(\phi_c^{-1}(Z_{\tau})),\Z/n_{\tau}\Z), \\
            \partial_{Z_{\tau'}}(b_{\tau,a}) &=a \cdot  \hat z \in H^1(k(Z_{\tau'}),\Z/n_{\tau}\Z), \text{ for all } \tau' \in \Gamma'[2]^{**},\\
            \partial_D(b_{\tau,a})&=a \cdot \sum_{\sigma' \in \Gamma',\, \sigma' \mapsto \sigma} \widehat{z^{\sigma'}} \in H^1(K(D),\Z/n_{\tau}\Z),
        \end{align*}
        and no other residues on $(Q')^{\text{aff}}$. The element $b_{\tau,a}$ is unique up to constant elements (i.e.\ elements in $\im \Br K$).
    \end{lemma}
    \begin{proof}
        Let $\Gamma'_{\tau}:= \{g \in \Gamma' : g|_{E}\in \{1,\tau\}\}$.
        The residues in the statement are realized geometrically (i.e.\ after base-changing to $\oK$) by the following element of $\Br \bar \eta'$:
        \begin{align}
         &\bar b_{\tau,a}:=\sum_{\sigma' \in \Gamma'_{\tau} \s \Gamma'}(a \cdot \bar  b_{\tau})^{\sigma'}, \text{ where}\label{LAbel}\\ 
        \notag &\bar b_{\tau}:=\left( \sum_{\substack{g_1,g_2: g_1g_2^{-1} \in \Gamma'[2]\\ g_1 \mapsto \id,\, g_2 \mapsto \tau}}(z^{g_2} \cup (z^{g_1}-z^{g_2}))_{n_{\tau}} + \sum_{\substack{g_1,g_2: g_1g_2^{-1} \in \Gamma'_+\\ g_1 \mapsto \id,\, g_2 \mapsto \tau}}(z^{g_2} \cup z^{g_1})_{n_{\tau}}-(y^{\tau} \cup (y-y^{\tau}))_{n_{\tau}}\right),
        \end{align}
        where the first sum is well-defined because $\bar b_{\tau}^{\sigma'}=\bar b_{\tau}$ for all $\sigma' \in \Gamma'_{\tau}$. Although this last identity is clear when $\sigma' \mapsto \id$, it is not immediate when $\sigma'\mapsto \tau$. However, we have
        \[
        \bar b_{\tau}^{\tau}=\left( \sum_{\substack{g_1,g_2: g_1g_2^{-1} \in \Gamma'[2]\\ g_1 \mapsto \id,\, g_2 \mapsto \tau}}(z^{g_1} \cup (z^{g_2}-z^{g_1}))_{n_{\tau}} + \sum_{\substack{g_1,g_2: g_1g_2^{-1} \in \Gamma'_-\\ g_1 \mapsto \id,\, g_2 \mapsto \tau}}(z^{g_2} \cup z^{g_1})_{n_{\tau}}- (y \cup (y^{\tau}-y))_{n_{\tau}}\right),
        \]
        and, using the identities
        \begin{align*}
            (z^{g_2} \cup (z^{g_1}-z^{g_2}))_{n_{\tau}}-(z^{g_1} \cup (z^{g_2}-z^{g_1}))_{n_{\tau}}&=(z^{g_2}\cup z^{g_1})_{n_{\tau}},\\
            (y^{\tau} \cup (y-y^{\tau}))_{n_{\tau}}-(y \cup (y^{\tau}-y))_{n_{\tau}} &= (y^{\tau} \cup y)_{n_{\tau}},
        \end{align*}
        we see that
        \begin{align*}
            \bar b_{\tau}-\bar b_{\tau}^{\tau} &= \left( \sum_{\substack{g_1,g_2: g_1g_2^{-1} \in \Gamma'[2]\\ g_1 \mapsto \id,\, g_2 \mapsto \tau}}(z^{g_2} \cup z^{g_1})_{n_{\tau}} + \sum_{\substack{g_1,g_2: g_1g_2^{-1} \in \Gamma'_+ \sqcup\Gamma'_- \\ g_1 \mapsto \id,\, g_2 \mapsto \tau}}(z^{g_2} \cup z^{g_1})_{n_{\tau}}- (y^{\tau} \cup y)_{n_{\tau}}\right)\\
            &= \left( \sum_{\substack{g_1,g_2: g_1g_2^{-1} \in \Gamma'\\ g_1 \mapsto \id,\, g_2 \mapsto \tau}}(z^{g_2} \cup z^{g_1})_{n_{\tau}} - \left(\left(\prod_{\substack{g_2 \in \Gamma',\\ g_2 \mapsto \tau}}z^{g}\right) \cup \left(\prod_{\substack{g_1 \in \Gamma',\\ g_1 \mapsto \id}}z^{g}\right)\right)_{n_{\tau}}\right)=0,
        \end{align*}
        proving that $\bar b'_{\tau}$ is $\Gamma'_{\tau}$-invariant as wished. Thus, by construction, $\bar b_{\tau,a}$ is $\Gamma'$-invariant and, a fortiori, $\Gamma_K$-invariant. So, by Lemma \ref{NextLemma} below, there exists $b_{\tau,a}' \in \Br \eta'$ mapping to $\bar b_{\tau,a}$ under $\Br \eta' \to \Br \bar \eta'$. 

        \vskip1mm
        
        The residues of $b_{\tau,a}'$ on $(Q')^{\text{aff}}$ thus coincide with those of the statement after base-change to $\oK$. Therefore, the difference in 
        \[
        \operatorname{Res}(\A^{d'}_K):=\bigoplus_{\substack{D \subset \A^{d'}_K \\  \text{irrreducible divisor}}}H^1(K(D),\qz)
        \]
        of the total residue $\partial_{\text{tot}}(b_{\tau,a}')$ of $b_{\tau,a}'$ and the element $r \in \operatorname{Res}(\A^{d'}_K)$ corresponding to the collection in the statement is algebraic, i.e.\ this difference lies in $\operatorname{Res}_{alg}(\A^{d'}_K):=\bigoplus_{D} H^1(\oK(D)/K(D),\qz)$. All elements of $\operatorname{Res}_{alg}(\A^{d'}_K)$ are total residues of elements of the algebraic Brauer group of $K(t_1,\ldots,t_d)$. Indeed, each element of $\operatorname{Res}_{alg}(\A^{d'}_K)$ corresponds to a collection of cyclic characters $\chi_i:\Gal(\oK(D_i)/K(D_i))\to \qz$ on finitely many divisors $D_1,\ldots,D_r$, and then it is the total residue of $\sum_i \cores_{K_i/K}(\chi_i,f_i)$, where $K_i:=\oK \cap K(D_i)$ and $f_i \in K_i[t_1,\ldots,t_d]$ is a uniformizer for a geometrically irreducible component of $D_i \otimes_K K_i$. 

        \vskip1mm

        From the discussion above, we have $r-\partial_{\text{tot}}(b_{\tau,a}')=\partial_{\text{tot}}(b)$ for some $b \in \Br \eta'$. Letting $b_{\tau,a}:=b_{\tau,a}'+b$ we get the desired Brauer element.
    \end{proof}

    \begin{lemma}\label{NextLemma}
        The map $\Br K(t_1,\ldots,t_r) \to (\Br \oK(t_1,\ldots,t_r))^{\Gamma_K}$ is surjective for all $r$.
    \end{lemma}
    \begin{proof}
        By the Hochschild--Serre spectral sequence $H^i(K,H^j(\oK(t_1,\ldots,t_r),\G_m)) \linebreak \Rightarrow  H^{i+j}(K(t_1,\ldots,t_r),\G_m)$ and the vanishing $H^1(\oK(t_1,\ldots,t_r),\G_m)=0$ by Hilbert 90, the cokernel of the map injects in $H^3(K,\oK(t_1,\ldots,t_r)^*)$. However, this last group is $0$ because $K$ is a number field (for the case $r=1$ see e.g. \cite[p.\ 241]{Harari94}, whose argument straightforwardly generalizes to all $r$). 
    \end{proof}

    \begin{lemma}\label{Lemma2}
        If, for some natural number $m$, $f$ admits a multi-section $s:Q^{\oK,m,(1)} \to X$, then, for all $c \in E^*$, $f'$ admits a multi-section $s':(Q')^{\oK,m,(1)} \to X'$. Moreover, for each $b \in (\Br X)_+$ such that $s^*b=0$, we have, for each $c \in E^*$, $\psi_c^*(b)=b'+(f')^*b_0$ for some $b' \in (\Br X')_+$ such that $(s')^*(b')=0$ and $b_0 \in \Br \eta'$.
    \end{lemma}
    \begin{proof}
        Fix $c \in E^*$. We may lift the morphism $\phi_c:Q' \to Q$ to a morphism $\tilde \phi_c:(Q')^{\oK,m,(1)} \to Q^{\oK,m,(1)}$. Indeed, writing $Q'=\Spec (\oK[(z^{\sigma'})^{\pm1}]_{\sigma'\in \Gamma'})^{\Gamma_K},\, \Gamma':=\Gal(E'/K)$, the morphism $\phi_c$ is $\Spec$-dual of the inclusion $(\oK[(y_{\sigma})^{\pm1}]_{\sigma\in \Gamma})^{\Gamma_K} \subset (\oK[(z^{\sigma'})^{\pm1}]_{\sigma'\in \Gamma'})^{\Gamma_K}$ defined by setting $y_{\sigma} = c^{\sigma} \cdot \prod_{\sigma' \mapsto \sigma} z_{\tau}$, and then $\tilde \phi_c$ may be obtained as the $\Spec$-dual of the inclusion $\oK[y^{\pm1},(y_{\sigma})^{\pm\frac1m}]_{\sigma \in \Gamma^*} \subset \oK[z^{\pm1},(z^{\sigma'})^{\pm\frac1m}]_{\sigma'\in (\Gamma')^*}$.

        \vskip1mm

        Then a multi-section $s'$ is given by the composition
        \[
        s':= (s\circ \tilde \phi_c,p'): (Q')^{\oK,m,(1)} \to X \times_Q Q' = X',
        \]
        where $p'$ denotes the natural morphism $(Q')^{\oK,m,(1)} \to Q'$.

        \vskip1mm

        Let us prove the last sentence. For $b \in (\Br X)_+$, and each $\tau \in \Gamma[2]^{**}$, let $a_{\tau}\cdot \widehat y \in H^1(k(Z'_{\tau}),\Z/n_{\tau}\Z),\, a_{\tau} \in (\Z/n_{\tau}\Z)(-1)$ be its residue at the divisor $Z'_{\tau}=f^{-1}(Z_{\tau})$. Taking $b_0 := \sum_{\tau \in \Gamma[2]^{**}}b_{\tau,a_{\tau}}$ (where each $b_{\tau,a_{\tau}}$ is defined as in Lemma \ref{ResLemma}), the residues at the divisors $(\psi_c)^{-1}(Z'_{\tau}),\,\tau \in \Gamma[2]^{**}$ of the difference $b':=\psi_c^*(b)-(f')^*b_0$ cancel out. Thus  the only residues  $b'$ has over $X$ are at the divisors $Z'_{\tau'},\,\tau' \in \Gamma'[2]^{**}$. Each residue at these divisors is a multiple of the Kummer character $\hat z$ of order $m_{\tau'}$ (see the residues of $b_0$ in Lemma \ref{ResLemma}), and so $b' \in (\Br X')_+$ as wished. 

        \vskip1mm

        It remains to verify that $(s')^*(b')=0$. We have $(s')^*(\psi_c)^*(b')=(\psi_c \circ s')^*(b)=0$ as $\psi_c \circ s'$ factors through $s$. It thus suffices to prove that the pullback of $b_0$ along $p':(Q')^{\oK,m,(1)} \to Q'$ is trivial. The morphism $p'$ factors as $(Q')^{\oK,m,(1)} \to \bar Q' \to Q'$. As a scolium of the proof of Lemma \ref{ResLemma}, the restriction $\bar b_0$ of $b_0$ to $\bar Q'$ is a weighted sum of the elements $\bar b_{\tau} \in \Br \bar \eta'$ presented in \eqref{LAbel}, and each of these pullbacks trivially to $Q^{\oK,m,(1)}$ because $m$ is a multiple of $n_{\tau}$ and each of the first factors of the cup-products appearing in the expression defining each $b_{\tau}$ becomes an $m$-th power after pullback to $Q^{\oK,m,(1)}$.
    \end{proof}

	\begin{proof}[Proof of Proposition \ref{SecondaProp}]
        Assume $f:X \to Q$ satisfies $(R_1)-(R_3)$, and let $B \subset (\Br X)_+$ be a set of representatives for $\Br X_{\eta}/f^*\Br \eta$. By Corollary \ref{Cor:ExistenceMultisection}, there exists a natural number $m$ and a multi-section $s:Q^{\oK,m,(1)} \to X$ such that $s^*B=0$. Then letting $E'/K$ be a Galois extension that contains the $m$-th roots of unity, Lemma \ref{Lemma2} guarantees that, for every $c \in E^*$, the associated Harpaz--Wittenberg base-change $X' \to Q'$ still has a multi-section $s':Q^{\oK,m,(1)} \to X'$, and that $\psi_c^*(B) \subset B' + f^*\Br \eta'$ for some $B' \subset \Br X'$ such that $(s')^*(B')=0$. Thus $\phi_c^*(r(B))=r(B')$, where $\phi_c^*$ denotes the map $\Br_{\text{hor}}(X/Q)\to \Br_{\text{hor}}(X'/Q')$. Since $r(B) = \Br_{\text{hor}}(X/Q)'$ and $\phi_c^*(\Br_{\text{hor}}(X/Q)')=\Br_{\text{hor}}(X'/Q')'$ by Lemma \ref{Lem1etrequarti}, then we infer that $B'$ surjects via $r$ onto $\Br_{\text{hor}}(X'/Q')'$, and so in particular, it surjects onto the quotient $\Br X_{\eta'}/(f')^*\Br \eta' \subset \Br_{\text{hor}}(X'/Q')'$, where the inclusion is induced by $r$, and we also infer that $r(\Br X_{\eta'})= \Br_{\text{hor}}(X'/Q')'$, i.e.\ that $(R_3)$ holds for $f'$. Thus, after the Harpaz--Wittenberg base-change, all properties $(R_1)-(R_4)$ hold, concluding the proof.
	\end{proof}

    \begin{proof}[Proof of Proposition \ref{PropReductions}]
        Combine Propositions \ref{PrimaProp} and \ref{SecondaProp}.
    \end{proof}

	\section{Cyclic strong approximation}\label{Sec9}

    The main result of this section is the approximation result Theorem \ref{MainApprThm}. Although independent, the result is motivated mostly by its application in the proof of Theorem \ref{Thm: fibrationnew}. Indeed,  in Section \ref{Sec10}, we shall use Theorem \ref{MainApprThm} to create rational fibers with suitable adelic points in the setting of Theorem \ref{Thm: fibrationnew}. 

    Observe also that Theorem \ref{MainApprThm} creates grids (see definition below) of approximating elements and not just a single one. Again, this is motivated by our application of Theorem \ref{MainApprThm} in the proof of Theorem \ref{Thm: fibrationnew}, as these are the grids that enable the ``triple variation argument''.
    
	\subsection{Arithmetic ingredients}\label{SSec9.1}
	Fix in this subsection a finite Galois extension $L/E$ of number fields.
	
	\begin{definition*}
		An idele $(x_w)_w \in \I_E$ and a conjugacy class $c \subset \Gal(L/E)$ are {\em compatible} if the image of $(x_w)_w$ under the reciprocity map
		\[
		rec_{L\cap E^{ab}/E}: \mathbb{I}_E \to C_E\to \Gal(L\cap E^{ab}/E)
		\]
		associated to the maximal abelian subextension $L\cap E^{ab}/E$ of $L/E$ coincides with the projection of $c$ to $\Gal(L\cap E^{ab}/E)$.
	\end{definition*}
	
	The following general result is proven in Appendix \ref{AppA}.
	
	\begin{theorem}\label{Thmneq10}
		Let $S$ be a finite set of places of $E$ containing the archimedean ones. Let $(x_w)_w \in \I_E$ be an idele with $x_w=1$ for all $w \notin S$, and $c \subset \Gal(L/E)$ be a compatible conjugacy class.
		
		Then there exist infinitely many prime elements $p \in \cO_{E,S}$ whose Frobenius class in $\Gal(L/E)$ is $c$, that approximate $x_w$ arbitrarily well for all finite $w \in S$, and whose direction in the Minkowski space $E \otimes \R$ approximates the direction of $(x_w)_{w \in M_E^{\infty}} \in E \otimes \R$ arbitrarily well.
	\end{theorem}
	
	The following lemma will occasionally prove useful in verifying the hypothesis of Theorem \ref{Thmneq10}.
	
	\begin{lemma}\label{LemBMidele}
		An idele $(x_w)_{w \in M_E} \in \mathbb I _E$ lies in the kernel of the reciprocity map
		\[
		rec_{L \cap E^{ab}/E}: \mathbb{I}_E \to C_E \to \Gal(L \cap E^{ab}/E)
		\]
		if and only if
		\begin{equation*}\label{BMortho}
		\sum_{w} \inv_w (\phi_w, x_w)=0 \ \text{for all }\phi\in \Hom(\Gal(L/E), \qz),
		\end{equation*}
		where $\phi_w$ denotes the restriction $\phi|_{\Gamma_w}$.
	\end{lemma}
	\begin{proof}
		Each $\phi$ factors through $\Gal(L \cap E^{ab}/E)$ and induces by Class Field Theory an idelic character $\Tilde{\phi}:\mathbb{I}_E \to[rec_{L \cap E^{ab}/E}] \Gal(L \cap E^{ab}/E) \to[\phi] \Q/\Z$. We have
		\[
		\sum_{w} \inv_w (\phi_w , x_w)=\Tilde{\phi}((x_w)_{w})
		\]
		Indeed, the compatibility of local and global Class Field Theories gives that $\Tilde{\phi}((x_w)_{w})=\sum_w \phi_w(x_w)$, and \cite[Corollary 9.6(a)]{HarariBook} gives that $\phi_w(x_w)=\inv_w (\phi_w, x_w)$. Since the $\phi$'s generate the dual abelian group of $\Gal(L \cap E^{ab}/E)$ (as they in fact form the whole group), the lemma follows.
	\end{proof}
	
	\subsubsection{A special case of Hecke-Chebotarev density}
	
	In this subsubsection, let $n$ be a natural number such that $\mu_n \subset E$, and $S$ be a finite set of places of $E$. 
	
	Consider $(\Z/n\Z)$-linearly independent characters 
	\[
	\chi_1,\ldots,\chi_r:\Gamma_E \to (\Z/n\Z)(1),
	\]
	and Heisenberg representations
	\[
	\rho_1,\ldots,\rho_s:\Gamma_E \to \operatorname{Heiss}_n, \quad \rho_j=\begin{pmatrix}
	1 &  \phi_j & \psi_j \\ 
	0 & 1 &  \phi'_j \\
	0 & 0 & 1
	\end{pmatrix},
	\]
	where, for each $j=1,\ldots,s$, the characters $\phi_j, \phi'_j$ lie in $V:=\langle \chi_1,\ldots,\chi_r\rangle$ (i.e.\ they are linear combinations of the $\chi_i$'s). 
	
	The characters and representations above fit into one big Galois representation as follows:
	
	\begin{lemma}\label{Lemma8point5}
		Let $\chi_i, \rho_j$ be as above. There exists a group structure $H$ on the set $(\Z/n\Z)(1)^r \times (\Z/n\Z)(2)^s$ sitting in a central exact sequence
		\[
		1 \to (\Z/n\Z)(2)^s \to H \to (\Z/n\Z)(1)^r \to 1
		\]
		such that
		\[
		\rho:\Gamma_E\to H, \quad \rho:=(\chi_1,\ldots,\chi_r,\psi_1,\ldots,\psi_s).
		\]
		is a homomorphism. Moreover, if the wedge-products $\phi_j\wedge \phi'_j,\, j=1,\ldots,s$ are $(\Z/n\Z)$-linearly independent in $\wedge^2 V$, then $\rho$ is surjective, the group structure is unique, and $[H,H]=(\Z/n\Z)(2)^s$.
	\end{lemma}
	
	\begin{proof}
		Let $\boldsymbol{\chi}:=(\chi_1,\ldots,\chi_r)$ and $\boldsymbol{\psi}:=(\psi_1,\ldots,\psi_s)$. For each $j$, let $\mathbf{v}_j,\mathbf{w}_j \in (\Z/n\Z)^r$ be such that $\phi_j=\mathbf{v}_j \cdot \boldsymbol{\chi}, \phi'_j=\mathbf{w}_j \cdot \boldsymbol{\chi}$.
		The cocycle conditions $\psi_j(\sigma \sigma')=\psi_j(\sigma)+\psi_j(\sigma') +\phi_j(\sigma)\phi'_j(\sigma')$ give 
		\begin{align*}
		\rho(\sigma \sigma')&=(\boldsymbol{\chi}(\sigma)+\boldsymbol{\chi}(\sigma'), \boldsymbol{\psi}(\sigma)+\boldsymbol{\psi}(\sigma') + \boldsymbol{\xi}(\sigma,\sigma')).\\
		\boldsymbol{\xi}(\sigma,\sigma')&:=(\phi_1(\sigma)\phi'_1(\sigma'),\ldots,\phi_s(\sigma)\phi'_s(\sigma')).\\
		&=(\alpha_1(\boldsymbol{\chi}(\sigma))\beta_1(\boldsymbol{\chi}(\sigma')),\ldots,\alpha_s(\boldsymbol{\chi}(\sigma))\beta_s(\boldsymbol{\chi}(\sigma'))),
		\end{align*}
		where $\alpha_j,\beta_j\in \Hom((\Z/n\Z)(1)^r, (\Z/n\Z)(1))$ denote the homomorphisms $\mathbf{a}\mapsto \mathbf{v}_j\cdot \mathbf{a}$ and $\mathbf{a}\mapsto \mathbf{w}_j\cdot \mathbf{a}$, respectively.
		So $\rho$ becomes a homomorphism if we give $H$ the group structure defined by
		\begin{gather*}
		\left(\lambda_1,\mu_1\right)\cdot \left(\lambda_2,\mu_2\right)=\left(\lambda_1+\lambda_2,\mu_1+\mu_2+c(\lambda_1,\lambda_2)\right), \ \text{with}\\
		c(\lambda_1,\lambda_2):=(\alpha_1(\lambda_1)\beta_1(\lambda_2),\ldots,\alpha_s(\lambda_1)\beta_s(\lambda_2)),
		\end{gather*}
		which fits in a central sequence as in the statement.\footnote{One also sees from the multiplication law that the associated cohomology class $[H] \in $\linebreak $H^2((\Z/n\Z)(1)^r,(\Z/n\Z)(2)^s)$ is represented by the $2$-cocycle $c=(\alpha_1\cup \beta_1,\ldots,\alpha_s\cup \beta_s)$ \cite[Ch.~IV, Sec.~3]{Brown}, but we do not need this here.}
		
		\vskip1mm
		
		To prove the final statement, we first prove that $[H,H]=(\Z/n\Z)(2)^s$ for the group structure on $H$ defined above. Assume by contradiction that $[H,H]$ is a proper subgroup of $(\Z/n\Z)(2)^s$, and let $(c_1,\ldots,c_s) \in (\Z/n\Z)^s \s \{(0,\ldots,0)\}$ be such that the linear form $\sum_j c_jx_j$ is identically $0$ on $(x_1,\ldots,x_s) \in [H,H]$. The commutator of two elements $\left(\lambda_1,\mu_1\right), \left(\lambda_2,\mu_2\right) \in H$ is the vector $([\alpha_1,\beta_1](\lambda_1,\lambda_2),\linebreak  \ldots,[\alpha_s,\beta_s](\lambda_1,\lambda_2))$ of $(\Z/n\Z)(2)^s$, where $[\alpha,\beta](\lambda_1,\lambda_2):=\alpha(\lambda_1)\beta(\lambda_2)-\alpha(\lambda_2)\beta(\lambda_1)$. Thus, we would have that
		\begin{equation}\label{Relation}
		\sum_j c_j\cdot [\alpha_j,\beta_j](\lambda_1,\lambda_2)
		\end{equation}
		is identically zero for $\lambda_1,\lambda_2 \in (\Z/n\Z)(1)^r$. 
		
		\vskip1mm Let $A:=(\Z/n\Z)(1)^r, A^D:=\Hom(A,(\Z/n\Z)(1))$. Note that $A^D$ is just isomorphic to $(\Z/n\Z)^r$ (with no Tate twists), and so is canonically isomorphic to $V$. An explicit isomorphism is the map $V \to[\simeq] A^D$ sending $\chi_1^{c_1}\cdots \chi_r^{c_r}$ to $(a_1,\ldots,a_r) \mapsto \sum_i a_ic_i$. By definition of $\alpha_j$ and $\beta_j$, this isomorphism sends $\phi_j$ to $\alpha_j$ and $\phi'_j$ to $\beta_j$.            
		Denote now by $\operatorname{Alt}_2(A)$ the module of alternating maps $A \times A \to (\Z/n\Z)(2)$, and consider the isomorphisms
		\begin{enumerate}[label=(\arabic*)]
			\item $\wedge^2V \to[\simeq] \wedge^2A^D,$ induced by the map $V \to[\simeq] A^D$ defined above,
		\end{enumerate}
		\begin{enumerate}[resume,label=(\arabic*)]
			\item $\wedge^2A^D \to[\simeq] \operatorname{Alt}_2(A), \quad \alpha \wedge \beta  \mapsto [\alpha,\beta]$.
		\end{enumerate}
		To see that the second an invertible, note that an inverse is given by $\mu \mapsto \sum_{i < j} \mu_{i,j}(e_i^*\wedge e_j^*), \ \mu_{i,j}:=\mu(e_i \wedge e_j)$, where $e_1,\ldots,e_r$ is a coordinate basis of $A$ (i.e.\ each $e_i$ is an element of the $i$-th coordinate copy of $(\Z/n\Z)(1)$ in $A=(\Z/n\Z)(1)^r$ of order $n$).
		
		\vskip1mm
		
		The composition of (1) and (2) sends, for each $j$, $\phi_j\wedge \phi'_j$ to $[\alpha_j,\beta_j]$. In particular, \eqref{Relation} would pull back to a non-trivial relation among $\phi_j\wedge \phi'_j,\, j=1,\ldots,s$, contradicting our assumption. So $[H,H]=(\Z/n\Z)(2)^s$ as wished.
		
		\vskip1mm
		
		Let us now prove that $\rho$ is surjective. Because the $\chi_i$'s are linearly independent, $\boldsymbol{\chi}=(\chi_1,\ldots,\chi_r):\Gamma_E \to (\Z/n\Z)(1)^r$ is surjective. But $\boldsymbol{\chi}=pr \circ \rho$, where $pr$ denotes the projection $H \to H/[H,H]=(\Z/n\Z)(1)^r$. Hence $pr(\im \rho)=H/[H,H]$, and so $[H,H]\cdot \im \rho=H$. The group $H$ is a product of $p$-groups, so its Frattini subgroup contains its derived subgroup, and we infer that $\im \rho =H$, i.e.\ $\rho$ is surjective. Finally, the uniqueness of the group structure on $H$ is a direct consequence of this surjectivity.
	\end{proof}
	
	We have the following corollary to Theorem \ref{Thmneq10}.
	
	\begin{corollary}\label{Thmneq10MainCoro}
		Let $\chi_i,\rho_j$ be as above. Assume further that, for each $i$, there exists a prime $\cP_i$ of $\cO_{E,S}$ at which $\chi_i$ is totally ramified, $\chi_{i'}$ is unramified for all $i'\neq i$, and at which $L$ is unramified. Assume further that the wedge-products $\phi_j\wedge \phi'_j,\, j=1,\ldots,s$ are $(\Z/n\Z)$-linearly independent in $\wedge^2 V$.
		
		Then for all $(a_1,\ldots,a_r,b_1,\ldots,b_s)\in (\Z/n\Z)(1)^r\times(\Z/n\Z)(2)^s$, there exist infinitely many prime elements $p\in\cO_{E,S}$ that are arbitrarily close to $1$ for all $v \in S$, have Minkowski direction arbitrarily close to that of $1$, and:
		\begin{itemize}
			\item are split in $L$ and are split over $\Q$;
			\item are such that for some absolute Frobenius $\Frob_p \in \Gamma_E$ of $p$:
			\[
			\text{$\chi_i(\Frob_p)=a_i$ for all $i$ and $\psi_j(\Frob_p)=b_j$ for all $j$.}
			\]
		\end{itemize}
	\end{corollary}        
	
	\begin{proof}[Proof of Corollary \ref{Thmneq10MainCoro}]            
		Let 
		\[
		\rho_L:\Gamma_L\to H, \ \rho_L=(\chi_1|_L,\ldots,\chi_r|_L,\psi_1|_L,\ldots,\psi_s|_L)
		\]
		be the restriction of $\rho$ to $\Gamma_L$. The assumption on the ramification gives in particular that the restricted characters $\chi_1|_L,\ldots,\chi_r|_L$ are still $(\Z/n\Z)$-linearly independent. Hence we may apply Lemma \ref{Lemma8point5} not only to $\rho$ but also to $\rho_L$, and infer that $\rho_L$ is a surjective homomorphism.
				
		Let now $c\subset H$ be the conjugacy class of $(a_1,\ldots,a_r,b_1,\ldots,b_s)$. We wish to apply Theorem \ref{Thmneq10} (for the field $L$ rather than $E$) to this class, and in order to do so, we need to construct a suitable compatible idele in $L$. The abelianization of $H$ is $(\Z/n\Z)(1)^r$, and the projection of $c$ to it is $(a_1,\ldots,a_r)$.
		Through Artin's reciprocity map, each character $\chi_i|_L:\Gamma_L \to (\Z/n\Z)(1)$ defines an idelic character $\tilde \chi_i:\I_L \to C_L \to[\text{rec}] \Gamma_L^{\text{ab}} \to[\chi_i|_L] (\Z/n\Z)(1)$. Recall that by Class Field Theory, $\chi_i(I_w)=\tilde\chi_i(\cO_w^*)$, where $I_w$ denotes the absolute inertia group of a place $w$. Denoting by $w_i$ the place of $L$ corresponding to the prime $\cQ_i$, and letting $\tilde x_i\in \cO_{w_i}^*$ be such that $\chi_i(x_i)=a_i$ and $\tilde x_w:=1$ for $w \notin \{w_1,\ldots,w_r\}$, the idele $(\tilde x_w)$ is compatible with $c$. 
		
		Now, by Theorem \ref{Thmneq10}, we have infinitely many prime elements $q$ of $L$ that split over $\Q$, that approximate the idele $(\tilde x_w)$ at finite places in $\tilde S:=S \cup \{w_1,\ldots,w_r\}$, approximate the direction of this idele in the Minkowski space, and whose Frobenius conjugacy class in $\Gal(\bar L^{\Ker \rho}/L)= H$ is $c$. So, for a prime $\bar q$ of $\bar L$ dividing $q$, we have $\rho(\Frob_{\bar q})=h \cdot (a_1,\ldots,a_r,b_1,\ldots,b_s)\cdot h^{-1}$ for some $h \in H$. Replacing, if necessary, $\bar q$ by $\gamma \cdot \bar q$ for $\gamma \in \Gamma_L$ with $\rho(\gamma)=h$, we have $\rho(\Frob_{\bar q})=(a_1,\ldots,a_r,b_1,\ldots,b_s)$.
		
		Since the primes $q\cdot \cO_{L,S}$ have degree $1$ over $\Q$ and a fortiori also over $E$, the norms $p=N_{L/E}(q)$ are prime elements of $\cO_{E,S}$ with the same absolute Frobenii in $\Gal(\oK/E)$ as the primes $q$ above them, and are then the sought prime elements of $E$.
	\end{proof}
	
	\subsection{Grids}
	
	\begin{definition}\label{DefGrid}
		For a number field $E$ and a finite set of places $T$, a $k$-dimensional ($k \geq 0$) {\em grid} of size $M$ in $\cO_{E,T}$ is an injective function $\cM:\{1,\ldots,M\}^k \to \cO_{E,T}\s \{0\}$ of the shape 
		\[
		\cM(i_1,\ldots,i_k)=x_0\  \cdot \ x_{1,i_1} \ \cdot \ x_{2,i_2}\ \cdot  \cdots \cdot \  x_{k,i_k},
		\]
		where $x_0 \in \cO_{E,T} \s \{0\}$, and  $\{x_{1,1}, \ldots, x_{1,M}\}, \ldots, \{x_{k,1}, \ldots,x_{k,M}\} \subset \cO_{E,T} \s \{0\}$ are ordered subsets. 
	\end{definition}
	
	We shall denote a grid as in Definition \ref{DefGrid} as follows:
	\[
	\cM=x_0 \cdot \{x_{1,1},\ldots,x_{1,M}\}\cdot \{x_{2,1},\ldots,x_{2,M}\}\cdots \{x_{k,1},\ldots,x_{k,M}\} \subset \cO_{E,T}\s \{0\}.
	\]
	
	\vskip1mm
	
	With a slight abuse of notation, we write $x \in \cM$ to mean that $x$ lies in the image of a grid $\cM$. 
	
	\begin{definition*}
		(1). A grid $\cM \subset \cO_{E,T}\s \{0\}$ is {\em square-free} if the norm-ideal $N_{E/\Q}(x \cdot \cO_{E,T})$ is square-free for all $x \in \cM$. 
		
		(2). A square-free grid $\cM \subset \cO_{E,S}\s \{0\}$ is {\em prime-variational} if all the $x_{i,j}$ ($i\geq 1$) are prime elements of $\cO_{E,T}$ whose underlying primes in $\Q$ are pairwise distinct and are all coprime to $N_{E/\Q}(x_0)$.
	\end{definition*}
		
	\subsection{The groups $B_{L}$ and $B'_{L}$}
	
	Let $K$ be a number field.
	
	\begin{notation*}
		For a Galois field extension $E/K$ with group $\Gamma$ and a $\tau \in \Gamma[2]^*$,
        let $n_{\tau}\coloneqq\max\{r \mid n : \mu_r \subset E^{\tau}\}$ and $n^-_{\tau}\coloneqq\max\{r \mid n : \zeta_r^{\tau}=\zeta_r^{-1}\}$, where  $n$ denotes the maximum order of roots of unity in $E$ and $\zeta_r$ denotes a primitive $r$-th root of unity. 
	\end{notation*}
	
	\begin{notation*}
		For a Galois quasi-trivial $K$-torus $Q:=R_{E/K}\G_m=(\Spec E[(y^{\sigma})^{\pm1}]_{\sigma \in \Gamma})^{\Gamma}$, and a Galois field extension $L/E$, we let $B_{L}(Q)\subset \Br K(Q)$ be the subgroup generated by the Brauer classes
		\begin{enumerate}[label=(\arabic*)]
			\setlength\itemsep{.2em}
			\item $\cores_{E/K} \left(\phi , y\right)$, for all $\phi\in \Hom(\Gal(L/E) , \qz)$,
			\item $\cores_{E/K} \left(y, y^\sigma\right)_n$, for all $\sigma\in\Gamma^*$,
			\item $\cores_{E/K} \left(y, y^\tau-y\right)_{n^-_{\tau}}$, for all $\tau\in\Gamma[2]^*$.
		\end{enumerate}
		We let $B'_{L}(Q)\subset \Br K(Q)$ be the subgroup generated by $B_L$ and the Brauer classes
		\begin{enumerate}[label=(\arabic*), resume]
			\setlength\itemsep{.2em}
			\item $\cores_{E/K}\left(y,{y^{\tau}-y}\right)_{n_{\tau}}$, for all $\sigma\in\Gamma[2]^{**}$.
		\end{enumerate}
	\end{notation*}
	
	We shall often just write $B_L,B'_L$ for $B_L(Q)$ and $B'_L(Q)$.
	
	\begin{lemma}
		The group $B_L$ (resp.\ $B'_L$) is contained in $\Br Q$ (resp.\ $\Br Q^{\sharp}$). 
	\end{lemma}
	\begin{proof}
		It is clear that elements of type (1) and (2) lie in $\Br Q$ and that those of type (4) lie in $\Br Q^{\sharp}$. To see that elements of type (3) also lie in $\Br Q$, we refer to \cite[Lemma 3.3]{DemeioBrauerTori} (the reference just computes the residues at the divisors $Z_{\tau}$ and shows that they vanish). 
	\end{proof}
	
	\subsection{Main result}
	
	Fix for the rest of Section \ref{Sec9} a number field $K$, an algebraic closure $\oK$, and a tower of finite extensions $L/E/K$, all contained in $\oK$, with both $L/E$ and $E/K$ Galois. Let $\Gamma=\Gal(E/K)$, and $n$ be the maximal order of roots of unity in $E$. 
	
	Let $S$ be a finite set of places of $K$ containing those dividing $n,$ $\infty$ or $\Delta_{L/K}$, and $T$ be the places of $E$ above $S$. Let $Q\coloneqq R_{E/K}\G_m$, and $\cQ\coloneqq R_{\cO_{E,T}/\cO_{K,S}}\G_m$. 
	
	\begin{theorem}\label{MainApprThm}
		For every $(q_v)_{v \in S} \in Q^{\sharp}(K_S)^{B'_L}$ and all natural numbers $k \geq 3,M \geq 1$ with $(k,n)=1$, there exists a $k$-dimensional prime-variational grid
		\[
		\cM=p_0 \cdot \{p_{1,1},\ldots,p_{1,M}\}\cdots \{p_{k,1},\ldots,p_{k,M}\} \subset \cO_{E,T}\s \{0\} \subset E^* =Q(K),
		\]
		of size $M$, where $p_0,p_{1,1},\ldots,p_{1,M},\ldots,p_{k,1},\ldots,p_{k,M}$ are prime elements of $\cO_{E,T}$ generating pairwise distinct prime ideals, such that:
		\begin{itemize}
			\item (the images of) all $q \in \cM$ approximate  arbitrarily well the image of $(q_v)_{v \in S}$ in the quotient $Q(K_S)/\R_{>0}$, 
			\item for all $q \in \cM$ and $v \notin S$ we have either $q \in \cQ(\cO_v)$ or $q \in \im(Q^{L,n,(1)}(K_v) \to Q(K_v))$,
			\item the half-spin symbols $\operatorname{HS}_{\tau,T,n_{\tau}}(q)$ are defined and trivial for all $\tau \in \Gamma[2]^{**}$ and $q \in \cM$.
		\end{itemize}
	\end{theorem}
	
	(The quotient in the first point is meant through the diagonal action of $\R_{>0}$ on the archimedean factors $\prod_{v\mid \infty}Q(K_v)=\prod_{w \mid \infty} E_w^*$ as in Section \ref{Sec2}.)
	
	\vskip1mm	
	
	{\bf Definition. }Given a grid $\cM= x_0 \cdot \{x_{1,1},\ldots,x_{1,M}\}\cdots \{x_{k,1},\ldots,x_{k,M}\}$ in $\cO_{E,T} \s \{0\}$, we call the function
	\begin{equation}\label{TripleRedVar}
	\begin{matrix}
	f:\{1,\ldots,M-1\}^3 &\to &\Fun_{\Gamma}(\Gamma^3\s\Delta_3,(\Z/n\Z)(2)), \\
	(a,b,c) &\mapsto &\left(\sigma_1,\sigma_2, \sigma_3 \mapsto [(x_{1,a+1}{x_{1,a}}^{-1})^{\sigma_1},(x_{2,b+1}{x_{2,b}}^{-1})^{\sigma_2},(x_{3,c+1}{x_{3,c}}^{-1})^{\sigma_3}]_{E,n}\right),
	\end{matrix}
	\end{equation}
	when defined (i.e.\ when all Redéi symbols appearing are defined), the {\em triple Redéi-variation of order $n$} of $\cM$.\footnote{It might be more appropriate to 
		call this the ``triple Redéi-variation {\em at the indices} $1,2,3$'', and define in general the ``triple Redéi-variation at indices $z,w,t$'' as \eqref{TripleRedVar} but after replacing the Redéi symbol with $[(x_{z,a+1}{x_{z,a}}^{-1})^{\sigma_1},(x_{w,b+1}{x_{w,b}}^{-1})^{\sigma_2},(x_{t,c+1}{x_{t,c}}^{-1})^{\sigma_3}]_{E,n}$. However, the variation at the first three indices is sufficient for the purposes of this paper, and so we focus on that.} Here $\Delta_3 \subset \Gamma^3$ denotes the multi-diagonal. 
	
	\vskip1mm
	
	The following enhancement of the theorem above will be used in later sections.
	
	\begin{theorem}\label{Thm:grid}
		For every function $g:\{1,\ldots,M-1\}^3 \to \Fun_{\Gamma}(\Gamma^3\s\Delta,(\Z/n\Z)(2))$, the grid $\cM$ in Theorem \ref{MainApprThm} may be found so that, in addition, 
		\begin{itemize}
			\item its triple Redéi-variation of order $n$ is defined and equal to $g$;
			\item for all $v \notin S,\, q,q' \in \cM \s \cQ(\cO_v)$, we have $q'q^{-1} \in \im(Q^{L,n}(K_v) \to Q(K_v))$.
		\end{itemize}
	\end{theorem}
	
	\begin{remark*}
		(1). Obviously Theorem \ref{Thm:grid} is stronger than Theorem \ref{MainApprThm}. We separated the statements because Theorem \ref{MainApprThm} might have some interest on its own, as it provides a generalization of the aforementioned result of Shafarevich, while enhancing Theorem \ref{MainApprThm} to Theorem \ref{Thm:grid} is a purely technical complication that is solely done for the later purposes of this paper.
		
		(2). The Redéi-variation condition in Theorem \ref{Thm:grid} is essential to the later arguments in this paper (namely, in achieving control of the triple variation), while the second condition is a very simple addition that simplifies the computations in Section \ref{Sec11} (see in particular Proposition \ref{PropVariation}).
	\end{remark*}
	
	\subsection{An arithmetic version}\label{SSec8.1}
	
	Theorem \ref{MainApprThm} follows from the following theorem of purely arithmetic nature.
	
	\begin{theorem}\label{ThmShafarevich}
		For every $M \geq 1$, every $k \geq 3$ coprime with $n$, every function $g:\{1,\ldots,M-1\}^3 \to \Fun_{\Gamma}(\Gamma^3\s\Delta,(\Z/n\Z)(2))$, and every $(x_w)_{w \in T} \in \prod_{w \in T}E_w^{*}$ such that
		\begin{enumerate}[label={\rm (H\arabic*)}]
			\item $\sum_{w \in T}\inv_w(\phi_w,x_w)=0$ for all characters $\phi \in \Hom(\Gal(L/E),\qz)$, where $\phi_w$ denotes the localization of $\phi$ at $w$,
			\item $\prod_{w \in T}(x_w,(x_{\sigma^{-1}(w)})^{\sigma})_{w,n}=1$ for all $\sigma \in \Gamma^*$,
			\item $\prod_{w \in T}(x_w,(x_{\tau^{-1}(w)})^{\tau}-x_w)_{w,n^-_{\tau}}=1$ for all $\tau \in \Gamma[2]^*$,
			\item $\prod_{w \in T}(x_w,(x_{\tau^{-1}(w)})^{\tau}-x_w)_{w,n_{\tau}}=1$ for all $\tau \in \Gamma[2]^{**}$,
		\end{enumerate}
		there exists a $k$-dimensional grid of size $M$
		\[
		\cM= p_0 \cdot \{p_{1,1},\ldots,p_{1,M}\} \cdot \ \cdots\  \cdot \{p_{k,1},\ldots,p_{k,M}\},
		\]
		where $p_0,p_{1,1},\ldots,p_{1,M},\ldots,p_{k,1},\ldots,p_{k,M}$ are prime elements of $\cO_{E,T}$ generating pairwise distinct prime ideals, and
		such that for every $x \in \cM$
		\begin{enumerate}[label={\rm (\arabic*)}, resume]
			\setcounter{enumi}{0}
			\item $x$ approximates arbitrarily well $x_w$ for all finite $w \in T$, 
			\item the Minkowski direction of $x$ approximates arbitrarily well the direction of the vector $(x_w)_{w \in M_E^{\infty}} \in \prod_{w \in M_E^{\infty}} E_w = E \otimes \R$,
			\item every prime ideal $\cP$ dividing the ideal $x \cdot \cO_{E,T}\subset \cO_{E,T}$ splits completely in $L$ and is split over $\mathbb Q$;
			\item $x$ reduces to a non-zero $n$-th power modulo $\cP^{\sigma}$ for all prime ideals $\cP$ dividing $x \cdot \cO_{E,T}$ and all $\sigma \in \Gamma^{*} \coloneqq \Gamma \s \{\id\}$;
			\item $\operatorname{HS}_{\tau,T,n_{\tau}}(x)=1$ for every $\tau \in \Gamma[2]^{**}$,
		\end{enumerate}
		and, moreover,
		\begin{enumerate}[label={\rm (\arabic*)}, resume]
			\item the triple Redéi-variation of $\cM$ of order $n$ is defined and equal to $g$.
			\item for all $i \in \{1,\ldots,n\},\, j,j'\in \{1,\ldots,M\}$, we have that $p_{i,j}p_{i,j'}^{-1}$ is an $n$-th power modulo the prime ideals generated by $p_0$ and all $p_{i',\alpha}$ with $i' \neq i$ and $\alpha \in \{1,\ldots,M\}$.
		\end{enumerate}
	\end{theorem}
	
	The proof of Theorem \ref{ThmShafarevich} (to be found in Subsection \ref{Ssec8.3}) is based on a combination of Chebotarev's density theorem and Hecke equidistribution (through Theorem \ref{Thmneq10} and its Corollary \ref{Thmneq10MainCoro}) to deal with conditions (1), (2), (3), and (6), and a carefully placed pidgeonhole principle to deal with conditions (4) and (5), following an old technique of Shafarevich \cite{Shafarevich} (see \cite[Theorem 9.3.2]{GermanBook} for an analog result in English).
	
	\vskip1mm
	
	\begin{remark*}
		Dropping the assumption (H4), then the same proof gives, with minimal adjustements, all points except (5). Analogously, 
		weaking the assumption ``$(q_v)_{v \in S} \in Q^{\sharp}(K_S)^{B'_L}$'' in Theorem \ref{MainApprThm} to ``$(q_v)_{v \in S} \in Q(K_S)^{B_L}$'', then this theorem holds with the exception of the third point concerning half-spin symbols.
	\end{remark*}
	
	\begin{notation*}
		Given a finite field extension $k'/k$, a $k$-scheme $X$ whose structural morphism factors as $X\to \Spec k' \to \Spec k$, and a $k$-embedding $\iota:k'\hookrightarrow F$, we let $X(F)_{\iota}\subset X(F)$ be the set of $F$-points that correspond to morphisms $\Spec F\to X$ of $k'$-schemes (and not just of $k$-schemes) when we view $\Spec F$ as a $k'$-scheme via $\iota$.
	\end{notation*}
	
	Observe that $X(F)=\bigsqcup_{\iota:k'\hookrightarrow F}X(F)_{\iota}$.
	
	\begin{lemma}\label{LemLift}
		For a field extension $F/K$, a $K$-embedding $\iota:L\hookrightarrow F$, and a point $q \in Q(K)$,  $q$ lies in $\im(Q^{L,n,(1)}(F)_{\iota}\to Q(F))$ (resp.\ in $\im(Q^{L,n}(F)_{\iota}\to Q(F))$) if and only if $\iota(q^{\sigma})$ is an $n$-th power for all $\sigma\in \Gamma^*$ (resp.\ all $\sigma\in \Gamma$).
	\end{lemma}
	\begin{proof}
		The map $Q^{L,n,(1)}(F)_{\iota}\to Q(F)$ factors as $Q^{L,n,(1)}(F)_{\iota}\to (Q\otimes_KL)(F)_{\iota} \to Q(F)$. 
		The last map is bijective. In fact, this is more general: for any $k$-scheme $X$, finite field extension $k'/k$, and $F,\iota$ as above, the natural map $(X\otimes_kk')(F)_{\iota}\to X(F)$ is a bijection by the universal property of fibered products.
		
		The point $q\in Q(K)$ corresponds to the homomorphism of $K$-algebras obtained by taking $\Gamma$-invariants of the morphism of $E$-algebras  $E[(y^{\sigma})^{\pm1}]_{\sigma\in\Gamma}\to E,\,y^{\sigma}\mapsto q^{\sigma}$. Its inverse image in $(Q\otimes_KL)(F)_{\iota}$ is then the point corresponding to the homomorphism of $L$-algebras $L[(y^{\sigma})^{\pm1}]_{\sigma\in\Gamma}\to F,\,y^{\sigma}\mapsto q^{\sigma}$. The point lifts to $Q^{L,n,(1)}(F)_{\iota}$ if and only if we may complete the following commutative diagram with the upper row:
		\[
		\begin{tikzcd}
		{L[y^{\pm1},(y^{\sigma})^{\pm\frac1n}]_{\sigma \in \Gamma^*}} \arrow[r, dashed]                                       & F                                \\
		{L[(y^{\sigma})^{\pm1}]_{\sigma\in\Gamma}} \arrow[r, "y^{\sigma}\mapsto q^{\sigma}"] \arrow[u, "\subset", hook] & F, \arrow[u, Rightarrow, no head]
		\end{tikzcd}
		\]
		which we may do if and only if $q^{\sigma}$ is an $n$-th power in $F$ for all $\sigma\in \Gamma^*$, proving the first statement. The proof of the $Q^{L,n}$-case is analogous.
	\end{proof}
	
	\begin{proof}[{Proof of Theorem \ref{Thm:grid} assuming Theorem \ref{ThmShafarevich}}]
		Let $(\tilde x_w)_{w \in T}\in \prod_{w \in T}E_w^{*}$ be the image of $(q_v)_{v \in S} \in \prod_{v \in S} Q(K_v)$ under the identifications $Q(K_v)=\prod_{w \mid v} E_w^{*}$. After a small approximation of $(q_v)_{v \in S}$, we assume that $\tilde x_w\neq (\tilde x_{\tau^{-1}(w)})^{\tau}$ for all $w \in T$ and $\tau \in \Gamma[2]^*$. The collection $(q_v)_{v \in S}$ is orthogonal to $B'_L$, i.e.\ is orthogonal to the Brauer classes of type (1)-(4) listed at the beginning of this section. All these Brauer classes are of the shape $\cores_{E/K}b$ for some $b \in \Br E(Q)$. 
		
		\vskip1mm
		
		The local Brauer pairing associated to an $E/K$-corestriction  is computed as follows \cite[Proposition 1.4.7]{BGbook}:
		\begin{center}
			$(\star)$ Let $p$ be a $K_v$-point on a $K$-variety $V$, for some place $v$ of $K$. Let $b \in \Br V_E$. Then
			\[
			\inv_v((\cores_{E/K}b)(p))=\sum_{w \mid v,\  w \in M_E}\inv_w(b(p_w)),
			\]
			where $p_w$ denotes the image of $p$ under the embedding $V(K_v) \subset V(E_w)$.
		\end{center}
		Using $(\star)$, we see that each hypothesis (H1)-(H4) for $(\tilde x_w)_{w \in T}$ is equivalent to the orthogonality of $(q_v)_{v \in S} \in Q(K_v)$ to the respective algebras of type (1)-(4), and so holds by assumption.
		
		\vskip1mm
		
		So we may apply Theorem \ref{ThmShafarevich} to $(\tilde x_w)_{w \in T}$. and get a grid
		\[
		\cM = p_0 \cdot \{p_{1,1},\ldots,p_{1,M}\} \cdot \ \cdots\  \cdot \{p_{k,1},\ldots,p_{k,M}\} \subset \cO_{E,T}\s \{0\}.
		\]
		We claim that $\cM$ satisfies the sought conditions for Theorem \ref{MainApprThm}.
		Following the identification
		\begin{align*}
		Q(K_S)/\R_{>0} = \prod_{w \in T\cap M_E^{\text{fin}}} E_w^* \times (E \otimes \R)/\R_{>0},
		\end{align*}
		points (1) and (2) of Theorem \ref{ThmShafarevich} give that every $x \in \cM$ approximates arbitrarily well $(q_v)_{v \in S}$ in the quotient $Q(K_S)/\R_{>0}$. Moreover, the last point of Theorem \ref{MainApprThm} holds by (5). 
		
		\vskip1mm 
		
		We now show that $q$ lies in $\cQ(\cO_v) \cup \im(Q^{L,n,(1)}(K_v) \to Q(K_v))$ for all $q \in \cM$ and $v\notin S$. When $w(q)=0$ for all places $w$ of $E$ dividing $v$, we clearly have $q \in \cQ(\cO_v)$. Otherwise we have $w(q) >0$ for some place $w\mid v$. Let $\cP\subset \cO_{E,T}$ be the associated prime ideal, which by (3) is split over $K$ (a fortiori, since it is split even over $\Q$) and splits in $L$. These splittings induces $K$-embeddings $E \subset L\subset K_v$. The images of $q^{\sigma},\,\sigma \in \Gamma^*$ in $K_v$ under these embedding are $n$-th power by (4), and so we get liftability to $Q^{L,n,(1)}(K_v)$ by Lemma \ref{LemLift}.
		
		\vskip1mm
		
		The Red\'ei condition of Theorem \ref{Thm:grid} is already given, so it only remains to show that $q'q^{-1} \in \im(Q^{L,n}(K_v) \to Q(K_v))$ for all $v \notin S,\, q,q' \in \cM \s \cQ(\cO_v)$. 
		We may assume that $\cM \not \subset \cQ(\cO_v)$ and so we either have $v(N_{E/K}(q_0))>0$ or there exists some $q_{i,j}$ with $v(N_{E/K}(q_{i,j}))>0$. Let $\cP$ (resp.\ $w$) be the prime (resp.\ place) of $E$ associated to $q_0$ in the first case, and to $q_{i,j}$ in the second case. 
		For all $\sigma \neq \id$, both $(q')^{\sigma}$ and $q^{\sigma}$ are $n$-th power modulo all conjugates $\cP$ with $\sigma \neq \id$ by (3), and so the same holds for their quotient $(q'q^{-1})^{\sigma}$. Moreover, $q'q^{-1}$ is an $n$-th power also modulo $\cP$ by (7). Putting these together, we have that the image of $(q'q^{-1})^{\sigma}$ in $E_w \cong K_v$ is an $n$-th power for all $\sigma \in \Gamma$, and we conclude by Lemma \ref{LemLift}.
	\end{proof}
	
	\subsection{Proof of Theorem \ref{ThmShafarevich}}\label{Ssec8.3}
	
	It is more convenient to carry out the proof in additive notation, working with $(\Z/n\Z)(1)$ instead of $\mu_n$. As in Section \ref{Sec2}, $\addleg{\star}{\star}$ stands for the additive variant  of the Legendre symbol $\leg{\star}{\star}$. Analogously, we denote by $\operatorname{hs}(x)$ the additive variant of the half-spin symbol $\operatorname{HS}(x)$ (possibly with subscripts) under the identification $\mu_n = (\Z/n\Z)(1)$. 
	
	\vskip1mm
	
	Recall that $\Gamma^*=\Gamma \s \{e\}$, that $\Gamma[2]^*$ denotes the set of elements of order $2$ in $\Gamma$, and that $\Gamma[2]^{**}=\{\tau \in \Gamma[2]^*: \chi(\tau) \equiv \pm 1 \bmod\,8\}$ if $8 \mid n$, while it is equal to $\Gamma[2]^*$ if $v_2(n)\in \{1,2\}$ and is empty if $n$ is odd. Here $\chi$ denotes the cyclotomic character modulo $n$.
	
	\vskip1mm
	
	We shall also use the following terminology in the proof.
	
	\begin{notation*}
		For a function $f:\Gamma^* \to (\Z/n\Z)(1)$, we let $f^{\dagger}:\Gamma^* \to (\Z/n\Z)(1)$ be defined by $f^{\dagger}(\sigma):=\chi(\sigma)f(\sigma^{-1})$. 
	\end{notation*}
	
	\begin{definition*}
		(1) A function $f:\Gamma^* \to (\Z/n\Z)(1)$ is {\em even} if $f^{\dagger}=f$ and $(1+\chi(\tau))|f(\tau)$ for all $\tau \in \Gamma[2]^*$.
		
		(2) An {\em augmented function} $\mathbf{f}$ is a pair $\mathbf{f}=(f,h)$ of a function $f:\Gamma^* \to (\Z/n\Z)(1)$ and a collection of elements $h(\tau) \in (\Z/n_{\tau}\Z)(1)$, indexed by $\tau \in \Gamma[2]^{**}$.
		
		(3) An {\em augmented function} $\mathbf{f}=(f,h)$ is {\em even} if $f$ is even and $2h(\tau) \equiv f(\tau) \bmod\,{n_{\tau}}$ for all $\tau \in \Gamma[2]^{**}$.
	\end{definition*}
	
	\begin{lemma}\label{LemAEF}
		Let $p \in \cO_{E,T}$ be a prime element of degree $1$ over $K$ such that
		\begin{enumerate}[label={\rm (PH\arabic*)}]
			\item $\prod_{w \in T}(p,p^{\sigma})_{w,n}=1$ for all $\sigma \in \Gamma^*$,
			\item $\prod_{w \in T}(p,p^{\tau}-p)_{w,n^-_{\tau}}=1$ for all $\tau \in \Gamma[2]^*$,
			\item $\prod_{w \in T}(p,p^{\tau}-p)_{w,n_{\tau}}=1$ for all $\tau \in \Gamma[2]^{**}$.
		\end{enumerate}
		Then the augmented function $(f_p,h_p)$ defined by
		\[
		f_p(\sigma):=\addleg{p^{\sigma}}{p}_n,\quad
		h_p(\tau)\coloneqq \operatorname{hs}_{\tau,n_{\tau}}(p) \in (\Z/n_{\tau}\Z)(1),
		\]
		is even.
	\end{lemma}
	\begin{proof}
		For $\sigma \in \Gamma^*$, (PH1) gives $\addleg{p^{\sigma}}{p}_n= \addleg{p}{p^{\sigma}}_n$ by power-symbol reciprocity. Since $\addleg{p}{p^{\sigma}}_n=\addleg{p^{\sigma^{-1}}}{p}_n^{\sigma}$, this gives $f_p^{\dagger}=f_p$. 
		
		For $\tau \in \Gamma[2]^*$, (PH2) gives $\addleg{p^{\sigma}}{p}_{n^-_{\tau}}=0$  by spin reciprocity (Lemma \ref{Lem:reciprocity}). Since $\addleg{p^{\tau}}{p}_{n^-_{\tau}}$ is the reduction of $\addleg{p^{\tau}}{p}_{n}$ under the reduction modulo $n^-_{\tau}$ map $(\Z/n\Z)(1)\to (\Z/n^-_{\tau}\Z)(1)$, this shows that $f_p(\tau)$ is a multiple of $n^-_{\tau}$. As $n^-_{\tau}=\gcd(n,1+\chi(\tau))$, this proves that $f_p$ is even.
		
		Finally, for $\tau \in \Gamma[2]^{**}$, (PH3) gives $2\operatorname{hs}_{\tau,n_{\tau}}(p)=\addleg{p^{\tau}}{p}_{n_\tau}$ by Proposition \ref{PropSelfResidues}(1). As above, $\addleg{p^{\tau}}{p}_{n_{\tau}}$ is the reduction of $\addleg{p^{\tau}}{p}_{n}=f_p(\tau)$ modulo $n_{\tau}$, and we conclude.
	\end{proof}
	
	\begin{proof}[Proof of Theorem \ref{ThmShafarevich}]
		By Lemma \ref{LemBMidele}, (H1) tells that the idele $(x_w)_{w \in M_E}$ extending $(x_w)_{w \in T}$ by setting $x_w:=1$ if $w \notin T$, lies in the kernel of the Artin reciprocity map
		\[
		\text{rec}_{L \cap E^{ab}/E}:\I_E \to C_E \to \Gal(L \cap E^{ab}/E). 
		\]
		So this idele is compatible with the trivial conjugacy class of $\Gal(L/E)$, and Theorem \ref{Thmneq10} (Hecke-Chebotarev density) gives a prime element $p_0\in \cO_{E,T}$ with degree $1$ over $\Q$, that splits completely in $L$, approximates $x_w$ arbitrarily well for finite $w \in T$, and whose Minkowski direction approximates the direction of $(x_w)_{w \in M_E^{\infty}} \in \prod_{w \in M_E^{\infty}}E_w=E \otimes \R$ arbitrarily well.
		
		\vskip1mm
		
		Let $E_T^*:=\prod_{w \in T}E_w^*$, and let $\R_{>0}$ act by diagonal multiplication on the archimedean factors.
		Since the natural image of $p_0$ in the quotient $E_T^*/\R_{>0}$ approximates the image of $(x_w)_{w \in T}$, and hypothesis (H2)-(H4) are invariant under the $\R_{>0}$-action, these hypotheses must also hold for the diagonal image of $p_0$ in $E_T$, and so translate into (PH1)-(PH3) for $p_0$. Thus $(f,h):=(f_{p_0},h_{p_0})$ is an even augmented function by Lemma \ref{LemAEF}.
		
		\vskip1mm
		
		Choose now arbitrarily small neighbourhoods $U_w\subset E_w^*$ of $1$ for all finite $w \in T$, and a neighbourhood $U_{\infty}$ of the trivial direction $1 \cdot \R_{>0}$ in $(E \otimes \R)/\R_{>0}$. We define
		\[
		\cS:=\left\{p\in \cO_{E,T} \text{ prime}: \, \begin{gathered}
		p \in U_w\text{ for all finite }w \in T,\  p\cdot \R_{>0} \in U_{\infty},\\
		p \text{ has degree $1$ over }\Q \text{ and splits completely in }L , \\ 
		\addleg{p_0^{\sigma}}{p}_n=-\frac{1}{k}\cdot f(\sigma) \text{ for all }\sigma\in \Gamma^*.
		\end{gathered} \right\}.
		\]
		Note that $\cS$ is infinite by Corollary \ref{Thmneq10MainCoro}. More precisely, since $\addleg{\alpha}{p}_n=\chi_{\alpha}(\Frob_p)$ for any $\alpha \in E^*/E^{*n}$ whose associated character $\chi_{\alpha}$ is unramified at $p$, the infinitude of $\cS$ follows from the corollary applied to the characters $\chi_{p_0^{\sigma}},\, \sigma \in \Gamma^*$ (and no Heisenberg representations).
		We shall create $\cM$ as a subset of $\{p_0\} \times \cS^k$.
		
		To an element $x \in \cO_{E,T}$ that factors as
		\[
		x=p_0p_1\cdots p_k, \ \ \text{ with }p_i \in \cS \text{ for }i \geq 1,
		\]
		and $p_1,\ldots,p_k$ generating pairwise distinct prime ideals, we associate a $(k \times k)$-matrix 
		\begin{align*}
		\mathbf X:= \left(
		\begin{array}{c|c|c}
		(X_{1,1},h_{1}) & \cdots & X_{1,k} \\
		\hline
		\cdots & \cdots & \cdots \\
		\hline
		X_{k,1} & \cdots & (X_{k,k},h_{k}) \\
		\end{array}
		\right),
		\end{align*}
		whose off-diagonal (resp.\ diagonal) entries are functions (resp.\ augmented functions) from $\Gamma^*$ to $(\Z/n\Z)(1)$,
		defined by
		\begin{align*}
		X_{i,j}(\sigma) &\coloneqq \addleg{{p}_i^{\sigma}}{{p}_j}_n=\addleg{{p}_j}{{p}_i^{\sigma}}_n \in (\Z/n\Z)(1), \ \sigma \in \Gamma^*, \ i,j = 1,\ldots,k,\\
		h_i(\tau) &\coloneqq \operatorname{hs}_{\tau}(p_i), \ \tau \in \Gamma[2]^{**}, \ i = 1,\ldots,k.
		\end{align*}
		
		The identity $\addleg{{p}_i^{\sigma}}{{p}_j}_n=\addleg{{p}_j}{{p}_i^{\sigma}}_n$ written above comes from power-symbol reciprocity, since both $p_i$ and $p_j$ are close to $1$ at all places in $T$.
		
		\vskip1mm
		
		The following hold true for any $x$ as above:
		\begin{align}
		\label{Eq5.10}
		&X_{i,j}=X_{j,i}^{\dagger}  &\text{for all } i,j;\\
		&(X_{i,i},h_i)  \text{ is an even augmented function }&\text{for all } i;
		\end{align}
		
		The first is the simple observation that $\addleg{{p}_j}{{p}_i^{\sigma}}_n=\addleg{({p}_j)^{\sigma^{-1}}}{{p}_i}_n^{\sigma}$. The second is just a special case of Lemma \ref{LemAEF}.
		
		\vskip1mm 
		
		For $x$ as above, we can compute $\addleg{x}{p_i^{\sigma}}_n$ in terms of $f$ and $\mathbf X$:
		\begin{align*}
		\addleg{x}{p_0^{\sigma}}_n&=\sum_{i=0}^k \addleg{p_i}{p_0^{\sigma}}_n=\sum_{i=0}^k \addleg{p_0^{\sigma}}{p_i}_n=\addleg{p_0^{\sigma}}{p_0}_n-k\cdot\frac{1}{k}\cdot \addleg{p_0^{\sigma}}{p_0}_n=0,\\
		\addleg{x}{p_i^{\sigma}}_n&=\addleg{p_0}{p_i^{\sigma}}_n+\sum_j\addleg{p_j}{p_i^{\sigma}}_n=\addleg{(p_0)^{\sigma^{-1}}}{p_i}_n^{\sigma}+\sum_j\addleg{p_i^{\sigma}}{p_j}_n\\
		&=-\frac1k\cdot\addleg{(p_0)^{\sigma^{-1}}}{p_0}_n^{\sigma}+\sum_j X_{i,j}(\sigma)=-\frac1kf^{\dagger}(\sigma)+\sum_j X_{i,j}(\sigma)\\
		&=-\frac1kf(\sigma)+\sum_j X_{i,j}(\sigma),
		\end{align*}
		for all $\sigma \in \Gamma^*$ and $i=1,\ldots,k$. In the identities above we repeatedly used power-symbol reciprocity to infer that $\addleg{p_j}{p_i^{\sigma}}_n=\addleg{p_i^{\sigma}}{p_j}_n$ for all $i,j \in \{0,\ldots,k\}$ (the obstruction $\prod_{w \in T}(p_i^{\sigma},p_j)_{w,n}$ to this identity vanishes for all pairs $i,j$, as when either $i$ or $j$ is $>0$ then at least one among $p_i$ or $p_j$ is close to $1$ at all places in $T$, while for $i=j=0$ this is (PH1) for $p_0$). 
		We can also compute the half-spin symbols of $x$ in terms of  $(f,h)$ and $\mathbf X$ for all $\tau \in \Gamma[2]^{**}$:
		\begin{align*}
		\operatorname{hs}_{\tau}(x) &\equiv \operatorname{hs}_{\tau}(p_0)+\sum_i \addleg{p_i}{p_0^{\tau}}_n+ \operatorname{hs}_{\tau}(p_1\cdots p_k) &\bmod\,n_{\tau} \\
		&\equiv h(\tau)- f(\tau)+\sum_i \operatorname{hs}_{\tau}(p_i) + \sum_{i<j} X_{i,j}(\tau)& \bmod\,n_{\tau}\\
		&\equiv -h(\tau)+\sum_i h_i(\tau) + \sum_{i<j} X_{i,j}(\sigma) & \bmod\,n_{\tau},
		\end{align*}
		where the first congruence is obtained applying recursively point (ii) of Proposition \ref{PropSelfResidues} to the products $p_0(p_1\cdots p_k), p_1(p_2\cdots p_k),\ldots, p_{k-1}p_k$, and the third holds because $(f,h)$ is an even augmented function.
		
		\vskip1mm
				
		In the following, $\Fun(\Gamma^*,(\Z/n\Z)(1))$ denotes the group of functions from $\Gamma^*$ to $(\Z/n\Z)(1)$.
		
		\begin{lemma}\label{Lem2}
			Let $(f,h), (\delta,r), f,\delta:\Gamma^* \to (\Z/n\Z)(1)$ be even augmented functions. There exist functions $X_{i,j} \in \Fun(\Gamma^*, (\Z/n\Z)(1)), 1 \leq i < j \leq k$, such that letting 
			\[
			\mathbf X := \left(
			\begin{array}{c|c|c}
			(X_{1,1},h_{1}) & \cdots & X_{1,k} \\
			\hline
			\cdots & \cdots & \cdots \\
			\hline
			X_{k,1} & \cdots & (X_{k,k},h_{k}) \\
			\end{array}
			\right)
			\]
			be the matrix defined by $X_{j,i}:=X_{i,j}^{\dagger}$ for $i > j$, and $(X_{i,i},h_i):=(\delta,r)$ for all $i$, the following hold:
			\begin{enumerate}[label={\rm (\arabic*)}]
				\item $k \cdot \sum_{j} X_{i,j}=f$ for all $i$;
				\item $\sum_i h_i(\tau) + \sum_{i<j} X_{i,j}(\tau) \equiv h(\tau) \bmod n_{\tau}$ for all $\tau \in \Gamma[2]^{**}$. 
			\end{enumerate}
		\end{lemma}
		
		By the computations before the lemma, (1) and (2) give that for an $x=p_0 p_1 \cdots p_k$ with associated matrix $\mathbf X$, all symbols $\addleg{x}{p_i^{\sigma}}_n$ and $\operatorname{hs}_{\tau}(x)$ become trivial. 
		
		\begin{proof}[Proof of Lemma \ref{Lem2}]
			For a matrix $M=(m_{i,j})_{1 \leq i,j\leq k},\, m_{i,j} \in \Fun(\Gamma^*, (\Z/n\Z)(1))$, we denote by $M^\dag$ the ``adjoint'' matrix $M^\dag=(m_{j,i}^{\dag})$. We say that $M$ is {\em even} if $m_{j,i}=m_{i,j}^{\dagger}$ for all $i,j$ and $m_{i,i}$ is even for all $i$.
			
			\vskip1mm
			
			We rewrite (2) as 
			\begin{itemize}
				\item[(2$^\prime$)]  $k \cdot r (\tau) + \sum_{i<j} X_{i,j}(\tau) \equiv {{h}}(\tau) \bmod n_{\tau}$ for all $\tau \in \Gamma[2]^{**}$.
			\end{itemize}
			We are then looking for an even matrix $(X_{i,j})_{1 \leq i,j \leq k}$ with $X_{i,i}=\delta$ for all $i$ and that satisfies (1) and (2$^\prime$). We restrict our search among matrices $(X_{i,j})$ of the shape
			\begin{equation}\label{Matrix}
			\left(
			\begin{array}{cccc|ccc}
			\delta & 0 & \cdots & 0 & 0 & 0 & x_1 \\
			0 & \delta & \cdots & 0 & 0 & 0 & x_2 \\
			\vdots & \vdots & \ddots & \vdots & \vdots & \vdots & \vdots \\
			0 & 0 & \cdots & \delta & 0 & 0 & x_{k-3} \\
			\hline
			0 & 0 & \cdots & 0 & \delta & a & b \\
			0 & 0 & \cdots & 0 & a^{\dagger} & \delta & c \\
			x_1^{\dagger} & x_2^{\dagger} & \cdots & x_{k-3}^{\dagger} & b^{\dagger} & c^{\dagger} & \delta
			\end{array}
			\right),
			\end{equation}
			where $a,b,c,x_1,\ldots,x_{k-3}$ indicate unknowns in $\Fun(\Gamma^*,(\Z/n\Z)(1))$, for whose values we solve below. For matrices of this shape, (1) and (2$^\prime$) reduce to the system:
			\[
			\begin{cases}
			\delta+x_1=\frac1k\cdot f \\
			\cdots \\
			\delta + x_{k-3}=\frac1k\cdot f \\
			\delta + a +b =\frac1k\cdot f\\
			a^{\dagger} + \delta + c=\frac1k\cdot f\\
			x_1^{\dagger}+ \cdots +x_{k-3}^{\dagger} + b^{\dagger} + c^{\dagger} + \delta = \frac1k\cdot f \\
			k \cdot r(\tau)+ a(\tau)+b(\tau)+c(\tau) + \sum_{i \leq k-3}x_i(\tau) \equiv  {{h}}(\tau) \bmod n_{\tau}, &\text{ for all }\tau \in \Gamma[2]^{**}.
			\end{cases}
			\]
			We use the first $k-2$ equations to eliminate the unknowns $b,c$ and $x_i, i \leq k-3$ by setting $x_i \coloneqq \frac1k\cdot f-\delta$ for $i \leq k-3$, $b \coloneqq  \frac1k\cdot f - \delta -a$, and $c \coloneqq \frac1k\cdot f - \delta -a^{\dagger}$. Substituting in the last two equations, and using $\delta^{\dagger}=\delta, f^{\dagger}=f$, the system reduces to:
			\[
			\begin{cases}
			a+a^{\dagger}  =  (k-2)\cdot (\frac1k\cdot f-\delta) \\
			a(\tau)\equiv -{{h}}(\tau)+k \cdot r(\tau) + (k-1)(\frac1k\cdot f(\tau)-\delta(\tau))  \bmod n_{\tau}, &\text{ for all }\tau \in \Gamma[2]^{**},
			\end{cases}
			\]
			or equivalently, using that $\delta(\tau) \equiv 2 r (\tau)$ and $f(\tau) \equiv 2{{h}}(\tau)$ modulo $n_{\tau}$ for all $\tau \in \Gamma[2]^{**}$ (both $(f,{{h}})$ and $(\delta,r)$ are even augmented functions):
			\[
			\begin{cases}
			a+a^{\dagger}  = (k-2)\cdot (\frac1k\cdot f-\delta)\\
			a(\tau) \equiv (k-2)\cdot (\frac1k\cdot {{h}}(\tau)-r(\tau))  \bmod n_{\tau}, &\text{ for all }\tau \in \Gamma[2]^{**}.
			\end{cases}
			\]
			Linear combinations with coefficients in $\Z/n\Z$ of augmented functions are also augmented functions, and thus such is the linear combination $(g,\eta):=(k-2)\cdot (\frac1k\cdot (f,{{h}})-(\delta,r))$. We may rewrite the system above as:
			\begin{equation}\label{LastSytem}
			\begin{cases}
			a+a^{\dagger}  = g\\
			a(\tau) \equiv \eta(\tau)  \bmod n_{\tau}, &\text{ for all }\tau \in \Gamma[2]^{**}.
			\end{cases}
			\end{equation}
			We observe that for an $a$ that satisfies the first equation, the second already automatically holds modulo $n_{\tau}/2$ as the first equation gives $(1+\chi(\tau))a(\tau)=g(\tau)$ for all $\tau \in \Gamma[2]^*$, and thus $2a(\tau) \equiv g(\tau) \equiv 2\eta(\tau) \bmod n_{\tau}$ for $\tau \in \Gamma[2]^{*}$.
			
			\vskip1mm
			
			We now define $a:\Gamma^* \to (\Z/n\Z)(1)$ as follows:
			\[
			a(\sigma):=\begin{cases}
			g(\sigma) & \text{ if } \sigma \in \Gamma_+, \\
			0 & \text{ if } \sigma \in \Gamma_+^{-1}, \\
			\frac{g(\tau)}{1+\chi(\tau)} & \text{ if } \tau \in \Gamma[2]^*,
			\end{cases}
			\]
			where the fraction indicates any element of $(\Z/n\Z)(1)$ that multiplied by $(1+\chi(\tau))$ is equal to $g(\tau)$. For each $\tau\in \Gamma[2]^*$, at least one such value of the fraction exists because $g$ is even. Any choice of value for the fractions solves the first equation of \eqref{LastSytem}. Let us argue that we may further choose these values so that the second holds as well.
			
			As noted above, we already have $a(\tau)\equiv \eta(\tau) \bmod n_{\tau}/2$ for all $\tau \in \Gamma[2]^{**}$  from the first equation. Now observe that $\frac{n_{\tau}}2(1+\chi(\tau))=0$ in $\Z/n\Z$ for all $\tau\in \Gamma[2]^{**}$ (this product being $0$ is equivalent to the definition of ``$**$''), and so we may add $n_{\tau}/2$ to the value of the fractions $\frac{g(\tau)}{1+\chi(\tau)}$, if necessary, and obtain the sought congruence.
			
			Now the system is solved and \eqref{Matrix} is the sought matrix $X_{i,j}$, concluding the proof.
		\end{proof}
		
		For each even augmented function $\boldsymbol{\delta}=(\delta,r),\, \delta:\Gamma^* \to (\Z/n\Z)(1)$, we fix a matrix $\mathbf X^{\boldsymbol{\delta}}$ as given by Lemma \ref{Lem2}. We indicate its off-diagonal entries by $X_{i,j}^{\boldsymbol{\delta}}$. 
		
		\vskip1mm
		
		To construct the sought grid, consider one copy of the following rectangular $(k \times M)$ tile for each $\boldsymbol{\delta}$:
		\[
		\begin{array}{|c|ccc|c|}
		\hline
		& & \cdots && \\
		\hline
		&&  \cdots && \\
		\hline
		\vdots &&  \vdots && \vdots  \\
		\hline
		&&  \cdots &&  \\
		\hline
		\end{array}_{\ \boldsymbol{\delta}}
		\]
		
		We shall use the indices $i,j,\ldots$ to indicate the rows of the tiles, and the indices $a,b,c,\ldots$ for the columns. We order the slots $(i,a)$ of each tile lexicographically, i.e.\ $(i,a) < (j,b)$ if $i < j$ or if $i=j$ and $a < b$. So whenever we say ``first slot of a tile such that ...'', we always refer to the first slot {\em in this lexicographic order.}
			
		We wish to fill at least one of the tiles with prime elements of $\cS$ whose associated even augmented functions are equal to the index of the tile, and such that, denoting by $A_1,\ldots,A_k$ the row vectors of this tile, the grid $p_0 \cdot A_1 \cdot\, \cdots\, \cdot A_k$ satisfies (1)-(6).
		
		To achieve this goal, we recursively fill the tiles lexicographically, i.e.\ whenever we place a new element we place it in the first free slot of some tile (but which tile we choose will depend on the element), with elements of $\cS$, until one is completely filled. 
		
		\vskip1mm
		
		At the first step, pick any $p \in \cS$, and then place it in the top left corner of the tile with index $\boldsymbol{\delta}=\boldsymbol{\delta}(p):=(f_p,h_p)$.
		
		\vskip1mm
		
		At the recursive step, each tile is partially filled as in the following figure:
		\[
		\begin{array}{|c|c|c|c|c|c|}
		\hline
		* & \cdots & * & * & \cdots  & * \\
		\hline
		\vdots &&&   && \vdots  \\
		\hline
		* & \cdots & * &  * & \cdots & * \\
		\hline
		* & \cdots & * &  &\cdots &  \\
		\hline
		&  & & & &  \\
		\hline
		\vdots &&&   && \vdots  \\
		\hline
		& \cdots & & &\cdots & \\
		\hline
		\end{array}_{\ \boldsymbol{\delta}},
		\]
		where the asterisks denote prime elements in $\cS$. 
        Now pick a prime element $p$ of $\cO_{E,T}$ such that (the existence of such a $p$ is argued right after the list of properties):
		\begin{itemize}
			\item $p$ approximates $1$ arbitrarily well for all finite places $w \in T$ and its Minkowski direction approximates arbitrarily well that of $1$;
			\item we have $\addleg{p_0^\sigma}{p}_n=-\frac{1}{k} \cdot\addleg{p_0^\sigma}{p_0}_n$ for all $\sigma \in \Gamma^*$;
			\item {\em for any $\boldsymbol{\delta}$}, letting $j \geq 0$ be the integer such that exactly the first $j$ rows $A_1,\ldots,A_j$ of the $\boldsymbol{\delta}$-tile are already filled, then $\addleg{p_i^{\sigma}}{p}_n=X_{i,j+1}^{\boldsymbol{\delta}}(\sigma)$ for all $i \in \{1,\ldots,j\},\, p_i \in A_i,\,\sigma \in \Gamma^*$;
			\item {\em for any $\boldsymbol{\delta}$} on whose corresponding tile the first and second rows are already completely filled, and whose first empty slot is of the form $(3,c+1),\, c \in \{1,\ldots,M-1\}$ (i.e.\ lies on the third row and is not on the first column), we have 
			\begin{equation}\label{ConditionRedei}
			[(p_{1,a+1}p_{1,a}^{-1})^{\sigma_1},(p_{2,b+1}p_{2,b}^{-1})^{\sigma_2},p\cdot p_{3,c}^{-1}]_n=g(a,b,c)(\sigma_1,\sigma_2,id),
			\end{equation}
			for all $a,b \in \{1,\ldots,M-1\}^2, \sigma_1,\sigma_2 \in \Gamma^*, \sigma_1 \neq \sigma_2$, where $p_{i,\star}$ denotes the element in the $(i,\star)$-position of the $\boldsymbol{\delta}$-tile;
			\item we have $\addleg{q}{p}_n=0$ for all prime elements $q$ already appearing on the partially filled tiles.
		\end{itemize}
		The existence of such a $p$ is an application of Corollary \ref{Thmneq10MainCoro}. More precisely, note that $\addleg{\alpha}{p}_n=\chi_{\alpha}(\Frob_p)$ for any element $\alpha \in E^*/E^{*n}$ coprime with $p$, and that, for $\alpha,\beta \in E^*/E^{*n}$ and any prime element $p' \in \cO_{E,T}$ such that $\alpha,\beta,p\cdot (p')^{-1}$ satisfies condition ($C'$)$_{T}$ from Section \ref{Sec4}, we have
		\begin{align*}
		[\alpha,\beta,p\cdot (p')^{-1}]_n &= \Frob_p(E(\sqrt[n]{\alpha},\sqrt[n]{\beta})_2/E)-\Frob_{p'}(E(\sqrt[n]{\alpha},\sqrt[n]{\beta})_2/E)\\
		&= \psi_{\alpha,\beta}(\Frob_p)-\psi_{\alpha,\beta}(\Frob_{p'}) , 
		\end{align*}
		where $E(\sqrt[n]{\alpha},\sqrt[n]{\beta})_2$ denotes any Heisenberg extension of $E$ as defined in Section \ref{Sec4} (see in particular Lemma \ref{Lem3.8}) and $\rho_{\alpha,\beta}=\begin{pmatrix}
		1 & \chi_{\alpha} & \psi_{\alpha,\beta} \\
		0 & 1 & \chi_{\beta} \\
		0 & 0 & 1
		\end{pmatrix}$ denotes its associated Heisenberg representation.    
		So the conditions on $p$ may be imposed by applying Corollary \ref{Thmneq10MainCoro} to 
		\begin{itemize}
			\item the characters $\chi_{\pi^{\sigma}},$ where $\pi$ varies among $p_0$ and the prime elements already appearing on the partially filled tiles and $\sigma$ varies in $\Gamma$, and 
			\item the Heisenberg representations
			\[
			\rho_{\boldsymbol{\delta},a,b,\sigma_1,\sigma_2}:=\begin{pmatrix}
			1 & \chi_{(p_{1,a+1}p_{1,a}^{-1})^{\sigma_1}} & \psi_{\boldsymbol{\delta},a,b,\sigma_1,\sigma_2} \\
			0 & 1 & \chi_{(p_{2,b+1}p_{2,b}^{-1})^{\sigma_2}} \\
			0 & 0 & 1
			\end{pmatrix},
			\]
			associated to (a choice of) Heisenberg extensions 
            \[
            E\left(\sqrt[n]{(p_{1,a+1}p_{1,a}^{-1})^{\sigma_1}},\sqrt[n]{(p_{2,b+1}p_{2,b}^{-1})^{\sigma_2}}\right)_2/E
            \]
			where $\boldsymbol{\delta}$ ranges among the augmented even functions for which the first and second row of the $\boldsymbol{\delta}$-tile are already filled, $a$ and $b$ range among elements of $\{1,\ldots,M-1\}$, and $\sigma_1,\sigma_2$ range among pairs in $\Gamma^*$ with $\sigma_1 \neq \sigma_2$.
		\end{itemize}
		
		We now place $p$ on the tile with index $\boldsymbol{\delta}=\boldsymbol{\delta}(p):=(f_p,h_p)$. Note that that the first two conditions guarantee that $p$ lies in $\cS$.
		
		\vskip1mm
		
		By the pidgeonhole principle, after enough steps, one of the tiles is completely filled. Let $A_1,\ldots,A_k$ be its rows, and $\cM:=p_0 \cdot A_1 \cdots A_k$. By construction,
		\begin{align*}\label{Conditionprime}
		&\addleg{{p}_i^{\sigma}}{{p}_j}_n= X_{i,j}^{\boldsymbol{\delta}}(\sigma) \text { for all }1 \leq i<j \leq k, \ p_i \in A_i,p_j \in A_j, \text{ and }\\
		&(f_p,h_p)=(\delta,r) \text{ for all } p \in A:=A_1 \cup \cdots \cup A_k.
		\end{align*}
		By these two conditions, the augmented matrix $\mathbf X(x)$ associated to every $x=p_0p_1\cdots p_k \in \cM$ coincides with the augmented matrix $\mathbf X^{\boldsymbol{\delta}}$ 
		on the entries $(i,j)$ with $i \leq j$. Thus $\mathbf X(x)=\mathbf X^{\boldsymbol{\delta}}$ by the skew-symmetry \eqref{Eq5.10}, and (1)-(5) are satisfied as wished. Moreover, condition \eqref{ConditionRedei} ensures that (6) holds as well when $\sigma_3=\id$, and the $\Gamma$-equivariance of Redéi symbols gives the case of a general $\sigma_3$. Finally, we have 
		\begin{align*}
		&\addleg{{p}_i}{{p}_j}_n= 0 \text { for all }0 \leq i<j \leq k, \ p_i \in A_i,p_j \in A_j, \text{ with }A_0:=\{p_0\},
		\end{align*}
		by the vanishing condition $\addleg{q}{p}_n=0$ that we imposed for all $q$ appearing before $p$. Thus $\addleg{{p}_i}{{p}_j}_n=0$ also for $0 \leq  j <i \leq k$ by power-symbol reciprocity, and condition (7) follows.
	\end{proof}
	
	\section{Lifting local points}\label{Sec10}
	
	Let $f:X \to Q$ be a smooth proper fibration over a Galois quasi-trivial torus $Q=R_{E/K}\G_m$ defined over a number field $K$, endowed with a multi-section $s:Q^{L,n,(1)}\to X$ for some $n$ such that $\mu_n \subset E^*$ and some field extension $L/K$ containing $E$, as in the following diagram
	\[
	\begin{tikzcd}
	& {Q^{L,n,(1)}} \arrow[ld, "s"'] \arrow[d, "p"] \\
	X \arrow[r, "f"] & Q          .                                  
	\end{tikzcd}
	\]
    
	\vskip1mm
	
	In this section, we show how 
	for such an $f$ one may employ the multi-section and the arithmetic results of Section \ref{Sec9} to produce parameters $q \in Q(K)$ and adelic points $(P_v(q))\in X_q(\A_K)^{(\Br X)_+}$.
	
	\vskip2mm
	
	Let $S \subset M_K$ be a large enough finite set of places. More explicitly, let $S$ be large enough so that:
	\begin{itemize}
		\item it contains all archimedean places, all places dividing $2n$, and all places ramifying in $L$;
		\item there exist a smooth proper $\cO_{K,S}$-model $f:\cX \to \cQ$ for the morphism $f$ (we still denote the model with $f$ with an abuse of notation);
		\item the map $\cX(\cO_v) \to \cQ(\cO_v)$ is surjective for all places $v \notin S$.
	\end{itemize}
	
	Combining routine spreading out results, with the addition of Lang--Weil estimates and Hensel's lemma for the third condition, such a choice of $S$ can always be made.
	
	\begin{definition*}
		(1).\ A $K$-point $q \in Q(K)$ is {\em $(L,n)$-liftable} outside $S$ if for all $v\notin S$, $q$ lies in $\cQ(\cO_v)\cup \im({Q^{L,n,(1)}(K_v) \to Q(K_v)})$.
        
		(2).\ A prime-variational grid $\cM\subset \cO_{E,S}$ is {\em $(L,n)$-liftable} outside $S$ if all of its elements are {\em $(L,n)$-liftable} and, for all $v \notin S$ and $q,q' \in \cM \s \cQ(\cO_v)$, the quotient $q'q^{-1}$ belongs to $\im(Q^{L,n}(K_v) \to Q(K_v))$.
	\end{definition*}
	
	\begin{definition*}
		For a group $G$ and a subset $\cS \subset E^*$, a function $f:\cM \to G$ is {\em multiplicative} if it extends to a homomorphism $f:\langle \cS \rangle \to G$, where $\langle \cS \rangle$ denotes the subgroup of $E^*$ generated by $\cS$.
	\end{definition*}
	
	The following lemma guarantees that, for prime-variational grids made of $(L,n)$-liftable elements, local lifts may be chosen multiplicatively.
	
	\begin{lemma}\label{LemCompatibleLifts}
		Let $\cM\subset \cO_{E,S}$ be a prime-variational grid all of whose elements are $(L,n)$-liftable outside $S$ for some $n\geq 2$. Then, for every place $v\notin S$ such that $\cM\not \subset \cQ(\cO_v)$, there exists an embedding $\iota:L \hookrightarrow K_v$ and lifts $q^{L,n}\in {Q^{L,n,(1)}(K_v)}_{\iota}$ for all $q\in \cM \s \cQ(\cO_v)$ such that the function $\cM \s \cQ(\cO_v) \to {Q^{L,n,(1)}(K_v)}_{\iota}: q \mapsto q^{L,n}$ is multiplicative.
		
		Moreover, if $\cM\subset \cO_{E,S}$ is {$(L,n)$-liftable}, then $(q')^{L,n}(q^{L,n})^{-1}$ belongs to $\im ({Q^{L,n}(K_v)}_{\iota} \to {Q^{L,n,(1)}(K_v)}_{\iota})$ for all $q,q' \in \cM \s \cQ(\cO_v)$.
	\end{lemma}
	
	The multiplications in ${Q^{L,n,(1)}(K_v)}_{\iota}$ are meant with respect to the $L$-group scheme structure of $Q^{L,n,(1)} \cong \G_{m,L}^{\Gamma}$.

	\begin{proof}
		Write
		\[
		\cM=q_0 \cdot \{q_{1,1},\ldots,q_{1,M}\}\cdot \{q_{2,1},\ldots,q_{2,M}\}\cdots \{q_{k,1},\ldots,q_{k,M}\} \subset \cO_{E,S}\s \{0\}.
		\]
		and let $\Sigma_0,\Sigma_1,\ldots,\Sigma_r$ be the places dividing $q_0, \prod_j q_{1,j}, \ldots, \prod_j q_{k,j}$, respectively. The $\Sigma_i$ are disjoint because $\cM$ is square-free (being prime-variational). The places $v$ in the statement are precisely those in the union $\Sigma_0\sqcup \Sigma_1\sqcup \ldots\sqcup \Sigma_r$. For such $v$ we have 
		\[
		\cM \s \cQ(\cO_v)=\begin{cases}
		\cM & \text{if }v \in \Sigma_0 \\
		q_0 \cdot \{q_{1,1},\ldots,q_{1,M}\}\cdots q_{i,j} \cdots \{q_{k,1},\ldots,q_{k,M}\}& \text{if }v \in \Sigma_i,\, i\geq 1, \\
		\end{cases}
		\]
		where $q_{i,j} \in \cO_{E,S}$ is the prime element whose underlying prime in $K$ is the place associated to $v$.
		In both cases, $\cM \s \cQ(\cO_v)$ is a grid, that we denote by $\cM_v$. Now, for all $v$ as above and all $q \in \cM_v$, our hypothesis says that $q$ lifts to ${Q^{L,n,(1)}(K_v)}$. Equivalently, as in the proof of Theorem \ref{MainApprThm} (see page 51), there exists a $K$-embedding $\iota:L \hookrightarrow K_v$ such that 
		\begin{center}
			$(\star)$ the image of $q^{\sigma}$ under $\iota$ is an $n$-th power for all $\sigma \neq \id$. 
		\end{center}
		Property $(\star)$ clearly depends only on the restriction of $\iota$ to $E$. Let $w$ be the place of $E$ associated to $\iota|_E$. We claim that $w$ does not depend on $q$. 
				
		In fact, the places of $E$ dividing $v$ form a unique orbit under the $\Gamma$-action, and since $q\notin \cQ(\cO_v)$, for at least one of these places $q$ must have non-zero valuation. Since the grid $\cM$ is square-free, the valuation of $q$ at each of these places is either $0$ or $1$.
		Since $(\star)$ forces $w^{\sigma}(q)=w(q^{\sigma^{-1}})$ to be a multiple of $n$ for all $\sigma \neq \id$, it must then be that $w(q)>0$ and so $w(q)=1$. But, if $v\in \Sigma_i,\,i\geq 1$ (resp.\ $v\in \Sigma_0$), we have $w(q)=w(q_{i,j})$ (resp.\ $w(q)=w(q_0)$) and so $w$ is the unique place associated to the prime $q_{i,j}$ (resp.\ the unique place of $E$ above $v$ with $w(q_0)=1$), proving the claim.
		
		In particular, fixing an embedding $\iota:L\hookrightarrow K_v$ whose restriction to $E$ is associated to $w$, we can lift all $q\in \cM_v$ to $Q^{L,n(1)}(K_v)_{\iota}$.
		
		We rewrite now $\cM_v$ as
		\begin{align*}
		\cM_v&=q_{\text{base}} \cdot \Delta_1\cdots \Delta_{i-1}\cdot \Delta_{i+1} \cdots \Delta_k,\text{ where }\\
		q_{\text{base}}:&=q_0 \cdot q_{1,1} \cdots q_{i-1,1} \cdot q_{i,j} \cdot q_{i+1,1} \cdots q_{k,1},\\
		\Delta_{\alpha}:&= \{1,\delta_{\alpha,1},\ldots,\delta_{\alpha,M-1}\},\, \delta_{\alpha,\beta}:=q_{\alpha,\beta+1}q_{\alpha,1}^{-1}.
		\end{align*}
		Let also $q_{\alpha}:=q_{\text{base}}\cdot \delta_{\alpha,\beta} \in \cM_v$. 
		Choose now arbitrary lifts $q_{\text{base}}^{L,n}$ and $q_{\alpha}^{L,n}$ to $Q^{L,n(1)}(K_v)_{\iota}$. Define $\delta_{\alpha,\beta}^{L,n}:=q_{\alpha,\beta}^{L,n} \cdot (q_{\text{base}}^{L,n})^{-1}$, and extend the lifts multiplicatively to all $\cM_v$. Now condition ($M$) holds.
		
		For the last statement, note that, since $Q^{L,n}/Q$ is Galois, the inverse images of $q'q^{-1}$ in $Q^{L,n}(\oK_v)$ are permuted transtitively by $\Gal(Q^{L,n}/Q)$, and in particular they all have the same definition field. Since $q'q^{-1}\in \im (Q^{L,n}(K_v) \to Q(K_v))$, the common field of definition is $K_v$. In particular, any point in $Q^{L,n}(\oK_v)$ that is an inverse image of $(q')^{L,n}(q^{L,n})^{-1}$ is in fact defined over $K_v$, i.e.\ lies in $Q^{L,n}(K_v) \subset Q^{L,n}(\oK_v)$. We conclude by observing that any inverse image of an element of $Q^{L,n,(1)}(K_v)_{\iota}$ in $Q^{L,n}(K_v)$ lies necessarily in $Q^{L,n}(K_v)_{\iota}$.
	\end{proof}
	
	{\bf Warning.~}The existence of even just one $(L,n)$-liftable element (with at least one prime divisor outside $S$) in $Q(K)$ is not immediate at all, but is rather the point of Section \ref{Sec9} : Theorem \ref{MainApprThm} produces grids with $(L,n)$-liftable elements.
	
	\vskip1mm
	
	Given opens $U_v \subset X(K_v),\,v \in S \cap M_K^{\text{fin}}, U_{\infty} \subset X(K_{\infty})_{dir}$ and a prime-variational grid $\cM\subset \cO_{E,S}$ that is contained in $f(U_v)$ for all finite $v \in S$, whose image in $Q(K_{\infty})_{dir}$ is contained in $f(U_{\infty})$, 
    and all of whose elements are $(L,n)$-liftable for some pair $(L,n)$ with $n \geq 2$, we use Lemma \ref{LemCompatibleLifts} to define adelic lifts as follows. For each $q \in \cM$, we let $(P_v(q))_{v}\in X_q(\A_K)$ be defined by:
	\begin{itemize}
		\item $P_v(q)$ lies in $U_v$ for all $v \in S$;
		\item $P_v(q)$ lies in $\cX(\cO_v) \subset X(K_v)$ for all $v$ such that $q \in \cQ(\cO_v)$;
		\item $P_v(q)=s(q_{L,n})$, where $q_{L,n} \in Q^{L,n(1)}(K_v)$ is as in Lemma \ref{LemCompatibleLifts}.
	\end{itemize}
	
	It is clear from our choice of $S$ that such adelic lifts exist for every $q \in \cM$.
	
	\subsection{Brauer--Manin pairing with $(\Br X)_+$}
	
	Recall that $(\Br X)_+:=(\Br X)_+(\mathfrak A \times_QX)$ from Subsection \ref{SSec6.3}.
	
	\begin{proposition}\label{Proprop}
		Let $q \in \cO_{E,S} \subset Q(K)$ be coprime in $\cO_{E,S}$ with its conjugates $q^{\tau}$ for all $\tau \in \Gamma[2]^{**}$, and let $(P_v)_{v\in M_K}\in X_q(\A_K)$ be such that $P_v \in s(Q^{L,n(1)}(K_v))$ for $v\notin S$ with $q \notin \cQ(\cO_v)$. For $b \in (\Br X)_+$ such that $s^*b=0$, we have
		\[
		((P_v)_v, b)_{BM}=\sum_{v \in S} \inv_v (b(P_v))+\sum_{\tau \in \Gamma[2]^{**}}a_{\tau}\operatorname{hs}_{S,n_{\tau}}(q),
		\]
		where $a_{\tau} \in (\Z/n_{\tau}\Z)(-1)$ is uniquely defined by $\partial_{Z'_{\tau}}(b)=a_{\tau}\cdot \hat y \in H^1(k(Z'_{\tau}),\qz)$.
	\end{proposition}
	
	\begin{proof}
		Write:
		\begin{multline}\label{EqINV}
		((P_v)_v, b)_{BM}=
		\sum_{v \in S} \inv_v (b(P_v)) 
		+
		\sum_{\substack{v \notin S, \\ q \in \cQ(\cO_v)}} \inv_v (b(P_v)) +
		\sum_{\substack{v \notin S, \\q \notin \cQ(\cO_v)}} \inv_v (b(P_v)).
		\end{multline}
		The third sum is zero as $P_v \in s(Q^{L,n(1)}(K_v))$ for the $v$ appearing there and $s^*b=0$ by assumption.
		By Lemma \ref{PropRes}, the second sum is
		\begin{equation}\label{Pairing}
		\sum_{\tau \in \Gamma[2]^{**}} a_{\tau}\sum_{\substack{v \notin S, \\ q \in \cQ(\cO_v)}}\sum_{\substack{w \mid v  \\w \in M_{E^{\tau}}}}\addleg{q}{\cI_w}_{n_{\tau}}
		\end{equation}
		By definition of half-spin symbols, we have
		\[
		\operatorname{hs}_{S,n_{\tau}}(q)=\sum_{v \notin S}\sum_{\substack{w \mid v  \\ w \in M_{E^{\tau}}}}\addleg{q}{\cI_w}_{n_{\tau}}.
		\]
		Thus to prove the statement, it suffices to prove that $\addleg{q}{\cI_w}_{n_{\tau}}=0$ for all $w \in M_{E^{\tau}}$ dividing a $v \notin S$ such that $q \notin \cQ(\cO_v)$.
        For such $w$, the existence of $P_v \in X_q(K_v) \cap s(Q^{L,n(1)}(K_v))$ implies that $q\in \im (Q^{L,n,(1)}(K_v) \to Q(K_v))$. I.e.\  there exists a $K$-embedding $\iota:L \to K_v$ such that $\iota(q^{\sigma}) \in K_v^{*n}$ for all $\sigma \in \Gamma^*$. In particular, $v$ splits completely in $E$ and $q\in E_{w'}^{*n}$ for all places $w'\mid v$ of $E$ that are {\bf different from} the one place $w_E \in M_E$ lying below the place $w_L$ of $L$ corresponding to $\iota$. At least one of the two places $w'$ of $E$ dividing the place $w$ of $E^{\tau}$ is different from $w_E$, and thus $\addleg{q}{w'}_{n}=0$ for this $w'$. Writing $\addleg{q}{\cI_w}_{n_{\tau}}=w'(q^{\sigma}-q) \cdot \addleg{q}{w}_{n_{\tau}}=w'(q^{\sigma}-q) \cdot \addleg{q}{w'}_{n_{\tau}}$, where the last identity holds because $w'$ has degree $1$ over $w$, we conclude since $n_{\tau} \mid n$.
	\end{proof}        
	
	\begin{remark*}        
		For a real place $v$, the local Brauer pairing $X^{\sharp}(K_v) \to \qz,\,P_v \mapsto \inv_v (b(P_v))$ might not factor through $\pi_0(X(K_v))$, but it rather only factors through $\pi_0(X^{\sharp}(K_v))$. The connected components of $X^{\sharp}(K_v)$ can be smaller than those of $X(K_v)$, and this ultimately forces us to have some degree of archimedean approximation in our arguments. Luckily, $Q^{\sharp}(K_v)\subset Q(K_v)$ 
		has a ``conic shape'', i.e.\ it is invariant under the action of $\R_{>0}$, and the connected components of $X^{\sharp}(K_v)$ are just inverse images of (some) of those of $Q^{\sharp}(K_v)$. 
		
		So to gain control over $\inv_v(b(P_v(q)))$ at real places, the working solution that we adopted is to make sure that $q$ approximates a given direction in $Q(K_v)$ in the sense of Section \ref{Sec2} (in fact, since it requires no further effort, we do even more, and approximate a simultaneous direction in $\prod_{v\mid \infty}Q(K_v)$). This is our motivation behind Appendix \ref{AppA}.
	\end{remark*}
	
	\section{Triple Variation theorem}\label{Sec11}
	
	The purpose of this section is to compute under certain technical hypotheses the triple variation over a grid of fibers of the Brauer--Manin pairing of the adelic points defined in Section \ref{Sec10}. (The ``technical hypotheses'' are satisfied by the grids produced by Theorem \ref{MainApprThm}.) See Theorem \ref{Thm:3variation} below. 
	
	\subsection{Main result}\label{SSec10.1}
	
	Let $Q=R_{E/K}\G_m$ be a Galois quasi-trivial torus defined over a number field $K$. Let $\Gamma=\Gal(E/K)$, and $n$ be a natural number such that $\mu_n \subset E^*$.
	
	\begin{notation*}
		For a fixed total ordering $<$ of $\Gamma$, an element $\bar\beta  \in H^3(Q \otimes_K \oK,\mu_n)$ may be written uniquely as
		\begin{equation}\label{defcoeffs}
		\bar\beta= \sum_{\sigma_1 < \sigma_2 < \sigma_3} a_{\sigma_1,\sigma_2,\sigma_3}(\bar\beta) \cdot (y^{\sigma_1} \cup y^{\sigma_2} \cup y^{\sigma_3})_n, \ \ a_{\sigma_1,\sigma_2,\sigma_3}(\bar \beta)\in (\Z/n\Z)(-2).
		\end{equation}
		For any three pairwise distinct elements $\sigma_1,\sigma_2,\sigma_3 \in \Gamma$, set $a_{\sigma_1,\sigma_2,\sigma_3}(\bar\beta):=\linebreak (-1)^{\sgn s}a_{s(\sigma_1),s(\sigma_2),s(\sigma_3)}(\bar \beta)$ for $s\in S_3$ a permutation such that $s(\sigma_1)< s(\sigma_2)< s(\sigma_3)$, and set $a_{\sigma_1,\sigma_2,\sigma_3}(\bar\beta):=0$ if two of the indices are repeated. Since the cup-product is alternating, the elements $a_{\sigma_1,\sigma_2,\sigma_3}(\bar \beta)$ just defined do not depend on the chosen ordering of $\Gamma$. 
	\end{notation*}
	
	Acting by ${{g}} \in \Gamma$ on \eqref{defcoeffs} gives
	\begin{equation}\label{Symbols}
	a_{\sigma_1{{g}},\sigma_2{{g}},\sigma_3{{g}}}(\bar \beta^{{{g}}})=a_{\sigma_1,\sigma_2,\sigma_3}(\bar\beta)^{{{g}}}
	\end{equation}
	for all $\sigma_1,\sigma_2,\sigma_3,{{g}} \in \Gamma$.
	
	\begin{notation*}
		For $\bar \beta \in H^3(Q \otimes_K \oK, \mu_n)^{\Gamma_K}$ and $x_1,x_2,x_3 \in E^*$, we let
		\[
		R_{\bar \beta}[x_1,x_2,x_3]_n \coloneqq \sum_{[(\sigma_1,\sigma_2, \sigma_3)] \in \Gamma^3/\Gamma} a_{\sigma_1,\sigma_2,\sigma_3}(\bar \beta) [x_1^{\sigma_1},x_2^{\sigma_2},x_3^{\sigma_3}]_n \in \Z/n\Z,
		\]
		when all the Red\`ei symbols appearing on the right hand side are defined. 
	\end{notation*}
	
	The indexing in the last sum is well-posed as
    \begin{multline*}
        a_{\sigma_1{{g}},\sigma_2{{g}},\sigma_3{{g}}}(\bar \beta) [x_1^{\sigma_1{{g}}},x_2^{\sigma_2{{g}}},x_3^{\sigma_3{{g}}}]_n= \\
         (a_{\sigma_1,\sigma_2,\sigma_3}(\bar \beta)[x_1^{\sigma_1},x_2^{\sigma_2},x_3^{\sigma_3}]_n)^{{{g}}}=a_{\sigma_1,\sigma_2,\sigma_3}(\bar \beta)[x_1^{\sigma_1},x_2^{\sigma_2},x_3^{\sigma_3}]_n
    \end{multline*}
	for all $\sigma_1,\sigma_2,\sigma_3,{{g}}\in \Gamma$, where in the first identity we used \eqref{Symbols} and $\bar \beta^{{{g}}}=\bar\beta$, and in the second we used that each summand $a_{\sigma_1,\sigma_2,\sigma_3}(\bar \beta)[x_1^{\sigma_1},x_2^{\sigma_2},x_3^{\sigma_3}]_n\in (\Z/n\Z)(-2)\otimes (\Z/n\Z)(2)=\Z/n\Z$ is Galois-invariant.
	
	\begin{theorem}[Triple variation]\label{Thm:3variation}
		Let $f: X \to Q$ be a smooth projective fibration over a Galois quasi-trivial torus $Q=R_{E/K}\G_m$ with a multi-section $\bar s:Q^{\oK,n,(1)} \to X$ for some $n$ such that $\mu_n \subset E$. There exists then a finite Galois extension $L/K,\, L \subset \oK$ such that  $\bar s=s \circ (Q^{\oK,n,(1)}  \to Q^{L,n,(1)})$ for some $s:Q^{L,n,(1)} \to X$ and for which, for all sufficiently large finite sets of places $S$ of $K$, the following holds. 
        
        For all $(P_v)_{v \in S} \in X(K_S)$, all $n$-torsion elements $b \in \Br_{\text{hor}}(X/Q)$, all sufficiently small neighbourhoods $U_v$ of $P_v$ for non-archimedean $v \in S$, all arbitrarily small neighbourhood $U_{\infty}$ of the image of $P_{\infty} = (P_v)_{v \in M_K^{\infty}} \in X(K_{\infty})$ in $X(K_{\infty})_{dir}$,
		any $(L,n)$-liftable prime-variational cube
		\[
		\cM:=q_0 \cdot \{q_{1,1},q_{1,2}\}\cdot \{q_{2,1},q_{2,2}\}\cdot \{q_{3,1},q_{3,2}\} \subset \cO_{E,S} \subset E^*= Q(K)
		\]
		that is contained in $f(U_v)$ for all finite $v \in S$, and whose image in $Q(K_{\infty})_{dir}$ is contained in $f(U_{\infty})$, satifies
		\[
		\sum_{\substack{q=q_0q_{1}^iq_{2}^jq_{3}^l\\ 0 \leq i,j,l \leq 1}} (-1)^{i+j+l}\cdot (b,(P_v(q))_{v \in M_K})_{BM} = R_{{\bar \partial (b)}}[\delta_1,\delta_2,\delta_3]_n,
		\]
		where $(P_v(q))_{v \in M_K} \in X_q(\A_K)$ is a compatible choice of adelic lifts constructed through $s$ as in Section \ref{Sec10}, and $\delta_i := q_{i,2}q_{i,1}^{-1}$.
	\end{theorem}
	
	The rest of this section is devoted to proving Theorem \ref{Thm:3variation} through an explicit computation with \v{C}ech cochains.
	
	\subsection{Vertical Brauer--Manin pairing via \v{C}ech cohomology}\label{SSec102}
	
	In this subsection, we describe how to compute the Brauer--Manin pairing attached to a horizontal Brauer element using \v{C}ech cocycles. In the computation, it will be more convenient to work with inverse systems of coverings rather than with a single covering. For that reason, we operate with what we call ``procoverings'' (pro-objects in the category of coverings of a scheme). It is straightforward that such inverse systems, just as usual \v{C}ech coverings, give rise to \v{C}ech complexes that can be used to approximate the cohomology of schemes, and to compute it exactly for quasi-projective varieties. Mainly just to settle notation, we briefly pause to lay out these properties. Afterwards, we shift our attention to relative \v{C}ech cohomology and relative cohomology, summarizing the content of Appendix \ref{AppB}. We then conclude this subsection with an explicit formula for the vertical Brauer--Manin pairing.
	
	\vskip2mm
	
	{\bf Procoverings and procovered schemes.~} A {\em procovering} of a scheme $Y$ is an object of the pro-category of coverings of $Y$ (i.e.\ the category of inverse systems of coverings of $Y$), and a refinement of procoverings is a morphism in this pro-category. A {\em procovered scheme} is a pair $(Y,\cU)$ of a scheme $Y$ and a procovering $\cU/Y$. A morphism between procovered schemes $(X,\cV)$ and $(Y,\cU)$ is a pair $(f,\psi)$ of a morphism $f:X \to Y$ and a refinement  $\psi:\cV \to \cU \times_YX$.
	
	If $Y' \to Y$ is a morphism of schemes with $Y'$ an inverse limit of finite étale covers of $Y$, we denote by $\{Y'/Y\}$ its associated procovering and by $(Y,Y')$ the associated procovered scheme.
	
	\vskip2mm
	
	{\bf Pro-\v{C}ech-complex and pro-\v{C}ech-to-étale maps.} Recall that to an étale covering $\cU$ of a scheme $Y$ and a sheaf $F$ on $Y$, we may associate a \v{C}ech complex $\check{C}^{\bullet}(\cU,F)$, and that there are natural maps 
	\begin{equation}\label{EqNatural0}
	\check{H}^{\bullet}(\cU,F) \to H^{\bullet}(Y,F)
	\end{equation}
	from its cohomology to that of $Y$ \cite[\S III.2]{LECcompleto}. For a procovering $\cU$, we define $\check{C}^{\bullet}(\cU,F):=\lim\limits_{\to} \check{C}^{\bullet}(\cU_i,F)$, where $\cU_i,\, i \in I$ is an inverse system of coverings with inverse limit $\cU$. The map \eqref{EqNatural0} extends naturally to procoverings.
	
	\vskip2mm
	
	{\bf Specialization to points.} 
	For a field $\Omega$ and a covering $\cU=\{U_i\to Y\}_{i\in I}$ of a scheme $Y$, we let
	\[
	\cU(\Omega):=\bigcup_i U_i(\Omega).
	\]
	For a procovering $\cU= \lim\limits_{\leftarrow} \cU_i$, we let $\cU(\Omega):=\lim\limits_{\leftarrow}\cU_i(\Omega)$. 
		
	If $y\in Y(\Omega)$ is an $\Omega$-point, then each point $\bar y\in \cU(\Omega^{\text{sep}})$ whose image in $Y$ is (the geometric point corresponding to) $y$ gives rise to a morphism of procovered schemes
	\[
	(y,\bar y):  (\Spec \Omega,\Spec \Omega^{\text{sep}}) \to (Y,\cU),
	\]
	inducing pullback maps
	\begin{equation}\label{EqPullback2}
	\check{C}^{\bullet}(\cU/Y,F) \to \check{C}^{\bullet}(\Omega^{\text{sep}}/\Omega,F|_y),
	\end{equation}
	on \v{C}ech cochains for any étale sheaf $F$ on $Y$. We denote the image of a cochain $\alpha$ under this map by $\alpha(y,\bar y)$.
	
	\vskip1mm
	
	A special example is when $\cU/Y$ is comprised of a single Galois covering $Y' \to Y$, say with group $G$. In this case, \eqref{EqPullback2} becomes
	\begin{equation}\label{EqPullbackgroup}
		\check{C}^{\bullet}(\cU/Y,F) \to \check{C}^{\bullet}(\Omega^{\text{sep}}/\Omega,F|_y),
	\end{equation}
	under the identifications $\check{C}^{\bullet}(\cU/Y,F)=C^{\bullet}(G,F(Y'))$ and $\check{C}^{\bullet}(\Omega^{\text{sep}}/\Omega,F|_y)=\linebreak {C}^{\bullet}(\Gal(\Omega^{sep}/\Omega),F|_y(\bar y))$ given by \eqref{EqIdentificationCechGroup}.
	
	\vskip2mm
	
	{\bf Relative étale cohomology.~}For a morphism of schemes $f:X \to Y$, we denote by $\Sh_{X/Y}$ the category of triples $(F,G,\alpha:F \to f_*G)$, where $F$ (resp.\ $G$) is a sheaf on $Y$ (resp.\ $X$). We call such a triple a {\em relative sheaf}. We define the {\em relative étale cohomology groups} $H^{\bullet}_X(Y;M)$ of a relative sheaf $M=(F,G,\alpha)$ as the derived functors of
	\[
	\Gamma_X(Y,-):\Sh_{X/Y} \to \cA b:(F,G,\alpha)\mapsto \Ker(\alpha:F(Y) \to G(X)).
	\]
	See Appendix \ref{AppB} for details. (See also \cite[Ch.~14]{Friedlander} for an alternative equivalent definition via simplicial schemes). These groups sit in a natural long exact sequence
	\begin{equation}\label{EqRCS}
	\cdots   H^n(Y,F) \to[\alpha^{\dagger} \circ f^*] H^n(X,G) \to H_X^{n+1}(Y,M) \to H^{n+1}(Y,F) \to \cdots .
	\end{equation}
	
	\vskip2mm
	
	{\bf Relative \v{C}ech-to-étale comparison.~}A {\em relative covering} of $f$ is a triple $(\cU,\cV,\psi:\cV \to \cU \times_YX)$, where $\cU$ (resp.\ $\cV$) is a covering of $Y$ (resp.\ $X$), and $\psi$ is a refinement. For a relative covering and a relative étale sheaf $(F,G,\alpha)$, the {\em relative \v{C}ech complex} $\check{C}^{\bullet}(\cV/X,\cU/Y;F,G)$ is the mapping cone of the map of complexes $\check{C}^{\bullet}(\cU/Y,F) \to \check{C}^{\bullet}(\cV/X,G)$ induced by pullback along $f$ and the left adjunct $\alpha^{\dagger}:f^{-1}F\to G$ of $\alpha$. 
	The cohomology of this cone sits in a long exact sequence
	\begin{equation}\label{EqRCCS}
	\cdots \to \check{H}^{n}(\cU/Y,F) \to \check{H}^{n}(\cV/X,G) \to \check{H}^{n+1}(\cV,\cU;F,G) \to \check{H}^{n+1}(\cU/Y,F) \to \cdots.
	\end{equation}
	There are natural maps 
	\begin{equation}\label{EqNatural}
	\check{H}^{\bullet}(\cV/X,\cU/Y;F,G) \to H^{\bullet}_X(Y;F,G)
	\end{equation}
	from {\em relative \v{C}ech cohomology} to {\em relative étale cohomology}. See Appendix \ref{AppB} for the details on their construction. These maps are compatible with \eqref{EqRCS} and \eqref{EqRCCS}.
	
	A {\em relative procovering} of $f:X \to Y$ is a pro-object in the category of relative coverings of $f$. Analogously as for \eqref{EqNatural0}, the maps \eqref{EqNatural} also naturally extend to relative procoverings of $f$.
	
	\vskip2mm
	
	{\bf Vertical Brauer--Manin pairing.~}Let $f:X \to Y$ be a smooth proper morphism with geometrically integral fibers. We denote by $\G_m$ denote the relative sheaf $(\G_m,\G_m,\id)$. We then have an identification (see \eqref{Eq:RelCoh2})
	\begin{equation*}
	{H}^{3}(X,Y;\G_m) = \Br_{\text{hor}}(X/Y).
	\end{equation*}
	
	The following lemma gives the formula computing the vertical Brauer--Manin pairing. Its proof is completely straightforward, and mostly just a matter of unraveling the notation involved.
		
	Assume here that $K$ is a number field.
	
	\begin{lemma}\label{Prop:BMpairing}
		Let $(\cU,\psi:\cV \to \cU \times_YX)$ be a relative procovering of $f$ with $\psi$ surjective. Let $\oK \subset \oK_v$ be given embeddings for each $v$, and let $y \in Y(K), (P_v)_v \in X_y(\A_K)$, $\bar y \in \cU(\oK), \bar P_v \in \cV(\oK_v)$ for each $v$ be such that the diagram of procovered schemes
		\begin{equation}\label{Diagrprocoverd}
		\begin{tikzcd}
		{(\Spec K_v,\Spec \oK_v)} \arrow[d] \arrow[r, "{(P_v,\bar P_v)}"] & {(X,\cV)} \arrow[d, "{(f,\psi)}"] \\
		{(\Spec K,\Spec \oK)} \arrow[r, "{(y,\bar y)}"]                   & {(Y,\cU)}  .                      
		\end{tikzcd}
		\end{equation}
		commutes. Then for any $\check{b} \in \check{H}^{2}(\cV/X,\cU/Y;\G_m)$, any representative $u = (\alpha,\beta)  \in \check{Z}^{2}(\cV/X,\cU/Y;\G_m)$ of $\check{b}$, and any $\gamma \in C^2(\oK/K,\G_m)$ such that
		\[
		\d \gamma = \beta(y,\bar y),
		\]
		we have
		\begin{equation}\label{EqExpre}
		((P_v)_v, b)_{BM} = \sum_v \inv_v \left(\alpha_v - \gamma_v\right),
		\end{equation}
		where $\alpha_v:=\alpha(P_v,\bar P_v), \ \gamma_v := \gamma|_{\Gamma_v}$, and $b$ denotes the image of $\check{b}$ under the map
		\[
		\check{H}^{2}(\cV/X,\cU/Y;\G_m) \to H^{2}(X/Y,\G_m) =\Br_{\text{hor}}(X/Y).
		\]
	\end{lemma}
	
	\begin{remark}\label{Rmk10.3}
		(1). We note that a $\gamma$ as in the lemma always exists, as $\beta(y,\bar y)$ lies in $\check{Z}^{3}(\oK/K,\G_m)$, and $\check{H}^{3}(\oK/K,\G_m)={H}^{3}(\oK/K,\G_m)=0$ since $K$ is a number field \cite[VII.11.4, p. 199]{CasselsFrohlich}.
		
		(2). Given a relative procovering with $\psi$ surjective, and points $y\in Y(K),(P_v) \in X_y(\A_K)$, one may always find lifts $\bar y,\overline{P_v(q)}$ satisfying the commutativity \eqref{Diagrprocoverd} as follows.
		
		First, given any covering $\cU/Y$ and a point $y \in Y(K)$, we may always find a lift $\bar y \in \cU(\oK)$ lying above it using the surjectivity of $\cU \to Y$.
		Then, given any relative covering $(\cU,\cV;\psi:\cV \to \cU \times_YX)$ with $\psi$ surjective, points $y \in Y(K), (P_v) \in X_y(\A_K)$, and a lift $\bar y \in \cU(\oK)$ of $y$, letting $\overline{P_v}\in \cV(\oK_v)$ be any inverse image of the pair $(\bar y \otimes_{\oK} \oK_v, P_v\otimes_{K_v}\oK_v) \in (\cU\times_YX)(\oK_v) \subset \cU(\oK_v) \times X(\oK_v)$ under the surjective map $\cV(\oK_v) \to (\cU\times_YX)(\oK_v)$, diagram \eqref{Diagrprocoverd} commutes automatically.
	\end{remark}
	
	\begin{proof}
		After pulling back all the objects involved along $y \hookrightarrow Y$, we may assume that $Y=\Spec K$. Morover, we may assume that $(\cU,\psi:\cV \to \cU \times_YX)$ is a relative covering.
		
		\vskip1mm
		
		We first prove that the right hand side of \eqref{EqExpre} does not depend on the choices of $u$ and $\gamma$. The independence on $\gamma$ follows immediately from the Albert--Brauer--Hasse--Noether theorem. For the independence on $u$, let $v=(\psi^*f^*\delta,\d \delta) \in \check{B}^{2}(\cV/X,\cU/Y;\G_m)$, $\delta \in \check{C^2}(\cV/Y,\G_m)$ be a coborder. 
		For each $v$, Diagram \eqref{Diagrprocoverd} induces by pullback a commutative diagram 
		\[
		\begin{tikzcd}
		{\check{C}^{\bullet}(\cV/X,\G_m)} \arrow[rr, "{\alpha\mapsto \alpha(P_v,\bar P_v)}"]                     &  & {\check{C}^{\bullet}(\oK_v/K_v,\G_m)}       \\
		{\check{C}^{\bullet}(\cU/Y,\G_m)} \arrow[u, "\psi^*f^*"] \arrow[rr, "{\alpha \mapsto \alpha(y,\bar y)}"] &  & {\check{C}^{\bullet}(\oK/K,\G_m)} \arrow[u]
		\end{tikzcd}
		\]
		on \v{C}ech cochains, where the horizontal maps are $\alpha \mapsto \alpha(P_v,\bar P_v)$ and $\alpha \mapsto \alpha(y,\bar y)$, respectively. Identifying the image of $\gamma$ under the two compositions, we get
		\[
		(\psi^*f^*)(\delta)(P_v, \bar P_v)=\delta(y,\bar y)|_v.
		\]
		So $\alpha'_v=\alpha_v-\delta(y,\bar y)|_v$, and we infer that the right hand side of \eqref{EqExpre} remains the same after replacing $(u,\gamma)$ by $(u',\gamma'):=(u-v,\gamma-\delta),$ showing the independence on $u$.
		
		\vskip1mm
		
		Finally, to prove formula \eqref{EqExpre}, let us first prove that $\cV_{\bar y}:=\cV \times_{\cU} \bar y$ is a procovering of $X$ via the composition $\cV \times_{\cU} \bar y \to \cV \to X$. Since it is clearly an inverse system of collections of étale maps over $X$, we only need to show that $\cV_{\bar y}:=\cV \times_{\cU} \bar y\to X$ is surjective. To do so, note first that the projection
		\[
		\cV_{\bar y}:=\cV \times_{\cU} \bar y=\cV \times_{(\cU\times_YX)}(\bar y\times_YX) \to \bar y\times_YX 
		\]
		is surjective because the map $\cV \to \cU\times_YX$ is. The projection $\bar y\times_YX \to X$ is surjective as $\bar y \to Y$ is, and so $\cV_{\bar y} \to X$ is surjective too, proving that $\cV_{\bar y}$ is a procovering of $X$ as wished.
		
		We may thus pullback the statement along the refinement of relative procoverings
		\[
		(\bar y/y, \cV_{\bar y} \to \bar y\times_YX ) \to (\cU,\psi:\cV \to f^*\cU),
		\]
		and in doing so reduce the general case to the case $\cU=\{\oK/K\}$.
		In this case, we have $\check{H}^{3}(\cU/K,\G_m)=\check{H}^{3}(\oK/K,\G_m)=0$ (as $K$ is a number field), and so we may choose a representative $u$ for $\check{b}$ with $\beta=0$. When $\beta=0$, the cocycle $\alpha$ is a \v{C}ech cocyle representing a class of the usual Brauer group $\Br X$ whose image in $\Br_{\text{hor}}(X/K)=\Br X/\Br K$ is $b$. Then we may take $\gamma=0$ and formula \eqref{EqExpre} is just the usual expression for the Brauer--Manin pairing, where we only explicitated the pullbacks of the Brauer class to the points $P_v$ via \v{C}ech cocycles.
	\end{proof}

	\subsection{Proof of Theorem \ref{Thm:3variation}}\label{SSec103}
	We divide the proof in five main parts (distributed among respective subsubsections):
	\begin{enumerate}[label=\arabic*.]
		\item in the first we find a suitable \v{C}ech representative $\check{b}$ in a group ${\check{H}^3(\cU/Q,\mu_n)}$ for $b$ with coefficients in $\mu_n$, with $\cU$ the universal pro-finite étale cover of $Q$;
		\item in the second we choose lifts as in Lemma \ref{Prop:BMpairing}, set up the computation, and divide the resulting expression in two parts, containing the contribution from the fibers and the base, respectively;
		\item in the third we show that the first of the two parts above vanishes;
		\item in the fourth we show how the second of the two parts above reduces to a sum of Redéi symbols; in doing so, an important tool is Proposition \ref{PropVariation}, where we compute the ``triple variation'' of a suitable representative cocycles in $Z^3(\pi_1(Q),\mu_n)$;
		\item the fifth part is dedicated to the proof of Proposition \ref{PropVariation}.
	\end{enumerate}
	
	\subsubsection{Choice of \v{C}ech representative}\label{SSSec1031}
	
	We let $Q^{ur}:=\Spec \oK[(y^{\sigma})^{\pm\frac1\omega}]_{\sigma \in \Gamma}$ be the universal cover of $Q$, and $\bar Q:=\Spec \oK[(y^{\sigma})^{\pm1}]_{\sigma \in \Gamma}=Q \otimes_K \oK$.
	
	\begin{proposition}\label{Propcreationalphabeta}
		There exists a relative procovering of $f$ of the form $(\cV,Q^{ur},\psi:\cV \to Q^{ur} \times_QX)$ with $\psi$ surjective, and an element
		\begin{align*}
		(\alpha,\beta)  \in &\check{Z}^3(\cV/X,Q^{ur}/Q;\mu_n) \\
		\subset  &\check{C}^2(\cV/X,\mu_n) \times \check{Z}^3(Q^{ur}/Q,\mu_n)
		\end{align*}
		that maps to $b$ under the composition $\check{Z}^3(\cV/X,Q^{ur}/Q;\mu_n) \to H^3(X/Q,\mu_n) \to \Br_{\text{hor}}(X/Q)$ and such that:
		\begin{enumerate}
			\item $\beta$ lies in the image of $\check{Z}^3(Q^{\oK,n}/Q,\mu_n) \to \check{Z}^3(Q^{ur}/Q,\mu_n)$;
			\item the image of $\beta$ under the composition $\check{Z}^3(Q^{\oK,n}/Q,\mu_n) \to \check{H}^3(Q^{\oK,n}/Q,\mu_n) \to \check{H}^3(Q^{\oK,n}/\bar Q,\mu_n)$ lies in the $(\Z/n\Z)(-2)$-linear span of the cup-products $(y^{\sigma_1}\cup y^{\sigma_2} \cup y^{\sigma_3})_n,\, \{\sigma_1,\sigma_2,\sigma_3\} \in \binom{\Gamma}{3}$;
			\item there exists a finite extension $L'/L$ such that the restriction of $\beta$ to $\check{Z}^3(Q^{ur}/Q^{L',n,(1)},\mu_n)$ is trivial;
		\end{enumerate}
	\end{proposition}
	
	We divide the proof of Proposition \ref{Propcreationalphabeta} in two parts: existence of $(\cV,Q^{ur},\psi:\cV \to Q^{ur} \times_QX)$ and sought representative $(\alpha,\beta)$, and modification of $(\alpha,\beta)$ to satisfy i-iii.
	
	We need the following for the first part.
	
	\begin{lemma}
		Fix a covering $\cU/Q$. The relative coverings of the form $(\cU,\cV,\psi:\cV \to \cU \times_QX)$ with $\psi$ surjective form a filtered system when ordered by refinement. Moreover, the coverings $\cV$ appearing in the relative coverings in this poset are cofinal among all coverings of $X$. The analagous statement holds with procoverings instead of coverings.
	\end{lemma}
	\begin{proof}
		Given relative coverings $(\cU,\psi_1:\cV_1 \to \cU \times_QX)$ and $(\cU,\psi_2:\cV_2 \to \cU \times_QX)$ with $\psi_1,\psi_2$ surjective, the fibered product $\tilde \cV:=\cV_1 \times_{\psi_1,\cU \times_QX,\psi_2}\cV_2$ is another covering of $X$, and it maps surjectively to $\cU \times_QX$. Then the relative covering $(\cU, \tilde \cV \to \cU \times_QX)$ refines both, proving that the relative coverings in the statement form a filtered system. For the cofinality, note that any covering $\cV/X$ is refined by $\cV \times_X \cU \times_QX$, which can be extended to the relative covering $(\cU, \cV \times_X \cU \times_QX \to \cU \times_QX)$. 
		
		The statement on procoverings follows from that on coverings.
	\end{proof}
	
	\begin{proof}[Proof of Proposition \ref{Propcreationalphabeta}, existence of $(\alpha,\beta)$.]
		Applying the hypercohomology functors $H^{\bullet}(Q,-)$ to the ``truncated Kummer'' exact triangle
		\[
		\tau_{\geq 1}Rf_*\mu_n \to \tau_{\geq 1}Rf_*\G_m \to[{[n]}] \tau_{\geq 1}Rf_*\G_m \to[+1],
		\]
		we infer that the natural map $H^2(Q,\tau_{\geq 1}Rf_*\mu_n) \to H^2(Q,\tau_{\geq 1}Rf_*\G_m)[n]$ is surjective. Equivalently (see \eqref{Eq:RelCoh2}), the map $H^3(X/Q,\mu_n) \to H^3(X/Q,\G_m)[n]$ is surjective. 
		
		\vskip1mm
		
		Let $b_0 \in H^3(X/Q,\mu_n)$ be an inverse image of $b$. For any relative procovering $(\cU,\cV,\psi:\cV \to \cU \times_QX)$ of $f$, we get a commutative diagram with exact rows by comparing  \eqref{EqLES} and (the direct limit of) \eqref{EqLES2}:
		\begin{equation}\label{Equatiopn}
		\begin{tikzcd}[column sep=tiny]
		{\check{H}^3(\cU/Q,\mu_n)} \arrow[r] \arrow[d] & {\check{H}^2(\cV/X,\mu_n)} \arrow[d] \arrow[r] & {\check{H}^3(\cV/X,\cU/Q;\mu_n)} \arrow[d] \arrow[r, "\partial"] & {\check{H}^3(\cU/Q,\mu_n)} \arrow[d] \arrow[r] & {\check{H}^3(\cV/X,\mu_n)} \arrow[d] \\
		{H^3(Q,\mu_n)} \arrow[r, "f^*"]                & {H^2(X,\mu_n)} \arrow[r]                       & {H^3(X/Q,\mu_n)} \arrow[r, "\partial"]                           & {H^3(Q,\mu_n)} \arrow[r, "f^*"]                & {H^3(X,\mu_n)}                      
		\end{tikzcd}
		\end{equation}     
		
		Let $Q^{\text{ur}}=\Spec \oK[(y^{\sigma})^{\pm\frac1\omega}]_{\sigma \in \Gamma}$ be the profinite étale universal cover of $Q$. The pro-\v{C}ech-to-étale maps $\check{H}^k(Q^{\text{ur}}/Q,\mu_n) \to H^k(Q,\mu_n)$ are isomorphisms for all $k$ by Hochschild-Serre's spectral sequence (or rather, its profinite version \cite[Remark III.2.21(b)]{LECcompleto}) and the vanishing $H^k(Q^{\text{ur}},\mu_n)=0,\,k \geq 1$ \cite[Corollary 3.3]{GP}.
		So the first and fourth vertical maps in \eqref{Equatiopn} are isomorphisms when $\cU:=\{Q^{\text{ur}}/Q\}$. 
		
		Taking now the direct limit of diagram \eqref{Equatiopn} over all relative procoverings as in the lemma (taking such limit is allowed by Lemma \ref{LemIndepofRefin}), the second and fifth vertical maps become isomorphisms by the quasi-projectivity of $X$. Then the middle one is an isomorphism as well by the lemma of five. In particular, there exists a relative procovering $(\cU,\cV,\psi:\cV \to \cU \times_QX)$ with $\psi$ surjective for which $b$ lifts to a \v{C}ech class $\check{b}$.
		Let
		\[
		(\alpha,\beta) \in \check{Z}^2(\cV/X,Q^{ur}/Q;\mu_n) \subset \check{C}^2(\cV/X,\mu_n) \times \check{Z}^3(Q^{ur}/Q,\mu_n) 
		\]
		be a representative for $\check{b}$. 
	\end{proof}

	Before proceeding to the next part, we introduce some notation. 
	
	\begin{notation*}
		We set by convention the {\em fundamental group} of an intermediate cover between $Q^{ur}$ and $Q$, i.e.\ a profinite étale integral cover $Q' \to Q$ sitting in a diagram $Q^{ur} \to Q' \to Q$, to be $\pi_1(Q):=\Gal(Q^{ur}/Q)$.
	\end{notation*}
	
	The tower $Q^{ur} \to \bar Q \to Q$ gives by Galois theory a short exact sequence
	\[
	1 \to \pi_1(\bar Q) \to \pi_1 (Q) \to \Gamma_K \to 1.
	\]
	We recall that $\pi_1(\bar Q)=X_{*}(Q) \otimes_{\Z}\widehat \Z(1)$ (the same formula applies for any torus over an algebraically closed field of characterstic $0$). 
	We also let 
	\[
	G_n:=\pi_1(Q)/n\bar \pi_1(\bar Q)=\Gal(Q^{\oK,n}/Q) \text{ and }\bar G_n:=\pi_1(\bar Q)/n\pi_1(\bar Q)=\Gal(Q^{\oK,n}/\bar Q).
	\]
	Note that $\bar G_n =X_{*}(Q) \otimes_{\widehat\Z} (\Z/n\Z)(1) \cong (\Z/n\Z)(1)[\Gamma]$. More explicitly, the last isomorphism is given by:
	\[
	\bar G_n = \bigoplus_{\sigma \in \Gamma} \bar G_{n, \sigma}, \ \text{ with }
	\bar G_{n, \sigma} := \Gal(Q^{\oK,n}/Q^{\oK,n,(\sigma)}),
	\]
	\[
	Q^{\oK,n,(\sigma)} := \Spec \oK[(y^{\sigma'})^{\pm\frac1n},(y^{\sigma})^{\pm1}]_{\sigma' \in \Gamma \s \{\sigma\}}.
	\]
	Note that $\bar G_{n,\sigma}\cong \Z/n\Z$ as an abstract group.
	
	\begin{lemma}
		Let $d:=\#\Gamma$ and $\sigma_1:=\id,\sigma_2,\ldots,\sigma_{d}$ be the elements of $\Gamma$. We have a natural decomposition
		\[
		H^k(\bar G_n,\Z/n\Z)= \bigoplus_{k_1+\cdots +k_d=k} \left(H^{k_1}(\bar G_{n,\sigma_1},\Z/n\Z) \otimes_{\Z/n\Z} \cdots \otimes_{\Z/n\Z} H^{k_d}(\bar G_{n,\sigma_d},\Z/n\Z) \right)
		\]
		given by the cross-product map from the right to the left hand side.
	\end{lemma}
	\begin{proof}
		This is just K\"unneth's formula \cite[Theorem 3.6.3]{Weibel} applied to the direct sum $\bar G_n=\bar G_{n,\sigma_1}\oplus \cdots \oplus \bar G_{n,\sigma_d}$ with coefficients in $\Z/n\Z$. The flatness hypothesis of K\"unneth's formula holds because each $G_{n,\sigma_i}$ is isomorphic as a group to $\Z/n\Z$ and $H^p(\Z/n\Z,\Z/n\Z)\cong \Z/n\Z$ for every $p \geq 0$.
	\end{proof}
	
	We define 
	$H^k(\bar G_n,\Z/n\Z)_{cup} \subset H^k(\bar G_n,\Z/n\Z)$ as the subgroup corresponding, in the K\"unneth decomposition above, to strings $(k_1,\ldots,k_d)$ where each $k_i$ is either $0$ or $1$. 
	
	\vskip1mm
	
	We shall need the following two lemmas to continue the proof of Proposition \ref{Propcreationalphabeta}.
	
	\begin{lemma}\label{LemTwoPoints} The following hold for all $k \geq 0$.
		\begin{enumerate}
			\item The restriction of the inflation map $H^k(\bar G_n,\mu_n) \to H^k(\pi_1(\bar Q),\mu_n)$ to 
			$H^k(\bar G_n,\mu_n) \supset H^k(\bar G_n,\mu_n)_{cup}:=H^k(\bar G_n,\Z/n\Z)_{cup}\otimes_{\Z/n\Z} \mu_n$
			is an isomorphism.
			\item The \v{C}ech-to-\'etale map $H^k(G_n,\mu_n) \to H^k(Q,\mu_n)$ is surjective.
            \item Every element in $H^k(Q,\mu_n)$ admits an inverse image in $H^k(G_n,\mu_n)$ whose restriction to $H^k(\bar G_n,\mu_n)$ lies in $H^k(\bar G_n,\mu_n)_{cup}$.
		\end{enumerate}
	\end{lemma}
	\begin{proof}[Proof of i.]
		After Tate-twisting, we may prove that the restriction of $\inf: \linebreak H^k(\bar G_n,\Z/n\Z) \to H^k(\pi_1(\bar Q),\Z/n\Z)$ to $H^k(\bar G_n,\Z/n\Z)_{cup}$ is an isomorphism. Let $\pi_1(Q)=\oplus_{\sigma \in \Gamma}\bar G_{\infty,\sigma}$ be a direct sum decomposition lifting the decomposition $\bar G_n=\oplus_{\sigma \in \Gamma}\bar G_{\infty,\sigma}$ with each $\bar G_{\infty,\sigma} \cong \widehat \Z$ (e.g.\ take $\bar G_{\infty,\sigma}:=\Gal(Q^{ur}/Q^{\oK,\infty,\sigma})$ where $Q^{\oK,\infty,(\sigma)}:=\Spec \oK[(y^{\sigma'})^{\pm\frac1\omega},(y^{\sigma})^{\pm1}]_{\sigma' \in \Gamma \s \{\sigma\}}$). The inflation map $\inf:H^k(\bar G_n,\mu_n) \to H^k(\pi_1(\bar Q),\mu_n)$ respects the K\"unneth decompositions associated to the direct sum decompositions, and so we get a commutative diagram:
		\[
		\begin{tikzcd}[column sep=tiny]
		{H^k(\bar G_n,\Z/n\Z)} \arrow[r] \arrow[r, Rightarrow, no head] \arrow[d, "\inf"] & {\bigoplus_{k_1+\cdots +k_d=k} \left(H^{k_1}(\bar G_{n,\sigma_1},\Z/n\Z) \otimes_{\Z/n\Z} \cdots \otimes_{\Z/n\Z} H^{k_d}(\bar G_{n,\sigma_d},\Z/n\Z) \right)} \arrow[d, "\inf"] \\
		{H^k(\pi_1(\bar Q),\Z/n\Z)} \arrow[r, Rightarrow, no head]                        & {\bigoplus_{k_1+\cdots +k_d=k} \left(H^{k_1}(\bar G_{\infty,\sigma_1},\Z/n\Z) \otimes_{\Z/n\Z} \cdots \otimes_{\Z/n\Z} H^{k_d}(\bar G_{\infty,\sigma_d},\Z/n\Z) \right)}       .
		\end{tikzcd}
		\]
		The natural inflation map $\inf:H^k(\Z/n\Z,\Z/n\Z) \to H^k(\widehat \Z,\Z/n\Z)$ is an isomorphism for $k=0,1$, where both sides are $\Z/n\Z$, and is trivial for $k \geq 2$, where $H^k(\widehat \Z,\Z/n\Z)=0$ as the cohomological dimension of $\widehat \Z$ is $1$. So the same holds for each inflation map $\inf:H^k(\bar G_{n,\sigma},\Z/n\Z) \to H^k(\bar G_{\infty,\sigma},\Z/n\Z)$, and the statement follows.
    \end{proof}
		
	\begin{proof}[Proof of ii.] 
        Since $H^k(\pi_1(Q),\mu_n)=H^k(Q,\mu_n)$, it suffices to prove that the inflation map $H^k(G_n,\mu_n) \to H^k(\pi_1(Q),\mu_n)$ is surjective. Let $E^{p,q}_r$ and $\tilde E^{p,q}_r$ be the Hochschild--Serre spectral sequences associated to the the pairs $\bar G_n \triangleleft G_n$ and $\pi_1(\bar Q)\vartriangleleft \pi_1(Q)$, respectively. Their second pages are
		\[
		E^{p,q}_2:=H^p(\Gamma_K,H^q(\bar G_n,\mu_n)), \quad \tilde E^{p,q}_2:=H^p(\Gamma_K,H^q(\pi_1(\bar Q),\mu_n))
		\]
		and they abut to $H^k(G_n,\mu_n)$ and $H^k(\pi_1(Q),\mu_n)$. 
		
		\vskip1mm
		
		Consider, for each pair $p,q$, the composition
		\[
		H^p(\Gamma_K,H^q(\bar G_n,\mu_n)_{cup})\to H^p(\Gamma_K,H^q(\bar G_n,\mu_n))\to H^p(\Gamma_K,H^q(\pi_1(\bar Q),\mu_n)),
		\]
		where the second is the natural map coming from confronting the two sequences.
		This composition is an isomorphism by point i. In particular, the first map is injective and the second is surjective. Since $E^{p,q}_2=H^p(\Gamma_K,H^q(\bar G_n,\mu_n))$ and $\tilde E^{p,q}_2=H^p(\Gamma_K,H^q(\pi_1(\bar Q),\mu_n))$, this proves that the natural maps $E^{p,q}_2 \to \tilde E^{p,q}_2$ are surjective.
		
		\vskip1mm
		
		Since $\bar G_n = \pi_1(\bar Q)/n\pi_1(\bar Q)= X_*(Q) \otimes_{\Z} (\Z/n\Z)(1)=\Ind^{\Gamma_E}_{\Gamma_K}((\Z/n\Z)(1))$ as $\Gamma_K$-modules and $\Ind^{\Gamma_E}_{\Gamma_K}((\Z/n\Z)(1)) \cong \Ind^{\Gamma_E}_{\Gamma_K}(\Z/n\Z)$ (as $\mu_n \subset E$),  the sequence $E^{p,q}_2$ splits at the second page by Nakaoka's Theorem \ref{ThmNakaoka}. Thus, since, as proven above, the maps $E^{p,q}_2 \to \tilde E^{p,q}_2$ are surjective for all $p,q$, the spectral sequence $\tilde E^{p,q}_r$ must also degenerate at the second page and then point ii follows from the surjectivities $E^{p,q}_2 \to \tilde E^{p,q}_2$.
	\end{proof}
		
	\begin{proof}[Proof of iii.]
	    Nakaoka's theorem gives a direct sum decomposition $H^k(G_n,\mu_n)$ $=\bigoplus_{p+q=k}E^{p,q}_2$. The proof above shows that the composition
		\[
		\bigoplus_{p+q=k}H^p(\Gamma_K,H^q(\bar G_n,\mu_n)_{cup}) \to \bigoplus_{p+q=k}E^{p,q}_2=H^k(G_n,\mu_n) \to H^k(\pi_1(\bar Q),\mu_n),
		\]
		is surjective (and in fact, even an isomorphism, but we do not need this). Since the restriction $H^k(G_n,\mu_n) \to H^k(\bar G_n,\mu_n)$ is also expressible as the composition 
		\[
		H^k(G_n,\mu_n) \to E^{0,k}_{2} =H^k(\bar G_n,\mu_n)^{\Gamma_K}\subset H^k(\bar G_n,\mu_n),
		\]
		any element lying in the image of  $\bigoplus_{p+q=k}H^p(\Gamma_K,H^q(\bar G_n,\mu_n)_{cup}) \to \bigoplus_{p+q=k}E^{p,q}_2$ satisfies the sought condition.
	\end{proof}
	
	\begin{lemma}\label{thelemma}
		Let $G$ be a profinite group, $H$ an open subgroup, $M$ a discrete $G$-module, and $k$ a natural number. Any element of $Z^k(G,M)$ that maps to a coborder under the restriction $Z^k(G,M) \to Z^k(H,M)$ is equivalent up to a coborder to an element that maps to $0$ under the same map.
	\end{lemma}
	\begin{proof}
		Write $\alpha|_H=\d (\beta),\, \beta \in C^{k-1}(H,M)$. The restriction $C^{n}(G,M) \to C^n(H,M)$ is surjective for all $n$ (elements of $C^n(H,M)$ correspond to continuous functions $H^n \to M$, and each of these clearly extends to continuous functions $G^n \to M$, e.g.\ by composing with a topological retraction $G \to H$ of $H \to G$). So we can find $\gamma \in C^{k-1}(G,M)$ with $\gamma|_H=\beta$, and $\alpha - \d \gamma$ maps to $0$ in $Z^k(H,M)$.
	\end{proof}    
	
	\begin{proof}[Proof of Proposition \ref{Propcreationalphabeta}, modification of $(\alpha,\beta)$ to satisfy i-iii.] Taking $k=3$ in Lemma \ref{LemTwoPoints}.ii, the map $H^3(G_n,\mu_n)\to  \check{H}^3(Q^{ur}/Q,\mu_n)$ is surjective. Thus we may immediately modify $\beta$ to assume i, and using Lemma \ref{LemTwoPoints}.iii we also get ii. 
		
		\vskip1mm
		
		To get iii, we view $\beta$ as an element of $Z^3(G_n,\mu_n)$ via $\check{Z}^3(Q^{\oK,n},\mu_n) = Z^3(G_n,\mu_n)$ (see \eqref{EqIdentificationCechGroup}). The image of $\beta$ under the composition
		\[
		Z^3(G_n,\mu_n)\to H^3(G_n,\mu_n) \to H^3(\bar G_n,\mu_n)
		\]
		maps to $H^3(\bar G_n,\mu_n)_{cup}$. All elements of $H^3(\bar G_n,\mu_n)_{cup}$ restrict to zero in \linebreak $H^3(\bar G_{n,\id},\mu_n)$ (recall $\bar G_{n,\id}=\Gal(Q^{\oK,n}/Q^{\oK,n,(1)})$), since in each tensor product $H^1(\bar G_{n,\sigma_1},\Z/n\Z) \otimes_{\Z/n\Z}H^1(\bar G_{n,\sigma_2},\Z/n\Z) \otimes_{\Z/n\Z}H^1(\bar G_{n,\sigma_3},\Z/n\Z),\, \{\sigma_1,\sigma_2,\sigma_3\}\in \binom{\Gamma}{3}$, at least one among $\sigma_1,\sigma_2,\sigma_3$ is different from $\id$, and the corresponding factor $H^1(\bar G_{n,\star},\Z/n\Z)$ maps trivially to $\bar G_{n,\id}$, and so thus so does the whole tensor product. Using Lemma \ref{thelemma}, we  we may therefore assume, after modifying by a coborder, that $\beta$ restricts to zero in $Z^3(\bar G_{n,\id},\mu_n)={Z}^3(\Gal(Q^{\oK,n}/Q^{\oK,n,(1)}),\mu_n)$. The subgroups $\Gal(Q^{\oK,n}/Q^{L',n,(1)})<G_n,\, \text{ with }L'/L$ finite form a fundamental system of neighbourhoods of the finite group $\Gal(Q^{\oK,n}/Q^{\oK,n,(1)})$ in $\Gal(Q^{\oK,n}/Q)$. Thus, the cocycle $\beta \in Z^3(\Gal(Q^{\oK,n}/Q),\mu_n)$, being continuous, restricts trivially to $Z^3(\Gal(Q^{\oK,n}/Q^{L',n,(1)}),\mu_n)$ for some $L'$.
	\end{proof}
	
	\subsubsection{Setup for the triple variation}\label{SSSec1032}
	
	Since \v{C}ech cochains with respect to procoverings are just direct limits of \v{C}ech cochains with respect to the underlying inverse systems of coverings, there exists a finite étale connected cover $\tilde Q/Q$ and a refinement $\phi:\tilde \cV\to \tilde Q \times_QX$ such that $(\alpha,\beta)$ lies in the image of $\check{Z}^2(\tilde \cV/X,\tilde Q/Q;\mu_n) \to \check{Z}^2( \cV/X, Q^{ur}/Q;\mu_n)$ with respect to some refinement $(\cV,Q^{ur},\psi) \to (\tilde \cV, \tilde Q, \phi)$. Since $Q^{ur}= \lim\limits_{\leftarrow,(F,m)} Q^{F,m}$, we may assume that $\tilde Q=Q^{F,m}$ for some pair $(F,m)$. We may assume $(F,m) \succcurlyeq (L,n)$ and that the restriction of $\beta$ to $\check{Z}^3(\tilde Q/Q^{L,n,(1)},\mu_n)$ is zero. We enlarge $S$ so that all places of $K$ ramifying in $F$ or dividing $m$ are contained in it, and such that the relative covering $(\tilde Q/Q,\cV/X,\cV \to \tilde Q \times_QX)$ spreads out to a relative covering $(\tilde \cQ/\cQ,\mathfrak V/\cX,\mathfrak V \to \tilde \cQ \times_{\cQ}\cX)$ of $f:\cX \to \cQ$, where $\tilde \cQ:=\Spec \cO_{F,S}[(y^{\sigma})^{\pm\frac 1 {m}}]_{\sigma \in \Gamma}$.
	
	\vskip1mm
	
	For $q \in \cM$, Lemma \ref{Prop:BMpairing} gives
	\begin{equation}\label{Equationornf}
	((P_v)_v,q)_{BM}=\sum_v \inv_v(\alpha(P_v(q),\overline{P_v(q)})-\gamma(q)_v),
	\end{equation}
	for any choice of embeddings  $\oK \subset \oK_v$, lifts $\bar q \in \tilde Q(\oK),\, \overline{P_v(q)} \in \cV(\oK_v)$ satisfying \eqref{Diagrprocoverd}, and any cochain $\gamma(q) \in C^2(\oK/K,\G_m)$ such that $\d \gamma(q)=\beta(q,\bar q)$. 
	We choose the embeddings and the lifts as in the following proposition, and we choose $\gamma(q)$ arbitrarily.
	
	\vskip1mm
	
	For each archimedean $v$, we endow $\tilde Q(\oK_v) = \bigsqcup_{F \hookrightarrow \oK_v} (\oK_v^*)^\Gamma$ with the diagonal $\R_{>0}$-action induced by the inclusions $\R_{>0} \subset \C^* \cong \oK_v^*$.
	
	\begin{proposition}\label{orieogjiounvf}
		There exist embeddings $\oK \subset \oK_v$, lifts $\bar q,\,\overline{P_v(q)}$ satisfying \eqref{Diagrprocoverd}, such that:
		\begin{enumerate}[label=\roman*.]
			\item the point $\bar q$ lies in $\tilde Q(\oK)_{incl} \subset \tilde Q(\oK)$ (``incl'' stands for the inclusion $F \subset \oK$) for each $q \in \cM$;
			\item (multiplicativity) the function $\cM \to \tilde Q(\oK)_{incl}: q \mapsto \bar q$ is multiplicative;
			\item (locality at non-archimedean places in $S$) for each non-archimedean $v \in S$, the localizations $\bar q_v:=\bar q \otimes_{\oK}\oK_v \in \tilde Q(\oK_v),\, q \in \cM$ are arbitrarily close to each other;
			\item (same connected component at archimedean $v$) for each archimedean $v$, the projections $\bar q_v\cdot \R_{>0} \in \tilde Q(\oK_v)/\R_{>0},\, q \in \cM$ of the localizations $\bar q_v \in \tilde Q(\oK_v)$ are arbitrarily close to each other;
			\item (compatibility with the multi-section) for every $v \notin S$ and every $q \in \cM\s \cQ(\cO_v)$, the image of $\bar q$ under the map $\tilde Q(\oK_v) \to Q^{L,n,(1)}(\oK_v)$ is (the $\oK_v$-point associated to) $q^{L,n}$;
			\item (good reduction) for every $v \notin S$ and every $q \in \cM$ such that $q \in \cQ(\cO_v)$, the points $\overline{P_v(q)}$ lie in the natural image of $\mathfrak V(\cO_v^{ur})$ in $\tilde \cV(\overline{K_v})$.
		\end{enumerate}
	\end{proposition}
	
	Before proving the proposition, it is helpful to rewrite $\cM$ as:
	\[
	\cM = q_b \cdot \{1,\delta_1\} \cdot \{1,\delta_2\} \cdot \{1,\delta_3\}, \, q_b:=q_0q_{1,1}q_{2,1}q_{3,1},\, \delta_i:=q_{i,2}q_{i,1}^{-1}.
	\]
	For $v \notin S$ such that $\cM \not \subset \cQ(\cO_v)$, we also rewrite the intersection $\cM_v:=\cM \cap \cQ(\cO_v)$ as in the proof of Lemma \ref{LemCompatibleLifts}:
	\[
	\cM_v=\begin{cases}
	\cM & \text{ if }i_v=0;\\
	q_{\text{base}(v)}\cdot \prod_{i \neq i_v}\{1,\delta_i\}& \text{ if }i_v \neq 0;
	\end{cases}
	\]
	where $i_v:=0$ if $v(q_0) >0$, otherwise $i_v$ is defined as the first coordinate of the pair $(i_v,j_v)$ for which the associated $q_{i_v,j_v}$ corresponds to the prime ideal associated to $v$, and $q_{\text{base}(v)}:=q_0 \cdot q_{i_v,j_v} \cdot \prod_{i \neq i_v}q_{i_v,1}$. 
		
	\vskip1mm
	
	We need the following two lemmas:
	
	\begin{lemma}\label{Lem1}
		Let $\Omega$ be a field, $\Omega^{sep}$ be a separable closure of $\Omega$, and $\tilde \Omega$ be a second separably closed field containing $\Omega$. Let $m$ be a natural number, and $D=\{y_1,\ldots,y_d\} \subset \Omega^*$ be a subset such that the algebra $\Omega[\sqrt[m]{y}]_{y \in D}:=\Omega[t_1,\ldots,t_d]/(t_i^m-y_i)_{i=1,\ldots,d}$ is a field. Let $\sqrt[m]{y_1},\ldots,\sqrt[m]{y_d}$ be $m$-th roots of $y_1,\ldots,y_d$ in $\Omega^{sep}$. Then for every choice of $m$-th roots $\widetilde{\sqrt[m]{y_1}},\ldots,\widetilde{\sqrt[m]{y_d}}$ of $y_1,\ldots,y_d$ in $\tilde \Omega$, there exists an $\Omega$-embedding $\Omega^{sep} \subset \tilde \Omega$ sending each $\sqrt[m]{y_i}$ to $\widetilde{\sqrt[m]{y_i}}$.
	\end{lemma}
	\begin{proof}
		The roots $\{\sqrt[m]{y}\}_{y \in D}$ and $\{\widetilde{\sqrt[m]{y}}\}_{y \in D}$ determine homomorphisms of $\Omega[\sqrt[m]{y}]_{y \in D}$ with values in $\Omega^{sep}$ and $\tilde \Omega$, respectively. Since $\Omega[\sqrt[m]{y}]_{y \in D}$ is a field, both homomorphisms are embeddings. Extending then the identity on $\Omega[\sqrt[m]{y}]_{y \in D}$ to a field embedding $\Omega^{sep} \subset \tilde \Omega$ gives the desired embedding.
	\end{proof}
	
	\begin{lemma}\label{Lem22}
		The algebra $L_A:=L[\sqrt[m]{y}]_{y \in A}, \quad A:=\{q_b^{\sigma}\}_{\sigma \in \Gamma}\cup \{\delta_1\}_{\sigma \in \Gamma}\cup \{\delta_2\}_{\sigma \in \Gamma}\cup \{\delta_3\}_{\sigma \in \Gamma}$ is a field.
	\end{lemma}
	
	\begin{proof}
		We may embed $L_A$ in
		\begin{equation*}
		L_B := L[\sqrt[m]{y}]_{y \in B}, \quad B:=\{q_{i,j}^{\sigma}\}_{{\sigma\in \Gamma,\,i \in \{1,2,3\},\, j \in \{1,2\}}},
		\end{equation*}
		via $\sqrt[m]{q_b^{\sigma}}\mapsto \sqrt[m]{q_{1,1}^{\sigma}}\sqrt[m]{q_{2,1}^{\sigma}}\sqrt[m]{q_{3,1}^{\sigma}},\, \sqrt[m]{\delta_i^{\sigma}} \mapsto \sqrt[m]{q_{i,2}^{\sigma}}\sqrt[m]{q_{i,1}^{\sigma}}$. Note that $L_B$ is a field. In fact, we have $L_B=L \otimes_E E[\sqrt[m]y]_{y\in B}$ and the fields $L,\, E[\sqrt[m]y],\, y \in B$ are linearly disjoint over $E$, since for each $y$ the field $E[\sqrt[m]{y}]$ is totally ramified over $w_y$ (as $w_y(y)=1$), and unramified over $w_{y'}$ for $y' \neq y$, and $L$ is unramified over each $w_y$. Thus $L_A$ is a field as well as wished.
	\end{proof}

	\begin{proof}[Proof of Proposition \ref{orieogjiounvf}]
		
		Via the identification $\tilde Q(\oK)_{incl}=(\oK^*)^{\Gamma}$, lifts $\bar q \in \tilde Q(\oK)_{incl}$ of a $q \in Q(K)=E^*$ correspond to vectors $(\sqrt[m]{q^{\sigma}})_{\sigma \in \Gamma}$ of $m$-th roots of the conjugates of $q$.
		
		\vskip1mm
		
		Choose arbitrary $m$-th roots in $\oK$ of all elements of $A$. Such choice of roots determines by multiplicative extension a choice of $m$-th roots $\sqrt[m]{q^{\sigma}}$ for all $q \in \cM$ and $\sigma \in \Gamma$. These determine lifts $\bar q \in \tilde Q(\oK)_{incl}$ that are multiplicative by construction, giving i and ii.
		
		\vskip1mm
		
		We shall get points iii-v by appropriately choosing the embeddings $\oK \subset \oK_v$ with the help of the lemmas above.
		
		\vskip1mm
		
		For non-archimedean $v \in S$, we have that each $\delta_i$ is arbitrarily close to $1 \in Q(K_v)$, so the vector $(\delta_i^{\sigma})_{\sigma \in \Gamma} \in E^{\Gamma}$ is $w$-adically arbitrarily close to $(1,\ldots,1)$ for any place $w$ of $E$ dividing $v$. Choosing one $w$ and local $w$-adic roots in $\oE_w=\oK_v$ that are all close to $1$, we may apply Lemma \ref{Lem1} with $E=\Omega$ to infer that there exists an $E$-embedding $\oK \subset \oK_v$ under which all global $m$-th roots of the conjugates of the $\delta_i$ map to the just determined local roots that are close to $1$. This gives iii.
		
		\vskip1mm
		
		For archimedean $v$, we argue analogously after scaling. More precisely, there exists, for each $i$, an element $t_i \in \R_{>0}$ such that $\delta_i\cdot t_i^{-1}$ is arbitrarily close to $1 \in Q(K_v)$. So, for any place $w$ of $E$ dividing $v$, the vector $(\delta_i^{\sigma}t_i^{-1})_{\sigma \in \Gamma} \in E_w^{\Gamma}$ is arbitrarily close to $(1,\ldots,1)$. Arguing as for the non-archimedean case, we find an embedding $\oK \subset \oK_v$ under which each global root $\sqrt[m]{\delta_i^{\sigma}}$ maps to the local root of $\delta_i^{\sigma}t_i^{-1}$ close to $1$ times $\sqrt[m]{t_i}$. This gives iv.
		
		\vskip1mm
		
		Finally, let $v \notin S$ be such that $\cM \cap \cQ(\cO_v)\neq \emptyset$, and $\iota$ be the embedding $L \subset K_v$ determined by the lifts $q^{L,n}$ as in Lemma \ref{LemCompatibleLifts}. The lifts $q^{L,n}$ determine $n$-th roots $\sqrt[n]{q^{\sigma}}$ in $K_v$ for all $q \in \cM_v,\, \sigma \in \Gamma^*$ that are multiplicative in the variable $q \in \cM_v$. Applying Lemma \ref{Lem1} with $L=\Omega$, we find an embedding $\oK \subset \oK_v$ extending the embedding $\iota$ under which each global root
		\[
		\sqrt[n]{q_{\text{base}(v)}^{\sigma}}:=\left(\sqrt[m]{q_{\text{base}(v)}^{\sigma}}\right)^{m/n},\,\sqrt[n]{\delta_i^{\sigma}}:=(\sqrt[m]{\delta_i^{\sigma}})^{m/n},\, \sigma \in \Gamma^*, i \neq i_v
		\]
		determined above maps to the corresponding local root in $\oK_v$. This gives v.
		
		\vskip1mm
		
		For $v \notin S$ such that $\cM \cap \cQ(\cO_v)= \emptyset$, we choose $\oK \subset \oK_v$ arbitrarily.
		
		\vskip1mm
		
		For each $q \in \cM$ and all $v \in S$ or $v$ such that $q \notin \cQ(\cO_v)$, we choose any lift $\overline{P_v(q)} \in \cV(\oK_v)$ as in Remark \ref{Rmk10.3}(2), guaranteeing the compatibility \eqref{Diagrprocoverd}.
		
		\vskip1mm
		
		Finally, for $v \notin S$ and $q \in \cM \cap \cQ(\cO_v)$, it remains to create the lifts $\overline{P_v(q)}$ in such a way that vi holds. Since $q \in \cQ(\cO_v)$ and $\tilde \cQ \to \cQ$ is finite étale, the localization $\bar q_v:\Spec \oK_v \to \tilde Q$ extends to a morphism $\Spec \overline{\cO_v} \to \tilde \cQ$ which arises from a morphism $\Spec \cO_v^{ur} \to \tilde \cQ$. We still denote the latter by $\bar q_v$ with a slight abuse of notation. Consider now the commutative diagram
		\[
		\begin{tikzcd}
		& \Spec \cO_v^{ur} \arrow[d, "{(P_v(q),\bar q_v)}"] \arrow[r] \arrow[ld, dashed] & \Spec \cO_v \arrow[d, "P_v(q)"] \\
		\mathfrak V \arrow[r, "\psi"] & \cX\times_{\cQ}\tilde \cQ \arrow[r]                                            & \cX                            
		\end{tikzcd}
		\]
		The dotted lift may be found by the Henselianity of $\Spec \cO_v^{ur}$ and the surjectivity of $\psi$. Choosing as lift $\overline{P_v(q)}:\Spec \oK_v \to \cV$ the one induced by the map $ \Spec \cO_v^{ur}  \to \mathfrak V$ above, both vi and the compatibility \eqref{Diagrprocoverd} hold.
	\end{proof}
	
	To simplify notation, we let $\Delta^{(3)}:=\sum_{q \in \cM}(-1)^{\sgn(q)}$. Applying $\Delta^{(3)}$ to \eqref{Equationornf}, we get
	\begin{equation}\label{BMPairing}
	\Delta^{(3)} ((P_v)_v,q)_{BM} = \sum_v \inv_v\left( \Delta^{(3)}\alpha(P_v(q),\overline{P_v(q)})-(\Delta^{(3)}\gamma(q))_v\right).
	\end{equation}
	We have $\d (\Delta^{(3)}\gamma(q))=\Delta^{(3)}\beta(q,\bar q)$. Let us prove that this is a completely unramified cochain (in the sense of Section \ref{Sec4}). In fact, for $v\in S$, the cocycle $\beta(q,\bar q)_v \in \check{Z}^3(\oK_v/K_v,\mu_n)$ is constant throughout $\cM$ by our locality assumption since $\check{Z}^3(\oK_v/K_v,\mu_n)=Z^3(\Gamma_v,\mu_n)$ is discrete as a topological space. So its triple variation vanishes. For $v \notin S$ and $q \in \cQ(\cO_v)$, the cocycle $\beta(q,\bar q)_v$ is unramified if $q \in \cQ(\cO_v)$ by our good reduction assumption. For $v \notin S$ and $q \notin \cQ(\cO_v)$, we have $\beta(q,\bar q)_v=0$ as $\beta(q,\bar q)_v$ is the pullback of $\beta$ along the composition
	\[
	{(\Spec K_v,\Spec \oK_v)} \to {(\Spec K,\Spec \oK)} \to[(q,\bar q)] {(Q,\tilde Q)}
	\]
	which, by our assumption of compatibility with the multi-section, coincides with the composition 
	\[
	(\Spec K_v, \Spec \oK_v) \to[(q^{L,n},\bar q_v)] (Q^{L,n,(1)},\tilde Q) \to {(Q,\tilde Q)},
	\]
	and in \ref{SSSec1031} we chose $\beta$ so that its restriction to $(Q^{L,n,(1)},\tilde Q)$ is trivial.
	
	\vskip1mm
	
	Since $\d (\Delta^{(3)}\gamma(q))$ is completely unramified, we may split \eqref{BMPairing} as:
    \begin{multline*}
        \Delta^{(3)} ((P_v)_v,q)_{BM} \\
        = \sum_v \inv_v\left( \Delta^{(3)}\alpha(P_v(q),\overline{P_v(q)})-\eta_v\right)-\sum_v \inv_v \left((\Delta^{(3)}\gamma(q))_v-\eta_v\right),
    \end{multline*}
	where $\eta_v$ is any unramified primitive of $\Delta^{(3)}\beta(q,\bar q)_v$. We treat the two contributions separately.
	
	\subsubsection{First contribution}\label{SSSec1033}
	
	We need the following two auxiliary lemmas. 
	
	\begin{lemma}
		Let $S$ be a scheme, and
		\vskip-2mm
		\[
		\cU'' \mathrel{\overset{\psi_1}{\underset{\psi_2}{\rightrightarrows}}} \cU' \to[\psi_0] \cU
		\]
		be a commutative diagram of procoverings of $S$. Let $P$ be an étale sheaf on $S$, and $\alpha \in \check{C}^{\bullet}({\cU'}/S,P)$ be such that $\d \alpha = \psi_0^* \beta$ for some $\beta \in \check{Z}^{\bullet}({\cU}/S,P)$. Then the two pullbacks $\psi_1^*\alpha$ and $\psi_2^*\alpha$ differ by a coborder.
	\end{lemma}
	
	\begin{proof}
		We have $
		\psi_1^*\alpha-\psi_2^*\alpha = \d (K_{\psi_1,\psi_2}(\alpha))+K_{\psi_1,\psi_2}(\d \alpha)
		$ with the homotopy $K_{\psi_1,\psi_2}:\check{C}^{\bullet}({\cU'}/S,P) \to \check{C}^{\bullet-1}({\cU''}/S,P)$ defined in \eqref{TheHomotopy}.
		Moreover, writing $\d \alpha=\psi_0^*\beta,\, \beta \in C^{\bullet+1}(\cU,P)$, we have 
		\[
		K_{\psi_1,\psi_2}(\d \alpha)=K_{\psi_1,\psi_2}(\psi_0^*(\beta))=K_{\psi_0 \circ \psi_1,\psi_0 \circ \psi_2}(\beta)=\d K^{(2)}_{\psi_0 \circ \psi_1}(\beta),
		\]
		where we used $\psi_0 \circ \psi_1= \psi_0\circ \psi_2$ in the end, and $K^{(2)}:\check{C}^{\bullet}({\cU}/S,P) \to \check{C}^{\bullet-2}({\cU''}/S,P)$ denotes the $2$-homotopy \eqref{Equation}  (the $2$-homotopy described therein applies to a refinement of coverings, but taking a direct limit we immediately get an analogous homotopy for a refinement of procoverings).
	\end{proof}
	
	\begin{lemma}\label{DueSpec}
		Let $S$ be a scheme, $\cU' \to[\psi] \cU$ be a refinement of coverings, and $P$ be an étale sheaf on $S$. Let $\alpha \in \check{C}^{\bullet}({\cU'}/S,P)$ be such that $\d \alpha = \psi^*\beta$ with $\beta \in \check{Z}^{\bullet}({\cU}/S,P)$. Let $\Omega$ be a field with separable closure $\Omega^{sep}$. Let $P\in S(\Omega)$, and $\bar P_1,\bar P_2 \in \cU'(\Omega^{sep})$ be two lifts of $P$ with the same projection in $\cU(\Omega^{sep})$. Then $\alpha(P,\bar P_1)-\alpha(P,\bar P_2)$ is a coborder.
	\end{lemma}
	
	\begin{proof}
		We may assume $S=P$. Then this reduces to the previous lemma by taking $\cU''=\Spec \Omega^{sep}$.
	\end{proof}
	
	We shall show that $\inv_v\left( \Delta^{(3)}\alpha(P_v(q),\overline{P_v(q)})-\eta_v\right)=0$ for all $v$.
	We start by proving that $\Delta^{(3)}\alpha(P_v(q),\overline{P_v(q)})$ is a coborder for $v \in S$. Choosing $\eta_v=0$, this gives the sought vanishing for $v \in S$.
	Define
	\[
	\Xi_v:=  X(K_v) \times_{X(\oK_v)} \tilde \cV(\oK_v) , \quad \Xi_v^0:= X(K_v) \times_{Q(\oK_v)} \tilde Q(\oK_v),
	\]
	which we view as topological spaces through the $v$-adic topology. The natural map $\Xi_v \to \Xi_v^0$ is continuous and open as it is locally defined by specializing étale maps on $K_v$-points, and it is also surjective by the surjectivity of $\psi$. Let $C^2:=\check{C}^2(\oK_v/K_v,\mu_n)$, which we endow with the discrete topology, and $B^2:=\check{B}^2(\oK_v/K_v,\mu_n)$. By Lemma \ref{DueSpec}, the function $\Xi_v \to C^2,\, (P,\bar P) \mapsto \alpha(P,\bar P)$ sits in a commutative diagram
	\[
	\begin{tikzcd}
	\Xi_v \arrow[d] \arrow[r, "{\alpha \mapsto \alpha(P,\bar P)}"] & C^2 \arrow[d] \\
	\Xi_v^0 \arrow[r]                                              & C^2/B^2  .    
	\end{tikzcd}
	\]
	Since $\Xi_v \to C^2$ is clearly continuous and $\Xi_v \to \Xi^0_v$ is open and surjective, the induced map $\Xi_v^0 \to C^2/B^2$ is also continuous (by the diagram, the inverse image of an open under this last map is the projection of an open under $\Xi_v \to \Xi^0_v$). 
	
	\vskip1mm
	
	The image of a pair $(P_v(q),\overline{P_v(q)}) \in \Xi_v$ in $\Xi^0_v$ is $(P_v(q), \bar q_v)$. 
	
	\vskip1mm
	
	For non-archimedean $v \in S$, the pairs $(P_v(q), \bar q_v)\in \Xi^0_v,\, q \in \cM$ are arbitrarily close to each other by condition iii. 
	
	\vskip1mm
	
	On the other hand, for archimedean $v$, the pairs $(P_v(q), \bar q_v)\in \Xi^0_v,\, q \in \cM$ all lie on the same connected component of $\Xi^0_v$. In fact, by condition iv, the localizations $\bar q_v$ all lie on the same connected component $\tilde V_v \subset \tilde Q(\oK_v)$ of the inverse image of $V_v$ under $\tilde Q(\oK_v) \to Q(\oK_v)$. The pairs $(P_v(q),\bar q_v)$ then all lie in
	\[
	\tilde U_v := U_v \times_{V_v} \tilde V_v \subset X(K_v) \times_{Q(\oK_v)} \tilde Q(\oK_v) = \Xi^0_v.
	\]
	The projection $\tilde U_v \to \tilde V_v$ has connected fibers as it is a base-change of $U_v \to V_v$, which has connected fibers by hypothesis. Thus, since $\tilde V_v$ is connected, $\tilde U_v$ is connected as well.
	
	\vskip1mm
	
	In both the archimedean and the non-archimedean case, the function $\cM \to C^2/B^2,  q \mapsto \alpha(q,\bar q)$ is constant. In particular, its triple variation vanishes. I.e.\ $\Delta^{(3)}\alpha(P_v(q),\overline{P_v(q)})$ is a coborder for $v \in S$, as wished.
	
	\vskip2mm
	
	Let now $v \notin S$. We may assume $\cM \cap \cQ(\cO_v) \neq \emptyset$, for otherwise the cochain $\Delta^{(3)}\alpha(P_v(q),\overline{P_v(q)})$ is unramified, and so $\inv_v\left( \Delta^{(3)}\alpha(P_v(q),\overline{P_v(q)})-\eta_v\right)=0$ as this is the local invariant of an unramified cocycle. 
	
	\vskip1mm
	
	For $q \in \cM_v=\cM \cap \cQ(\cO_v)$, we have $\beta(q,\bar q)_v=0$, and so $\alpha(P_v(q),\overline{P_v(q)})$ is a cocycle.
	Let $Q_{sec}:=Q^{L,n,(1)}$. This maps to $X$ via $s:Q_{sec} \to X$. Pulling back $\tilde \cV$ along $s$, we obtain a covering $\tilde \cV \times_{X} Q_{sec}$ of $Q_{sec}$. We let $\tilde \cV_{sec} \subset \tilde \cV \times_{X} Q_{sec}$ be defined by the cartesian diagram
	\[
	\begin{tikzcd}
	\tilde \cV_{sec} \arrow[d] \arrow[r, hook] & \cV \times_X Q_{sec} \arrow[d, "{(\psi,\id)}"] \\
	\tilde Q \arrow[r, "\Delta", hook]         & \tilde Q \times_{Q}Q_{sec}           ,         
	\end{tikzcd}
	\]
	where $\Delta: \tilde Q \hookrightarrow \tilde Q \times_{Q}Q_{sec}$ is the ``diagonal map'', i.e.\ the map that is the identity on the first coordinate and the natural projection $\tilde Q \to Q_{sec}$ on the second. This is an open-closed embedding, and thus $\tilde \cV_{sec} \to \tilde \cV \times_X Q_{sec}$ is as well. Thus the projection $\cV_{sec} \to \tilde Q$ is a collection of étale maps. Moreover, the surjectivity of $\psi$ implies the surjectivity of  $\tilde \cV_{sec} \to \tilde Q$. In particular, $\tilde \cV_{sec}$ is a covering of $Q_{sec}$.
	
	\vskip1mm
	
	Now the definition of $\tilde \cV_{sec}$ gives the commutative diagram of procovered schemes in display on the left, which thus induces by pullback a commutative diagram on \v{C}ech cochains on the right:
	\[
	\begin{tikzcd}
	{(X,\tilde \cV)} \arrow[r]                       & {(Q, \tilde Q)}                    \\
	{(Q_{sec},\tilde \cV_{sec})} \arrow[r] \arrow[u] & {(Q^{L,n,(1)},\tilde Q)} \arrow[u]
	\end{tikzcd}\rightsquigarrow 
	\begin{tikzcd}
	{\check{C}^{\bullet}(\tilde Q/Q, \mu_n)} \arrow[r] \arrow[d] & {\check{C}^{\bullet}(\cV/X,\mu_n)} \arrow[d]   \\
	{\check{C}^{\bullet}(\tilde Q/Q_{sec},\mu_n)} \arrow[r]       & {\check{C}^{\bullet}(\tilde \cV_{sec}/Q_{sec},\mu_n)}.
	\end{tikzcd}
	\]
	In particular, since $\d \alpha \in \check{C}^3(\cV/X,\mu_n)$ is the image of $\beta \in \check{Z}^{3}(\tilde Q/Q, \mu_n)$ under the upper horizontal map, the coborder of the pullback $\alpha_{sec}$ of $\alpha$ to ${\check{C}^{2}(\tilde \cV_{sec}/Q_{sec},\mu_n)}$ is the image of the pullback $\beta_{sec}$ of $\beta$ to ${\check{Z}^{3}(\tilde Q/Q_{sec},\mu_n)}$. Bu $\beta_{sec}$ is zero by our assumptions, and so $\alpha_{sec}$ is a cocycle. We denote by $b_{\alpha}$ its image in $\Br Q_{sec}$ under the composition ${\check{Z}^{2}(\cV_{sec}/Q_{sec},\mu_n)} \to H^2(Q_{sec},\mu_n) \to H^2(Q_{sec},\G_m)$. Since, by our assumptions of compatibility with the multi-section, the morphism $(P_v(q),\overline{P_v(q)}):(\Spec K_v,\Spec \oK_v) \to (X, \tilde \cV)$ factors as
	\[
	(\Spec K_v,\Spec \oK_v) \to[(q^{L,n},\tilde q^{L,n})] (Q_{sec},\tilde \cV_{sec}) \to (X, \tilde \cV),
	\]
	we have
	\[
	\inv_v(\alpha(P_v(q),\overline{P_v(q)}))=\inv_v(b_{\alpha}(q^{L,n})).
	\]
	
	For each $q \in \cM_v$, the point $q^{L,n} \in Q^{n,L,(1)}(K_v)_{\iota}$ corresponds to a point of the form $(\iota(q),(\sqrt[n]{\iota(q^{\sigma})})_{\sigma \in \Gamma^*}) \in (K_v^*)^{\Gamma}$, for some choice of $n$-th roots, under the identification $Q^{n,L,(1)}(K_v)_{\iota}=(K_v^*)^{\Gamma}$ induced by the isomorphism $Q^{L,n,1}=\Spec L[y^{\pm1},(y^{\sigma})^{\pm\frac1n}]_{\sigma \in \Gamma^*} \cong \G_{m,L}^{\Gamma}$ of $L$-schemes (see the proof of Lemma \ref{LemLift}). The valuation of $\iota(q)$ is $1$, while the valuation of $\iota(q^{\sigma})$ is $0$ for all $\sigma \neq \id$, so the quotient $(q')^{L,n}(q^{L,n})^{-1}$ lies in $(\cO_v^*)^{\Gamma} \subset (K_v^*)^{\Gamma}$ for all $q,q' \in \cM_v$.
	
	\vskip1mm
	
	We need the following lemma.
	
	\begin{lemma}\label{LemDoubleVar}
		Let $d$ be a natural number, $b \in \Br \G_{m,K_v}^d$ be of order coprime with $Nv$, and $t_1,t_2,s_1,s_2$ be elements of $(K_v^*)^d=\G_{m,K_v}^d(K_v)$. If the quotients $t_2t_1^{-1},s_2s_1^{-1}$ lie in $(\cO_v^*)^d \subset (K_v^*)^d$, then
		\[
		\sum_{i,j \in \{1,2\}} (-1)^{i+j}\inv_v(b(t_is_j))=0
		\]
	\end{lemma}
	\begin{proof}
		Write $\G_{m,K_v}^d=\Spec K_v[y_1^{\pm1},\ldots,y_d^{\pm1}]$. If $b \in \Br_1 \G_{m,K_v}^d$, write $b=b_0+b_1$ with $b_0 \in \im \Br K_v$, $b_1 = \sum_i (\chi_i,y_i)$ for characters $\chi_i:\Gamma_{K_v} \to \qz$. Then the sum in the statement vanishes for both $b=b_0$, as the function $\G_{m,K_v}^d(K_v) \to \qz:P\mapsto \inv_v(b_0(P))$ is constant, and for $b=b_1$, as the function $\G_{m,K_v}^d(K_v) \to \qz:P\mapsto \inv_v(b_1(P))$ is linear. 
		
		\vskip1mm
		
		It thus suffices to prove the lemma for $b=c\cdot (y_a,y_b)_m,\, a <b$ with $m$ coprime with $Nv$ such that $\mu_m \subset K_v$ and $c \in (\Z/m\Z)(-1)$, since the transcendental Brauer group of split tori is generated by such twisted cup-products \cite[Proposition 9.1.2]{BGbook}. We assume $b=c\cdot (y_1,y_2)_m$ after pemuting the coordinates, and $c=1$ after fixing an $m$-th root of unity. Then $\inv_v(b((p_1,\ldots,p_d)))=[p_1,p_2]_v$, where the latter denotes the additive variant of the Tate pairing (taken with values in $\Z/m\Z$ instead of $\mu_m$ via the identification $\Z/m\Z=\mu_m$). So, writing $t_i=(t_{1,i},\ldots,t_{d,i}),\,s_i=(s_{1,i},\ldots,s_{d,i})$, the sum in the statement becomes
		\[
		\sum_{i,j \in \{1,2\}} (-1)^{i+j}[t_{1,i}s_{1,j},t_{2,i}s_{2,j}]_v =  [t_{1,2}t_{1,1}^{-1},s_{2,2}s_{2,1}^{-1}]_v + [s_{1,2}s_{1,1}^{-1},t_{2,2}t_{2,1}^{-1}]_v,
		\]
		which is trivial bacause the coprime-to-the-residual-characteristic Tate pairing of $v$-adic units vanishes.
	\end{proof}
	
	As the rank of $\cM_v$ is at least $2$ and the variations of the grid $\{q^{n,L}\}_{q \in \cM_v}$ are $v$-adically integral as proven above, Lemma \ref{LemDoubleVar} gives 
	\[
	\sum_{q \in \cM_v} (-1)^{\sgn(q)}\inv_v(b_{\alpha}(q^{L,n}))=0.
	\]
	I.e.\ the contributions of the $q \in \cM_v$ in $\inv_v\left(\Delta^{(3)}\alpha(P_v(q),\overline{P_v(q)})-\eta_v\right)$ cancel each other out. On the other hand, when $q \in \cM \s \cM_v = \cM \cap \cQ(\cO_v)$, the cochain $\alpha(P_v(q),\overline{P_v(q)})$ is unramified, and so contributes trivially. We thus get $\inv_v\left(\Delta^{(3)}\alpha(P_v(q),\overline{P_v(q)})-\eta_v\right)=0$ as wished.
	
	\subsubsection{Second contribution}\label{SSSec1034}
	
	The cochain $\Delta^{(3)}\gamma(q) \in C^2(\oK/K,\G_m)$ is a primitive for $\Delta^{(3)}\beta(q,\bar q)$. If $R \in C^2(\oK/K,\G_m)$ is any other primitive for $\Delta^{(3)}\beta(q,\bar q)$, then $\sum_{v} \inv_v(R_v - \Delta^{(3)}\gamma(q)_v)=0$ by the Albert--Brauer--Hasse--Noether theorem, and so
	\begin{equation}\label{R}
	\sum_v \inv_v(\Delta^{(3)}\gamma(q)_v-\eta_v)= \sum_v \inv_v(R_v-\eta_v).
	\end{equation}
	
	\vskip1mm
	
	Recall that
	\[
	\cM=q_b \cdot\left\{1, \delta_1\right\} \cdot\left\{1, \delta_2\right\} \cdot\left\{1, \delta_3\right\},\, q_b:=q_0 q_{1,1} q_{2,1} q_{3,1},\, \delta_i=q_{i,2}q_{i,1}^{-1}.
	\]
	In order to proceed further with the computation, it is useful to consider the point $p_{\cM}:=(q_b,\delta_1,\delta_2,\delta_3) \in Q^4(K)$.
	
	\begin{notation*}
		For $k \geq 0$, and a vector $\mathbf v = (v_1,\ldots,v_k) \in \{0,1\}^k$, we let $m_{\mathbf v}:Q^{k+1} \to Q$ denote the morphism $(q_0,\ldots,q_k)\mapsto q_0 \cdot \prod_{i \geq 1} q_i^{v_i}$. 
	\end{notation*}
	
	The maps just defined may be used to recover the elements of $\cM$ from $p_{\cM}$. Namely, for $\mathbf v := (v_1,v_2,v_3) \in \{0,1\}^3$, we have
	\[
	m_{\mathbf{v}}(p_{\cM})=q_b\delta_1^{v_1}\delta_2^{v_2}\delta_3^{v_3} \in \cM.
	\]
	
	Proposition \ref{PropVariation} below, proved in \ref{SSSec1035}, shall give us a suitable $R$ for \eqref{R}.
	Before stating it, we need to introduce suitable coordinates for the powers of $Q$. 
	
	\begin{notation*}
		For a power $Q^{k+1},\, k \geq 0$ of $Q$, and $i \in \{0,\ldots,k\}$, we let $\{y_i^{\sigma}\}_{\sigma \in \Gamma}$ be the set of generating characters of the $i$-th factor of $Q^{k+1}$, so that
		\[
		Q^{k+1}=\Spec \left(\oK[(y_i^{\sigma})^{\pm1}]_{\sigma \in \Gamma,0\leq i\leq k}\right)^{\Gamma_K}.
		\]
		We also let
		\[
		(Q^{k+1})^{ur}:=\Spec \oK[(y_i^{\sigma})^{\pm\frac 1 \omega}]_{\sigma \in \Gamma,0\leq i\leq k}, \quad \pi_1(Q^{k+1}):=\Gal((Q^{k+1})^{ur}/Q^{k+1})
		\]
		be the universal cover of $Q^{k+1}$ and its étale fundamental group, respectively.
	\end{notation*}

	For any intermediate cover $Q'$ between $(Q^{k+1})^{ur}$  and $Q^{k+1}$ we set, as before, $\pi_1(Q'):=\Gal((Q^{k+1})^{ur}/Q')$ by convention.
	
	\vskip1mm
	
	In the following, we view $\beta$ as an element of $Z^3(\pi_1(Q),\mu_n)$ via the identification $Z^3(Q^{ur}/Q,\mu_n)=Z^3(\pi_1(Q),\mu_n)$ given by \eqref{EqIdentificationCechGroup}.
	
	\begin{proposition}\label{PropVariation}
		We have
		\[
		\sum_{\mathbf v \in \{0,1\}^3}(-1)^{|\mathbf v|}m_{\mathbf v}^*\beta = \sum_{\sigma_1,\sigma_2\in \Gamma^*, \sigma_1 \neq \sigma_2}\cores_{E/K} (a_{\id,\sigma_1,\sigma_2}(\beta) \cdot(  y_1 \cup y_2^{\sigma_1}\cup y_3^{\sigma_2})_n)+\d \delta
		\]
		in $C^3(\pi_1(Q^4),\mu_n)$, where each cup-product $(y_1 \cup y_2^{\sigma_1}\cup y_3^{\sigma_2})_n$ is viewed as an element of $Z^3(\pi_1(Q^4_E),\mu_n)$, and  $\delta \in C^2(\pi_1(Q^4),\mu_n)$ is such that for each $i=1,2,3$, the restriction of $\delta$ to
		\[
		C^2(H_i,\mu_n),\,
		H_i:=\pi_1(Q^i \times_K Q^{K,n} \times_K Q^{3-i})
		\]
		is a coborder.
	\end{proposition}
	
	{\bf Warning.}~Here $(y_1 \cup y_2^{\sigma_1}\cup y_3^{\sigma_2})_n \in Z^3(\pi_1(Q_E),\mu_n)$ really denotes the cup-product as a {\em group cocycle} and not as a mere cohomology class. Analogously, $\cores_{E/K}$ denotes the corestriction operator on {\em group cochains} \cite[p.48]{GermanBook}.
	
	\vskip1mm
	
	To the lifts $\bar q_b,\bar \delta_1,\bar \delta_2,\bar \delta_3$ defined in \ref{SSSec1032}, we may associate the point $\bar p_{\cM} \in \tilde Q^{\times_F 4}(\oK)_{incl}$ with $\tilde Q^{\times_F 4}=\tilde Q \times_F \tilde Q \times_F \tilde Q \times_F \tilde Q$. We may lift this further to a $\oK$-point $\bar p_{\cM} \in (Q^4)^{ur}(\oK)_{id}$ along the morphism $(Q^4)^{ur} \to Q^{\times_F 4}$ (viewed as a morphism of $F$-schemes). We then define, for each $q=q_b\delta_1^{v_1}\delta_2^{v_2}\delta_3^{v_3} \in \cM$, a lift $\bar q \in Q^{ur}(\oK)_{id}$ as $m_{\mathbf v}(\bar p_{\cM})$, where $m_{\mathbf v}$ denotes the morphism $m_{\mathbf v}:(Q^4)^{ur} \to Q$ with $\mathbf v=(v_1,v_2,v_3)$. By the mulitplicativity imposed on the lifts in $\tilde Q$ in \ref{SSSec1032}, each $\bar q$ projects to the lift $\bar q \in \tilde Q(\oK)_{incl}$ defined there.
	
	\vskip1mm
	
	Observe how, for each $q \in \cM$, the lift $\bar q \in Q^{ur}(\oK)$ defines a representation $D_q:\Gamma_K \to \pi_1(Q)$ defined by $D_q(\gamma)\cdot \bar q=\bar q^{\gamma}$.
	
	Analogously, for the point $p_{\cM} \in (Q^4)(\oK)$, the lift $\bar p_{\cM} \in (Q^4)^{ur}(\oK)$ defines a decomposition representation $D_{{\cM}}:\Gamma_K \to \pi_1(Q^4)$ defined by $D_{{\cM}}(\gamma) \cdot \bar p_{\cM}=(\bar p_{\cM})^{\gamma}.$ This recovers the representations $D_q,\, q \in \cM$: for each $\mathbf v=(v_1,v_2,v_3) \in \{0,1\}^3$, the homomorphism $D_q:\Gamma_K \to \pi_1(Q)$ coincides with the composition
	\begin{equation*}
	\Gamma_K \to[{D_{{\cM}}}] \pi_1(Q^4) \to[m_{\mathbf v}] \pi_1(Q)
	\end{equation*}
	by the definition of our lifts $\bar q,q \in \cM$.
	
	Pulling back the identity of Proposition \ref{PropVariation} along $D_{{\cM}}$ and using that for all $q=q_b\delta_1^{v_1}\delta_2^{v_2}\delta_3^{v_3}$, we have $m_{\mathbf v}\circ D_{\cM}=D_q$, we get
	\[
	\sum_{q \in \cM}(-1)^{\sgn(q)}D_q^*(\beta) = \sum_{\sigma_1,\sigma_2 \in \Gamma^*,\sigma_1 \neq \sigma_2} \cores_{E/K}(a_{\id,\sigma_1,\sigma_2}\cdot (\delta_1 \cup \delta_2^{\sigma_1} \cup \delta_3^{\sigma_2})_n) + d (D_{\cM}^*(\delta))
	\]
	in $C^3(K,\mu_n)$. Via \eqref{EqPullbackgroup}, the left hand side is equal to $\sum_{q \in \cM}(-1)^{\sgn(q)}\beta(q, \bar q)=\Delta^{(3)}\beta(q, \bar q)$. 
	
	By Lemma \ref{Lem:KMWellPosed}, there exists a primitive $R_{\sigma_1,\sigma_2} \in C^2(E,(\Z/n\Z)(3))$ for the cocycle $(\delta_1 \cup \delta_2^{\sigma_1} \cup \delta_3^{\sigma_2})_n \in Z^3(E,(\Z/n\Z)(3))$. So now
	\[
	R:=\sum_{\sigma_1,\sigma_2 \in \Gamma,\sigma_1 \neq \sigma_2} \cores_{E/K}(a_{\id,\sigma_1,\sigma_2}\cdot R_{\sigma_1,\sigma_2}) + D_{\cM}^*(\delta)
	\]
	gives another primitive for $\Delta^{(3)}\beta(q,\bar q)$. Letting $\eta'_v:=D_{\cM}^*(\delta)_v+\eta_v$ for each $v$, we get
	\[
	\eqref{R}=\sum_v \inv_v \left(\sum_{\sigma_1,\sigma_2}\nolimits^*(\cores_{E/K}(a_{\id,\sigma_1,\sigma_2}\cdot R_{\sigma_1,\sigma_2})|_{\Gamma_v} + \eta'_v\right),
	\]
	where we are abbreviating via $\sum_{\sigma_1,\sigma_2}^*:=\sum_{\sigma_1,\sigma_2 \in \Gamma^*,\sigma_1 \neq \sigma_2}$.
	
	The cochain $D_{\cM}^*(\delta)$ is everywhere unramified. In fact, this is clear for $v \notin S$ such that for which $\cM \subset \cQ(\cO_v)$. For $v \in S$ and $v \notin S$ such that $\cM \cap \cQ(\cO_v) \neq \emptyset$, the elements $\delta_i \in Q(K)=E^*$ with $i \neq i_v$ (recall that $i_v \in \{0,1,2,3\}$ denotes the index such that $q_{i',j} \in \cQ(\cO_v)$ for all $i' \neq i_v$) lie in the image of $Q(K_v) \to[{[n]}] Q(K_v)$: while for $v \in S$ this is clear because all $\delta_i$'s are close to $1$ (resp.\ close to $1$ after scaling by a positive real factor) if $v$ is non-archimedean (resp.\ archimedean), for $v \notin S$ this holds because the quotients $q' \cdot q^{-1},\, q,q' \in \cM_v$ are assumed to lie in $Q^{L,n}(K_v) \to Q(K_v)$, which factors through $Q(K_v) \to[{[n]}] Q(K_v)$.

	Thus the composition $\Gamma_v \to \Gamma_K \to[D_{\cM}] \pi_1(Q^4)$ lands in $H_i$ for all $i\in\{1,2,3\}\s\{i_v\}$. So the cochain $D_{\cM}^*(\delta)$ is everywhere unramified as wished. Thus $\eta'_v$ is unramified for all $v$.
	
	The cocycle $(\delta_1 \cup \delta_2^{\sigma_1} \cup \delta_3^{\sigma_2})_n$ is everywhere unramified ($\delta_1, \delta_2^{\sigma_1}, \delta_3^{\sigma_2}$ satisfy condition (C)), so for each place $w$ of $E$, there exists $\xi_{\sigma_1,\sigma_2,w} \in C^2(\Gal(E_w^{ur}/E_w),(\Z/n\Z)(3))$ with $\d \xi_{\sigma_1,\sigma_2,w}=(\delta_1 \cup \delta_2^{\sigma_1} \cup \delta_3^{\sigma_2})_n|_{\Gamma_w}$. We may also assume that $\xi_{\sigma_1,\sigma_2,w}=0$ when $w \mid S$, since $(\delta_1 \cup \delta_2^{\sigma_1} \cup \delta_3^{\sigma_2})_n|_{\Gamma_w}=0$ for such $w$ as all conjugates of all $\delta_i$'s are $n$-th powers in $E_w$. So, rewriting  \eqref{R} as
	\begin{align*}
	  &\sum_v \inv_v \left(\sum_{\sigma_1,\sigma_2}\nolimits^*(\cores_{E/K}(a_{\id,\sigma_1,\sigma_2}\cdot R_{\sigma_1,\sigma_2})|_{\Gamma_v} + \eta'_v\right)\\
	=&\sum_v  \inv_v \left(\sum_{w \mid v} \left(\sum_{\sigma_1,\sigma_2}\nolimits^*\cores_{E_w/K_v}(a_{\id,\sigma_1,\sigma_2}\cdot R_{\sigma_1,\sigma_2}|_{\Gamma_w})\right) + \eta'_v\right)\\
	=&\sum_v  \inv_v \left(\sum_{w \mid v} \left(\sum_{\sigma_1,\sigma_2}\nolimits^*\cores_{E_w/K_v}(a_{\id,\sigma_1,\sigma_2}\cdot R_{\sigma_1,\sigma_2}|_{\Gamma_w}-a_{\id,\sigma_1,\sigma_2}\cdot \xi_{\sigma_1,\sigma_2,w})\right) +\eta''_v \right)
	\end{align*}
	with $\eta''_v:=\sum_{w \mid v} \sum_{\sigma_1,\sigma_2}\nolimits^*\left(\cores_{E_w/K_v} (a_{\id,\sigma_1,\sigma_2}\cdot \xi_{\sigma_1,\sigma_2,w})\right) + \eta'_v$, this $\eta''_v$ is unramified for all $v$ as it is the corestriction of an unramified cochain via the unramified extension $E_w/K_v$ for $v \notin S$, and it is $0$ for $v \in S$. Since each $R_{\sigma_1,\sigma_2}|_{\Gamma_w}-\xi_{\sigma_1,\sigma_2,w}$ is a cocycle for each $w$ by construction, we get that $\eta''_v$ is also a cocycle and $\inv_v(\eta''_v)=0$ for all $v$ because it is unramified. Thus, we may rearrange the sum as:
	\begin{align*}
	\eqref{R}   &= \sum_{\sigma_1,\sigma_2}\nolimits^*\sum_v  \inv_v \left(\sum_{w \mid v} \cores_{E_w/K_v}(a_{\id,\sigma_1,\sigma_2}\cdot (R_{\sigma_1,\sigma_2}|_{\Gamma_w}-\xi_{\sigma_1,\sigma_2,w}))\right)\\
	&= \sum_{\sigma_1,\sigma_2}\nolimits^*\sum_w \inv_w (a_{\id,\sigma_1,\sigma_2}\cdot (R_{\sigma_1,\sigma_2}|_{\Gamma_w}-\xi_{\sigma_1,\sigma_2,w}))\\
	&= \sum_{\sigma_1,\sigma_2}\nolimits^* a_{\id,\sigma_1,\sigma_2}\cdot \sum_w  \inv_w (R_{\sigma_1,\sigma_2}|_{\Gamma_w}-\xi_{\sigma_1,\sigma_2,w})\\
	&= \sum_{\sigma_1,\sigma_2 \in \Gamma^*, \sigma_1 \neq \sigma_2}a_{\id,\sigma_1,\sigma_2}\cdot [\delta_1,\delta_2^{\sigma_1},\delta_3^{\sigma_2}]_n=R_{{\bar \partial(b)}}[\delta_1,\delta_2,\delta_3],
	\end{align*}
	concluding the proof.
	
	\subsubsection{Proof of Proposition \ref{PropVariation}}\label{SSSec1035}
	
	We need some additional notation before going to the heart of the proof.
	
	Let $\bar Q:=Q \otimes_K \oK=\Spec \oK[y_{\sigma}^{\pm1}]_{\sigma \in \Gamma}$ and $\pi_1(\bar Q):=\Gal(Q^{ur}/\bar Q) \cong \widehat \Z [\Gamma](1)$ be its fundamental group. Galois theory provides a split short exact sequence
	\[
	1 \to \pi_1(\bar Q) \to \pi_1(Q) \to \Gamma_K \to 1.
	\]
	For $k \geq 0$, we let 
    \[
    G_n^{(k+1)}:=\pi_1(Q^{k+1})/n\pi_1(\bar Q^{k+1})\text{ and }\bar G_n^{(k+1)}:=\pi_1(\bar Q^{k+1})/n\pi_1(\bar Q^{k+1}).
    \]These corresponds via Galois Theory to the Galois groups of 
	\[
	\Spec \oK[(y_{i}^{\sigma})^{\pm\frac1n}]_{\sigma \in \Gamma,0 \leq i \leq k}/Q^{k+1}, \quad \Spec \oK[(y_{i}^{\sigma})^{\pm\frac1n}]_{\sigma \in \Gamma,0 \leq i \leq k}/\bar Q^{k+1},
	\]
	respectively. By Galois theory, we have $G_n^{(k+1)}= \bar G_n^{(k+1)} \rtimes \Gamma_K$ and $\bar G_n^{(k+1)} = \bar G_n^{[0]} \oplus \cdot \oplus \bar G_n^{[k]}$, where each $\bar G_n^{[j]}$ denotes the Galois group of
	\[
	\Spec \oK[(y_{i}^{\sigma})^{\pm\frac1n}]_{\sigma \in \Gamma,0 \leq i \leq k}/\Spec \oK[y_j^{\pm1},(y_{i}^{\sigma})^{\pm\frac1n}]_{\sigma \in \Gamma,0 \leq i \leq k,\, i\neq j}.
	\]

	\vskip1mm
	
	For each subset $J \subset \{0,\ldots,k\}$, we let $\bar G_n^{[J]}:=\oplus_{j \in J}\bar G_n^{[j]}$ and $G_n^{[J]}:=\bar G_n^{[J]} \rtimes \Gamma_K$.
	
	\vskip1mm
	
	Recall that we are assuming that $\beta$ lies in the image of the embedding \linebreak $Z^3(G_n,\mu_n) \to Z^3(\pi_1(Q),\mu_n)$. 
	We shall prove a slightly stronger statement than Proposition \ref{PropVariation}:
	\begin{proposition}\label{PropVariation2}
		Viewing $\beta$ as an element of $Z^3(G_n,\mu_n)$, the identity
		\[
		\sum_{\mathbf v \in \{0,1\}^3}m_{\mathbf v}^*(\beta) = - \sum_{\sigma_1,\sigma_2 \in \Gamma}\cores_{E/K}(a_{\id,\sigma_1,\sigma_2}(\beta)\cdot (y_1 \cup y_2^{\sigma_1}\cup y_3^{\sigma_2})_n)+\d \delta
		\]
		holds in $Z^3(G_n^{(4)},\mu_n)$ for some $\delta \in C^2(G_n^{(4)},\mu_n)$ restricting trivially to $C^2(G_n^{(I)},\mu_n)$ for $I=\{0,1,2\},\{0,1,3\}$ and $\{0,2,3\}$.
	\end{proposition}
	
	\begin{proof}
		We divide the proof in three steps: we first prove that the two sides reduce to the same class in $H^3(\bar G_n^{(4)},\mu_n)$, then that they reduce to the same class in $H^3(G_n^{(4)},\mu_n)$, and finally that we may find the primitive $\delta$ for the difference satisfying the sought vanishings.
		
		\vskip1mm
		
		{\em Step 1.} Choose a total ordering $<$ of $\Gamma$. The image of $\beta$ in 
		$H^3(\bar G_n,\mu_n)$ lies in $H^3(\bar G_n,\mu_n)_{cup}$, and so we may write it as 
		\[
		\sum_{\sigma_1 < \sigma_2 < \sigma_3} a_{\sigma_1,\sigma_2,\sigma_3}(\beta)\cdot (y^{\sigma_1} \cup y^{\sigma_2} \cup y^{\sigma_3})_n.
		\]
		In this proof, all cup-products shall denote cup-products of order $n$, so we drop the corresponding subscript from notation.
		For a fixed triple $\sigma_1 < \sigma_2 < \sigma_3$, the triple variation of $y^{\sigma_1} \cup y^{\sigma_2} \cup y^{\sigma_3}$ in $H^3(\bar G_n^{(4)},\mu_n)$ is:
		\begin{align*}
		&\sum_{\mathbf v \in \{0,1\}^3}(-1)^{|\mathbf v|}\cdot m_{\mathbf v}^*(y^{\sigma_1} \cup y^{\sigma_2} \cup y^{\sigma_3})=\sum_{\mathbf v \in \{0,1\}^3} (-1)^{|\mathbf v|}\cdot (y_0^{\sigma_1} + v_1 y_1^{\sigma_1} +v_2y_2^{\sigma_1}+v_3y_3^{\sigma_1}) \\ 
		&\phantom{\sum_{a}}\cup (y_0^{\sigma_2} + v_1 y_1^{\sigma_2} +v_2y_2^{\sigma_2}+v_3y_3^{\sigma_2}) \cup (y_0^{\sigma_3} + v_1 y_1^{\sigma_3} +v_2y_2^{\sigma_3}+v_3y_3^{\sigma_3}) \\ 
		&=\sum_{\mathbf v \in \{0,1\}^3} \sum_{{h}:\{\sigma_1,\sigma_2,\sigma_3\} \to \{0,\ldots,4\}}(-1)^{|\mathbf v|}\cdot v_{{h}(\sigma_1)}v_{{h}(\sigma_2)}v_{{h}(\sigma_3)} \cdot y_{{h}(\sigma_1)}^{\sigma_1}\cup y_{{h}(\sigma_2)}^{\sigma_2}\cup y_{{h}(\sigma_3)}^{\sigma_3},
		\end{align*}
		where the last internal sum is over all functions $h$ as in the index, and in the summand we adopt the convention $v_0:=1$. Swapping the sums we get:
		\begin{align*}
		&=\sum_{{h}:\{\sigma_1,\sigma_2,\sigma_3\} \to \{0,\ldots,4\}} \left(\sum_{\mathbf v \in \{0,1\}^3} (-1)^{|\mathbf v|}\cdot v_{{h}(\sigma_1)}v_{{h}(\sigma_2)}v_{{h}(\sigma_3)} \right) \cdot y_{{h}(\sigma_1)}^{\sigma_1}\cup y_{{h}(\sigma_2)}^{\sigma_2}\cup y_{{h}(\sigma_3)}^{\sigma_3},
		\end{align*}
		Note that the internal sum vanishes if the image of ${h}$ is different than $\{1,2,3\}$ since when this is the case, there exists $i \in \{1,2,3\}\s \im {h}$ and then the summands with $v_i=0$ cancel out those with $v_i=1$. On the other hand, when the image of $h$ is $\{1,2,3\}$, the internal sum becomes $-1$. Since the functions ${h}:\{\sigma_1,\sigma_2,\sigma_3\} \to \{0,\ldots,4\}$ with image $\{1,2,3\}$ are permuted on transitively and faithfully by the $S_3$-action on $\{1,2,3\}$, the sum above then becomes:
		\[
		- \sum_{s \in S_3} y_{s(1)}^{\sigma_1} \cup y_{s(2)}^{\sigma_2} \cup y_{s(3)}^{\sigma_3} = -\sum_{s \in \operatorname{Sym}(\sigma_1,\sigma_2,\sigma_3)} (-1)^{\sgn(s)}  \cdot y_1^{s(\sigma_1)} \cup y_2^{s(\sigma_2)} \cup y_3^{s(\sigma_3)}.
		\]
		The image of the triple variation $\sum_{\mathbf v \in \{0,1\}^3}m_{\mathbf v}^*(\beta)$ of $\beta$ in $H^3(\bar G_n^{(4)},\mu_n)$ is then the weighted sum of the expression above over all triples $\sigma_1 < \sigma_2 < \sigma_3$, with weights $a_{\sigma_1,\sigma_2,\sigma_3}(\beta)$, which is just:
		\[
		-\sum_{\sigma_1,\sigma_2,\sigma_3\, \in \Gamma} a_{\sigma_1,\sigma_2,\sigma_3}(\beta) \cdot (y^{\sigma_1} \cup y^{\sigma_2} \cup y^{\sigma_3})_n,
		\]
		finishing the computation of image of the left hand side in $H^3(\bar G_n,\mu_n)$. 
		
		The opposite of the right hand side is 
		\begin{multline*}
			\sum_{\sigma_1,\sigma_2}\cores_{E/K}(a_{\id,\sigma_1,\sigma_2}(\beta)\cdot (y \cup y^{\sigma_1}\cup y^{\sigma_2})_n)= \sum_{\sigma_1,\sigma_2,\rho}a_{\id,\sigma_1,\sigma_2}(\beta)^{\rho}\cdot (y \cup y^{\sigma_1}\cup y^{\sigma_2})_n^{\rho}\\
			=\sum_{\sigma_1,\sigma_2,\rho} a_{\rho,\sigma_1\rho,\sigma_2\rho}(\beta)\cdot (y^{\rho} \cup y^{\sigma_1\rho}\cup y^{\sigma_2\rho})_n
			=\sum_{\sigma_1,\sigma_2,\sigma_3} a_{\sigma_1,\sigma_2,\sigma_3}(\beta) \cdot (y^{\sigma_1} \cup y^{\sigma_2} \cup y^{\sigma_3})_n
		\end{multline*}
		and coincides with the left hand side as wished.
		
		\vskip1mm
		
		{\em Step 2.} By the point above, the sum 
        \[
        \sum_{\mathbf v \in \{0,1\}^3}m_{\mathbf v}^*(\beta) + \sum_{\sigma_1,\sigma_2 \in \Gamma}\cores_{E/K}(a_{\id,\sigma_1,\sigma_2}(\beta)\cdot (y \cup y^{\sigma_1}\cup y^{\sigma_2})_n) 
        \]
        lies in the kernel of $H^3(G_n^{(4)},\mu_n) \to H^3(\bar G_n^{(4)},\mu_n)$.
		By Nakaoka's theorem, this kernel decomposes as 
		\[
		\bigoplus_{i=0,1,2}H^i(\Gamma_K,H^{2-i}(\bar G_n^{(4)},\mu_n)).
		\]
		Write $\bar G_n^{(4)}=\bar G_n^{[0]} \oplus \bar G_n^{[1]} \oplus \bar G_n^{[2]} \oplus \bar G_n^{[3]}$, where each $\bar G_n^{[i]}$ is a copy of $\bar G_n$. Using K\"unneth's formula, the sum above decomposes as
		\[
		\bigoplus_{i+j_1+\ldots+j_4=2} H^i\left(\Gamma_K,(\Z/n\Z)(1) \otimes_{\Z/n\Z}\bigotimes_{k=0,1,2,3} H^{j_k}(\bar G_n^{[k]},\Z/n\Z)\right).
		\]
		Since in the latter tensor product, for any fixed string $(i,j_1,\ldots,j_4)$, at most two factors appear, it is clear that any element that restricts to zero in $H^3(G_n^{[J]},\mu_n)$ for all subgroups $J \subset \{0,1,2,3\}$ of cardinality $2$ must be zero. Such are both $\sum_{\mathbf v \in \{0,1\}^3}m_{\mathbf v}^*(\beta)$ and $ \sum_{\sigma_1,\sigma_2 \in \Gamma}\cores_{E/K}(a_{\id,\sigma_1,\sigma_2}(\beta)\cdot (y \cup y^{\sigma_1}\cup y^{\sigma_2})_n)$, and thus so is their sum.
		
		\vskip1mm
		
		{\em Step 3.} Let $\delta$ be any primitive of 
		\[
		(\star)= \sum_{\mathbf v \in \{0,1\}^3}m_{\mathbf v}^*(\beta)+ \sum_{\sigma_1,\sigma_2 \in \Gamma}\cores_{E/K}(a_{\id,\sigma_1,\sigma_2}(\beta)\cdot (y \cup y^{\sigma_1}\cup y^{\sigma_2})_n).
		\]
		For each subset $J \subset \{0,1,2,3\}$ with $0 \in J$, we define
		\[
		\delta_J = pr_J^*(\delta)|_{G_n^{[J]}},
		\]
		where $pr_J:G_n^{(4)} \to G_n^{[J]}$ denotes the natural projection. Note that 
		\[
		\d \delta = \sum_{\mathbf v \in \{0,1\}^3}m_{\mathbf v}^*(\beta)+ \sum_{\sigma_1,\sigma_2 \in \Gamma}\cores_{E/K}(a_{\id,\sigma_1,\sigma_2}(\beta)\cdot (y \cup y^{\sigma_1}\cup y^{\sigma_2})_n)
		\]
		restricts trivially to $G^{[J]}_n$ for all proper subgroups $J \subset \{0,1,2,3\}$ with $0 \in J$. Thus 
		\[
		\d \left(\sum_{J}(-1)^{\# J }\delta_J \right) = \d \delta_{\{0,1,2,3\}} = \d \delta
		\]
		and so $\sum_{J}(-1)^{\# J }\delta_J $ is also a primitive of $(\star)$. Replacing $\delta$ with other primitive proves the statement, since this sum restricts trivially to $G_I$ for all $I$ as in the statement.  
	\end{proof}
	
	\section{Proof of Main Theorem}\label{Sec12}
	
	We conclude here the proof of Theorem \ref{Thm: fibrationnew}, and thus also of Theorem \ref{Thm: fibration}, by proving the following:
	\begin{theorem}\label{Thm: fibrationreduced}
		Let $f:X \to Q$ be a smooth projective morphism with rationally connected fibers, with $Q$ a quasi-trivial torus. Further assume that $(R_1)-(R_4)$ hold.
		Then 
		\begin{equation}\label{Eq:fibration2}
		X(K_{\Omega})_{dir}^{\Br_{ur} X} = \overline{\bigcup_{q \in Q(K)} X_q(\A_K)_{\bullet}^{\Br X_q}}. \tag{$\star$}
		\end{equation}
	\end{theorem}
	
	Let $B_{tr} \coloneqq \Coker(\Br X_{\eta} \to[r] \Br_{\text{hor}}(X/Q))$. We start with the following:
	
	\begin{corollary}[of Theorems \ref{Thm:grid} and \ref{Thm:3variation}]\label{ThmTripleVarCustom}
		In the setting of Theorem \ref{Thm: fibrationreduced}, for a sufficiently large set of places $S$, the following holds. For every collection of local points $(P_v)_{v \in S} \in X(K_S)$, and every function $g:\{1,\ldots,M-1\}^3 \to B_{tr}^D$, there exists a grid of parameters
		\[
		\mathcal M=\{q_0\} \cdot \{ q_{1,1}, \cdots, q_{1,M} \} \cdot \{ q_{2,1}, \cdots, q_{2,M} \} \cdot \{q_{3,1}, \cdots, q_{3,M} \} \subset Q(K)
		\]
		and adelic points $(P_v(q)) \in X_q(\A_K)^{r(\Br X_{\eta})}, q \in \cM$, each approximating $(P_v)_{v \in S}$ arbitrarily well at finite places and lying on the same connected of $f^{-1}(Q^{\sharp})(K_v)$ as $P_v$ for real $v$, such that
		\[
		\Delta^{(3)}(-,(P_v(q)))_{BM}=g 
		\]
		as elements of $\Fun(\{1,\ldots,M-1\}^3, B_{tr}^D)$.
	\end{corollary}
	\begin{proof}
		We enlarge $S$ and $(P_v)_{v \in S}$ via Harari's formal lemma and assume that the local family $(P_v)_{v \in S}$ is orthogonal 
		
		We assume that $S$ is large enough and that $(P_v)_{v \in S}$ has been modified through Harari's Formal lemma, and that the multi-section  $Q \leftarrow Q^{L,n,(1)} \to[s] X$ has been refined as in Proposition \ref{Proprop}. We also assume that $S$ has been enlarged and $s$ refined as in Theorem \ref{Thm:3variation}.
		
		The map $\Br_{\text{hor}}(X/Q) \to[\bar \partial] H^3(Q \otimes_K \oK,\G_m)^{\Gamma_K}$ has kernel $\im \Br X_{\eta}$ by assumption $(R_3)$, and so induces an injection 
		\[
		\iota: B_{tr} \hookrightarrow H^3(Q \otimes_K \oK,\G_m)^{\Gamma_K}.
		\]
		Following the discussion of subsection \ref{SSec10.1}, the homomorphism
		\[
		B_{tr} \to \Fun_{\Gamma}(\Gamma^3\s \Delta_3,(\Z/n\Z)(-2)), \quad b \mapsto a_{\sigma_1,\sigma_2,\gamma}(\iota(b))
		\]
		is also injective, where $\Delta_3 \subset \Gamma^3$ denotes the multi-diagonal. Identifying the dual of $\Fun_{\Gamma}(\Gamma^3\s \Delta_3,(\Z/n\Z)(-2))$ with $\Fun_{\Gamma}(\Gamma^3\s \Delta_3,(\Z/n\Z)(2))$, the dual of the map above is 
		\[
		\begin{matrix}
		&\Fun_{\Gamma}(\Gamma^3\s \Delta_3,(\Z/n\Z)(2)) &\to &B_{tr}^D, \\
		&c(\sigma_1,\sigma_2,\sigma_3) &\mapsto &\left(b\mapsto \sum a_{\sigma_1,\sigma_2,\sigma_3}(\iota(b))c(\sigma_1,\sigma_2,\sigma_3)\right),
		\end{matrix}
		\]
		which is thus surjective. Let $f:\{1,\ldots,M-1\}^3  \to \Fun_{\Gamma}(\Gamma^3\s \Delta_3,(\Z/n\Z)(2))$ be a lift of $g$ through this map.
		
		By Theorem \ref{Thm:grid} there exists an $(L,n)$-liftable grid
		\[
		\mathcal M=\{q_0\} \cdot \{ q_{1,1}, \cdots, q_{1,M} \} \cdot \{ q_{2,1}, \cdots, q_{2,M} \} \cdot \{q_{3,1}, \cdots, q_{3,M} \} \subset Q(K),
		\]
		whose elements approximate the collection $(f(P_v))_{v \in S}$, and on which the Redéi symbols $[\partial p, (\partial q)^{\sigma_1},(\partial r)^{\sigma_2}] (i,j,k)$ are equal to $g'(i,j,k)(\sigma_1,\sigma_2)$. Then by Theorem \ref{Thm:3variation}, for a suitable choices of adelic points $(P_v(q)) \in X_q(\A_K)$, the triple variation of the function
		\[
		\cM \to B_{tr}^D, \ \ q \mapsto (b,(P_v(q)))_{BM}
		\]
		is $g$. All the adelic points $(P_v(q))_v$ approximate $P_v, v \in S$ by construction.
	\end{proof}
	
	To conclude the argument we use the following lemma (this is Lemma \ref{Lem:Smith0} from the introduction, but we also added an explicit expression for $M_0(A)$), a baby-case of work of Alexander Smith:
	\begin{lemma}\label{LemSmith}
		Let $A$ be a finite abelian group of cardinality $a$. For every natural number $M > k\cdot  a  \log a$, there exists $g:\{1,\ldots,M-1\}^3 \to A$ such that every $f:\{1,\ldots,M\}^k \to A$ with $\Delta^{(k)}(f)=g$ is surjective.
	\end{lemma}
	\begin{proof}
		Consider the operator
		\[
		\Delta^{(k)}_{ns}:\Fun_{\text{non-surj}}(\{1,\ldots,M\}^k,A) \to \Fun(\{1,\ldots,M-1\}^k,A), \ f \mapsto \Delta^{(k)}(f),
		\]
        where the domain denotes the set of non-surjective functions $\{1,\ldots,M\}^k\to A$.
		A function $g$ satisfies the desired condition in the statement if and only if it does not lie in the image of $\Delta^{(k)}_{ns}$. The cardinality of the domain is $\leq a \cdot (a-1)^{M^k}$, while that of the codomain is $a^{(M-1)^k}$. An exercise in calculus shows that $a^{(M-1)^k} > a \cdot (a-1)^{M^k}$ when $M > k\cdot a   \log a$, and so the sought $g$ exists.
	\end{proof}
	
	\begin{proof}[Proof of Theorem \ref{Thm: fibrationreduced}.]
		Let $g$ be as in Smith's lemma. By Corollary \ref{ThmTripleVarCustom}, we may find a grid $\cM$ for which the triple variation of the function
		\[
		\cM \to B_{tr}^D, \ \ q \mapsto (b,(P_v(q)))_{BM}
		\]
		is equal to $g$. So, this function attains a $0$ on some $q\in \cM$ by the lemma. The corresponding adelic point then lies in $X_q(\A_K)^{\Br_{\text{hor}}(X/Q)}=X_q(\A_K)^{\Br X_q}$, and approximates the point $(P_v)_{v \in S}$ as wished.
	\end{proof}

    \begin{proof}[Proof of Theorem \ref{Thm: fibrationnew}]
        Combine Theorem \ref{Thm: fibrationreduced} and Proposition \ref{PropReductions}.
    \end{proof}
	
	\newpage
	\appendix
	
	\section{Combining Chebotarev and Hecke equidistribution}\label{AppA}
	
	Let $E$ be a number field. 
	Recall that the reciprocity map $rec:C_E \to \Gamma_E^{ab}$ identifies $\Gamma_E^{ab}$ with the profinite completion of $C_E$, and via this map to every finite quotient $A$ of $C_E$, we can associate a finite abelian extension $L/E$ such that $\Gal(L/E) \cong A$, and the quotient map $rec_{L/E}:C_E \to \Gal(L/E)$ is the Artin reciprocity map for $L/E$.
	
	For a prime $\cP$ of $\cO_E$, we denote by $[\cP]$ an idele $(i_v)_v \in \mathbb I_E$ such that $i_v \in E_{v}^{*}$ is a local uniformizer when $v$ is the valuation defined by $\cP$, and such that $i_v=1$ for all other $v$. ($[\cP]$ is not uniquely defined, but this ambiguity is irrelevant for our purposes.)
	
	\begin{theorem}[Chebotarev-Hecke density]\label{Lem:ChebHec}
		Let $L/E$ be a finite Galois extension with group $G$, and let $L \cap E^{ab}$ be its maximal abelian subextension. 
		Let $q:C_E \to B$ be a compact quotient. Assume that $B$ is a real Lie group, 
		and that the abelian extension corresponding to the projection $C_E \to \pi_0(B)$ is $L \cap E^{ab}$. Let $\mu$ be the Haar measure of $B$ normalized so that its connected components have mass $1$. Let $c$ be a conjugacy class in $\Gal(L/E)$, and $U \subseteq B$ be a non-empty open that lies entirely in the connected component of $B$ corresponding to the projection of $c$ in  $\pi_0(B)=\Gal(L \cap E^{ab}/E)$. Then the primes $\cP$ of $E$ such that $q([\cP])$ lies in $U$ and their Frobenius class in $\Gal(L/E)$ is $c$ have density $\frac{|c|}{|G|} \cdot \mu(U)$ when ordered by norm.
	\end{theorem}
	
	For a subset $W$ of $\Spec \cO_E$, recall that its {\em density }is the limit (if it exists):
	\[
	\delta_E(W)\coloneqq \lim _{n \rightarrow \infty} \frac{\text { Number of } \mathfrak{p} \in W \text { with } \mathbf{N} p \leqq n}{\text { Number of } \mathfrak{p} \in \Spec \cO_E \text { with } \mathbf{N} \mathfrak{p} \leqq n}.
	\]
	Equivalently, by Landau's prime number theorem, it is equal to the limit:
	\[
	\lim _{n \rightarrow \infty} \frac{\text { Number of } \mathfrak{p} \in W \text { with } \mathbf{N} p \leqq n}{n/\log n}.
	\]
	
	\begin{proof}
		We follow the lines of a simple argument due to Duering as presented in \cite[p.169]{Lang}, but using \cite[p.317, Theorem 6]{Lang} (which combines Dirichlet and Hecke equidistribution) instead of Dirichlet's theorem.
		When $L/E$ is abelian, the statement follows directly from \cite[p.317, Theorem 6]{Lang}.
		
		\vskip1mm
		
		We go to the general case. Let $\Delta_{L/E}\subseteq \cO_E$ be the discriminant of the extension $L/E$. Let $\sigma$ be a representative of $c$, $f$ be its order and $Z=L^{\sigma}$. The extension $L/Z$ is cyclic of group $\Gal(L/Z) \cong \langle \sigma \rangle$, and so induces a character $\chi:C_Z \to \Gal(L/Z) \cong \Z/f\Z$.
		
		Let $W$ be the set of primes $\cP$ as in the statement. Define
		\[
		W_Z:=
		\left\{ \mathfrak{p} \in \Spec \cO_Z :\begin{gathered}
		\mathfrak{p} \text{ has degree $1$ over }E,\ (\mathfrak{p}, \Delta_{L/E})=1, \\
		\Frob_{\mathfrak{p}}(L/Z)=\sigma,\  q([\mathfrak{p} \cap E]) \in U
		\end{gathered}\right\}.
		\]
		The map $\mathfrak p \mapsto \mathcal P = \mathfrak p\cap E$ defines a finite-to-one map $W_Z \to W$. The fibers of this map have constant cardinality $\frac{|G|}{|c|f}$. In fact, for a fixed prime $\mathcal P$ of $E$ of Frobenius class $c$, the number of primes $\beta$ above it in $L$ is $|G|/|G_{\beta}|=|G|/f$, where $G_\beta$ denotes the decomposition group of $\beta$. 
		The Frobenii elements $\Frob_{\beta}(L/K),\,\beta|\cP$ are uniformely distributed in the conjugacy class $c$, so exactly a proportion $1/|c|$ of the primes $\beta$ have $\Frob_{\beta}(L/E)=\sigma$. I.e.\ there are $|G|/(|c|f)$ such $\beta$. The association $\beta \mapsto \mathfrak p=\beta \cap E$  gives a bijection between these $\beta$ and the $\mathfrak p$ in the fiber of $W_Z\to W$ above $\mathcal P$, showing that the fibers have the desired cardinality.
		
		\vskip1mm
		
		Note that $\mathbf{N}(\mathfrak p)=\mathbf{N}(\mathcal P)$ for $\mathfrak p\in W_Z$. Thus $\delta_E(W)=\frac{|c|f}{|G|}\delta_Z(W_Z)$. We now compute $\delta_Z(W_Z)$.
		The quotient map $q:C_E \to B$ induces by pullback a homomorphism $q_Z:C_Z \xrightarrow{N_{Z/E}} C_E \to B$. The cokernel of $C_Z \xrightarrow{N_{Z/E}} C_E$ is isomorphic to $\Gal(Z\cap E^{ab}/E)$ by the norm limitation theorem, and is thus finite.     
		Hence the image of the homomorphism $q'_Z:C_Z \to[(q_Z,(-,L/Z))] B \times_{\pi_0(B)} \Gal(L/Z)$ is a subgroup of finite index of the real Lie group $B \times_{\pi_0(B)} \Gal(L/Z)$ and so contains its connected component of the identity. Moreover, the image of $q'_Z$ surjects onto the second factor $\Gal(L/Z)$, which parametrizes the connected components of the fibered product $B \times_{\pi_0(B)} \Gal(L/Z)$, and we deduce that $q'_Z$ is surjective.
		
		The condition $q([\mathfrak{p} \cap E]) \in U$ in the definition on $W_Z$ is equivalent to $q_Z([\mathfrak{p}]) \in U$ (since $N_{Z/E}([\mathfrak{p}])=[\mathfrak{p}\cap E]$ as $\mathfrak{p}$ has degree $1$ over $E$). Since only the primes that have degree $1$ over $\Q$ contribute to the density, the condition ``$\mathfrak{p} \text{ has degree $1$ over }E$'' may be removed when computing the density of $W_Z$, and this density becomes the density of $\mathfrak{p}$ such that $q_Z([\mathfrak{p}]) \in U$ and $\Frob_{\mathfrak{p}}(L/Z)=\sigma$. This is $\frac{1}{f}\mu(U)$ by the abelian (or even cyclic) case applied to $q'_Z$, which yields the statement.
	\end{proof}
	
	In the cyclic strong approximation Theorem \ref{ThmShafarevich}, we use the following consequence of Chebotarev-Hecke density about approximating $v$-adic numbers, $v \in S$ with prime elements of $\cO_{E,S}$. Recall that the {\em direction} of a non-zero element of a real vector space $V$ is its image in $(V \s\{0\})/\R_{>0}$.
	
	\begin{theorem}[Theorem \ref{Thmneq10}]\label{Thmneq1}
		Let $L/E$ be a Galois field extension, and $S$ a finite set of places of $E$ containing the archimedean ones and those where $L/E$ ramifies. Let, for each place $v \in S$, $x_v \in E_v^{*}$. Assume that the idele $(\tilde x_v)_{v} \in \mathbb I _E, \tilde x_v \coloneqq x_v$ if $v \in S$ and $\tilde x_v \coloneqq 1$ if $v \notin S$, is in the kernel of the reciprocity map
		\[
		rec_{L \cap E^{ab}/E}: \mathbb{I}_E \to C_E\to  \Gal(L \cap E^{ab}/E)
		\]
		associated to the maximal abelian subextension $L \cap E^{ab}/E$ of $L/E$. Let $c \subset \Gal(L/E)$ be a conjugacy class which projects to the identity of $\Gal(L \cap E^{ab}/E)$.
		Then there exist prime elements $x \in \cO_{E,S}$ whose Frobenius class in $\Gal(L/E)$ is $c$, approximate $x_v$ arbitrarily well for all finite $v \in S$, and whose direction in the real vector space $E \otimes \R$ approximates the direction of $(x_v)_{v \in M_E^{\infty}} \in E \otimes \R$ arbitrarily well.
	\end{theorem}
	
	\begin{proof}
		Let $\m$ be the conductor of the abelian extension $L \cap E^{ab}/E$. Recall that this is the minimal modulus of $E$ such that $L \cap E^{ab}$ is contained in the ray class field $E_{\m}$ associated to $\m$. Let $F$ be the compositum $LE_\m$. By Galois theory, we may identify $\Gal(F/E)$ with the group of pairs $(\gamma, \delta) \in \Gal(E_\m/E) \times \Gal(L/E)$ such that the images of $\gamma$ and $\delta$ in $\Gal(L\cap E_{\m}/E)=\Gal(L\cap E^{ab}/E)$ coincide. Recall (see e.g.\ the sentence after \cite[Definition 15.22]{HarariBook}) that the reciprocity map $rec_{E_\m/E}$ gives an identification 
		\[
		\Gal(E_{\m}/E)= E^{*}\backslash \I_E /U_\m.
		\]
		Let $\gamma \in \Gal(E_{\m}/E)$ be the image of $(\tilde x_v)_v$ in this quotient. Note that $c' \coloneqq \{\gamma\} \times c$ is a conjugacy class in $\Gal(F/E)$.
		
		Let $U^0_\m$ be the subgroup of $U_\m \coloneqq \prod_v U_v$ (recall the notation from Section \ref{Sec2}) of ideles with archimedean components equal to $1$. 
		We apply Chebotarev-Hecke density (Theorem \ref{Lem:ChebHec}) to
		\[
		\pi:C_E \to C_E/U^0_\m\R_{>0}=E^{*}\backslash \I_E /U^0_\m \R_{>0}
		\]
		(note that this is a compact real Lie group) and the Galois extension $F/E$ to infer that there are primes $\cP$ of $E$ with Frobenius class $c'$ in $\Gal(F/E)$ such that $[\cP] \in E^{*}\backslash \I_E /U_\m\R_{>0}$ approximates the image of $(\tilde x_v)_v$. The compatibility assumption of Theorem \ref{Lem:ChebHec} here is satisfied because $c'$ and $(\tilde x_v)_v$ map to the same element of
		\[
		\Gal(E_{\m}/E)= E^{*}\backslash \I_E /U_\m = \pi_0(E^{*}\backslash \I_E /U^0_\m \R_{>0}).
		\]
		by construction of $c'$. (See the sentence after \cite[Definition 15.22]{HarariBook} for the first identity.)
		
		By definition of quotient topology there exists then $x\cdot \mathfrak{a}\cdot t \in E^{*}U^0_{\m}\R_{>0}$ such that $[\cP]\cdot x^{-1} \cdot \mathfrak{a}^{-1} \cdot t^{-1} \in \mathbb I_E$ approximates $(\tilde x_v)_v$. For such $[\cP], x,  \mathfrak{a},  t$, we have $x \cdot \cO_{E,S}=\cP$, $x\cdot x_v^{-1}$ is congruent to $1 \bmod \m$ for every finite $v \in S$, and $tx$ approximates $(x_v)_{v \in M_E^{\infty}}\in \prod_{v \in M_E^{\infty}}E_v^{*}$. The last condition shows that $x$ approximates the direction of $(x_v)_{v \in M_E^{\infty}}$ as wished. Refining $\m$, $x$ approximates $x_v$ at $S$ arbitrarily well, completing the proof.
	\end{proof}
	
	\section{Relative cohomology}\label{AppB}
	
	Relative (étale) cohomology was defined by Friedlander in \cite[Ch.~14]{Friedlander} in the general context of simplicial schemes. In this appendix we give an alternative definition that avoids the use of simplicial schemes. At the end of this appendix we discuss how to represent relative étale cohomology classes via \v{C}ech cocycles in an analog way as one does for usual (non-relative) cohomology.
	
	\begin{definition}
		A {\em relative abelian (pre)sheaf} of a morphism of schemes $f:X \to Y$ is a triple $(F,G,\alpha)$, with $F \in \operatorname{(P)Sh}_Y,\, G \in \operatorname{(P)Sh}_X,\, \alpha:F \to f_*G$.
	\end{definition}
	
	\begin{notation*}
		We denote by $\operatorname{(P)Sh}_{X/Y}$ the category of relative abelian (pre)sheaves of $f$.
		For $(F,G,\alpha) \in \Sh_{X/Y}$, we denote by $\alpha^{\dagger}:f^{-1}F \to G$ the left adjunct of $\alpha$.
	\end{notation*}
	
	\vskip1mm
	
	The category $\operatorname{(P)Sh}_{X/Y}$ is the comma category $\left(\operatorname{Id}_{\operatorname{(P)Sh}(Y)} \downarrow f_*\right)$. The comma category $(U \downarrow V)$ of a diagram of functors $\cA \to[U] \cC \xleftarrow{V} \cB$ is the category whose objets are triples $(A \in \mathscr{A}, B \in \mathscr{B}, \theta: U A \rightarrow V B)$ \cite[II.6]{MacLane}. Comma categories are abelian if $U$ is right exact and $V$ is left exact \cite[Prop.~5.1]{BH}, so in particular $\Sh_{X/Y}$ is abelian. (One may also easily prove this abelianity directly.)
	
	\vskip1mm
	
	The following lemma proves, in particular, that $\Sh_{X/Y}$ and $\Psh_{X/Y}$ have enough injectives:
	
	\begin{lemma}\label{LemEnoughInjectives}
		Let $V:\cB \to \cA$ be a left exact functor between abelian categories, and $\mathfrak C:=(\id \downarrow V)$ be the comma category of $\cA \to[\id] \cA \xleftarrow{V} \cB$. The injectives objects of $\mathfrak C$ are those of the form $(I\oplus V(J),J,I \oplus V(J) \to V(J))$, with $I$ and $J$ injective in $\cA, \cB$, respectively. Moreover, if $\cA$ and $\cB$ have enough injectives, then $\mathfrak C$ does too.
	\end{lemma}
	\begin{proof}
		See \cite[p.16]{Abelian}. \qedhere
	\end{proof}
	
	We give the following definition, borrowing Friedlander's notation.
	
	\begin{definition}[Relative cohomology]
		Let $H^k_X(Y,-),\, k\geq0$ be the $k$-th derived functor of
		\[
		\Gamma_{Y,X}:\Sh_{X/Y} \to \cA b, \ \ (F,G,\alpha) \mapsto \Ker(\alpha: F(Y) \to G(X)).
		\]
	\end{definition}
	
	Below, we shall argue that, for $M^{\bullet}=(F^{\bullet},G^{\bullet},\alpha) \in \cD^+(\Sh_{X/Y})$, the total derived functor $R\Gamma_{Y,X}(M^{\bullet})$ is, in a sense, naturally isomorphic to the cone of the natural transformation $R\Gamma_Y(F^{\bullet}) \to[\alpha] R\Gamma_X(G^{\bullet})$. A slight problem arises: to give precise statements we need to talk about cones in derived categories, which notoriously cannot be defined functorially in general \cite[Prop.~1.2.13]{Verdier}. In the following paragraph, we present a way to circumvent this lack of functoriality by defining the cone functor on the derived category of the morphism category instead of on the morphism category of the derived category.\footnote{Other ways would be to pass to dg-enhancements of the categories in question or use the theory of derivators, but for the sake of self-containement of the paper I preferred to avoid both.}
	
	\vskip2mm
	
	\noindent {\bf Derived category of the morphism category.~}Let $\cA$ be an abelian category, and $\Mor(\cA)$ be its category of morphisms. This is still an abelian category \cite{Tohoku}.
	
	Since $\Ch(\Mor(\cA))=\Mor(\Ch(\cA))$, we shall represent elements of $\Ch(\Mor(\cA))$ as morphisms $A^{\bullet} \to B^{\bullet}$ of chain complexes. (The same identity does not pass to the derived category! Namely, $\cD(\Mor(\cA))$ and $\Mor(\cD(\cA))$ are different categories in general: in view of \cite[Prop.~1.2.13]{Verdier}, this follows for instance from the proposition below.)
	
	\begin{proposition}\label{PropCone}
		There is a functor
		\[
		\Cone(-): \cD(\Mor(\cA)) \to \cD(\cA), \ 
		\]
		and natural exact triangles 
		\[
		A^{\bullet} \to B^{\bullet} \to \Cone(A^{\bullet}\to B^{\bullet}) \to[+1] \ \ \text{ in }\cD(\cA)
		\]
		for $(A^{\bullet}\to B^{\bullet}) \in \cD(\Mor(\cA))$.
	\end{proposition}
	\begin{proof}
		The usual mapping cone functor $\Ch(\Mor(\cA)) \to \Ch(\cA),\,(A^{\bullet} \to B^{\bullet})\mapsto \linebreak\Cone(A^{\bullet}\to B^{\bullet})$ defined on the category of chain complexes sends a quasi-isomorphism $(A_1^{\bullet} \to B_1^{\bullet}) \to[\text{q.i.}] (A_2^{\bullet} \to B_2^{\bullet})$ to a quasi-isomorphism of cones $\Cone(A_1^{\bullet} \to B_1^{\bullet}) \to[\text{q.i.}] \Cone(A_2^{\bullet} \to B_2^{\bullet})$, and so passes to the derived category.
	\end{proof}
	
	We observe that, if $\cA$ has enough injectives, then $\Mor(\cA)$ does as well, as follows from Lemma \ref{LemEnoughInjectives}, since the morphism category is just the comma category $(\id \downarrow \id)$. So the statement of the following lemma is well-posed.
	
	\begin{lemma}\label{LemConeIsDerivedKer}
		The functor $\Ker:\Mor(\cA) \to \cA$ sending a morphism to its kernel is left exact. Moreover, if $\cA$ has enough injectives, then the total right derived functor $R\Ker: \cD^+(\Mor(\cA)) \to \cD^+(\cA)$ is the shifted cone $\Cone(-)[-1]$.
	\end{lemma}
	\begin{proof}
		The first part is clear from the snake lemma. 
		For the second part, let $\cI \subset \Mor(\cA)$ be the full subcategory of injective objects of $\Mor(\cA)$. By Lemma \ref{LemEnoughInjectives}, any morphism $A \to B$ that is injective as an object of $\Mor(\cA)$ is an epimorphism.  In particular, $A^n \to B^n$ is surjective for all $n$ and all $(A^{\bullet} \to B^{\bullet}) \in \Ch(\cI)$. 
		
		We denote by $\cD(\cI) \subset \cD(\Mor(\cA))$ the full subcategory of complexes with entries in $\cI$. We have $R\Ker((A^{\bullet} \to B^{\bullet}))=\Ker(A^{\bullet} \to B^{\bullet})$ for $(A^{\bullet} \to B^{\bullet}) \in \cD^+(\cI)$. By the surjectivity of the maps $A^n \to B^n$, the natural chain map
		\[
		\Ker(A^{\bullet} \to B^{\bullet}) \to \Cone(A^{\bullet} \to B^{\bullet})[-1],
		\]
		defined degree-wise as the composition $\Ker(A^n \to B^n) \to A^n \hookrightarrow A^n\oplus B^{n-1}$, is a quasi-isomorphism. So $R\Ker \cong \Cone(-)[-1]$ on the full subcategory $\cD^+(\cI) \subset \cD^+(\Mor(\cA))$. Since $\Mor(\cA)$ has enough injectives, the inclusion $\cD^+(\cI)\subset \cD^+(\Mor(\cA))$ is essentially surjective, and so the two functors are naturally isomorphic on the whole $\cD^+(\Mor(\cA))$.            
	\end{proof}
	
	\vskip2mm
	
	\noindent {\bf A long exact sequence.~}We express the functor $\Gamma_{Y,X}$ as the composition
	\begin{align*}
	&\Gamma_{Y,X}: \Sh_{X/Y} \to[\Gamma_{X/Y}] \Mor(\cA b) \to[\Ker] \cA b, \quad \text{where}\\
	&\Gamma_{X/Y}(F,G,\alpha):=(F(Y) \to[\alpha(Y)] G(X)).
	\end{align*}
	The functor $\Gamma_{X/Y}$ is left exact and has an exact left adjoint $\Mor(\cA b) \to \Sh_{X/Y}$ sending a morphism $A \to[\alpha] B$ to the constant relative sheaf $(A_Y,B_X,A_Y \to f_*B_X)$, where $A_Y$ and $B_X$ denote the constant sheaves associated to $A$ and $B$ on $X$ and $Y$, and the morphism $A_Y \to f_*B_X$ is adjoint to the morphism $f^{-1}A_Y=A_X \to[\alpha] B_X$. In particular, $\Gamma_{X/Y}$  sends injectives to injectives, so Grothendieck's theorem on the composition of derived functors gives $R\Gamma_{Y,X} = \Cone[-1] \circ R\Gamma_{X/Y}$.
	
	\vskip1mm
	
	For a diagram $\cA \to[\id] \cA \xleftarrow{V} \cB$ with $V$ left exact, we denote by $pr_1:(\id\downarrow V)\to \cA,\,pr_2:(\id\downarrow V)\to \cB$ the two ``projection'' functors $(A,B,\theta)\mapsto A$ and $(A,B,\theta)\mapsto B$. These are exact functors, so $R(pr_i)=pr_i$ for $i=1,2$. Moreover, if $V$ sends injectives to injectives then they also both send injectives to injectives by Lemma \ref{LemEnoughInjectives} (if $V$ does not send injectives to injectives, then only the second projection does). 
	
	\vskip1mm
		
	By these general observations on projections in comma categories, we get $pr_1\circ R\Gamma_{X/Y}\cong R(pr_1 \circ \Gamma_{X/Y})\cong R(\Gamma_{Y}\circ pr_1)\cong R\Gamma_{Y}\circ pr_1$. So $pr_1\circ R\Gamma_{X/Y}\cong R\Gamma_{Y}\circ pr_1$ and, analogously, $pr_2\circ R\Gamma_{X/Y}\cong R\Gamma_{X}\circ pr_2$.
	
	For $M^{\bullet}=(F^{\bullet},G^{\bullet},\alpha^{\bullet})\in\cD^+(\Sh_{X/Y})$, we may consider $R\Gamma_{X/Y}(M^{\bullet})\in \cD^+(\Mor(\cA b))$. The discussion above gives natural derived isomorphisms 
	\begin{equation}\label{NatIsos}
	pr_1(R\Gamma_{X/Y}(M^{\bullet})) \cong R \Gamma_Y(F^{\bullet}),\,pr_2(R\Gamma_{X/Y}(M^{\bullet})) \cong R \Gamma_X(G^{\bullet}),
	\end{equation}
	and the image of $R\Gamma_{X/Y}(M^{\bullet}) \in \cD^+(\Mor(\cA b))$ in $\Mor(\cD^+(\cA b))$ thus defines a map $R\Gamma_Y(F^{\bullet}) \to R \Gamma_X(G^{\bullet})$. We already have a natural map between the two: the composition $R\Gamma_Y(F^{\bullet}) \to[f^*] R \Gamma_X(f^{-1}F^{\bullet}) \to[\alpha^{\dagger}]R \Gamma_X(G^{\bullet}).$
	As one would expect, these two coincide:
	\begin{lemma*}
		The image under $\cD^+(\Mor(\cA b)) \to \Mor(\cD^+(\cA b))$ of $R\Gamma_{X/Y}(M^{\bullet})$ corresponds to the composition $R\Gamma_Y(F^{\bullet}) \to[f^*] R \Gamma_X(f^{-1}F^{\bullet}) \to[\alpha^{\dagger}]R \Gamma_X(G^{\bullet})$ under the natural isomorphisms \eqref{NatIsos}.
	\end{lemma*}
	\begin{proof}
		The statement is equivalent to the commutativity of the diagram
		\[
		\begin{tikzcd}
		pr_1 \circ R\Gamma_{X/Y} \arrow[rr] \arrow[d, "\cong"] &                                                               & pr_2 \circ R\Gamma_{X/Y} \arrow[d, "\cong"] \\
		R\Gamma_Y \circ pr_1 \arrow[r, "f^*\circ pr_1"]                          & R\Gamma_X \circ f^{-1} \circ pr_1 \arrow[r] & R\Gamma_X \circ pr_2                       
		\end{tikzcd}
		\]
		of natural transformations $\cD^+(\Sh_{X/Y}) \to \cD^+(\cA b)$ (the upper row specializes to $[R\alpha^{\bullet}]$ and the lower row to $\alpha^{\dagger}\circ f^*$). But this commutes as it is the ``derived'' of the commutative diagram
		\[
		\begin{tikzcd}
		pr_1 \circ \Gamma_{X/Y} \arrow[rr] \arrow[d, "\cong"] &                                                              & pr_2 \circ \Gamma_{X/Y} \arrow[d, "\cong"] \\
		\Gamma_Y \circ pr_1 \arrow[r, "f^*\circ pr_1"]                          & \Gamma_X \circ f^{-1} \circ pr_1 \arrow[r] & \Gamma_X \circ pr_2      .                 
		\end{tikzcd}
		\]
	\end{proof}
	
	The lemma above gives a natural triangle
	\[
	R\Gamma_Y(F^{\bullet}) \to[\alpha^{\dagger} \circ f^*] R\Gamma_X(F^{\bullet}) \to \Cone(R\Gamma_{X/Y}(M^{\bullet})) \to[+1],
	\]
	and, after shifting:
	\[
	R\Gamma_{Y,X}(M^{\bullet})\to R\Gamma_Y(F^{\bullet}) \to[\alpha^{\dagger} \circ f^*] R\Gamma_X(F^{\bullet}) \to[+1],
	\]
	Specializing at $M=(F,G,\alpha) \in \Sh_{X/Y}$ and taking cohomology, we get a natural long exact sequence:
	\begin{align}
	\notag 0 \to & H_X^0(Y,M) \to H^0(Y,F) \to[\alpha^{\dagger} \circ f^*] H^0(X,G)   \\
	\label{EqLES} \to & H_X^1(Y,M) \to H^1(Y,F) \to[\alpha^{\dagger} \circ f^*] H^1(X,G) \to \cdots .
	\end{align}
	
	\vskip2mm
	
	\begin{proposition}
		When $F=f_*G$, we have natural isomorphisms
		\begin{equation}\label{Eq:RelCoh2}
		H_X^k(Y,M) \cong H^{k-1}(Y,\tau_{\geq 1}Rf_*G)
		\end{equation}
		for all $k \geq 0$ (with the convention that $H^{-1}(Y,\tau_{\geq 1}Rf_*G)=0$).
	\end{proposition}
	
	\begin{proof}
		We first express the functor $\Gamma_{Y,X}$ as the composition
		\begin{align*}
		&\Gamma_{Y,X}: \Sh_{X/Y} \to[{p}] \Mor(\Sh_{Y})\to[\Ker] \Sh_{Y} \to[\Gamma_Y] \cA b, 
		\end{align*}
		where ${p}((F,G,\alpha:F \to f_*G)):= \alpha$. All functors appearing are left-exact, so $R\Gamma_{Y,X} \cong R\Gamma_Y \circ \Cone \circ R{p}$, and we get a natural isomorphism
		\[
		R\Gamma_{Y,X}(M^{\bullet})\cong R\Gamma_Y(\Cone(R{p}(M^{\bullet})))
		\]
		for $M^{\bullet} \in \cD^+(\Sh_{X/Y})$.
		
		Analogously as in our discussion on $R\Gamma_{X/Y}$, we have natural derived isomorphisms $pr_1(R{p}(M^{\bullet})) \cong F^{\bullet}$ and $pr_2(R{p}(M^{\bullet})) \cong Rf_*(G^{\bullet})$. In this case, the image of $[R{p}(M^{\bullet})] \in \cD^+(\Mor(\Sh_{Y}))$ in $\Mor(\cD^+(\Sh_{Y}))$ is the composition $F^{\bullet} \to[\alpha^{\bullet}] f_*G^{\bullet} \to Rf_*G^{\bullet}$. 
		So we get a natural triangle
		$
		F^{\bullet} \to Rf_*G^{\bullet} \to \Cone(R{p}(M^{\bullet})) \to[+1]
		$
		in $\cD^+(\Sh_Y)$. Specializing at $M=(f_*G,G,id)$ (with everything supported in degree $0$), this triangle becomes
		\[
		f_*G \to Rf_*G \to \Cone(R{p}(M)) \to[+1].
		\]
		Comparing it with the triangle $f_*G \to Rf_*G\to \tau_{\geq 1}Rf_*G \to[+1]$, we get an isomorphism 
		\[
		\Cone(R{p}(M)) \cong \tau_{\geq 1}Rf_*G.
		\]
		Because $\tau_{\geq 1}Rf_*G$ is supported in degree $\geq 1$ and $f_*G$ is supported in degree $0$, we have $\Hom_{\cD}(f_*G,\tau_{\geq 1}Rf_*G[-1])=0$, so this isomorphism is uniquely defined \cite[IV.1.5]{GM} and thus natural.
		Thus we get natural isomorphisms
		\[
		R\Gamma_{Y,X}(M) \cong R\Gamma_Y(\tau_{\geq 1}Rf_*G)[1],
		\]
		and taking cohomology we get \eqref{Eq:RelCoh2}.
	\end{proof}
	
	\vskip1mm
	
	\noindent {\bf Relative coverings and relative \v{C}ech cochains.} 
	
	\begin{definition}
		A {\em relative covering} of a morphism of schemes $f:X\to Y$ is a triple $(\cV,\cU,\psi)$, where $\cU$ (resp.\ $\cV$) is an étale covering of $Y$ (resp.\ $X$), and $\psi$ is a refinement $\cV \to \cU \times_YX$, where we view $\cU \times_YX$ as a covering of $Y$. 
	\end{definition}
	For a morphism $f$, we define the {\em relative \v{C}ech complex} of a relative sheaf $(F,G,\alpha)\in \Sh_{X/Y}$ with respect to a relative covering $(\cV,\cU,\psi)$ as
	\begin{equation}\label{EqRelCech}
	\check{C}^{\bullet}(\cV,\cU;F,G) \coloneqq \Cone(\check{C}^{\bullet}(\cU/Y,F) \to \check{C}^{\bullet}(\cV/X,G)),
	\end{equation}
	where $\Cone$ denotes the {\em mapping cone} in the category of chain complexes, and the map $\check{C}^{\bullet}(\cU,F) \to \check{C}^{\bullet}(\cV,G)$ is $\psi^*\circ\alpha^{\dagger}\circ f^*$ (this is just the natural pullback). We denote the cohomology of this complex by $\check{H}^k(\cV,\cU;F,G), k\geq 0$. Taking hypercohomology of the exact triangle 
	\[
	\check{C}^{\bullet}(\cU/Y,F) \to \check{C}^{\bullet}(\cV/X,G) \to \check{C}^{\bullet}(\cV,\cU;F,G)  \to[+1],
	\]
	we obtain a natural long exact sequence
	\begin{align}\label{EqLES2}
	0 &\to \check{H}^0(\cV,\cU;F,G) \to \check{H}^0(\cU/Y,F) \to[\alpha^{\dagger}\circ \psi^*\circ f^*] \check{H}^0(\cV/X,G) \\
	\notag &\to \check{H}^1(\cV,\cU;F,G) \to \check{H}^1(\cU/Y,F) \to[\alpha^{\dagger}\circ \psi^*\circ f^*] \check{H}^1(\cV/X,G) \to \cdots.
	\end{align}
	
	A {\em refinement} map between two relative coverings $(\cV_1,\cU_1,\psi_1)$ and $(\cV_2,\cU_2,\psi_2)$ of a morphism $f:X \to Y$ is a pair $\theta=(\theta_X,\theta_Y)$ of refinements $\theta_X: \cV_1 \to \cV_2,\, \theta_Y: \cU_1 \to \cU_2$ commuting with $\psi_1$ and $\psi_2$ as in the following diagram
	\[
	\begin{tikzcd}[column sep=large]
	\cV_1 \arrow[d, "\psi_1"'] \arrow[r, "\theta_X"] & \cV_2 \arrow[d, "\psi_2"] \\
	\cU_1 \times_YX \arrow[r, "\theta_Y \times_YX"]  & \cU_2\times_YX   .       
	\end{tikzcd}
	\]
	
	Clearly, a refinement $\theta$ induces by pullback a refinement map $\check{C}^{\bullet}(\cV_2,\cU_2;F,G) \to[\theta^*] \check{C}^{\bullet}(\cV_1,\cU_1;F,G)$ on \v{C}ech complexes, which in turn induces a refinement map on cohomology.
	The following lemma shows that, just as for the non-relative case, the pullback map induced on relative cohomology among two relative coverings do not depend on the choice of refinement. 
	
	\begin{lemma}\label{LemIndepofRefin}
		Let $\theta,\theta':(\cV_1,\cU_1,\psi_1) \to (\cV_2,\cU_2,\psi_2)$ be two refinements between relative coverings of a morphism $f$. The pullbacks $\theta^*,(\theta')^*:\check{H}^{\bullet}(\cV,\cU;F,G) \to \check{H}^{\bullet}(\cV',\cU';F,G)$ coincide.
	\end{lemma}
	\begin{proof}
		Write $\theta=(\theta_X,\theta_Y),\, \theta'=(\theta'_X,\theta'_Y).$ From the well-known analogue statement in the usual ``non-relative'' setting (see \eqref{TheHomotopy}), we have homotopies $K_{\theta_X,\theta'_X}:{\check{C}^{\bullet}(\cU/Y,F)} \to {\check{C}^{\bullet-1}(\cU'/Y,F)}$ and $K_{\theta_Y,\theta'_Y}:{\check{C}^{\bullet}(\cV/X,G)} \to {\check{C}^{\bullet-1}(\cV'/X,G)}$ satisfying
		\[
		\d K_{\theta_X,\theta'_X} + K_{\theta_X,\theta'_X} \d =\theta_{2,Y}^*-\theta_{1,Y}^*,\quad \d K_{\theta_Y,\theta'_Y} + K_{\theta_Y,\theta'_Y} \d = \theta_{2,X}^*-\theta_{1,X}^*.
		\]
		By the naturality of the homotopies $K_{\star,\star}$, the diagram
		\[
		\begin{tikzcd}
		{\check{C}^{\bullet}(\cU/Y,F)} \arrow[d, "{K_{\theta_X,\theta'_X}}"'] \arrow[r, "\psi^*\circ f^*"] & {\check{C}^{\bullet}(\cV/X,G)} \arrow[d, "{K_{\theta_Y,\theta'_Y}}"] \\
		{\check{C}^{\bullet-1}(\cU',F)} \arrow[r, "\psi^*\circ f^*"']                                      & {\check{C}^{\bullet-1}(\cV'/X,G)}              .                     
		\end{tikzcd}
		\]
		commutes. Taking cones of the upper and lower rows, the pair $(K_{\theta_X,\theta'_X},K_{\theta_Y,\theta'_Y})$ defines a chain map $K:C^{\bullet}(\cV,\cU;F,G) \to C^{\bullet-1}(\cV,\cU;F,G)$, and by construction $\d K+K \d=\theta^*-(\theta')^*$.
	\end{proof}
	
	Just as in the non-relative case, we have maps from relative \v{C}ech cohomology to relative étale cohomology:
	
	\begin{proposition}\label{PropNatTransf}
		For $M=(F,G,\alpha) \in \Sh_{X/Y}$, there exist natural maps
		\[
		\check{H}^{\bullet}(\cV,\cU;F,G) \to H^{\bullet}_X(Y;F,G).
		\]
		For fixed $X,Y,\cU,\cV$, these form a natural transformation of $\delta$-functors. Moreover, these morphisms are compatible with the long exact sequences \eqref{EqLES} and \eqref{EqLES2}.
	\end{proposition}
	
	Before proving the proposition, let us recall how these maps are constructed in the non-relative setting. For an étale covering $\cU=\{U_i \to Y\}_{i\in I}$ of a scheme $Y$, we may factor the functor $\Gamma_Y:\Sh_Y \to \cA b$ as
	\[
	\Sh_Y \xhookrightarrow{\iota} \Psh_Y \to[\check{H}^0(\cU,-)] \cA b,
	\]
	with $\check{H}^0(\cU,P):=\Ker\left(\prod_{i \in I}P(U_i) \rightrightarrows \prod_{i,j} P(U_i \times_Y U_j)\right)$. Since $\iota$ sends injectives to injectives, we have $R\Gamma_Y = R\check{H}^0(\cU,-) \circ R\iota$. The derived functor $R\check{H}^0(\cU,-)$ is the functor $\check{C}^{\bullet}(\cU,-)$ associating to a complex its total \v{C}ech complex \cite[Ch.\ III.§2]{LECcompleto},
	and so $R\Gamma_Y = \check{C}^{\bullet}(\cU,-)\circ R \iota$. The natural transformation $\iota \to R\iota$ then defines a natural transformation $\check{C}^{\bullet}(\cU,-)\circ \iota \to R\Gamma_Y$ of functors $\Ch^+(\Sh_Y) \to \cD^+(\cA b)$. Equivalently, it gives natural derived maps
	\[
	\check{C}^{\bullet}(\cU,F^{\bullet}) \to R\Gamma_Y(F^{\bullet})
	\]
	for $(F^{\bullet}) \in \Ch^+(\Sh_Y)$. Taking cohomology and specializing to the case where $F^{\bullet}=F$ is supported in degree $0$, we get natural maps:
	\[
	\check{H}^{\bullet}(\cU,F) \to H^{\bullet}(Y,F).
	\]
	
	\begin{proof}[Proof of Proposition \ref{PropNatTransf}]
		Just as the derived functor $R\check{H}^0(\cU,-)$ is the functor $\check{C}^{\bullet}(\cU,-)$ associating to a complex its total \v{C}ech complex, the content of \cite[Ch.\ III.§2]{LECcompleto} also shows that the derived functor $R\check{H}^0(\cU,\cV;-)$ is the functor associating to a complex of relative sheaves its total relative \v{C}ech complex defined via \eqref{EqRelCech}.
		
		The same argument as for the non-relative case works when applied to the composition
		\[
		\Sh_{X/Y} \xhookrightarrow{\iota} \Psh_{X/Y} \to[\check{H}^0(\cU,\cV;-)] \Mor(\cA b),
		\]
		with $\check{H}^0(\cU,\cV;F,G,\alpha):=(\check{H}^0(\cU,F) \to[\check{\alpha}] \check{H^0}(\cV,G)),\,\check{\alpha}:=\alpha^{\dagger}\circ \psi^*\circ f^*$, using that the derived functor $R\check{H}^0(\cU,\cV;-)$ is the functor associating to a complex of relative sheaves its total relative \v{C}ech complex defined via \eqref{EqRelCech}.
	\end{proof}

	\section{Nakaoka's theorem}\label{AppC}
	
	The following theorem is essentially due to Nakaoka \cite{Nakaoka} (see also Leary's subsequent work \cite[Theorem 2.1]{Leary}), but we present it with a slightly modified statement, with profinite groups instead of standard ones, and with more general coefficients. We give a simple full proof below, freely adapted from the original argument of Nakaoka.
	
	\begin{theorem}[Nakaoka]\label{ThmNakaoka}
		Let $G$ be a profinite group, and $R$ be a constant commutative ring. Let $Y$ be a finite set on which $G$ acts continuously by permutations, and $H$ be a profinite group such that $H^{\bullet}(H,R)$ is $R$-projective. Then for any $R[G]$-module $M$ that is flat over $R$, the Hochschild--Serre spectral sequence
		\[
		H^p(G,H^q(H^Y,M)) \Rightarrow H^{p+q}(H^Y \rtimes G,M)
		\]
		degenerates at the second page and the additive extensions in the reconstruction of the cohomology from the $E_{\infty}$-page are all trivial. Moreover, for fixed $H$ and $R$, the additive extensions can be reconstructed in a way that is natural in $Y,G$ and $M$.
	\end{theorem}
	
	\begin{lemma}\label{LemSD}
		Let $H \rtimes G$ be a semi-direct product of profinite groups, and $M$ be a $G$-module. Let $M \to M^{\bullet}$ be a resolution in acyclic $H$-modules (where we are endowing $M$ with the trivial $H$-action) that is endowed with a $G$-action. Assume further that, for all $p \geq 0$, the surjection 
		\[
		Z^p((M^{\bullet})^H) \to H^p((M^{\bullet})^H) = H^p(H,M)
		\]
		splits as a map of $G$-modules. Then the Hochschild--Serre spectral sequence 
		\[
		H^p(G,H^q(H,M)) \Rightarrow H^{p+q}(H \rtimes G,M)
		\]
		degenerates at the second page and the additive extensions in the reconstruction of the cohomology from the $E_{\infty}$-page are all trivial. Moreover, the additive extensions can be reconstructed in a way that is natural in all the data ($H,G,M,$ the resolution $M \to M^{\bullet}$ and the splittings of $Z^p((M^{\bullet})^H) \to H^p((M^{\bullet})^H)$).
	\end{lemma}
	
	\begin{proof}
		The Hochschild-Serre spectral sequence for $H \rtimes G$ is the Grothendieck spectral sequence associated to the composition
		\[
		(H \rtimes G)-\operatorname{Mod} \to[F_H] G-\operatorname{Mod} \to[F_G] \cA b,
		\]
		where the first functor is $F_H:M \mapsto M^H$, and the second is $F_G:M \mapsto M^G$. 
		
		The resolution $M \to M^{\bullet}$ is a right resolution of $M$ in $(H \rtimes G)$-modules that are $H$-acyclic, so it computes $RF_H$ and we have $RF_H(M)=((M^{\bullet})^H)$. The splittings $s_p: H^p(H,M) = H^p((M^{\bullet})^H) \to Z^p((M^{\bullet})^H),\,p \geq 0$ given by our hypothesis provide a chain map of complexes of $G$-modules
		\[
		\begin{tikzcd}
		\cdots \arrow[r] & H^{p-1}((M^{\bullet})^H) \arrow[r, "0"] \arrow[d, "s_{p-1}"] & H^{p}((M^{\bullet})^H) \arrow[d, "s_p"] \arrow[r, "0"] & H^{p+1}((M^{\bullet})^H) \arrow[d, "s_{p+1}"] \arrow[r] & \cdots \\
		\cdots \arrow[r] & C^{p-1}((M^{\bullet})^H) \arrow[r, "\partial"]               & C^{p}((M^{\bullet})^H) \arrow[r, "\partial"]           & C^{p+1}((M^{\bullet})^H) \arrow[r]                      & \cdots
		\end{tikzcd}
		\]
		that is a quasi-isomorphism by construction, and so in fact $RF_H(M)\cong (H^{\bullet}(H,M))$, where the complex $H^{\bullet}(H,M)$ has trivial differentials. The Grothendieck spectral sequence under consideration coincides with the hypercohomology spectral sequence  ${H}^p(G, H^q(RF_H(M))) \Rightarrow {H}^{p+q}(G,RF_H(M))$, and the hypercohomology spectral sequence of a complex with trivial differentials degenerates at the second page and gives trivial extensions at the $E_{\infty}$-page.
	\end{proof}
	
	\begin{lemma}\label{TheLemma}
		Let $R$ be a commutative ring, $G_1,G_2$ be profinite groups, and, for $i=1,2$, $M_i$ be $R[G_i]$-modules. If $M_2$ and $(M_1)^{G_1}$ are $R$-flat, then $(M_1 \otimes_R M_2)^{G_1 \times G_2}=M_1^{G_1} \otimes_R M_2^{G_2}$. 
		If, additionally, $M_1$ and $M_2$ are $G_1$- and $G_2$-acyclic, respectively, then $M_1 \otimes_R M_2$ is $G_1 \times G_2$-acylic.
	\end{lemma}
	\begin{proof}
		In general, if $G$ is a profinite group, $M$ an $R[G]$-module, and $N$ is an $R$-flat constant module, then $H^n(G,M \otimes_RN)=H^n(G,M) \otimes_RN$ for all $n \geq 0$ (as follows by tensoring  by $N$ the complex $C^{\bullet}(G,M)$ computing the cohomology of $M$). The case $n=0$ gives $(M\otimes_RN)^G=M^G \otimes_RN$. So $(M_1 \otimes_R M_2)^{G_1}=M_1^{G_1} \otimes_R M_2$ and $(M_1^{G_1} \otimes_R M_2)^{G_2}=M_1^{G_1} \otimes_R M_2^{G_2}$, which combine to the first statement. 
		
		For the second part, consider the Hochschild-Serre spectral sequence $H^p(G_2,\linebreak H^q(G_1, M_1 \otimes_R M_2)) \Rightarrow H^{p+q}(G_1 \times G_2, M_1 \otimes_R M_2)$. If $q>0$, we have $H^q(G_1,M_1 \otimes_R M_2)=H^q(G_1,M_1) \otimes_R M_2=0$. If $q=0$ and $p>0$, we have $H^p(G_2,H^q(G_1,M_1 \otimes_R M_2))=H^p(G_2,(M_1)^{G_1}\otimes_RM_2)=(M_1)^{G_1}\otimes_R H^p(G_2,M_2)=0$. So $H^p(G_2,H^q(G_1,\linebreak M_1 \otimes_R M_2))=0$ unless $p=q=0$, and the spectral sequence gives the statement.
	\end{proof}
	
	\begin{proof}[Proof of Theorem \ref{ThmNakaoka}]
		The proof of Theorem \ref{ThmNakaoka} is an application of Lemma \ref{LemSD}.
		
		Let $R \to C^{\bullet}$ be any resolution into $H$-acyclic flat $R$-modules such that $(C^n)^H$ is $R$-flat for every $n \geq 0$. (For instance, take the standard resolution $R \to X^{\bullet}(H,R)$.) Then $R \to D^{\bullet}:=(C^{\bullet})^{\otimes Y}$ is a resolution in flat $R$-modules, and is $H^Y$-acylic by Lemma \ref{TheLemma}. 
		Let $\mathbf{Z}(C^{\bullet}):=\oplus_p Z^p((C^{\bullet})^H)$ and $\mathbf{H}(C^{\bullet}):=\oplus_p H^p((C^{\bullet})^H)$. Define analogously $\mathbf{Z}(D^{\bullet})$ and $\mathbf{H}(D^{\bullet})$. The resolutions $C^{\bullet}$ and $D^{\bullet}$ are, respectively, $H$- and $H^Y$-acyclic, so $\mathbf{H}(C^{\bullet})= \oplus_{p \geq 0} H^{p}(H,R)$ and $\mathbf{H}(D^{\bullet})=\oplus_{p \geq 0} H^{p}(H^Y,R)$. 
		
		By assumption, all the graded pieces of $\mathbf{H}(C^{\bullet})$ are $R$-projective, hence the natural surjection $\mathbf{Z}(C^{\bullet})\to\mathbf{H}(C^{\bullet})$ admits a graded section $s:\mathbf{H}(C^{\bullet})\to\mathbf{Z}(C^{\bullet})$.
		Consider now the commutative diagram
		\[
		\begin{tikzcd}
		\mathbf{H}(C^{\bullet})^{\otimes Y} \arrow[d, "\cong"] \arrow[r, "s^{\otimes Y}"] & \mathbf{Z}(C^{\bullet})^{\otimes Y} \arrow[d] \\
		\mathbf{H}(D^{\bullet}) \arrow[r, dashed]                                         & \mathbf{Z}(D^{\bullet})                      
		\end{tikzcd}
		\]
		where the vertical maps are cross-products. The left one is an isomorphism by K\"unneth's theorem \cite[Theorem 3.6.3]{Weibel}, which applies since the $R$-module $\mathbf{H}(C^{\bullet})$ is $R$-projective and hence $R$-flat. The bottom composition is now a graded section of $\mathbf{Z}(D^{\bullet}) \to \mathbf{H}(D^{\bullet})$ that is equivariant under permutations of $Y$, and so in particular under the action of $G$.
		
		So the resolution $R \to D^{\bullet}$ is a resolution in $H^Y$-acyclics with a $G$-action, and the surjection $\mathbf{Z}(D^{\bullet}) \to \mathbf{H}(D^{\bullet})$ admits a $G$-equivariant graded section. Applying the exact functor $\otimes_RM$ to the resolution $R \to D^{\bullet}$, we obtain a resolution of $M \to D^{\bullet} \otimes_RM$ with the same properties, and we conclude by Lemma \ref{LemSD}.
		
		The naturality holds since, for fixed $H$ and $R$, both $C^{\bullet}$ and $s$ can be fixed, and the rest of the proof is natural in all the data.
	\end{proof}
	
	\bibliographystyle{alpha}
	\bibliography{homspaces}
	
\end{document}